\documentclass[12pt]{smfart}
\usepackage[T1]{fontenc}
\usepackage[all]{xy}
\usepackage[francais]{babel}
\usepackage{verbatim,amscd}
\usepackage{amssymb}
\usepackage{enumitem}
\usepackage[colorlinks,linkcolor=blue,citecolor=blue,urlcolor=red]{hyperref}
\makeindex

\newcommand{\cf}{\emph{cf. }}
\newcommand{\resp}{\emph{resp. }}
\newcommand{\loccit}{\emph{loc. cit. }}
\newcommand{\ie}{\emph{i.e. }}

\newcommand{\Li}{\operatorname{Li}}
\newcommand{\ord}{\operatorname{ord}}
\newcommand{\Ab}{\operatorname{Ab}}
\newcommand{\Spec}{\operatorname{Spec}}
\newcommand{\Dir}{\operatorname{Dir}}
\newcommand{\Gr}{\operatorname{Gr}}
\newcommand{\Div}{\operatorname{Div}}
\newcommand{\Corr}{\operatorname{Corr}}
\newcommand{\Aut}{\operatorname{Aut}}
\newcommand{\End}{\operatorname{End}}
\newcommand{\Hom}{\operatorname{Hom}}
\newcommand{\uHom}{\operatorname{\underline{Hom}}}
\newcommand{\Fonct}{\operatorname{Fonct}}
\newcommand{\Ind}{\operatorname{Ind}}
\newcommand{\Cl}{\operatorname{Cl}}
\newcommand{\Pic}{\operatorname{Pic}}
\newcommand{\NS}{\operatorname{NS}}
\newcommand{\Alb}{\operatorname{Alb}}
\newcommand{\car}{\operatorname{car}}
\newcommand{\Ker}{\operatorname{Ker}}
\newcommand{\Coker}{\operatorname{Coker}}
\newcommand{\IM}{\operatorname{Im}}
\newcommand{\Sch}{\operatorname{Sch}}
\newcommand{\Tr}{\operatorname{Tr}}
\newcommand{\Trd}{\operatorname{Trd}}
\newcommand{\cl}{\operatorname{cl}}
\newcommand{\eff}{{\operatorname{eff}}}
\newcommand{\op}{{\operatorname{op}}}
\newcommand{\et}{{\operatorname{\acute{e}t}}}
\newcommand{\Zar}{{\operatorname{Zar}}}
\newcommand{\dR}{{\operatorname{dR}}}
\newcommand{\cris}{{\operatorname{cris}}}
\renewcommand{\Vec}{\operatorname{\mathbf{Vec}}}
\newcommand{\Prim}{\operatorname{Prim}}
\newcommand{\rg}{\operatorname{rg}}
\newcommand{\Gal}{\operatorname{Gal}}
\newcommand{\Rep}{\operatorname{Rep}}
\newcommand{\resprod}{\operatornamewithlimits{\coprod\hspace{-0.55cm}\prod}}
\newcommand{\modmult}[1]{\quad (\bmod^*\ #1\, )}
\newcommand{\leg}[2]{\displaystyle\left(\frac{#1}{#2}\right)}
\newcommand{\oo}{\operatornamewithlimits{\otimes}}
\newcommand{\RHom}{\operatorname{R\underline{Hom}}}
\newcommand{\rat}{{\operatorname{rat}}}
\newcommand{\alg}{{\operatorname{alg}}}
\newcommand{\tnil}{{\operatorname{tnil}}}
\newcommand{\num}{{\operatorname{num}}}

\newcommand{\plim}{\varprojlim}
\newcommand{\colim}{\varinjlim}

\newcommand{\tto}{\dashrightarrow}
\newcommand{\by}{\xrightarrow}
\newcommand{\yb}{\xleftarrow}
\newcommand{\iso}{\by{\sim}}
\newcommand{\osi}{\yb{\sim}}
\newcommand{\surj}{\rightarrow\!\!\!\!\!\!\!\rightarrow}

\newcommand{\inj}{\hookrightarrow}

\newcommand{\un}{\mathbf{1}}
\newcommand{\A}{\mathbf{A}}
\newcommand{\C}{\mathbf{C}}
\newcommand{\F}{\mathbf{F}}
\newcommand{\G}{\mathbf{G}}
\renewcommand{\L}{\mathbb{L}}
\newcommand{\N}{\mathbf{N}}
\renewcommand{\P}{\mathbf{P}}
\newcommand{\Q}{\mathbf{Q}}
\newcommand{\R}{\mathbf{R}}
\newcommand{\T}{\mathbb{T}}
\newcommand{\V}{\mathbf{V}}
\newcommand{\Z}{\mathbf{Z}}

\newcommand{\bA}{\mathbb{A}}
\newcommand{\bH}{\mathbb{H}}
\newcommand{\bI}{\mathbb{I}}

\newcommand{\fa}{\mathfrak{a}}
\newcommand{\fA}{\mathfrak{A}}
\newcommand{\ff}{\mathfrak{f}}
\newcommand{\fm}{\mathfrak{m}}
\newcommand{\fP}{\mathfrak{P}}
\newcommand{\fp}{\mathfrak{p}}
\newcommand{\fS}{\mathfrak{S}}

\newcommand{\sA}{\mathcal{A}}
\newcommand{\sB}{\mathcal{B}}
\newcommand{\sC}{\mathcal{C}}
\newcommand{\sD}{\mathcal{D}}
\newcommand{\sE}{\mathcal{E}}
\newcommand{\sL}{\mathcal{L}}
\newcommand{\sM}{\mathcal{M}}
\newcommand{\sO}{\mathcal{O}}
\newcommand{\sP}{\mathcal{P}}
\newcommand{\sS}{\mathcal{S}}
\newcommand{\sT}{\mathcal{T}}
\newcommand{\sU}{\mathcal{U}}
\newcommand{\sX}{\mathcal{X}}
\newcommand{\sZ}{\mathcal{Z}}

\renewcommand{\epsilon}{\varepsilon}
\renewcommand{\phi}{\varphi}

\newenvironment{thlist}{\begin{list}{\rm{(\roman{enumi})}}%
{\usecounter{enumi}}}%
{\end{list}}

\swapnumbers
\newtheorem{thm}{Th\'eor\`eme}[subsection]
\newtheorem{lemme}[thm]{Lemme}
\newtheorem{prop}[thm]{Proposition}
\newtheorem{conj}[thm]{Conjecture}
\newtheorem{cor}[thm]{Corollaire}
\theoremstyle{definition}
\newtheorem{defn}[thm]{D\'efinition}
\theoremstyle{remark}
\newtheorem{rque}[thm]{Remarque}
\newtheorem{rques}[thm]{Remarques}
\newtheorem{ex}[thm]{Exemple}
\newtheorem{exs}[thm]{Exemples}
\newtheorem{meg}[thm]{Mise en garde}
\newtheorem{exo}[thm]{Exercice}

\numberwithin{equation}{section}

\begin{document}

\title{Fonctions z\^eta et $L$ de vari\'et\'es et de motifs}
\author{Bruno Kahn}
\address{IMJ-PRG\\Case 247\\4 place Jussieu\\75252 Paris Cedex 05}
\email{bruno.kahn@imj-prg.fr}
\date{3 septembre 2016}
\maketitle

%\hfill Notes de cours, version presque finale du \today

%\hfill Jussieu, mars-avril 2013

\tableofcontents

\section*{Introduction}

Gauss a d\'ecrit la th\'eorie des nombres comme la reine des math\'ematiques. Et en effet, la quantit\'e de math\'ematiques invent\'ees pour des raisons arithm\'etiques est étonnante. Citons:

\begin{itemize}
\item Une bonne partie de l'analyse complexe (Cauchy, Riemann, Weierstra\ss, Hadamard, Hardy-Littlewood\dots)
\item La th\'eorie des diviseurs (Kummer, Dedekind) puis des id\'eaux (E. Noether).
\item La th\'eorie des surfaces de Riemann (Riemann\dots)
\item La version non analytique du th\'eor\`eme de Riemann-Roch\index{Théorème!de Riemann-Roch} pour les courbes (F.K. Schmidt).
\item La refondation de la g\'eom\'etrie alg\'ebrique italienne sur la base de l'alg\`ebre commutative (Weil).\footnote{van der Waerden et Zariski ont aussi participé à ce mouvement, pour des raisons différentes.}
\item La th\'eorie des vari\'et\'es ab\'eliennes (Weil).
\item Une partie de la th\'eorie des repr\'esentations lin\'eaires des groupes finis (E. Artin, Brauer).
\item Une bonne partie de l'alg\`ebre homologique: cohomologie des grou\-pes, th\'eorie des faisceaux sur un site quelconque (Cartan, Eilenberg, Serre, Tate, Grothendieck\dots)
\item Pour une part, la th\'eorie des sch\'emas (Grothendieck).
\item Le d\'eveloppement de la cohomologie \'etale (M. Artin, Grothendieck, Verdier, Deligne\dots), puis cristalline (Grothendieck, Bloch, Berthelot, Ogus, Deligne, Illusie\dots)
\item Pour une part, la notion de cat\'egorie d\'eriv\'ee (Grothendieck, Verdier).
\item Le formalisme des six op\'erations (Grothendieck).
\item La th\'eorie de la monodromie dans le monde des sch\'emas (Grothendieck, Serre, Deligne, Katz\dots)
\end{itemize}

Et bien s\^ur, la th\'eorie des motifs!

Les fonctions zêta et les fonctions $L$ sont le point de rencontre des théories citées ci-dessus: elles couronnent la reine des mathématiques. Il est remarquable qu'une fonction de définition aussi élémentaire que
\[\zeta(s)= \sum_{n=1}^\infty \frac{1}{n^s}\]
ait joué un rôle aussi profond, connaissant de vastes généralisations qui ont modelé l'évolution de la théorie des nombres jusqu'à aujourd'hui, et reste l'objet d'une conjecture dont les analogues pour les variétés sur les corps finis ont été fameusement démontrées par Weil et Deligne,  mais que personne ne sait approcher dans le cas originel de Riemann.

Le texte qui suit est la rédaction d'un cours de M2 donné à Jussieu en mars-avril 2013. J'ai essayé d'y exposer une partie des résultats connus sur les fonctions zêta et $L$, mais aussi la problématique complexe qui les entoure, de manière ontogénétique: l'ontogenèse décrit le développement progressif d'un organisme depuis sa conception jusqu'à sa forme mûre. Dans ce but, j'ai parsemé le texte de citations et de commentaires qui, j'espère, offriront au lecteur ou à la lectrice une petite fenêtre sur l'histoire des idées dans ce domaine.

Après avoir rappelé les résultats (et hypothèse) classiques sur la fonction zêta de Riemann au chapitre \ref{s1}, j'introduis les fonctions zêta des $\Z$-schémas de type fini au chapitre suivant, qui est consacré pour l'essentiel à la démonstration de l'hypothèse de Riemann pour les courbes sur un corps fini. Je me suis efforcé de reproduire fidèlement la démonstration originelle de Weil \cite{weil2}. 

Le chapitre \ref{s3} est consacré aux conjectures de Weil. Elles y sont toutes démontrées sauf la plus difficile: l'``hypothèse de Riemann''. J'y donne aussi un aperçu de la preuve $p$-adique donnée par Dwork de la rationalité des fonctions zêta de variétés sur un corps fini (avant le développement des méthodes cohomologiques de Grothendieck!).

Le chapitre \ref{s4} revient à des mathématiques plus élémentaires, en introduisant les fonctions $L$ de Dirichlet, de Hecke et d'Artin. J'y donne la démonstration du théorème de Dirichlet\index{Théorème!de Dirichlet} sur la progression arithmétique, par la méthode de Landau exposée par Serre dans \cite{serre-arith}; il serait cependant dommage, comme me l'a souligné Pierre Charollois, d'oublier la méthode originelle de Dirichlet qui donnait des renseignements supplémentaires anticipant la formule analytique du nombre de classes (voir remarque \ref{r4.2} 2)). J'introduis ensuite les deux généralisations des fonctions $L$ de Dirichlet: celle de Hecke et celle d'Artin. J'énonce sans démonstration le théorème principal de Hecke: prolongement analytique et équation fonctionnelle (théorème \ref{t:hecke}), puis explique comment Artin et Brauer en déduisent les mêmes résultats pour les fonctions $L$ non abéliennes (théorème \ref{t4.2}).

Le chapitre suivant est la pièce de résistance de ce texte. Il commence par introduire l'idée approximative des fonctions $L$ à la Hasse-Weil, et se termine en décrivant leur définition précise par Serre \cite{serre-dpp}. Entretemps, je décris les apports fondamentaux de Grothendieck et Deligne: rationalité et équation fonctionnelle des fonctions $L$ de faisceaux $l$-adiques en caractéristique $p$ avec des démonstrations essentiellement complètes, théorie des poids et théorèmes de Deligne sur l'hypothèse de Riemann (dernière conjecture de Weil), cette fois sans démonstrations. Notamment, on trouvera au \S \ref{s5.3.4} une exposition de l'équation fonctionnelle des fonctions $L$ développée par Grothendieck dans  \cite[lettre du 30-9-64]{corresp}, et au théorème \ref{t5.5} un énoncé précis, et des éléments assez complets de démonstration, du théorème de Grothendieck et Deligne sur la rationalité et l'équation fonctionnelle des fonctions $L$ de Hasse-Weil en caractéristique $>0$, confirmant une conjecture de Serre (conjecture \ref{c5.3}) dans ce cas.

Le dernier chapitre est consacré aux motifs et à leurs fonctions zêta. Je me suis limité à un  cas élémentaire: celui des motifs purs de Grothendieck associés aux variétés projectives lisses sur un corps fini. On peut aller beaucoup plus loin en utilisant les catégories triangulées de motifs introduites par Voevodsky et développées par Ivorra, Ayoub et Cisinski-Déglise, mais cela dépasserait le cadre de ce cours (voir \cite{zetaL}). J'explique tout de même comment ce point de vue clarifie considérablement l'usage d'une cohomologie de Weil pour démontrer la rationalité et l'équation fonctionnelle, et ne résiste pas au plaisir d'appliquer cette théorie pour démontrer un théorème un peu oublié de Weil: la conjecture d'Artin pour les fonctions $L$ non abéliennes en caractéristique positive (\S \ref{cor-loco}).

Enfin, deux appendices donnent des compléments d'algèbre catégorique. J'ai aussi agrémenté le texte d'exercices, sans rien essayer de systématique.

Mise à part la théorie de Hecke, le versant automorphe de l'histoire n'est essentiellement pas abordé; je me suis contenté de brèves allusions ici et là.

Il reste une question à éclaircir: quelle est la différence entre une fonction zêta et une fonction $L$? Moralement, une fonction $L$ est une fonction zêta ``à coefficients'' (dans un faisceau, dans une représentation\dots) et on retrouve une fonction zêta en prenant des coefficients triviaux (cf. \S \ref{s5.2}). Mais une fonction zêta de Dedekind est aussi une fonction $L$ d'Artin\dots\ Et quelle terminologie et notation adopter dans le cas des motifs? Je suis resté prudemment vague sur cette question; il serait sans doute plus simple d'unifier la terminologie en laissant tomber l'une des deux notations, mais c'est évidemment délicat pour des raisons de tradition.  %La situation se complique quand on considère les produits eulériens associés par Hasse et Weil à une courbe sur un corps de nombres

J'ai profité des expositions antérieures de la théorie, auxquelles j'ai largement emprunté: elles seraient trop nombreuses pour les citer toutes, je renvoie pour cela à la bibliographie.  Je remercie aussi le public du cours, et tout particuli\`erement Matthieu Rambaud, de m'avoir signal\'e quantit\'e de coquilles. Je remercie enfin Joseph Ayoub,  Pierre Charollois, Luc Illusie, Amnon Neeman, Ram Murty et Jean-Pierre Serre pour des commentaires pertinents sur ce texte.

\section{La fonction z\^eta de Riemann}\label{s1}\index{Fonction zêta!de Riemann}
\enlargethispage*{20pt}

C'est
\[\zeta(s)=\sum_{n=1}^\infty \frac{1}{n^s}.\]

Elle a \'et\'e \'etudi\'ee par Euler, puis par Riemann.

\subsection{Un peu d'histoire}

\subsubsection{Euler} 
\begin{itemize}
\item $\zeta(2) =\displaystyle \frac{\pi^2}{6}$ (probl\`eme de B\^ale, 1735).
\item \emph{Produit eul\'erien}: $\zeta(s) = \displaystyle\prod_{p\text{ premier}} (1-p^{-s})^{-1}$ (th\`ese, 1737).
\end{itemize}

Formule plus g\'en\'erale: 
\[\zeta(2n)=(-1)^{n+1}\frac{B_{2n}(2\pi)^{2n}}{2(2n)!}\]
$B_n$ les nombres de Bernoulli-Seki\footnote{Seki Takakazu ou Seki K\^owa est un math\'ematicien japonais contemporain de Bernoulli, qui a d\'ecouvert les ``nombres de Bernoulli'' \`a la m\^eme p\'eriode, soit vers la fin du 17\`eme si\`ecle.}.

\subsubsection{Riemann (\protect{\cite{riemann}}, 1859)} \label{s:riemann} \'Etudie $\zeta(s)$ pour $s\in\C$.

\begin{itemize}
\item Elle converge absolument pour $\Re(s)>1$.
\item Prolongement m\'eromorphe \`a $\C$, p\^ole simple en $s=1$, r\'esidu $=1$.
\item \'Equation fonctionnelle (conjectur\'ee essentiellement par Euler en 1749): si $\Lambda(s)=\pi^{-s/2}\Gamma(s/2)\zeta(s)$, alors \label{zetacompl}\index{Equation fonctionnelle@\'Equation fonctionnelle}
\[\Lambda(s) = \Lambda(1-s).\]
\item Hypoth\`ese de Riemann \index{Hypothèse de Riemann} sur la distribution des z\'eros de $\zeta(s)$.
\end{itemize}

L'\'equation fonctionnelle donne une formule pour $\zeta(s)$ aux entiers n\'egatifs (devin\'ee par Euler):
\[\zeta(-n)= -\frac{B_{n+1}}{n+1} \quad (n>0),\quad \zeta(0) = -\frac{1}{2}.\]

La motivation de Riemann: le th\'eor\`eme des nombres premiers (distribution des nombres premiers). Anticip\'e par \v Ceby\v sev (1848 -- 1850) et d\'emontr\'e par  Hadamard et de la Vall\'ee Poussin (1896).

Une large part de l'exposition qui suit est empruntée à Serre \cite{serre-arith}.

\subsection{Convergence absolue}

\begin{prop}\phantomsection\label{p5} $\zeta(s)$ converge absolument et uniform\'ement sur tout compact du domaine $\Re(s)>1$ o\`u elle d\'efinit une fonction holomorphe, et  diverge pour $s=1$. 
\end{prop}

\begin{proof} Pour la convergence et la divergence, si $s=\sigma+it$ avec $\sigma,t\in \R$ on a $|n^{-s}|=n^{-\sigma}$, donc on peut supposer $s\in \R$. On utilise alors le

\begin{lemme}[test int\'egral] \phantomsection Soit $f:\R^+\to \R^+$ une fonction d\'ecroissante. Alors $\sum_{n=1}^\infty f(n)<\infty$ $\iff$ $\int_1^\infty f(t)dt<\infty$.
\end{lemme}

\begin{proof} On a les in\'egalit\'es
\[f(n)=\int_n^{n+1} f(n)dt \ge \int_n^{n+1}f(t)dt \ge \int_n^{n+1} f(n+1) dt=f(n+1)\]
soit
\[\int_{n-1}^{n}f(t)dt\ge f(n) \ge \int_n^{n+1}f(t)dt\]
et on somme.
\end{proof}

Cette d\'emonstration donne aussi l'in\'egalit\'e
\[\left |\sum_{n=N}^\infty n^{-s}\right |\le \int_{N-1}^\infty t^{-\Re(s)}dt= \frac{(N-1)^{1-\Re(s)}}{\Re(s)-1}\]
d'o\`u la convergence uniforme. Pour l'holomorphie, on applique aux som\-mes partielles le

\begin{lemme}\phantomsection\label{l2.7.3} Soit $U$ un ouvert de $\C$, et soit $f_{n}:U\to \C$ une suite de fonctions holomorphes, convergeant
uniform\'ement sur tout compact de $U$ vers $f$. Alors $f$ est holomorphe dans
$U$, et les d\'eriv\'ees $f_n'$ des $f_{n}$ convergent uniform\'ement sur tout
compact vers la d\'eriv\'ee $f'$ de $f$.
\end{lemme}  

\begin{proof} Soient $D$ un disque ferm\'e contenu dans $U$, de bord $\partial D$,
orient\'e trigo\-no\-m\'e\-tri\-quement, et $z_0\in \overset{\circ}{D}$. On
applique aux $f_n$ la formule de Cauchy:  
\[f_n(z_0) =  \frac{1}{ 
2i\pi}\int_{\partial D}\frac{f_n(z)  }{ z-z_0}dz.\]
 
Par passage \`a la limite, on obtient 
\[f(z_0) = \frac{1}{2i\pi}\int_{\partial D}\frac{f(z) }{ z-z_0}dz\]
ce qui montre que $f$ est holomorphe dans $U$, d'o\`u la premi\`ere partie du
lemme. La formule 
\[f'(z_0) = \frac{1  }{ 2i\pi}\int_{\partial D}\frac{f(z) }{ (z-z_0)^2}dz \]
donne la seconde partie de la m\^eme mani\`ere.
\end{proof}
\end{proof}

\subsection{Produit eul\'erien}

\begin{prop}[Euler]\phantomsection\label{p3} On a la factorisation
\[\zeta(s) = \prod_p \frac{1}{1-p^{-s}}\]
o\`u le produit, sur l'ensemble des nombres premiers, converge absolument pour $\Re(s)>1$.
\end{prop}

Pour la d\'emonstration, disons qu'une fonction $a:\N-\{0\}\to \C$ est \emph{multiplicative} (\emph{resp. } \emph{compl\`etement multiplicative}) si elle v\'erifie l'identit\'e $a(mn)=a(m)a(n)$ pour tous $m,n$ premiers entre eux (\emph{resp. } pour tous $m,n$). La proposition r\'esulte du lemme suivant, appliqu\'e \`a $a(n)=n^{-s}$:

\begin{lemme}\phantomsection\label{t2.8.1} Soit $a$ une fonction
multiplicative. Les conditions suivantes sont \'equivalentes:
\begin{thlist}
\item $\displaystyle\sum_{n=1}^\infty|a(n)|<+\infty$;
\item $\displaystyle\prod_p(1+|a(p)|+\cdots+|a(p^m)|+\cdots)<+\infty$.
\end{thlist}
Si ces conditions sont v\'erifi\'ees, on a
\[\sum_{n=1}^\infty a(n)=\prod_p(1+a(p)+\cdots+a(p^m)+\cdots)\]
et si $a$ est compl\`etement multiplicative:
\[\sum_{n=1}^\infty a(n)=\prod_p\frac{1}{1-a(p)}.\]
\end{lemme}

\begin{proof} Supposons (i) v\'erifi\'e. En particulier, $\sum a(p^m)$ converge
abso\-lument pour tout $p$. Pour $x\in\N$, on a donc 
\[\sum_{n\in
E(x)}a(n)=\prod_{p<x}\sum_m a(p^m)\]
o\`u 
\[E(x)=\{n\in\N-\{0\}\mid
\text{tous les facteurs premiers de $n$ sont $<x$}\}.\] 

D'o\`u
\[\left|\sum_{n= 1}^\infty a(n)-\prod_{p<x}\sum_m
a(p^m)\right|=\left|\sum_{n\notin E(x)}a(n)\right|\leq\sum_{n\geq x}|a(n)|\]
ce qui montre que le produit infini converge vers
$\displaystyle\sum_{n=1}^\infty a(n)$. En appliquant ceci \`a $|a(n)|$, on voit
que la convergence est absolue, donc (ii) est vrai. 

Supposons (ii) v\'erifi\'e. Alors $\displaystyle\sum_{n<x}|a(n)|\leq\sum_{n\in
E(x)}|a(n)|=\prod_{p<x}(1+|a(p)|+\cdots+|a(p^m)|+\cdots)$, donc (i) est vrai.
\end{proof}

\begin{cor}\phantomsection $\zeta(s)$ ne s'annule pas pour $\Re(s)>1$.\qed
\end{cor}

\subsection{S\'eries formelles de Dirichlet}\label{section2.3} Cette notion est tr\`es pratique pour faire des calculs formels sans (pour commencer) se soucier de convergence:

\begin{defn}\phantomsection\label{d2.3.1}On appelle \emph{s\'erie formelle de Dirichlet} (\`a
coefficients complexes) une
expression
\[f=\sum_{n\geq 1}\frac{a_n}{ n^s}\]
o\`u $n$ d\'ecrit l'ensemble des entiers positifs, et les $a_n$ sont des nombres
complexes. Si $\displaystyle f=\sum_{n\geq 1}\frac{a_n}{ n^s}$ et
$\displaystyle g=\sum_{n\geq 1}\frac{b_n}{ n^s}$ sont deux s\'eries formelles de
Dirichlet, on d\'efinit leur somme et leur produit par
$$f+g=\sum_{n\geq 1}\frac{a_n+b_n}{ n^s};$$
$$fg=\sum_{n\geq 1}\frac{c_n}{ n^s}$$
avec $c_n=\sum_{pq=n}a_pb_q$.\smallskip

Les s\'eries formelles de Dirichlet forment un anneau commutatif unitaire:
l'alg\`ebre des s\'eries formelles de Dirichlet \`a coefficients complexes,
not\'ee $\Dir(\C)$.
\end{defn} 

\begin{rque}\phantomsection L'alg\`ebre $\Dir(\C)$ n'est autre
que l'\emph{alg\`ebre large} du mono\"{\i}de multiplicatif $\N-\{0\}$
\cite{bourbaki1}. En \'ecrivant tout entier comme produit de
facteurs premiers, on voit donc qu'elle est isomorphe \`a l'alg\`ebre des
s\'eries formelles \`a coefficients complexes en une infinit\'e d\'enombrable de
variables (cf. \cite[ex. 2 de l'appendice]{bbki-lie}).
\end{rque}

\begin{defn}\phantomsection\label{d2.3.2} Soit $\displaystyle f=\sum_{n\geq 1}\frac{a_n}{n^s}$ une
s\'erie formelle de Dirichlet. Si $f\neq 0$, on appelle \emph{ordre} de $f$, et
on note $\omega(f)$, le plus petit entier $n$ tel que $a_n\neq 0$. Si $f=0$, on
pose $\omega(f)=+\infty$.
\end{defn}

\begin{prop}\phantomsection{}\label{p2.3.1} Les sous-ensembles $\{f\mid\omega(f)\geq N\}$ de
$\Dir(\C)$ sont des id\'eaux de $\Dir(\C)$. Ils
d\'efinissent une topologie sur $\Dir(\C)$ pour laquelle
$\Dir(\C)$ est un anneau topologique complet.\qed
\end{prop}

\begin{cor}\phantomsection\label{c2.3.1} a) Une suite $(f_n)_{n\in \N}$ est sommable dans
$\Dir(\C)$ si et seulement si
$\displaystyle\lim_{n\rightarrow\infty}\omega(f_n)=+\infty$.\\
b) Une suite $(1+f_n)_{n\in \N}$, o\`u $\omega(f_n)>1$ pour tout $n$, est
multipliable dans $\Dir(\C)$ si et seulement si
$\displaystyle\lim_{n\rightarrow\infty}\omega(f_n)=+\infty$.\qed
\end{cor}

Le lemme \ref{t2.8.1} donne:

\begin{prop}\phantomsection\label{p1} Soit $a:\N-\{0\}\to \C$ une fonction compl\`etement multiplicative, et soit $\alpha\in \R$. Dans l'identit\'e formelle
\[\sum_{n\ge 1} \frac{a(n)}{n^s} = \prod_p \left(1- \frac{a(p)}{p^{s}}\right)^{-1}\]
le membre de gauche converge absolument pour $\Re(s)>\alpha$ si et seulement si le membre de droite converge absolument pour $\Re(s)>\alpha$.
\end{prop}

\subsection{Prolongement \`a $\Re(s)>0$; p\^ole et r\'esidu en $s=1$}

\begin{prop}\phantomsection\label{p2} La fonction $\zeta(s)$ se prolonge m\'e\-ro\-mor\-phi\-que\-ment au domaine $\Re(s)>0$, avec un p\^ole simple en $s=1$ de r\'esidu $1$.
\end{prop}

Pour ceci, on a besoin du

\begin{lemme}[lemme d'Abel]\phantomsection\label{l2.7.1} Soient $(a_{n})$ et $(b_{n})$ deux
suites de nombres complexes. Posons $A_{m,p}  =
\sum^{p}_{n=m}a_{n}$ et  $S_{m,m'} =  \sum^{m'}_{n=m}a_{n}b_{n}$.
Alors on a
\[ S_{m,m'} = \sum^{m'-1}_{n=m}
A_{m,n}(b_{n}-b_{n+1}) + A_{m,m'} b_{ m'}.\]
En particulier, supposons $|A_{m,n}|\leq \varepsilon$ pour tout
$n\leq m'$, les $b_n$ r\'eels et la suite $(b_n)$ d\'ecroissante. Alors:
\[| S_{m,m'}| \leq \varepsilon(b_m-b_{m'}).\qed\]
\end{lemme}

\begin{proof}[D\'emonstration de la proposition \protect{\ref{p2}}] Consid\'erons 
\[ \zeta_{2}(s) =\sum_{n=1}^\infty\frac{(-1)^{n+1} }{
n^{s}}.\] 

On applique le lemme \ref{l2.7.1} avec $a_n=(-1)^{n+1}$, $b_n=n^{-s}$: on conclut que $ \zeta_{2} $ converge pour
$\Re(s)  >  0$. De plus,
$$
\zeta(s) - \zeta_{2}(s)= 2^{1-s}\zeta(s)   \Rightarrow
 \zeta_{2}(s) = (1-2^{1-s})\zeta(s).
$$

Par cons\'equent la fonction $(s-1)\zeta(s)$ se prolonge m\'eromorphiquement dans
le demi plan $\Re(s)  >  0$. En $s = 1$, on a $ \zeta_{2}(s) = \log 2$. La r\`egle de l'H\^opital donne alors
$$\lim_{s\rightarrow1}(s-1)\zeta(s) = 1.$$

Plus g\'en\'eralement, consid\'erons pour $r \in \N- \{0,1\}$ 
$$\zeta_r(s) = \frac{1}{ 1^s}+ \frac{1}{ 2^s}+\cdots+\frac{1}{ (r-1)^s}
 - \frac{r-1}{ r^s}+\frac{1}{ (r+1)^s}+\cdots$$

On voit tout de suite que les sommes partielles des coefficients
de $\zeta_r$ sont born\'ees par $r-1$. Donc $ \zeta_r $ est
analytique pour $\Re(s)>0$. De plus, on a comme ci-dessus:
\begin{equation}\label{eq2.2}
\left(1- \frac{1}{ r^{s-1}}\right)\zeta(s)= \zeta_r(s)
\end{equation}

L'\'equation \eqref{eq2.2} pour $r=2$ montre qu'un p\^ole $s\neq 1$ de $\zeta(s)$
doit v\'erifier $2^{s-1}= 1$,  \emph{i.e.} $\displaystyle s=\frac{2i\pi k }{ \log
2}+ 1$  pour un entier $k\neq 0$. De m\^eme, en utilisant l'\'equation 
\eqref{eq2.2} pour $r=3$, on voit que $\displaystyle s=\frac{2i\pi \ell }{ \log 3}+
1$ pour $\ell$ entier $\neq 0$. Mais alors on a $3^{k}=2^{\ell}$ , ce qui est
impossible.
\end{proof}

\subsection{\'Equation fonctionnelle}\index{Equation fonctionnelle@\'Equation fonctionnelle} Rappelons la fonction gamma (d\'e\-fi\-nie par Euler en 1730):
\[\Gamma(s) = \int_0^\infty t^s e^{-t} \frac{dt}{t}.\]

Cette int\'egrale converge pour $\Re(s)>0$; on a $\Gamma(1)=1$ et l'\'equation fonctionnelle
\begin{equation}
\Gamma(s+1)=s\Gamma(s).
\end{equation}

Il en r\'esulte que  $\Gamma(n)=(n-1)!$ pour $n\in \N-\{0\}$, et que  $\Gamma(s)$ se prolonge m\'eromorphiquement \`a tout le plan complexe avec  un p\^ole simple en $s=-n$ pour tout $n\in \N$, de r\'esidu $(-1)^n/n!$. Enfin, on a la \emph{formule des compl\'ements}:
\begin{equation}\label{eq:compl}
\Gamma(s)\Gamma(1-s) = \frac{\pi}{\sin(\pi s)}
\end{equation}
et la \emph{formule de distribution}:
\begin{equation}\label{eq:distr}
\prod_{k=0}^{m-1} \Gamma(s+\frac{k}{m})= (2\pi)^{\frac{m-1}{2}} m^{\frac{1}{2}-ms}\Gamma(ms).
\end{equation}

Posons
\[\xi(s)=\frac{s(s-1)}{2}\pi^{-s/2}\Gamma(s/2)\zeta(s).\]

\begin{thm}[Riemann]\phantomsection\label{triemann} a) La fonction $\xi(s)$ se prolonge ho\-lo\-mor\-phi\-que\-ment \`a tout le plan complexe, avec l'\'equation fonctionnelle
\[\xi(1-s) = \xi(s).\]
b) La fonction $\zeta(s)$ se prolonge m\'eromorphiquement \`a tout le plan complexe, avec l'\'equation fonctionnelle
\[\zeta(1-s)=\frac{2}{(2\pi)^s}\cos(\frac{\pi s}{2}) \Gamma(s)\zeta(s).\]
Elle a un z\'ero simple en $s=-n$ pour tout entier $n> 0$ pair (ces z\'eros sont appel\'es ``triviaux''), et pas d'autres z\'eros en dehors de la ``bande critique'' $0\le \Re(s)\le 1$.
\end{thm}

(La fonction $\xi(s)$ diff\`ere de la fonction $\Lambda(s)$ du \S \ref{s:riemann} par un facteur $\frac{s(s-1)}{2}$:  leurs \'equations fonctionnelles sont \'equivalentes et $\xi$ a l'avantage d'\^etre enti\`ere.)

\begin{proof} Les formules \eqref{eq:compl} et \eqref{eq:distr} r\'eduisent b) \`a a). Je suivrai la preuve donn\'ee dans Ellison \cite[preuve du th. 5.2]{ellison}:

Pour tout $s$ v\'erifiant $\Re(s)>0$ et pour tout entier $n>0$, on a
\[\Gamma\left(\frac{s}{2}\right)(\pi n^2)^{-\frac{s}{2}} =\int_0^\infty t^{\frac{s}{2}-1}e^{-\pi n^2 t}dt.  \]

Supposant $\Re(s)>1$, on obtient en sommant sur $n>0$:
\[\Gamma\left(\frac{s}{2}\right)\pi^{-\frac{s}{2}}\zeta(s) =\sum_{n=1}^\infty\int_0^\infty t^{\frac{s}{2}-1}e^{-\pi n^2 t}dt.  \]

Comme la s\'erie de droite est convergente pour $s$ r\'eel $>1$, on peut intervertir somme et int\'egrale, ce qui donne:
\[\Gamma\left(\frac{s}{2}\right)\pi^{-\frac{s}{2}}\zeta(s) =\int_0^\infty \sum_{n=1}^\infty t^{\frac{s}{2}-1}e^{-\pi n^2 t}dt= \int_0^\infty  t^{\frac{s}{2}-1}\omega(t)dt \]
avec $\omega(t)= \sum_{n=1}^\infty e^{-\pi n^2 t}$. En \'ecrivant $\int_0^\infty=\int_0^1+\int_1^\infty$ dans le membre de droite, il reste \`a d\'emontrer la formule
\begin{equation}\label{eq:omega}
1+2\omega(t) =\frac{1}{\sqrt{t}}\left(1+2\omega\left(\frac{1}{t}\right)\right) 
\end{equation}
qui impliquera
\[\xi(s) = \frac{1}{2} +\frac{s(s-1)}{2} \int_1^\infty \left(t^{\frac{s}{2}}+t^{\frac{1-s}{2}}\right)\omega(t)\frac{dt}{t}.\]

(Comme $\omega(t)=O(e^{-\pi t})$, le membre de droite de cette identit\'e s'\'etend \`a tout le plan complexe, et est invariant par $s\mapsto 1-s$.)

L'\'equation fonctionnelle \eqref{eq:omega} est le cas $z=0$ de l'identit\'e
\[\sum_{n=-\infty}^{+\infty} e^{-\pi(n+z)^2t} =\frac{1}{\sqrt{t}}\sum_{n=-\infty}^{+\infty} e^{-\pi n^2/t+2i\pi nz} \]
o\`u le membre de droite est le d\'eveloppement en s\'erie de Fourier de la fonction de gauche, p\'eriodique de p\'eriode $1$ en $z$ (\emph{cf.} \cite[preuve du th. 5.1]{ellison}).
\end{proof}

\begin{rque}\phantomsection Il y a une d\'emonstration plus \'el\'ementaire du prolongement analytique de $\zeta(s)$, dans l'esprit de la proposition \ref{p2}: 
\[(1-2^{1-s})\zeta(s)=\sum_{n=1}^\infty \frac{\Gamma(s+n)}{\Gamma(s)n!} \frac{\zeta(n+s)}{2^{n+s}} \]
\cf \cite[ch. 2, ex. 2.3]{ellison}. Malheureusement, elle ne semble pas s'\'etendre aux fonctions z\^eta de Dedekind\index{Fonction zêta!de Dedekind}\dots
\end{rque}

\subsection{L'hypoth\`ese de Riemann}\index{Hypothèse de Riemann}

Ce qui pr\'ec\`ede permet de formuler:

\begin{conj}[Hypoth\`ese de Riemann]\phantomsection Les z\'eros non triviaux de $\zeta(s)$ sont situ\'es sur la ``droite critique'' $\Re(s)=\frac{1}{2}$.
\end{conj}

Pourquoi l'hypoth\`ese de Riemann est-elle importante? D'abord \`a cause de la ``formule explicite'' reliant les z\'eros de $\zeta(s)$ \`a la distribution des nombres premiers.

Pour $x\in \R^+$, posons 
\[\pi(x)= |\{p\mid p \text{ premier}, p\le x\}|\] 
et
\[\Pi(x) = \sum_{n=1}^\infty \frac{1}{n} \pi(x^{1/n})\]
de sorte que
\[\pi(x) = \sum_{n=1}^\infty \mu(n)\frac{1}{n} \Pi(x^{1/n})\]
o\`u $\mu$ est la fonction de M\"obius. Posons aussi
\[\Li(z) = \int_0^z \frac{dt}{\log t} \quad \text{(logarithme int\'egral).}\footnote{Il faut un peu de soin pour donner un sens \`a cette int\'egrale divergente: elle d\'efinit une fonction multiforme avec points de branchement en $0$ et $1$.}\]

\begin{thm}[Riemann, 1859]\phantomsection On a
\[\Li(x) -\sum_\rho \Li(x^\rho) - \log(2) + \int_x^\infty \frac{dt}{t(t^2-1)\log t}= \lim_{\epsilon \to 0} \frac{\Pi(x-\epsilon)+\Pi(x+\epsilon)}{2}\]
o\`u la somme de gauche (non absolument convergente) porte sur les z\'eros non triviaux de $\zeta(s)$, ordonn\'es par valeurs absolues croissantes de leurs parties imaginaires.
\end{thm}

On peut en d\'eduire:

\begin{thm}[von Koch \protect{\cite{vonkoch}}, Schoenfeld \protect{\cite{schoenfeld}}]\phantomsection\label{t2} L'hypoth\`ese de Riemann est vraie si et seulement si
\begin{equation}\label{eq:npfort}
\left|\pi(x)-\Li(x)\right| <\frac{1}{8\pi} x^{1/2} \log x \text{ pour } x\ge 2657.
\end{equation}
\end{thm}

Voici une autre formule explicite un peu plus maniable \cite[ch. 5, th. 5.9]{ellison}: soit
\[Z(s) = -\frac{\zeta'(s)}{\zeta(s)} = \sum_{n=1}^\infty \frac{\lambda(n)}{n^s}\]
o\`u 
\[\lambda(n) = \begin{cases}
\log p &\text{si $n$ est de la forme $p^\nu$}\\
0 &\text{sinon}
\end{cases}
\]
(fonction de von Mangoldt)\footnote{Elle est traditionnellement not\'ee $\Lambda(n)$, mais cette notation entre en conflit avec celle de la fonction z\^eta compl\'et\'ee p. \pageref{zetacompl}.}. Posons
\[\psi(x) = \sum_{n\le x} \lambda(n) = \sum_{p^\nu\le x} \log p\]
et 
\[\psi_0(x) = \lim_{\epsilon\to 0} \frac{1}{2}\left(\psi(x+\epsilon) + \psi(x-\epsilon)\right)=
\begin{cases}
\psi(x) &\text{si $x\ne p^\nu$}\\
\psi(x)-\frac{1}{2}\log p &\text{si $x= p^\nu$.}
\end{cases}\]

\begin{thm}\phantomsection Pour tout $x>1$, on a
\[\psi_0(x) = x+Z(0)-\frac{1}{2} \log(1-x^{-2}) -\lim_{T\to \infty} \sum_{|\Im(\rho)|<T} \frac{x^\rho}{\rho}\]
o\`u $\rho$ d\'ecrit les z\'eros non triviaux de $\zeta(s)$.
\end{thm}

Pour une g\'en\'eralisation consid\'erable aux fonctions $L$ de Hecke, voir Weil \cite{weil-expl}.

Plusieurs \'enonc\'es arithm\'etiques ont \'et\'e d\'emontr\'es par des raisonnements du type: l'\'enonc\'e est vrai si l'hypoth\`ese de Riemann est vraie, et d'autre part il est vrai si l'hypoth\`ese de Riemann est fausse. Le plus c\'el\`ebre est le fait qu'il n'y a qu'un nombre fini de corps quadratiques imaginaires de nombre de classes donn\'e; Siegel en a finalement donn\'e en 1935 une d\'emonstration qui ne fait pas intervenir l'hypoth\`ese de Riemann.

\subsection{R\'esultats et approches}

Tous les z\'eros non triviaux de $\zeta(s)$ qui ont \'et\'e trouv\'es sont sur la droite critique. On a les r\'esultats suivants:

\begin{thm}\phantomsection\
\begin{enumerate}
\item (Hardy, 1914 \cite{hardy}) $\zeta(s)$ a une infinit\'e de z\'eros sur la droite critique.
\item (Selberg, Levinson, Conrey 1989 \cite{conrey}) Au moins $2/5$ des z\'eros non triviaux sont sur la droite critique.
\end{enumerate}
\end{thm}

\begin{thm}\phantomsection\label{t1}\
\begin{enumerate}
\item (Hadamard \cite{hadamard}, de la Vall\'ee Poussin \cite{poussin}, 1896) $\zeta(s)$ ne s'annule pas pour $\Re(s)=1$.
\item (Ford, 2002 \cite{ford}) Si $s=\sigma +it$ est un z\'ero non trivial de $\zeta(s)$ avec $|t|\ge 3$, alors
\[\sigma \le 1- \frac{1}{57,54(\log|t|)^{2/3}(\log\log|t|)^{1/3}}. \]
\end{enumerate}
\end{thm}

Il y a eu de nombreuses tentatives pour d\'emontrer l'hypoth\`ese de Riemann, qui ont toutes \'et\'e infructueuses jusqu'\`a maintenant. Plusieurs sont une \'elaboration d'une id\'ee de Hilbert et P\'olya: trouver un op\'erateur autoadjoint $u$ sur un espace de Hilbert tel que les z\'eros non triviaux de $\zeta(s)$ correspondent aux valeurs propres de $u$. 

\subsection{Le th\'eor\`eme des nombres premiers} C'est la version faible suivante de \eqref{eq:npfort}:

\begin{thm}[Hadamard \protect{\cite{hadamard}}, de la Vall\'ee Poussin \protect{\cite{poussin}}, 1896]\phantomsection
On a 
\[\pi(x)\sim \frac{x}{\log x}, \quad x\to +\infty.\]
\end{thm}

(Noter que la fonction $\Li(x)$ est \'equivalente \`a $x/\log x$ quand $x$ tend vers l'infini.)

La d\'emonstration repose essentiellement sur le th\'eor\`eme \ref{t1} 1): voir \cite[ch. 2]{ellison} pour une d\'emonstration de cet ordre, et \cite[ch. 3]{ellison} pour une d\'emonstration ``\'el\'ementaire''.  R\'eciproquement, on peut d\'eduire le th\'eor\`eme \ref{t1} 1) du th\'eor\`eme des nombres premiers  \cite[ch. 3, \S 4.4]{ellison}: la situation est parall\`ele \`a celle du th\'eor\`eme \ref{t2}.

\subsection{Les fonctions z\^eta de Dedekind}\index{Fonction zêta!de Dedekind}\label{s:dedekind}

Si $K$ est un corps de nombres, on d\'efinit
\[\zeta_K(s)=\sum_{\fA} \frac{1}{N(\fA)^{s}}\]
o\`u $\fA$ d\'ecrit les id\'eaux de l'anneau des entiers $O_K$ et $N(\fA):=(O_K:\fA)$ est la norme de $\fA$: c'est la \emph{fonction z\^eta de Dedekind} de $K$. Elle a les propri\'et\'es suivantes:

\begin{description}
\item[Produit eul\'erien] 
\[\zeta_K(s) = \prod_\fP\frac{1}{1-N(\fP)^{-s}}\] 
o\`u $\fP$ d\'ecrit les id\'eaux premiers de $O_K$. Il converge absolument pour $\Re(s)>1$.
\item[Convergence absolue] pour $\Re(s)>1$.
\item[P\^ole simple] en $s=1$.
\item[\'Equation fonctionnelle\index{Equation fonctionnelle@\'Equation fonctionnelle} (Hecke, 1920)] on pose
\begin{equation}\label{eq:gammar}
\Gamma_\R(s) = \pi^{-s/2}\Gamma(s/2),\quad \Gamma_\C(s) = 2(2\pi)^{-s} \Gamma(s)
\end{equation}
\[\Lambda_K(s) = \left|d_K\right|^{s/2} \Gamma_\R(s)^{r_1}\Gamma_\C(s)^{r_2} \zeta_K(s)\]
o\`u $d_K$ est le discriminant absolu de $K$ et $(r_1,r_2)$ est la signature de $K$ (nombre de plongements r\'eels et complexes). Alors
\[\Lambda_K(s)=\Lambda_K(1-s). \]
\item[R\'esidu en $s=1$] par l'\'equation fonctionnelle, il est \'equivalent de donner la partie principale de $\zeta_K(s)$ en $s=0$. On a
\[\ord_{s=0}\zeta_K(s)=r_1+r_2-1=:r,\quad \lim_{s\to 0} s^{-r}\zeta_K(s) = -\frac{h_KR_K}{w_K}\]
o\`u $h_K=|Cl(O_K)|$ est le nombre de classes de $K$, $w_K$ est le nombre de racines de l'unit\'e de $K$, et $R_K$ est le \emph{r\'egulateur}.
\end{description}

La d\'emonstration du produit eul\'erien est la m\^eme que pour la fonction z\^eta de Riemann.\index{Fonction zêta!de Riemann} Montrons la convergence absolue pour $\Re(s)>1$: on peut supposer $s$ r\'eel et, d'apr\`es la proposition \ref{p1},  travailler avec le produit eul\'erien. Soit $p$ un nombre premier: si $\fP\mid p$, on a
\[\frac{1}{1-N(\fP)^{-s}}\le \frac{1}{1-p^{-s}}\]
puisque $N(\fP)$ est une puissance de $p$. Si $n=[K:\Q]$, il y a au plus $n$ id\'eaux premiers $\fP$ divisant $p$, d'o\`u
\[\prod_{\fP\mid p} \frac{1}{1-N(\fP)^{-s}}\le \frac{1}{(1-p^{-s})^n}\]
et
\[\zeta_K(s)\le \zeta(s)^n\]
d'o\`u l'\'enonc\'e.  Le fait que $\zeta_K(s)$ ait un pôle simple en $s=1$ est relativement facile à voir, cf. \cite[ch. 7]{marcus}. Les autres points sont beaucoup plus difficiles (Dedekind, Hecke, Tate: \cite{lang}, \cite{tate}).

\begin{exo} En utilisant la formule \eqref{eq:distr}, vérifier l'identité
\[\Gamma_\C(s) = \Gamma_\R(s) \Gamma_\R(s+1).\]
\end{exo}

\section{La fonction z\^eta d'un $\Z$-sch\'ema de type fini}\index{Fonction zêta!d'un schéma de type fini sur $\Z$}

\subsection{Un peu d'histoire}\

R\'ef\'erences: Roquette \cite{roquette}, Serre \cite{serre-weil}, Osserman \cite{osserman}, Audin \cite{audin}.

\subsubsection{Th\`ese d'Emil Artin (1921)} Sur une suggestion de son directeur de th\`ese Herglotz, il d\'eveloppe la th\'eorie de la fonction z\^eta des corps de fonctions quadratiques (sur $\F_p(t)$) en parall\`ele \`a celle de $\zeta_K(s)$ pour $K$ corps de nombres quadratique. Il d\'emontre sa rationalit\'e (en $p^{-s}$) et formule l'``hypoth\`ese de Riemann''\index{Hypothèse de Riemann!pour les courbes} dans ce contexte. Il la v\'erifie num\'eriquement sur quelques exemples.

\subsubsection{Th\`eses de Sengenhorst, F.K. Schmidt et Rauter} Ils \'etudient les corps de fonctions [d'une variable sur un corps fini] et y transf\`erent la plupart des th\'eor\`emes connus alors sur les corps de nombres.

\subsubsection{Friedrich Karl Schmidt et les fonctions z\^eta \protect{\cite{schmidt}}} Il d\'emontre le th\'eor\`eme de Riemann-Roch\index{Théorème!de Riemann-Roch} pour un corps de fonctions d'une variable $K$ sur un corps parfait $k$ quelconque (apr\`es Dedekind-Weber sur $\C$). Quand $k$ est fini, il g\'en\'eralise la th\'eorie d'Artin des fonctions z\^eta \`a $K$ d'un point de vue ``birationnel'' (sans choisir d'anneau d'entiers dans $K$), et utilise Riemann-Roch \`a la fois pour d\'emontrer l'\'equation fonctionnelle \index{Equation fonctionnelle@\'Equation fonctionnelle} de $\zeta_K(s)$ et sa rationalit\'e. Pour cela, il prouve que $K$ admet toujours un diviseur de degr\'e $1$.

\subsubsection{Helmut Hasse et Max Deuring} Hasse donne \emph{deux} preuves de l'hypoth\`ese de Riemann\index{Hypothèse de Riemann!pour les courbes} pour le corps de fonctions d'une courbe elliptique $E$. La premi\`ere pour $p>3$, non publi\'ee, utilise une m\'ethode d'uniformisation (rel\`evement en caract\'eristique z\'ero), puis la th\'eorie de la multiplication complexe.\footnote{Voir \protect{\cite[II, 5.3]{roquette}} pour une description de cette preuve: Roquette y explique qu'elle g\'en\'eralise (ind\'ependamment) une preuve ant\'erieure d'Herglotz pour le corps des fonctions de la lemniscate.} La seconde \cite{hasse} utilise des propri\'et\'es de l'anneau des endomorphismes de $E$, et en particulier  l'endomorphisme de Frobenius. Deuring simplifie ses d\'emonstrations \`a l'aide des correspondances alg\'ebriques \cite{deuring}.

\subsubsection{Andr\'e Weil} Il d\'emontre l'hypoth\`ese de Riemann pour toute courbe $C$ sur un corps fini, introduisant au passage le langage g\'eom\'etrique. La note \cite{weil1} est une annonce incompl\`ete; il la compl\`ete dans \cite{weil2} et \cite{weil2bis}, apr\`es avoir entretemps refond\'e la g\'eom\'etrie alg\'ebrique dans ce but \cite{weil3}. Il donne deux preuves: la premi\`ere, dans \cite{weil2}, repose sur la th\'eorie des correspondances sur $C$ et est reproduite ci-dessous; la deuxi\`eme, dans \cite[no 48]{weil2bis}, utilise essentiellement le module de Tate de la jacobienne de $C$ et s'applique plus g\'en\'eralement \`a une vari\'et\'e ab\'elienne: j'en expose une version au \S \ref{s:hrab}. 

Pour les vari\'et\'es quelconques sur un corps fini, il formule les \emph{conjectures de Weil} \cite{weil4}.\index{Conjectures!de Weil}

\subsection{Propri\'et\'es \'el\'ementaires de $\zeta(X,s)$}\

R\'ef\'erence: Serre \cite{serre-zeta}.

\begin{defn}\phantomsection Soit $X$ un sch\'ema de type fini sur $\Z$. On pose
 \[\zeta(X,s) = \prod_{x\in X_{(0)}} \frac{1}{1-N(x)^{-s}}\]
o\`u $X_{(0)}$ d\'esigne l'ensemble des points ferm\'es de $X$ et $N(x)$ est le cardinal du corps r\'esiduel $\kappa(x)$ (qui est fini par hypoth\`ese sur $X$).
\end{defn}

(Cette d\'efinition ne d\'epend que de la structure r\'eduite de $X$, et m\^eme de l'``atomisation'' $X_{(0)}$ de $X$.)

\begin{ex}\phantomsection Pour $X=\Spec O_K$, o\`u $K$ est un corps de nombres, on retrouve la fonction z\^eta de Dedekind $\zeta_K(s)$ du \S \ref{s:dedekind}.\index{Fonction zêta!de Dedekind}
\end{ex}

\begin{prop}\phantomsection\label{p4} a) Le produit d\'efinissant $\zeta(X,s)$ est multipliable dans $\Dir(\C)$.\\
b) Si $X_{(0)}=\coprod_{r=1}^\infty (X_r)_{(0)}$, o\`u les $X_r$ sont des sous-sch\'emas, on a
\[\zeta(X,s)=\prod_{r=1}^\infty\zeta(X_r,s).\]
En particulier,
\[\zeta(X,s)=\prod_p \zeta(X_p,s)\]
o\`u $X_p$ est la fibre de $X\to \Spec \Z$ au nombre premier $p$.\\
c) (F.K. Schmidt pour une courbe) Si $X$ est un $\F_q$-sch\'ema, on a
\[\zeta(X,s) = \exp\left(\sum_{n=1}^\infty\left|X(\F_{q^n})\right| \frac{q^{-ns}}{n}\right).\]
d) Notons $\A^1_X$ (\resp $\P^1_X$) la droite affine (\resp projective) sur $X$. Alors on a
\[\zeta(\A^1_{\F_q},s)= \frac{1}{1-q^{1-s}}, \quad \zeta(\P^1_{\F_q},s)= \frac{1}{(1-q^{-s})(1-q^{1-s})}.\] 
e) On a l'identit\'e
\[\zeta(\A^1_X,s)=\zeta(X,s-1)\]
o\`u $\A^1_X$ est la droite affine sur $X$. 
\end{prop}

\begin{proof} a) Comme $X$ est de type fini sur $\Z$, il n'y a qu'un nombre fini de $x\in X_{(0)}$ tels que $N(x)\le N$, o\`u $N$ est un entier donn\'e. Il en r\'esulte que le produit est multipliable (\cf corollaire \ref{c2.3.1} b)).  En d\'eveloppant, on obtient la formule
\begin{equation}\label{eq2.3}
\zeta(X,s)=\sum_{c\in Z_0(X)^+} \frac{1}{N(c)^s}
\end{equation}
comme s\'erie de Dirichlet formelle, o\`u $Z_0(X)$ est le groupe des $0$-cycles de $X$, $Z_0(X)^+$ est le sous-mono\"\i de des cycles effectifs, et 
\[N(c) = \prod N(x)^{n_x}\]
si $c=\sum n_x x$.

b) est \'evident. Pour c), notons $\deg(x)=[\kappa(x):\F_q]$ le degr\'e d'un point ferm\'e $x$. On a
\begin{multline*}
\log \zeta(X,s) = \sum_{x\in X_{(0)}} -\log(1-N(x)^{-s}) =\sum_{x\in X_{(0)}} \sum_{m=1}^\infty \frac{N(x)^{-ms}}{m}\\
= \sum_{m=1}^\infty \sum_{x\in X_{(0)}} \frac{N(x)^{-ms}}{m} = \sum_{m=1}^\infty \sum_{x\in X_{(0)}} \frac{q^{-m\deg(x)s}}{m}=\sum_{n=1}^\infty \sum_{\deg(x)\mid n} \deg(x) \frac{q^{-ns}}{n}
\end{multline*}
et il reste \`a observer que $\left|X(\F_{q^n})\right|=\sum_{\deg(x)\mid n} \deg(x)$.

 Pour d), on applique c) qui donne
\begin{multline*}\zeta(\A^1_{\F_q},s)= \exp\left(\sum_{n=1}^\infty\left|\A^1(\F_{q^n})\right| \frac{q^{-ns}}{n}\right)\\
= \exp\left(\sum_{n=1}^\infty q^n \frac{q^{-ns}}{n}\right)= \frac{1}{1-q^{1-s}}.
\end{multline*}
puis
\[\zeta(\P^1_{\F_q},s)=\zeta(\F_q,s)\zeta(\A^1_{\F_q},s)= \frac{1}{(1-q^{-s})(1-q^{1-s})}.\] 

Enfin, pour e) on applique b) qui donne
\[\zeta(\A^1_X,s)=\prod_{x\in X_{(0)}} \zeta(\A^1_x,s)\]
d'o\`u la formule en utilisant c).
\end{proof}

\begin{rque} Le point c) de la proposition \ref{p4} montre l'int\'er\^et des fonctions z\^eta: elles codent le nombre de points rationnels sur les corps finis. Il montre aussi leur limite: deux vari\'et\'es ayant le m\^eme nombre de points en r\'eduction modulo $\fp$ pour tout $\fp$ ont la m\^eme fonction z\^eta. Ainsi, $\zeta(\P^1_{\Z[1/n]},s)=\zeta(C_{\Z[1/n]},s)$ si $C$ est une conique anisotrope sur $\Q$ d'\'equation $x_0^2=ax_1^2+bx_2^2$, o\`u $a,b$ sont des entiers premiers \`a $n$ et $C_{\Z[1/n]}$ est le $\Z[1/n]$-sch\'ema projectif lisse de m\^eme \'equation. (En effet, toute conique sur un corps fini a un point rationnel par le th\'eor\`eme de Chevalley-Warning.) Voir aussi proposition \ref{p5.2}. Pour des exemples [d'anneaux d'entiers] de corps de nombres non isomorphes ayant même fonction zêta, voir par exemple \cite{perlis1,perlis2}.
\end{rque}

\begin{thm}\phantomsection\label{t:abs} $\zeta(X,s)$ converge absolument (comme s\'erie ou com\-me produit infini) pour $\Re(s)>\dim X$.
\end{thm}

\begin{proof} On proc\`ede ainsi (\cf \cite[d\'em. du th. 1]{serre-zeta}):
\begin{enumerate}
\item Si $X=X_1\cup X_2$ o\`u $X_1$ et $X_2$ sont ferm\'es, la proposition \ref{p4} a) donne la formule 
\[\zeta(X,s)=\frac{\zeta(X_1,s)\zeta(X_2,s)}{\zeta(X_1\cap X_2,s)}\]
donc, par r\'ecurrence sur $\dim X$, l'\'enonc\'e pour $X_1$ et $X_2$ l'implique pour $X$. On se ram\`ene ainsi \`a $X$ irr\'eductible.
\item Si $U$ est un ouvert de $X$, de compl\'ementaire $Z$, on a de m\^eme
\[\zeta(X,s)=\zeta(U,s)\zeta(Z,s)\]
donc, par r\'ecurrence sur $\dim X$, l'\'enonc\'e est \'equivalent pour $X$ et $U$. Ceci nous ram\`ene \`a $X$ affine (et int\`egre).
\item Si $f:X\to Y$ est un morphisme fini et plat et si l'\'enonc\'e est vrai pour $Y$, il est vrai pour $X$. En effet, si $d$ est le degr\'e de $f$, le m\^eme raisonnement qu'au \S \ref{s:dedekind} donne $\zeta(X,s)\le \zeta(Y,s)^d$ pour $s\in ]n,+\infty[$ o\`u $n=\dim Y$.
\item Par le lemme de normalisation de Noether, on est ainsi ramen\'e \`a $X=\A^n_\Z$ ou $X=\Spec \A^n_{\F_q}$ selon que $X$ est plat ou non sur $\Z$\footnote{Plus pr\'ecis\'ement, dans le premier cas un ouvert affine de $X$ est fini sur $\A^n_U$ o\`u $U$ est un ouvert de $\Spec \Z$; on se ram\`ene de l\`a au cas de $\A^n_\Z$ comme en 2.}: dans le second cas, la proposition \ref{p4} e) donne
\[\zeta(X,s)=\frac{1}{1-q^{n-s}}\]
qui converge absolument pour $\Re(s)>n=\dim X$; dans le premier cas, la m\^eme r\'ef\'erence donne
\[\zeta(X,s)=\zeta(s-n)\]
qui converge absolument pour $s>n+1=\dim X$ d'apr\`es la proposition \ref{p5}.
\end{enumerate}
\end{proof}

\subsection{Cas d'une courbe sur un corps fini: \'enonc\'e}

Si $X$ est un $\F_q$-sch\'ema de type fini, on peut \'ecrire d'apr\`es la proposition \ref{p4} c):
\[\zeta(X,s) = Z(X,q^{-s})\]
avec $Z(X,t)\in \Q[[t]]$. La m\^eme r\'ef\'erence donne
\[Z(X,t) = \exp\left(\sum_{n=1}^\infty \nu_n \frac{t^{n}}{n}\right),\quad \nu_n = \left|X(\F_{q^n})\right|\]
et la preuve de la proposition \ref{p4} a) donne une autre expression:
\begin{equation}\label{eq:Z2}
Z(X,t) = \sum_{c\in Z_0(X)^+} t^{\deg(c)}=\sum_{n\ge 0} b_nt^n
\end{equation}
o\`u $b_n = \left|\{c\in Z_0(X)^+\mid \deg(c)=n \} \right|$.

\begin{meg} La fonction $\zeta(X,s)$ est ``absolue'' alors que $Z(X,t)$ est relative au choix du corps de base $\F_q$: si $X$ est un $\F_{q'}$-sch\'ema de type fini avec $q'=q^r$, alors $Z(X/\F_q,t) = Z(X/\F_{q'},t^r)$.
\end{meg}

\begin{thm}\phantomsection\label{t:courbe} Soit $C$ une courbe projective, lisse et g\'eom\'etriquement connexe, de genre $g$ sur $\F_q$. Alors
\[Z(C,t) =\frac{P(t)}{(1-t)(1-qt)}\]
o\`u $P\in \Z[t]$ est un polyn\^ome de degr\'e $2g$ et de terme constant $1$, dont les racines inverses sont de valeur absolue $\sqrt{q}$ pour tout plongement de $\bar \Q$ dans $\C$. On a l'\'equation fonctionnelle\index{Equation fonctionnelle@\'Equation fonctionnelle}
\[Z(C,1/qt) = q^{1-g} t^{2-2g} Z(C,t). \]
\end{thm}

(Les racines inverses de $P$ sont les nombres $\alpha_1,\dots,\alpha_{2g}$ tels que 
\[P(t)=\prod_{i=1}^{2g}(1-\alpha_i t).)\]

\begin{cor}[Hypoth\`ese de Riemann]\index{Hypothèse de Riemann!pour les courbes}\phantomsection Les z\'eros de $\zeta(C,s)$ sont tous situ\'es sur la droite $\Re(s)=1/2$.
\end{cor}

\subsection{Strat\'egie de la preuve}

\begin{enumerate}
\item Rationalit\'e et \'equation fonctionnelle: r\'esultent de Riemann-Roch.
\item Hypoth\`ese de Riemann: borne sur les $a_n:= 1+q^n-\nu_n$, plus pr\'ecis\'ement:
\begin{equation}\label{eq:hr}
\left|a_n\right|\le 2g\sqrt{q^n}.
\end{equation}
\end{enumerate}

Dans tout ce qui suit, ``courbe sur $k$'' signifie ``courbe projective, lisse et g\'eom\'etriquement connexe sur (le corps) $k$''.

\subsection{Le th\'eor\`eme de Riemann-Roch}\index{Théorème!de Riemann-Roch} Dans le cas d'une courbe, les z\'ero-cycles co\"\i ncident avec les diviseurs. Rappelons l'\'enonc\'e du th\'eor\`eme de Riemann-Roch (\cf \cite[ch. IV, \S 1]{hartshorne}) pour une courbe $C$ sur un corps $k$:

Notons $F$ le corps des fonctions de $C$. Si $D$ est un diviseur sur $C$, on note $\sL(D)$ le faisceau inversible correspondant et $L(D)= H^0(C,\sL(D))$: c'est un $k$-espace vectoriel de dimension finie $l(D)$. Par d\'efinition, $\sL(D)$ est un sous-faisceau du faisceau constant $F$ et on a
\[L(D)=  \{f\in F^* \mid D+(f)\in Z_0(C)^+\}\cup\{0\}.\]

Notons $\Omega^1$ le faisceau des formes diff\'erentielles de K\"ahler, localement libre de rang $1$, et $K$ un diviseur correspondant \`a $\Omega^1$ (diviseur canonique). Le th\'eor\`eme de Riemann-Roch est la formule
\begin{equation}\label{eq:rr}
l(D) -l(K-D) = \deg(D) +1-g.
\end{equation}

On en d\'eduit notamment:
\begin{equation}\label{eq:rr2}
l(K)=g,\quad \deg(K)=2g-2
\end{equation}
\begin{equation}\label{eq:rr3}
l(D)=\deg(D)+1-g \text{ si } \deg(D)>2g-2
\end{equation}
\cf \cite[ch. IV, ex. 1.3.3 et 1.3.4]{hartshorne}.

(Chez Hartshorne, $k$ est suppos\'e alg\'ebriquement clos, mais on se ram\`ene imm\'ediatement \`a ce cas puisque les $l(D)$ sont invariants par extension du corps de base.)

\subsection{Rationalit\'e et \'equation fonctionnelle  (F.K. Schmidt)}\label{s:req} Revenons au cas qui nous int\'eresse: $k=\F_q$. Ce qui suit est une retranscription de \cite[nos 18--19]{weil2}. \index{Equation fonctionnelle@\'Equation fonctionnelle}

\begin{lemme}\phantomsection\label{l2.2} Soit $D\in Z_0(C)$. Alors le nombre de diviseurs effectifs rationnellement \'equivalents \`a $D$ est \'egal \`a 
\[\frac{q^{l(D)}-1}{q-1}.\]
\end{lemme}

\begin{proof} En effet, ce nombre est par d\'efinition le cardinal du syst\`eme lin\'eaire $|D|= \P(L(D))$ attach\'e \`a $D$.
\end{proof}

Soit $\delta>0$ le pgcd des degr\'es des \'el\'ements non nuls de $Z_0(C)$\footnote{On verra plus bas que $\delta=1$, mais il n'est pas n\'ecessaire de le savoir pour le moment.}; choisissons un diviseur  $D_0$ de degr\'e $\delta$. Choisissons aussi un entier $\nu$ tel que $\nu\delta\ge g$, et un ensemble maximal $D_1,\dots, D_h$ de diviseurs effectifs de degr\'e $\nu\delta$, deux \`a deux non rationnellement \'equivalents.

\begin{lemme}\phantomsection \label{l2.1}Tout diviseur $D$ de degr\'e $\nu\delta$ est rationnellement \'equivalent \`a un et un seul $D_i$, $1\le i\le h$.
\end{lemme}

\begin{proof} D'apr\`es \eqref{eq:rr}, $l(D)>0$ donc $D$ est rationnellement \'equivalent \`a un diviseur effectif.
\end{proof}

Soit $D$ un diviseur effectif non nul: on a $\deg(D)=n\delta$, $n>0$. Alors $D-(n-\nu)D_0$ est de degr\'e $\nu\delta$, donc on peut lui appliquer le lemme \ref{l2.1}. En tenant compte du lemme \ref{l2.2}, ceci donne
\[b_{n\delta} = \sum_{i=1}^h\frac{q^{l(D_i+(n-\nu)D_0)}-1}{q-1}.\]

Mais, pour $n\delta>2g-2$, \eqref{eq:rr} donne $l(D_i+(n-\nu)D_0)= n\delta+1-g$, d'o\`u 
\[b_{n\delta} = h\frac{q^{n\delta+1-g}-1}{q-1}.\]

Donc \eqref{eq:Z2} donne que $(q-1)Z(C,t)$ est somme d'un polyn\^ome
 et de la s\'erie
\[h \sum_{n\delta>2g-2} (q^{n\delta+1-g}-1)t^{n\delta}=h\left(q^{1-g} \frac{(qt)^{(\rho+1)\delta}}{1-(qt)^\delta}- \frac{t^{(\rho+1)\delta}}{1-t^\delta}\right)\]
o\`u $\rho=\inf\{n\mid n\delta\ge 2g-2\}$. Donc $Z(C,t)\in \Q(t)$.

On remarque ensuite que $\delta\mid 2g-2$ (gr\^ace \`a \eqref{eq:rr2}), donc que 
\[\rho=\frac{2g-2}{\delta}.\]

On peut alors \'ecrire plus pr\'ecis\'ement $(q-1)Z(C,t)= F(t)+hR(t)$, avec
\begin{equation}\label{eq:rat}
F(t)=\sum_{n=0}^\rho \sum_{i=1}^h q^{l(D_i+(n-\nu)D_0)}t^{n\delta},\quad R(t) = \frac{1}{t^\delta-1}+q^{1-g} \frac{(qt)^{(\rho+1)\delta}}{1-(qt)^\delta}.
\end{equation}

Un calcul direct donne
\[R(1/qt) = q^{1-g} t^{2-2g} R(t).\]

Pour traiter $F(t)$, on r\'eutilise la formule de Riemann-Roch\index{Théorème!de Riemann-Roch}: jointe au lemme \ref{l2.1}, elle implique que les familles  $(l(K-D_i-(n-\nu)D_0))_{1\le i\le h}$ et  $(l(D_i+(\rho-n-\nu)D_0))_{1\le i\le h}$ co\"\i ncident \`a permutation pr\`es. On en d\'eduit que $F(t)$, et donc $Z(C,t)$, v\'erifie la m\^eme \'equation fonctionnelle que $R(t)$.

\subsection{Hypothèse de Riemann: réduction à %\eqref{eq:hr} 
(2.3) (Hasse, Schmidt, Weil)}\index{Hypothèse de Riemann!pour les courbes}
\label{s2.7} Supposons  \eqref{eq:hr} connu; posons $P(t)=(1-t)(1-qt) Z(C,t)$, qui v\'erifie l'\'equation fonctionnelle
\begin{equation}\label{eq:ef}h
P(1/qt) = q^{-g} t^{-2g} P(t).
\end{equation}

D'apr\`es le \S \ref{s:req}, $P$ est une fraction rationnelle  donc est m\'eromorphe dans tout le plan complexe. Par construction, on a
\[\frac{d\log{P(t)}}{dt} = -\sum_{n=1}^\infty a_nt^{n-1}.\]

D'apr\`es \eqref{eq:hr}, cette s\'erie converge pour $|t|<q^{-1/2}$; $P$ n'a donc ni z\'ero ni p\^ole dans ce domaine. Mais l'\'equation fonctionnelle montre que $P$ n'a pas non plus de z\'eros ou de p\^oles dans le domaine $|t|>q^{-1/2}$; ses z\'eros et p\^oles sont donc concentr\'es sur le cercle $|t|=q^{-1/2}$ et $Z(C,t)$ a en plus deux p\^oles simples, en $t=1$ et en $t=1/q$. En examinant \eqref{eq:rat}, on en déduit:
\[\delta=1;\quad P \text{ est un polyn\^ome.}\]

Enfin, ce polyn\^ome est \`a coefficients entiers par construction, et il est de degr\'e $2g$ d'apr\`es \eqref{eq:ef}.

\subsection{Hypoth\`ese de Riemann: la premi\`ere d\'emonstration de Weil}\label{s:hr1} Elle repose sur la \emph{th\'eorie des correspondances} et plus pr\'ecis\'ement sur l'\emph{in\'egalit\'e de Castelnuovo-Severi}.

\subsubsection{Th\'eorie \'el\'ementaire des correspondances} (Voir aussi \cite[I]{weil-criteres}.)

\begin{defn}\phantomsection Soient $C,C'$ deux courbes (projectives, lisses, g\'eom\'etriquement connexes) sur un corps $k$. Une \emph{correspondance alg\'ebrique} de $C$ vers $C'$ est un diviseur sur $C\times C'$. Notation: $\Corr(C,C')$.
\end{defn}

\begin{ex}\phantomsection Si $f:C\to C'$ est un $k$-morphisme, son graphe $\Gamma_f$, image de l'immersion ferm\'ee $C\by{\left(\begin{smallmatrix}1_G\\ f \end{smallmatrix}\right)} C\times C'$, est une correspondance not\'ee $f_*$.
\end{ex}

\begin{ex}\phantomsection Si $\gamma\in \Corr(C,C')$, le diviseur sous-jacent \`a $\gamma$ d\'efinit une correspondance de $C'$ vers $C$, not\'ee ${}^t\gamma\in \Corr(C',C)$ (correspondance transpos\'ee). Si $\gamma$ est de la forme $\Gamma_f$, on note ${}^t \gamma=f^*$.
\end{ex}

\begin{defn}\phantomsection\label{d2.1} Soient $C,C',C''$ trois courbes sur le corps $k$, et soient $\gamma\in \Corr(C,C')$ et $\gamma'\in \Corr(C',C'')$.\\
a) On dit que $\gamma$ et $\gamma'$ sont \emph{composables} si les composantes irréductibles des diviseurs $\gamma\times C''$ et $C\times \gamma'$ de $C\times C'\times C''$ sont deux \`a deux distinctes.\\
b) Si $\gamma$ et $\gamma'$ sont composables, on d\'efinit leur compos\'ee $\gamma'\circ \gamma$ comme le diviseur
\[(p_{C\times C''}^{C\times C'\times C''})_*(C\times \gamma'\cap \gamma\times C'')\]
o\`u $p_*$ est l'image directe des cycles \cite[\S 1.4]{fulton} et $\cap$ d\'esigne le produit d'intersection (d\'efini \`a l'aide des multiplicit\'es d'intersection)\footnote{Sauf lecture trop rapide de \protect{\cite{weil3}}, Weil d\'efinit celles-ci ``\`a l'italienne'': sur une vari\'et\'e que nous appellerions maintenant quasi-projective, par r\'eduction \`a la diagonale et intersection avec un espace lin\'eaire g\'en\'eral de dimension compl\'ementaire; puis sur une vari\'et\'e quelconque, par recollement \`a partir du cas quasi-projectif.}.
\end{defn}

Si $\gamma$ et $\gamma'$ sont irr\'eductibles, les diviseurs $\gamma\times C''$ et $C\times \gamma'$ sont \'egaux si et seulement si $\gamma$ est de la forme $C\times P$ et $\gamma'$ de la forme $P\times C''$, o\`u $P$ est un point de $C'$. Ceci donne:

\begin{prop}[Weil, \protect{\cite[th. 6]{weil2}}]\phantomsection\label{p:comp} Il existe une et une seule loi bilin\'eaire (composition)
\[\circ:\Corr(C',C'')\times \Corr(C,C')\to \Corr(C,C'')\]
telle que
\begin{thlist}
\item Si $\gamma\in \Corr(C,C')$ et $\gamma'\in \Corr(C',C'')$ sont composables, leur composition est donn\'ee par la r\`egle de la d\'efinition \ref{d2.1} b).
\item $(P\times C'')\circ (C\times P) = 0$ pour tout $P\in C'_{(0)}$.
\end{thlist}
Cette loi est associative (quand on fait intervenir quatre courbes) et la correspondance diagonale $\Delta_C\subset C\times C$ est \'el\'ement neutre \`a gauche et \`a droite. Si $f:C\to C'$ et $f':C'\to C''$ sont deux morphismes, on a
\[(f'\circ f)_*=f'_*\circ f_*.\]
Enfin, pour toutes correspondances $\gamma\in \Corr(C,C')$ et $\gamma'\in \Corr(C',C'')$, on a 
\[{}^t(\gamma'\circ \gamma)={}^t\gamma\circ {}^t\gamma'.\]
\end{prop}

\begin{proof} L'existence et l'unicit\'e sont \'evidentes; la v\'erification des propri\'et\'es formelles suit de celle des produits d'intersection (\cf \cite[prop. 16.1.1]{fulton}).
\end{proof}

\subsubsection{Indices}

\begin{defn}[\cf \protect{\cite[ex. 16.1.4]{fulton}}]\phantomsection Soit $\gamma\in \Corr(C,C')$ une correspondance. On lui associe ses \emph{indices} (ou degr\'es) $d(\gamma),d'(\gamma)$: ce sont les entiers tels que $(p^{C\times C'}_C)_*\gamma=d(\gamma) [C]$ et $(p^{C\times C'}_{C'})_*\gamma=d'(\gamma) [C']$, respectivement. 
\end{defn}

Si $\gamma$ est irr\'eductible, on a donc
\[d(\gamma)=
\begin{cases}
[k(\gamma):k(C)] &\text{si $\gamma$ est dominant sur $C$}\\
0&\text{sinon}
\end{cases}
\]
et de m\^eme pour $d'(\gamma)$.

\begin{lemme}\phantomsection \label{l:deg} On a les identit\'es
\[d({}^t\gamma)=d'(\gamma), \quad d(\gamma'\circ \gamma) = d(\gamma')d(\gamma).\qed\]
\end{lemme}

\begin{prop}[Weil, \protect{\cite[th. 6]{weil2}}]\phantomsection\label{p:proj} Soit $\gamma\in \Corr(C,C')$ une correspondance effective, sans composantes de la forme $P\times C'$ o\`u $P$ est un point ferm\'e de $C$, et telle qu'on ait $d(\gamma)=1$. Alors
\[\gamma\circ {}^t\gamma = d'(\gamma)\Delta_C.\]
\end{prop}

\begin{cor}\phantomsection\label{c:proj} Si $f:C\to C'$ est un morphisme non constant, on a $f_*\circ f^*=\deg(f)\Delta_{C'}$.
\end{cor}

\begin{proof} En effet, on a $d(\Gamma_f)=1$ et $d'(\Gamma_f)=\deg(f)$, et $\Gamma_f$, irr\'eductible, n'est clairement pas de la forme $P\times C'$. 
\end{proof}

\subsubsection{La trace}

\begin{defn}\phantomsection Soient $\gamma_1,\gamma_2\in \Corr(C,C')$, consid\'er\'es comme diviseurs. On note $I(\gamma_1,\gamma_2)\in\Z$ leur \emph{produit d'intersection}.
\end{defn}

Chez Weil, ce produit est d\'efini ainsi \cite[p. 30]{weil2}: si $\gamma,\gamma'$ sont deux courbes r\'eduites et irr\'eductibles distinctes, leur intersection $\gamma\cap \gamma'$ est un z\'ero-cycle et $I(\gamma,\gamma')=\deg(\gamma\cap \gamma')$: ceci s'\'etend aux diviseurs sans composante commune par lin\'earité. Si $\gamma=\gamma'$, alors il existe une fonction rationnelle $f$ telle que $v_\gamma(f)=1$, et on d\'efinit $I(\gamma,\gamma)$ comme $I(\gamma,\gamma-(f))$. Cette d\'efinition marche parce que, si $\gamma$ et $(f)$ n'ont pas de composante commune, alors $I(\gamma,(f))=0$.

\begin{prop}[Weil, \protect{\cite[prop. 2]{weil2}}]\phantomsection\label{p:int} On a l'identit\'e
\[I(\gamma_1,\gamma_2)=I({}^t\gamma_2\circ \gamma_1,\Delta_C).\]
\end{prop}

\begin{proof} C'est un exercice combinatoire, \cf \cite[ex. 16.1.3]{fulton} pour un \'enonc\'e un peu plus pr\'ecis (et une d\'emonstration).
\end{proof}

\begin{defn}\phantomsection\label{d:trace} Soit $C$ une courbe et soit $\gamma\in \Corr(C,C)$. On d\'efinit 
\[\sigma(\gamma) = d(\gamma)+d'(\gamma)-I(\gamma,\Delta_C).\]
C'est la \emph{trace} de $\gamma$.
\end{defn}

\begin{defn}\phantomsection Deux correspondances $\gamma_1,\gamma_2\in \Corr(C,C')$ sont \emph{\'equivalentes} (notation: $\gamma\equiv \gamma'$) si il existe une fonction rationnelle $f$ et des diviseurs $D$ sur $C$ et $D'$ sur $C'$ tels que
\[\gamma-\gamma'=(f)+D\times C'+C\times D'.\]
C'est la \emph{three-line equivalence} de Weil (terminologie de J.P. Murre).
\end{defn}

\begin{prop}\phantomsection\label{p:sigma} a) (\cite[p. 38]{weil2}). Dans $\Corr(C,C)$, l'ensemble des $\gamma\equiv 0$ est un id\'eal bilat\`ere. 

Plus g\'en\'eralement, si on note $\Corr_\equiv(C,C')=\Corr(C,C')/\equiv$, la composition de la proposition \ref{p:comp} induit une composition
\[\circ:\Corr_\equiv(C',C'')\times \Corr_\equiv(C,C')\to \Corr_\equiv(C,C'').\]

b) (\cite[th. 7]{weil2}) La trace $\sigma$ induit une fonction
\[\sigma:\Corr_\equiv(C,C)\to \Z.\]
qui v\'erifie l'identit\'e $\sigma({}^t\gamma)=\sigma(\gamma)$.\\
c) (ibid.) Si $\gamma\in \Corr_\equiv(C,C')$ et $\gamma'\in \Corr_\equiv(C',C)$, on a la relation
\[\sigma(\gamma'\circ \gamma)=\sigma(\gamma\circ \gamma').\]
\end{prop}

\subsubsection{La trace de l'identit\'e}

\begin{thm}[Weil,  \protect{\cite[th. 8]{weil2}}]\phantomsection\label{t:trid} Si $C$ est de genre $g$, on a $\sigma(\Delta_C)=2g$.
\end{thm}

\begin{proof}[Esquisse] Comme $d(\Delta_C)=d'(\Delta_C)=1$, il faut prouver que $I(\Delta_C,\Delta_C)\allowbreak=2-2g$. Cela r\'esulte du th\'eor\`eme de Riemann-Roch\index{Théorème!de Riemann-Roch} \eqref{eq:rr} car, si $f$ est une fonction rationnelle sur $C\times C$ telle que $v_{\Delta_C}(f)=1$, $(p^{C\times C}_C)_*\big(((f)-\Delta_C)\cap \Delta_C)\big)$ (premi\`ere projection) est un diviseur canonique \cite[p. 15]{weil2}.
\end{proof}

\subsubsection{L'in\'egalit\'e de Castelnuovo-Severi}

\begin{thm}[th\'eor\`eme de l'indice de Hodge pour une surface]\phantomsection\label{t:hodge} Soit $S/k$ une surface projective lisse, et soit $H$ un diviseur ample sur $S$. Notons $<,>$ le produit d'intersection sur les diviseurs. Soit $D$ un diviseur sur $S$ tel que $<D, H>=0$. Alors $<D,D>\le 0$, avec \'egalit\'e si et seulement si  $D$ est num\'eriquement \'equivalent \`a z\'ero (c'est-\`a-dire $<D,E>=0$ pour tout diviseur $E$).
\end{thm}

\begin{proof} Voir \cite[ch. V, th. 1.9]{hartshorne} ou \cite[ex. 15.2.4]{fulton}.  Ce th\'eor\`eme est d\^u \`a Beniamino Segre \cite{segre} et Jacob Bronowski \cite{bronowski}, red\'ecouvert par Alexander Grothendieck \cite{groth-mt}. La preuve utilise le th\'eor\`eme de Riemann-Roch pour $S$.
\end{proof}

\begin{thm}[Weil,  \protect{\cite[th. 10]{weil2}}]\phantomsection\label{t:lemmeimportant} Soit $\gamma\in \Corr_\equiv(C,C')-\{0\}$. Alors $\sigma({}^t\gamma\circ \gamma)>0$.
\end{thm}

\begin{proof}[D\'emonstration moderne]  D'apr\`es la proposition \ref{p:int}, il s'agit de montrer:
\[I(\gamma,\gamma)< d({}^t\gamma\circ\gamma) +d'({}^t\gamma\circ\gamma) \]
si $\gamma\in \Corr(C,C)$ n'est pas $\equiv 0$. \'Etant donn\'e le lemme \ref{l:deg}, cette in\'egalit\'e est \'equivalente \`a
\[I(\gamma,\gamma)< 2d(\gamma) d'(\gamma) \]
(\emph{in\'egalit\'e de Castelnuovo-Severi}). Cela r\'esulte du th\'eor\`eme \ref{t:hodge} pour la surface $S=C\times C'$, appliqu\'e (sur $\bar k$) au diviseur $D=\gamma-d(\gamma)(P\times C')-d'(\gamma)(C\times P')$ et au diviseur ample $H=C\times P'+P\times C'$, o\`u $P$ (\resp $P'$) est un point de $C$ (\resp de $C'$).
\end{proof}

\begin{rque}\phantomsection Mattuck et Tate d\'eduisent dans \cite{mattuck-tate} le th\'eor\`eme \ref{t:lemmeimportant} du th\'eor\`eme de Riemann-Roch\index{Théorème!de Riemann-Roch} pour $C\times C'$; Grothendieck remarque dans \cite{groth-mt}  que cela r\'esulte directement du th\'eor\`eme de l'indice de Hodge, comme ci-dessus: j'ai reproduit la version de son argument donn\'ee  par Fulton dans \cite[ex. 16.1.10 (a)]{fulton}.
\end{rque}

\begin{cor}\phantomsection\label{cauchy-schwarz} a) L'application
\[(\gamma_1,\gamma_2)\mapsto \sigma({}^t\gamma_2\circ \gamma_1)\]
de $\Corr_\equiv(C,C')\times \Corr_\equiv(C,C')$ vers $\Z$ est une forme bilin\'eaire sym\'etrique d\'efinie positive.\\
b) Pour toute classe de correspondance $\gamma\in \Corr_\equiv(C,C)$, on a l'in\'egalit\'e
\[\sigma(\gamma)^2\le 2g\sigma({}^t\gamma \circ\gamma) \]
avec \'egalit\'e si et seulement si $\gamma$ est lin\'eairement d\'ependante de $\Delta_C$.
\end{cor}

\begin{proof} a) La sym\'etrie r\'esulte de la proposition \ref{p:sigma} et la positivit\'e du th\'eor\`eme \ref{t:lemmeimportant}.

b) La forme quadratique
\[(a,b)\mapsto \sigma\left({}^t(a\Delta_C+b\gamma)(a\Delta_C+b\gamma)\right)=2ga^2+2\sigma(\gamma)ab+\sigma({}^t\gamma \circ\gamma)b^2\]
(\cf th\'eor\`eme \ref{t:trid} et proposition \ref{p:sigma}) est positive, son discriminant est donc n\'egatif ou nul; il est nul si et seulement si la forme n'est pas d\'efinie, ce qui arrive exactement quand $\gamma$ et $\Delta_C$ sont lin\'eairement d\'ependants (dans $\Corr_\equiv(C,C)$).
\end{proof}

\subsubsection{Cas du graphe d'un endomorphisme} Soit $f:C\to C$ un endomorphisme, de graphe $\Gamma_f$. En appliquant le corollaire \ref{cauchy-schwarz} \`a $\Gamma_f$ en tenant compte du corollaire \ref{c:proj} et en r\'eappliquant le th\'eor\`eme \ref{t:trid}, on trouve:

\begin{thm}\phantomsection\label{ineg} On a l'in\'egalit\'e
\[|\sigma(f_*)|\le 2g\sqrt{\deg(f)}. \]
Si $\Gamma_f$ est transverse \`a $\Delta_C$, cette in\'egalit\'e devient
\[|1+\deg(f)-N(f)|\le 2g\sqrt{\deg(f)} \]
o\`u $N(f)$ est le nombre de points fixes de $f$.
\end{thm}

\begin{proof} Comme $d(f_*)=1$ et $d'(f_*)=\deg(f)$ (voir preuve du corollaire \ref{c:proj}), pour d\'eduire la seconde in\'egalit\'e de la premi\`ere il reste \`a remarquer que $I(f_*,\Delta_C)=N(f)$ par l'hypoth\`ese de transversalit\'e (\cf proposition \ref{p:int}).

\end{proof}

\subsubsection{Fin de la d\'emonstration}\phantomsection\label{s2.8.7} Pour d\'emontrer \eqref{eq:hr}, Weil applique le th\'eor\`eme \ref{ineg} aux puissances de la \emph{correspondance de Frobenius} $\pi_C$, graphe du morphisme de Frobenius
\[F_C:C\to C\]
d\'efini (en termes grothendieckiens) par l'identit\'e sur les points de $C$ et l'application $f\mapsto f^q$ sur le faisceau d'anneaux $\sO_C$. Comme $F_C$ est radiciel, son graphe est bien transverse \`a $\Delta_C$ ainsi que celui de $F_C^n$ pour tout $n>0$: voir \cite[preuve du th. 13]{weil2}. Il reste \`a observer que l'ensemble des points fixes de $F_C^n$ est $C(\F_{q^n})$.

\subsection{Premières applications}

\begin{cor}\phantomsection a) Soit $C$ une courbe projective lisse de genre $g$, g\'eom\'etriquement connexe sur $\F_q$. Alors
\[1+q-2g\sqrt{q}\le |C(\F_q)|\le 1+q+2g\sqrt{q}.\] 
b) Soit $C$ une courbe projective lisse de genre $g$, g\'eom\'etriquement irr\'eductible  sur un corps de nombres $k$. Si $v$ est une place finie de $k$, de bonne r\'eduction pour $C$, $C(k_v)$ a un point rationnel d\`es que $N(v)\ge 4g^2$.
\end{cor}

\begin{proof} a) est imm\'ediat \`a partir de \eqref{eq:hr}. Cela implique que $X(\F_q)\neq \emptyset$ d\`es que $1+q-2g\sqrt{q}>0$. Ceci se produit d\`es que
\[\sqrt{q}> g+\sqrt{g^2-1}\]
ce qui est vrai d\`es que $\sqrt{q}\ge 2g$, soit $q\ge 4g^2$. Dans  b), on en d\'eduit que la fibre sp\'eciale $C(v)$ a un point rationnel d\`es que $N(v)>4g^2$, et alors $C(k_v)$ a \'egalement un point rationnel.
\end{proof}

\begin{ex}\phantomsection Si $C$ est une courbe de genre $1$ sur $\Q$ qui a bonne r\'eduction en $p$, alors $C(\Q_p)$ a un point rationnel d\`es que $p\ge 5$.
\end{ex}

%\begin{comment}
\begin{cor}\phantomsection\label{c:courbesg} Soit $C$ une courbe sur $\F_q$, \'eventuellement ouverte et singuli\`ere, mais g\'eom\'etriquement int\`egre. Alors $\zeta(C,s)$ est une fraction rationnelle en $q^{-s}$; ses p\^oles sont simples et situ\'es aux points $s=1+\frac{2i\pi r}{\log q}$, $r\in\Z$, et elle ne s'annule pas pour $\Re(s)>\frac{1}{2}$.
\end{cor}

\begin{proof} Supposons d'abord $C$ propre, et soit $\tilde C$ sa normalisation. La projection $f:\tilde C\to C$ est un isomorphisme en dehors d'un nombre fini de points $c_1,\dots,c_n\in C$. On a
\[Z(c_i,t) = (1-t^{d_i})^{-1}\]
o\`u $d_i=\deg(c_i)$ et
\[Z(f^{-1}(c_i),t)= \prod_{c\in f^{-1}(c_i)} (1-t^{d_c})^{-1}\]
o\`u $d_c=\deg(c)$. comme $d_i\mid d_c$ pour tout $c$, $(1-t^{d_i})\mid \prod (1-t^{d_c})$. On en conclut que 
\begin{equation}\label{eq:courbesg}
Z(C,t)=Z(\tilde C,t)Q(t)
\end{equation}
o\`u $Q$ est un polyn\^ome dont les zéros sont des racines de l'unit\'e.

Si $C$ est quelconque, on consid\`ere sa compl\'et\'ee $\bar C$: si $\tilde C$ est la normalis\'ee de $\bar C$, on voit que \eqref{eq:courbesg} reste vrai, avec la m\^eme propri\'et\'e pour $Q$. La conclusion r\'esulte alors du th\'eor\`eme \ref{t:courbe}.
\end{proof}
%\end{comment}

\subsection{Les théorèmes de Lang-Weil}

Le corollaire suivant est d\^u ind\'ependamment \`a Lang-Weil \cite{lang-weil} et Nisnevi\v c \cite{nisnevic} (ce dernier a une borne un peu moins fine).

\begin{cor}\phantomsection\label{c:lw} Soit $X$ une sous-vari\'et\'e de $\P^N_{\F_q}$ de dimension $n$, de degr\'e $d$ et g\'eom\'etriquement irr\'eductible. Alors
\[\big| |X(\F_q)|-q^n\big |\le (d-1)(d-2)q^{n-1/2}+Aq^{n-1}\]
o\`u $A$ est une constante ne d\'ependant que de $n,d,N$.\\
En particulier, $X$ admet un z\'ero-cycle de degr\'e $1$.
\end{cor}

\begin{proof}[Esquisse de démonstration] Lang et Weil procèdent par récurrence sur $n$. Le cas $n=1$ est une variante de la démonstration du corollaire \ref{c:courbesg}, utilisant l'identité
\[\frac{(d-1)(d-2)}{2}=g+\sum_P \delta_P\]
où $g$ est le genre de la désingularisée $\tilde X$ de $X$ et, pour tout $P\in X$, $\delta_P$ est la longueur du $\sO_{X,P}$-module $\sO_{\tilde X,P}/\sO_{X,P}$ \cite[I, ex. 7.2; IV, ex. 1.8]{hartshorne}.

Pour $n>1$, on peut supposer que sans perte de généralité que $X$ n'est contenu dans aucun hyperplan $H$ de $\P^N$.  Le lieu des $H$ tels que $H\cap X$ soit réductible est alors un fermé algébrique non vide de l'espace projectif dual de $\P^N$.\footnote{Dans  \cite[exp. XVII, \S 3]{SGA7}, une variante de ce lieu est étudiée, quand $X$ est lisse, sous le nom de \emph{variété duale}: la condition ``intersection non irréductible'' est remplacée par ``intersection non transverse''.} On voit facilement qu'il est contenu dans une hypersurface, ce qui borne son nombre de $\F_q$-points par $B q^{N-1}$ où $B$ est une constante dépendant de $(n,d,N)$. Un argument de comptage élémentaire permet alors de conclure.
\end{proof}

\begin{rque} Pour obtenir l'existence d'un zéro-cycle de degré $1$ sur $X$, il n'est pas nécessaire d'appliquer l'estimation du corollaire \ref{c:lw}: il suffit d'appliquer le résultat du  \S \ref{s2.7} à la normalisée d'une courbe tracée sur $X$.
\end{rque}

\begin{cor}\phantomsection\label{c:lw2} Soit $K$ un corps de fonctions de $n$ variables sur $\F_q$, $\F_q$ étant algébriquement fermé dans $K$. Alors il existe une constante $\gamma$ telle que, pour tout $\F_q$-modèle $V$ de $K$ et tout $r\ge 1$, on ait l'inégalité
\[\big | |V(\F_{q^r})|-q^{rn}\big|\le \gamma q^{r(n-1/2)} + Bq^{r(n-1)}\]
où la constante $B$ dépend de $V$. Si $K$ admet un modèle projectif de degré $d$, on peut prendre $\gamma\le (d-1)(d-2)$.
\end{cor}

\begin{proof} On raisonne par récurrence sur $n$, le cas $n=0$ étant trivial (avec $\gamma=B=0$). Pour $n>0$, cela montre que l'énoncé est invariant par passage à un ouvert non vide de $V$; on peut donc supposer $V$ affine, puis projectif, et on est ramené au corollaire \ref{c:lw}.
\end{proof}

\begin{cor}\phantomsection\label{c:tw3} Si $V'$ est un autre modèle de $K$, la fonction $\zeta(V,s)/\zeta(V',s)$ converge absolument pour $\Re(s)> n-1$; donc l'ordre des zéros et pôles de $\zeta(V,s)$ dans cette région sont des invariants birationnels de $K$.\qed
\end{cor}

Le corollaire suivant est d\^u \`a Lang-Weil \cite{lang-weil} quand $V$ est comme ci-dessus, et g\'en\'eralis\'e par Serre dans \cite{serre-zeta}.

\begin{cor}\phantomsection\label{c:tw4} a) Soit $V$ un $\Z$-sch\'ema de type fini de dimension $n$. Alors $\zeta(V,s)$ se prolonge analytiquement au demi-plan $\Re(s)>n-\frac{1}{2}$.\\
b) Supposons $V$ int\`egre, de corps des fonctions $E$.
\begin{thlist}
\item Si $\car E=0$,  $\zeta(V,s)$ a dans ce domaine un unique p\^ole simple en $s=n$.
\item Si $\car E=p$, soit $\F_q$ son corps des constantes (fermeture alg\'ebrique de $\F_p$ dans $E$). Alors les seuls p\^oles de $\zeta(V,s)$ dans ce domaine sont simples et situ\'es aux points 
\[s =n+\frac{2i\pi r}{\log q},\quad r\in\Z.\]
\end{thlist}
\end{cor}

\begin{proof} Serre suggère une démonstration utilisant des fibrations en courbes; procédons différemment, en appliquant le corollaire \ref{c:lw2}. On ramène d'abord a) à b) en raisonnant par récurrence sur $n$ comme dans le point 1 de la preuve du théorème \ref{t:abs}. Dans le cas (ii) de b), l'énoncé résulte immédiatement du corollaire \ref{c:lw2} en utilisant la formule de la proposition \ref{p4} c): plus précisément, celle-ci montre comme dans \cite[p. 826]{lang-weil} que la série $Z(V,t)(1-q^nt)$ converge absolument pour $|t|< q^{-(n-\frac{1}{2})}$. 

Dans le cas (i), soit $K$ la fermeture algébrique de $\Q$ dans $E$: d'après \cite[Prop. 9.7.8]{EGA4}\footnote{Je remercie Colliot-Thélène pour cette référence.}, les fibres du morphisme $V\to \Spec O_K$ sont géométriquement irréductibles de dimension $n-1$,
 à l'exception d'un nombre fini d'entre elles. On déduit alors du corollaire \ref{c:lw2} et de l'abscisse de convergence absolue de $\zeta(O_K,s)$ (\S \ref{s:dedekind}) que $\zeta(V,s)\zeta(O_K,s-n+1)$ converge absolument pour $\Re(s)> n-\frac{1}{2}$, ce qui donne la conclusion de nouveau grâce à \ref{s:dedekind}.
\end{proof}

\section{Les conjectures de Weil}\label{s3}\index{Conjectures!de Weil}

\subsection{Des courbes aux vari\'et\'es ab\'eliennes}

Rappelons bri\`evement les points les plus importants de la th\'eorie des vari\'et\'es ab\'eliennes (\cite{weil2bis}, \cite{lang-ab}, \cite{mumford}, \cite{milne-ab1,milne-ab2}): pour un aper\c cu historique, voir les notes bibliographiques de Milne \`a la fin de \cite{milne-ab2} et les commentaires de Weil dans ses \oe uvres scientifiques.

\subsubsection{G\'en\'eralit\'es}

\begin{defn}\phantomsection Soit $k$ un corps. Une \emph{vari\'et\'e ab\'elienne} sur $k$ est un $k$-groupe alg\'ebrique, propre et int\`egre.
\end{defn}

Notons que int\`egre $\Rightarrow$ g\'eom\'etriquement int\`egre (car $A$ admet un point rationnel, l'\'el\'ement neutre) $\Rightarrow$ lisse (car l'ouvert de lissit\'e de $A$ est non vide, et on utilise des translations pour recouvrir $A$.)

\begin{thm}\phantomsection a) Une vari\'et\'e ab\'elienne est projective et sa loi de groupe est commutative.\\
b) Soit $f:A\to B$ un $k$-morphisme entre vari\'et\'es ab\'eliennes. Pour $b\in B(k)$, notons $t_b$ la translation (\`a gauche ou \`a droite) par $b$. Alors $t_{-f(0)}\circ f:a\mapsto f(a)-f(0)$ est un homomorphisme.
\end{thm}

\begin{defn}\phantomsection Soient $A,B$ deux vari\'et\'es ab\'eliennes sur $k$. Une \emph{isog\'enie} de $A$ vers $B$ est un homomorphisme $f:A\to B$, surjectif [fid\`element plat] et de noyau fini.
\end{defn}

\begin{thm}\phantomsection Soit $A$ une vari\'et\'e ab\'elienne de dimension $g$. Soit $l$ un nombre premier diff\'erent de $\car k$, et soit $\bar k$ une cl\^oture alg\'ebrique de $k$. Alors $A(\bar k)$ est divisible et, pour tout $n\ge 1$,  le groupe ${}_{l^n}A(\bar k)$ est (non canoniquement) isomorphe \`a $(\Z/l^n)^{2g}$. On note 
\[T_l(A) = \plim {}_{l^n} A(\bar k)\]
son \emph{module de Tate}: c'est un $\Z_l$-module libre de rang $2g$.\footnote{Si $l=p=\car k$, ${}_{p^n}A(\bar k)\simeq (\Z/p^n)^r$ avec $0\le r\le g$. De plus, ${}_{p^n}A$ a une structure de \emph{$k$-sch\'ema en groupe fini} de $\Z/p^n$-rang $2g$: voir \cite[fin \S 6 et \S 15, ``The $p$-rank'']{mumford}.\label{fn:8}}
\end{thm}

\begin{thm}\phantomsection Soient $A,B$ deux vari\'et\'es ab\'eliennes sur $k$. Alors $\Hom_k(A,B)$ est un $\Z$-module libre de type fini, et l'homomorphisme
\[\Hom_k(A,B)\otimes \Z_l\to \Hom_{\Z_l}(T_l(A),T_l(B))\]
est injectif pour tout $l\ne \car k$. De plus, l'alg\`ebre
\[\End^0(A)=\End(A)\otimes \Q\]
est semi-simple.
\end{thm}

\begin{cor}\phantomsection Tout endomorphisme  surjecitf d'une vari\'et\'e ab\'elienne est une isog\'enie.
\end{cor}

\begin{cor}[Th\'eor\`eme de compl\`ete r\'eductibilit\'e de Poincar\'e]\phantomsection Soit $A$ une $k$-vari\'et\'e ab\'elienne, et soit $B$ une sous-$k$-vari\'et\'e ab\'elienne de $A$. Alors il existe une sous-$k$-vari\'et\'e ab\'elienne $C$ de $A$ telle que $B+C=A$ et que $B\cap C$ soit fini.
\end{cor}

Sur $k=\C$, tous ces r\'esultats sont faciles \`a obtenir \`a partir de la description transcendante d'une vari\'et\'e ab\'elienne: quotient de $\C^g$ par un r\'eseau avec formes de Riemann. Sur un corps quelconque (surtout de caract\'eristique $>0$), ils sont nettement plus difficiles, et d\^us \`a Weil \cite{weil2bis} (sauf celui de la note \ref{fn:8}; la projectivit\'e est \'egalement due ind\'ependamment \`a Matsusaka et \`a Barsotti).

Pour \^etre complet, citons:

\begin{thm}[Tate, Zarhin, Mori, Faltings] \phantomsection\label{tt}\index{Conjecture!de Tate} Soit $k$ un corps de type fini (sur son sous-corps premier), et soient $A,B$ deux $k$-vari\'et\'es ab\'eliennes. Notons $k_s$ une cl\^oture s\'eparable de $k$ et $G=\Gal(k_s/k)$ le groupe de Galois absolu de $k$. Enfin, soit $l$ un nombre premier diff\'erent de $\car k$. Alors l'homomorphisme
\[\Hom_k(A,B)\otimes \Z_l\to \Hom_{\Z_l}(T_l(A),T_l(B))^G\]
est un isomorphisme.
\end{thm}

Lorsque $k$ est fini, ceci donne une classification compl\`ete des vari\'et\'es ab\'eliennes sur $k$, voir \cite{tate-fini} ou \cite[App. 2, th. 2 et 3]{mumford}.

\subsubsection{Le polyn\^ome caract\'eristique d'un endomorphisme}

\begin{defn}\phantomsection Soit $f\in \End(A)$. On note  $\deg(f)$ le degr\'e de $f$ si $f$ est surjectif, et $0$ sinon.
\end{defn}

\begin{ex}\phantomsection $\deg(n 1_A)=n^{2g}$.
\end{ex}

\begin{thm}\phantomsection\label{t3.2} Soit $A$ une vari\'et\'e ab\'elienne de dimension $g$, et soit $f\in \End(A)$. Alors 
\[\det T_l(f) = \deg(f) \]
pour tout $l\ne \car k$. Par cons\'equent,
\[\deg(n1_A-f) = P(n)\quad\text{pour tout } n\in\Z\]
o\`u $P$ est le polyn\^ome caract\'eristique de $T_l(f)$.\\
En particulier, $P$ est \`a coefficients entiers, de degr\'e $2g$, et ind\'ependant de $l$.
\end{thm}

\begin{proof} Voir \cite[\S 19, th. 4]{mumford} ou \cite[prop. 12.9]{milne-ab1}.
\end{proof}

\subsubsection{Vari\'et\'es de Picard; dualit\'e}

\begin{thm}\phantomsection\label{t3.5} Soit $X$ une $k$-vari\'et\'e projective lisse. Si $k$ est alg\'ebriquement clos, le groupe de Picard $\Pic(X)$ se d\'ecrit comme une extension
\[0\to \Pic^0(X)\to \Pic(X)\to \NS(X)\to 0\]
o\`u $\NS(X)$, le groupe de N\'eron-Severi de $X$, est un groupe ab\'elien de type fini et $\Pic^0(X)=\Pic^0_{X/k}(k)$ est le groupe des $k$-points d'une vari\'et\'e ab\'elienne, la \emph{vari\'et\'e de Picard} de $X$.\\
Si $k$ est quelconque, $\Pic^0_{X_{\bar k}/\bar k}$ est d\'efinie sur $k$ (notation: $\Pic^0_{X/k}$).\\
Si $X$ est une courbe de genre $g$, $\Pic^0_{X/k}$ est de dimension $g$: on l'appelle la \emph{jacobienne} de $X$ et on la note (ici) $J(X)$.
\end{thm}

\begin{defn}\phantomsection Soit $A$ une vari\'et\'e ab\'elienne: on note $A^*$\footnote{ou parfois $A'$, $\hat A$, $A^\vee$\dots} sa vari\'et\'e de Picard: c'est la \emph{vari\'et\'e duale} de $A$.
\end{defn}

Une justification partielle de la terminologie est que tout homomorphisme $f:A\to B$ induit un homomorphisme dual $f^*:B^*\to A^*$.

\begin{thm}[bidualit\'e]\phantomsection On a $A^{**}=A$. Si $X$ est une courbe, $J(X)$ est canoniquement isomorphe \`a sa duale. (Voir exemple \ref{r3.2}.)
\end{thm}

\begin{cor}\phantomsection\label{c3.1} Si $f:A\to B$ est une isog\'enie, $f^*$ est aussi une isog\'enie, et $\deg(f^*)=\deg(f)$.
\end{cor}

Plus pr\'ecis\'ement, les sch\'emas en groupes $\Ker f$ et $\Ker f^*$ sont duals l'un de l'autre. La dualit\'e en question est la dualit\'e de Cartier pour les sch\'emas en groupes finis: le dual d'un tel sch\'ema en groupes $G$ est $G^*=\Hom(G,\G_m)$. Voir \cite[\S 15, th. 1]{mumford} ou \cite[\S 11]{milne-ab1} pour la preuve.

La meilleure explication de la dualit\'e des vari\'et\'es ab\'eliennes est \emph{via} la th\'eorie des biextensions (Mumford, Grothendieck \cite[exp. VII-VIII]{SGA7}), qui conduit \`a la notion de $1$-motif (Deligne \cite[\S 10]{hodgeIII}).

\begin{defn}\phantomsection\label{d3.2} Soit $l$ un nombre premier diff\'erent de $\car k$. L'accouplement
\[T_l(A)\times T_l(A^*)\to \Z_l(1):=\plim \mu_{l^n}\]
obtenu en appliquant le corollaire \ref{c3.1} aux groupes ${}_{l^n} A=\Ker(l^n)$ et \`a leurs duaux s'appelle l'\emph{accouplement de Weil}.
\end{defn}

\subsubsection{Morphismes vers une vari\'et\'e ab\'elienne}

\begin{thm}\phantomsection\label{t3.3} Soient $X$ une $k$-vari\'et\'e lisse, $U$ un ouvert dense de $X$ et $f:U\to A$ un morphisme de $U$ vers une vari\'et\'e ab\'elienne. Alors $f$ s'\'etend (de mani\`ere unique) \`a $X$.
\end{thm}

\begin{proof} Voir \cite[th. 3.1]{milne-ab1}.
\end{proof}

\subsubsection{Vari\'et\'e d'Albanese}

La vari\'et\'e d'Albanese d'une vari\'et\'e projective lisse $X/k$ est une vari\'et\'e ab\'elienne qui est, approximativement, universelle pour les morphismes de $X$ vers une vari\'et\'e ab\'elienne. Malheureusement le probl\`eme universel juste \'enonc\'e est incorrect, et donner le bon \'enonc\'e est un peu plus d\'esagr\'eable.

Dans le cas o\`u $X$ a un point rationnel $x_0$, l'\'enonc\'e correct est simple: les $k$-morphismes $f:X\to A$ tels que $f(x_0)=0$. Quel est le bon \'enonc\'e quand $X$ n'a pas n\'ecessairement de point rationnel?

En voici deux, \'equivalents:

\begin{defn}\phantomsection\label{d3.1}
Soit $X$ une $k$-vari\'et\'e projective lisse, et soit $A$ une $k$-vari\'et\'e ab\'elienne. Un morphisme $f:X\times_kX\to A$ est  \emph{admissible} si $f(\Delta_X)=0$.
\end{defn}

(Vu le th\'eor\`eme \ref{t3.3}, cela revient \`a \'enoncer le probl\`eme universel ci-dessus ``sur $K=k(X)$'' avec le point g\'en\'erique de $X$, devenu point $K$-rationnel.)

\begin{thm}\phantomsection Le probl\`eme universel donn\'e par la d\'efinition \ref{d3.1} a une solution $\Alb(X)$.
\end{thm}

Pour \'enoncer le second probl\`eme universel, nous avons besoin de la notion de \emph{torseur} (ou \emph{espace homog\`ene principal}) sous un sch\'ema en groupes:

\begin{defn}\phantomsection Soit $S$ un sch\'ema, et soit $G$ un $S$-sch\'ema en groupes.\\
a) Le \emph{$G$-torseur trivial} sur $S$ est $G$, muni de l'action \`a gauche de $G$.\\
b) Un \emph{$G$-torseur de base $S$} est un $S$-sch\'ema fid\`element plat $E$, muni d'une action de $G$ (au-dessus de $S$)
\[\mu:G\times_S E\to E\]
tel que le diagramme
\[\begin{CD}
G\times_S E@>\mu>> E\\
@V{p_2}VV @VVV\\
E@>>> S
\end{CD}\]
soit cart\'esien.
\end{defn}

(La condition ``cart\'esien'' signifie que le morphisme $\phi:G\times_S E\to E\times_S E$ d\'efini par $p_2$ et  $\mu$ est un isomorphisme: autrement dit, le pull-back du $G$-torseur $E$ de $S$ \`a $E$ est trivial.)

\begin{thm}\phantomsection \label{t3.1} Soit $X$ une $k$-vari\'et\'e projective lisse. Alors le probl\`eme universel d\'efini par les triplets $(A,P,f)$ o\`u $A$ est un $k$-vari\'et\'e ab\'elienne, $P$ est un $A$-torseur de base $k$ et $f:X\to P$ est un $k$-morphisme, a une solution $(\Alb(X),P_X,f_X)$.
\end{thm}

\begin{ex}\phantomsection Partant d'une vari\'et\'e ab\'elienne $A$ et d'un $A$-torseur $P$ de base $k$, on a $\Alb(P)=A$, $P_P =P$ et $f_P=1_P$.
\end{ex}

\begin{thm}[dualit\'e Picard-Albanese]\phantomsection a) Si $P$ est un torseur sous la vari\'et\'e ab\'elienne $A$, $\Pic^0_{P/k}$ est canoniquement isomorphe \`a $\Pic^0_{A/k}=A^*$.\\
b) Soient $X$, $(\Alb(X),P_X,f_X)$ comme dans le th\'eor\`eme \ref{t3.1}. Alors l'homomorphisme
\[f_P^*:\Pic(P)\to \Pic(X)\]
induit un isomorphisme
\[\Alb(X)^*\simeq \Pic^0_{P_X/k}\iso \Pic^0_{X/k}.\]
\end{thm}

\begin{thm}\phantomsection Toute vari\'et\'e ab\'elienne $A$ est quotient de la jacobienne d'une courbe.
\end{thm}

\begin{proof}[Esquisse] (\cf \cite[\S 10]{milne-ab2}.) On choisit un plongement projectif $A\inj \P^N$, donn\'e par un fibr\'e en droites tr\`es ample sur $A$. Si $k$ est infini, le th\'eor\`eme de Bertini fournit une section hyperplane multiple $C$ de $A$, passant par $0$, qui soit une courbe projective, lisse et connexe. [Si $k$ est fini, on arrive aussi \`a une telle courbe $C$ d\'efinie sur $k$ quitte \`a remplacer $\P^N$ par un plongement de V\'eron\`ese, c'est-\`a-dire \`a couper $A$ par une intersection compl\`ete de multidegr\'e assez grand.] On obtient alors un homomorphisme
\[J(C)=\Alb(C)\to A\]
et un lemme de connexit\'e \cite[lemma 10.3]{milne-ab2} implique qu'il est surjectif.
\end{proof}

\subsubsection{Polarisations; l'involution de Rosati}

Soit $A$ une vari\'et\'e ab\'elienne et soit $\sL$ un fibr\'e en droites sur $A$. On lui associe l'application
\begin{align*}
\phi_\sL:A(k)&\to \Pic(A)\\
x&\mapsto t_x^*\sL \otimes \sL^{-1}.
\end{align*}

Le th\'eor\`eme du carr\'e (\cite[\S 6, cor. 4]{mumford} ou \cite[th. 6.7]{milne-ab1}) implique que $\phi_\sL$ est un homomorphisme, qui tombe dans $\Pic^0(A)$. On montre que $\phi_\sL$ d\'efinit un homomorphisme de vari\'et\'es ab\'eliennes $A\to A^*$, qui est une isog\'enie si (et seulement si) $\sL$ est ample.\footnote{C'est une version quelque peu distordue de la th\'eorie: on utilise $\phi_\sL$ pour $\sL$ ample pour construire la vari\'et\'e ab\'elienne duale $A^*$ comme quotient de $A$ par un sous-sch\'ema en groupes fini, \cf \cite[p. 124]{mumford} ou \cite[\S\S 9, 10]{milne-ab1}.}

En particulier, $A$ et $A^*$ sont toujours isog\`enes.

\begin{defn}\phantomsection Une \emph{polarisation} de $A$ est une isog\'enie $\lambda:A\to A^*$ telle que $\lambda_{\bar k}$ soit de la forme $\phi_\sL$ pour un fibr\'e en droites $\sL$ sur $A_{\bar k}$ (ce fibr\'e n'est pas n\'ecessairement d\'efini sur $k$). On dit qu'une polarisation est \emph{principale} si c'est un isomorphisme.
\end{defn}

\begin{ex}\phantomsection\label{r3.2} Supposons que $A=J(C)$, où $C$ est une courbe projective lisse de genre $g>0$ munie d'un point rationnel $c$; soit $\iota:C\to A$ le plongement associé (en considérant $A$ comme la variété d'Albanese de $C$). On pose 
\[\Theta_C=\iota(C)+\dots+\iota(C) \quad (g-1 \text{ termes})\]
où la somme est relative à l'addition de $A$: c'est le \emph{diviseur theta associé}. 
Alors $\Theta_C$ définit une polarisation principale sur $A$, indépendante du choix de $c$   \cite[\S 6]{milne-ab2}.
\end{ex}

\begin{defn}\phantomsection Soit $\lambda$ une polarisation de $A$. L'\emph{involution de Rosati} associ\'ee \`a $\lambda$ est l'application
\[f\mapsto {}^\rho f=\lambda^{-1}f^*\lambda\]
de $\End^0(A)$ dans lui-m\^eme.
\end{defn}

Il est clair que l'involution de Rosati est un anti-automorphisme de $\End^0(A)$; de plus:

\begin{prop}\phantomsection Soit $\lambda:A\to A^*$ une polarisation. Notons
\[E^\lambda:T_l(A)\times T_l(A)\to \Z_l(1)\]
l'accouplement d\'eduit de l'accouplement de Weil (d\'efinition \ref{d3.2}) via l'homomorphisme $T_l(\lambda):T_l(A)\to T_l(A^*)$. Alors:
\begin{thlist}
\item $E^\lambda$ est antisym\'etrique.
\item Pour tout $f\in \End^0(A)$, $f$ et ${}^\rho f$ sont adjoints pour $E^\lambda\otimes \Q_l$.
\item L'involution de Rosati est une involution.
\end{thlist}
\end{prop}

\begin{proof} Voir \cite[\S 20, th. 1 et ff.]{mumford} ou \cite[d\'ebut \S 17]{milne-ab1}.
\end{proof}

Soit $R$ une alg\`ebre semi-simple (de dimension finie) sur un corps $K$. On d\'efinit une trace $\Tr:R\to K$ comme suit:
\begin{enumerate}
\item Si $R$ est simple et de centre $Z$, $\Tr(x)=\Tr_{Z/K}\Trd(x)$, o\`u $\Trd$ est la trace r\'eduite de $R$ (sur $Z$).
\item Si $R=R_1\times\dots \times R_n$, o\`u les $R_i$ sont simples, 
\[\Tr(x_1,\dots,x_n) =\Tr(x_1)+\dots +\Tr(x_n).\]
\end{enumerate}

 Le r\'esultat principal est:

\begin{thm}[Positivit\'e de l'involution de Rosati]\phantomsection \label{t:rosati} Soit $\lambda:A\to A^*$ une polarisation. Alors la forme quadratique 
\[\End^0(A)\ni f\mapsto \Tr({}^\rho f f)\]
est d\'efinie positive. Si $\lambda_{\bar k} = \phi_\sL$ et si $f\in \End(A)$, on a
\[\Tr({}^\rho f f)=\frac{2g}{(\sL)^g}((\sL)^{g-1}\cdot f^*\sL).\]
\end{thm}

\begin{proof} Voir \cite[\S 21]{mumford} ou \cite[th. 17.3]{milne-ab1}.
\end{proof}

\begin{cor}\phantomsection Soit $\lambda:A\to A^*$ une polarisation; soit $f\in \End(A)$ tel que ${}^\rho ff=a\in \Z$. Soient $\alpha_1,\dots,\alpha_{2g}$ les racines (dans $\C$) du polyn\^ome caract\'eristique de $f$ (\cf th\'eor\`eme \ref{t3.2}). Alors la sous-alg\`ebre de $\End^0(A)$ engendr\'ee par $f$ est semi-simple, et
\begin{thlist}
\item $|\alpha_i|^2=a$ pour tout $i$;
\item l'application $\alpha_i\mapsto a/\alpha_i$ est une permutation des $\alpha_i$.
\end{thlist}
\end{cor}

\begin{proof} C'est une cons\'equence presque imm\'ediate du th\'eor\`eme de positivit\'e, voir \cite[\S 21, Application II]{mumford} ou \cite[lemma 19.3]{milne-ab1}.
\end{proof}

\subsection{L'hypoth\`ese de Riemann pour une vari\'et\'e ab\'elienne}\label{s:hrab}\index{Hypothèse de Riemann!pour les variétés abéliennes}

\subsubsection{La fonction z\^eta d'une vari\'et\'e ab\'elienne}\index{Fonction zêta!d'un schéma de type fini sur $\Z$}

On suppose maintenant que $k=\F_q$ est fini. Soient $A$ une $k$-vari\'et\'e ab\'elienne et $\pi_A$ son endomorphisme de Frobenius. Alors $\pi_A$ agit par $0$ sur l'espace tangent en $0$, donc $1-\pi_A^n$ est \'etale pour tout $n>0$ et
\[A(\F_{q^n})=\Ker(1-\pi_A^n)\]
donc
\begin{equation}\label{eq3.1}
|A(\F_{q^n})| = P^{(n)}(1) = \prod_{i=1}^{2g} (1-\alpha_i^n)
\end{equation}
o\`u $P^{(n)}$ est le polyn\^ome caract\'eristique de $\pi_A^n$, de racines $\alpha_1^n,\dots,\alpha_{2g}^n$.

\begin{lemme}\phantomsection Pour toute involution de Rosati, on a ${}^\rho\pi_A\pi_A=q$.
\end{lemme}

\begin{proof} Si $\lambda$ est une polarisation, on peut montrer (parce que $k$ est fini) que $\lambda=\phi_\sL$ pour un fibr\'e en droites $\sL$ d\'efini sur $k$. Alors on a $\pi_A^*\sL\simeq \sL^q$ et le lemme r\'esulte d'un petit calcul (\cf \cite[lemme 19.2]{milne-ab1}).
\end{proof}

(On pourra s'amuser \`a comparer les d\'emonstrations donn\'ees dans \cite{lang-ab}, \cite{mumford} et \cite{milne-ab1}, qui sont de plus en plus simples\dots)

\begin{thm}[Weil]\phantomsection\label{t3.4} a) Les $\alpha_i$ v\'erifient $|\alpha_i|=\sqrt{q}$ pour tout $i$, et $\alpha_i\mapsto q/\alpha_i$ est une permutation des racines de $P^{(1)}$.\\
b) Pour $n\in \{0,\dots,2g\}$, soit $P_n$ le polyn\^ome dont les racines inverses sont les produits $\alpha_{i_1}\dots\alpha_{i_n}$ pour $i_1<\dots<i_n$ (donc $P_1=P^{(1)}$). Alors on a la factorisation
\[Z(A,t) = \frac{P_1(t)\dots P_{2g-1}(t)}{P_0(t)\dots P_{2g}(t)}\]
et l'\'equation fonctionnelle\index{Equation fonctionnelle@\'Equation fonctionnelle}
\[Z(A,1/q^gt) = Z(A,t).\]
\end{thm}

\begin{proof} C'est imm\'ediat \`a partir de \eqref{eq3.1} et de la formule de la proposition \ref{p4} c).
\end{proof}

\begin{rques}\phantomsection\label{r3.3} a) Avant l'invention des modules de Tate, Hasse, Deuring et Weil raisonnaient en termes de ``matrices $l$-adiques'' d\'eduites de l'action des endomorphismes de $A$ sur ses points de torsion, en utilisant que
\[\End_\Z(\Q_l/\Z_l)= \Z_l.\]
b) Quand on conna\^\i tra l'interpr\'etation de la factorisation du th\'eor\`eme \ref{t3.4} b) en termes d'une cohomologie de Weil $H$ \index{Cohomologie!de Weil}\eqref{eq3.4}, elle s'expliquera de la mani\`ere suivante: les $\alpha_i$ sont les valeurs propres de l'action de Frobenius sur $H^1(A)$, et $H^*(A)$ est une alg\`ebre ext\'erieure sur $H^1(A)$ (ce qui se d\'eduit pour toute cohomologie de Weil de la structure du motif de $A$, \cite{kunnemann}).
\end{rques}

\subsubsection{Raccord avec le cas des courbes} Soient $C,C'$ deux $k$-courbes projectives lisses. Les correspondances op\`erent sur les groupes de Picard par l'op\'eration
\begin{align*}
\Corr(C,C')\times \Pic(C)&\to \Pic(C')\\
(\gamma,D)&\mapsto (p^{C\times C'}_{C'})_*(\gamma\cap D\times C')
\end{align*}
(si le produit d'intersection est d\'efini; dans le cas exceptionnel $\gamma=P\times C'$, le second membre est d\'efini comme $0$). C'est d'ailleurs ainsi que Weil introduit la composition des correspondances.

Cette action envoie $\Pic^0(C)$ dans $\Pic^0(C')$ et induit en fait un homomorphisme
\[\Corr_\equiv(C,C')\to \Hom_k(J(C),J(C')).\]

\begin{thm}[Weil]\phantomsection\label{t3.7} Cet homomorphisme est un i\-so\-mor\-phis\-me.
\end{thm}

\begin{proof} Voir \cite[ch. VI, th. 22]{weil2bis} ou \cite[cor. 6.3]{milne-ab2}.
\end{proof}

Lorsque $C'=C$, notons $J=J(C)$ et $\lambda=J\iso J^*$ la polarisation principale canonique de l'exemple \ref{r3.2}. On peut montrer que l'isomorphisme
\[\Corr_\equiv(C,C)\iso \End(J)\]
transforme l'involution canonique du membre de gauche (transposition) en l'involution de Rosati associ\'ee \`a $\lambda$, et que la trace $\sigma$ de la d\'efinition \ref{d:trace} correspond \`a la trace $\Tr$ de l'alg\`ebre semi-simple $\End^0(J)$ \cite[ch. VI, \S 3, th. 6]{lang-ab}. Ainsi, on retrouve le th\'eor\`eme \ref{t:lemmeimportant} \`a partir du th\'eor\`eme \ref{t:rosati}. On voit aussi que le num\'erateur de $Z(C,t)$ est le polyn\^ome caract\'eristique inverse de $\pi_{J}$.

%\begin{exo}[\protect{\cite[cor. 3 au th. 1]{weil-criteres}}]\phantomsection\label{exo3.1} Soit $\gamma$ comme dans le théorème \ref{t:lemmeimportant}, vu comme homomorphisme de $J(C)$ vers $J(C')$ via le théorème \ref{t3.7}. Montrer que la restriction de ${}^t\gamma$ à $\gamma J(C)$ a un noyau fini, donc définit une isogénie de $\gamma J(C)$ sur son image. (Comparer les composantes neutres des noyaux de $\gamma$ et de ${}^t\gamma\gamma$: si la seconde était différente de la première, elle contiendrait une courbe non contenue dans la première; utiliser le théorème \ref{t:lemmeimportant} pour en tirer une contradiction.) 
%\end{exo}

\subsection{Les conjectures de Weil}\label{s3.3}\index{Conjectures!de Weil}

Elles sont \'enonc\'ees \`a la fin de \cite{weil4}: 

\begin{quote} \it
This, and other examples which we cannot discuss here, seem to lend some support to the following conjectural statements, which are known to be true for curves, but which I have not so far been able to prove for varieties of higher dimension.

Let $V$ be a variety\footnote{Curieusement, le mot ``projective'' est omis.} without singular points, of dimension $n$, defined over a finite field $k$ with $q$ elements. Let $N_\nu$ be the number of rational points on $V$ over the extension $k_\nu$ of $k$ of degree $\nu$. Then we have
\[\sum_1^\infty N_\nu U^{\nu-1} = \frac{d}{dU} \log Z(U)\]
where $Z(U)$ is a rational function in $U$, satisfying a functional equation
\[Z(\frac{1}{q^nU}) = \pm q^{n\chi/2} U^\chi Z(U)\]
with $\chi$ equal to the Euler-Poincar\'e characteristic of $V$ (intersection number of the diagonal with itself on the product $V\times V$).

Furthermore, we have:
\[Z(U) = \frac{P_1(U)P_3(U)\dots P_{2n-1}(U)}{P_0(U)P_2(U)\dots P_{2n}(U)}\]
with $P_0(U)= 1-U$, $P_{2n}(U)=1-q^nU$, and, for $1\le h\le 2n-1$:
\[P_h(U)= \prod_{i=1}^{B_h} (1-\alpha_{hi} U)\]
where the $\alpha_{hi}$ are algebraic integers of absolute value $q^{h/2}$.\index{Nombres!de Weil}

Finally, let us call the degrees $B_h$ of the polynomials $P_h(U)$ the \emph{Betti numbers}\index{Nombres!de Betti} of the variety $V$; the Euler-Poincar\'e characteristic $\chi$ is then expressed by the usual formula $\chi=\sum_h (-1)^hB_h$. The evidence at hand seems to suggest that, if $\overline V$ is a variety without singular points, defined over a field $K$ of algebraic numbers, the Betti numbers of the varieties $V_\fp$, derived from $\overline V$ by reduction modulo a prime ideal $\fp$ in $K$, are equal to the Betti numbers of $\overline V$ (considered as a variety over complex numbers) in the sense of combinatorial topology, for all except at most a finite number of prime ideals $\fp$. For instance, consider the Grassmannian variety\dots
\end{quote}

Le cas des vari\'et\'es grassmanniennes para\^\i t maintenant facile dans la mesure o\`u ce sont des vari\'et\'es cellulaires. Curieusement, Weil ne cite pas les variétés abéliennes. L'exemple dont il part au d\'ebut de son article est celui des hypersurfaces de Fermat g\'en\'eralis\'ees:
\begin{equation}\label{eq:fermat}
H: \quad a_0x_0^n+\dots+a_rx_r^n=0
\end{equation}
pour lesquelles il prouve la conjecture en utilisant les sommes de Jacobi. Il apprend de Pierre Dolbeault la valeur des nombres de Betti de $H$ vue comme vari\'et\'e complexe, ce qui l'encourage \`a formuler sa conjecture. 

\begin{exo}[\cf \protect{\cite[pp. 826--827]{lang-weil}}]\phantomsection Montrer que les con\-jec\-tu\-res de Weil impliquent que la plus petite constante $\gamma$ apparaissant dans le corollaire \ref{c:lw2} est égale au degré de $P_{2n-1}(U)$. (Ce degré est égal à $2g$, où $g$ est la dimension de la variété de Picard de $V$, \cf remarques \ref{r3.3} b) et \ref{r3.4}.)
\end{exo}

Dans le commentaire de ses \oe uvres scientifiques sur cet article, Weil \'ecrit:\index{Conjectures!de Weil}

\begin{quote} \it Je crois que je fis un raisonnement heuristique bas\'e sur la formule de Lefschetz. En g\'eom\'etrie alg\'ebrique classique, soit $\phi$ une application g\'en\'eriquement surjective de degr\'e $d$ d'une vari\'et\'e $V$ de dimension $n$, lisse et compl\`ete, dans elle-m\^eme. Pour chaque $\nu$, soit $N_\nu$ le nombre de solutions, suppos\'e fini, de $P=\phi^\nu(P)$, ou pour mieux dire, le nombre d'intersection avec la diagonale, sur $V\times V$, du graphe de la $n$-i\`eme it\'er\'ee de $\phi$. Des raisonnements de topologie combinatoire classique font voir alors que la fonction $Z(U)$ d\'efinie par (I)\footnote{Il s'agit de la formule de la proposition  \ref{p4} c).} satisfait \`a une \'equation fonctionnelle\index{Equation fonctionnelle@\'Equation fonctionnelle}
\[Z(d^{-1}U^{-1})= \pm (\sqrt{d}U)^\chi Z(U)\]
o\`u $\chi$ est la caract\'eristique d'Euler-Poincar\'e de $V$, ou autrement dit (comme je le rappelais dans \cite[p. 507]{weil4}) le nombre d'intersection de la diagonale avec elle-m\^eme sur $V\times V$. De l\`a \`a esp\'erer que la m\^eme formule pouvait s'appliquer, sur le corps de base fini $\F_q$, \`a l'application de Frobenius (pour laquelle on a $d=q^n$), et que mes calculs donnaient donc la valeur de $\chi$ pour l'hypersurface \eqref{eq:fermat}, il n'y avait qu'un pas \`a faire. Quand le r\'esultat de Dolbeault m'eut permis de v\'erifier ce point, l'inspection de mes formules sugg\'era de lui-m\^eme les conjectures sur les nombres de Betti\index{Nombres!de Betti} et l'``hypoth\`ese de Riemann''\index{Hypothèse de Riemann!pour les variétés projectives lisses}; la comparaison avec les r\'esultats connus sur les courbes, ainsi que l'examen des Grassmanniennes et de quelques autres exemples simples, ne tarda pas \`a les confirmer.
\end{quote}

Weil ne le savait pas \`a l'\'epoque, mais les hypersurfaces de Fermat sont ``de type ab\'elien'': leur motif s'exprime en termes de motifs de vari\'et\'es ab\'eliennes (Katsura-Shioda \cite{ks}).

\subsection{Cohomologies de Weil} Commen\c cons par deux citations de Serre:\index{Conjectures!de Weil}

\begin{quote} {\it \dots the idea of counting points over $\F_q$ by a Lefschetz formula is entirely \emph{an idea of Weil}. I remember how enthousiastic I was when he explained it to me, and a few years later I managed to convey my enthousiasm to Grothendeick (whose taste was not a priori directed towards finite fields).} (Lettre \`a S. Kleiman, 25 mars 1992, cit\'ee dans \cite[p. 8, note 3]{kleiman-standard}.)
\end{quote}

\begin{quote} \it \dots Weil formule ce qu'on a tout de suite appel\'e les \emph{conjectures de Weil}. Ces conjectures portent sur les vari\'et\'es (projectives, non singuli\`eres) sur un corps fini. Elles reviennent \`a supposer que les m\'ethodes topologiques de Riemann, Lefschetz, Hodge\dots\ s'appliquent en caract\'eristique $p>0$; dans cette optique, le nombre de solutions d'une \'equation $\pmod{p}$ appara\^\i t comme un nombre de points fixes, et doit donc pouvoir \^etre calcul\'e par la formule des traces de Lefschetz. Cette id\'ee, vraiment r\'evolutionnaire, a enthousiasm\'e les math\'ematiciens de l'\'epoque (je peux en t\'emoigner de premi\`ere main); elle a \'et\'e \`a l'origine d'une bonne partie des progr\`es de la g\'eom\'etrie alg\'ebrique qui ont suivi. \cite[no 5]{serre-weil}.
\end{quote}

Avec le recul, on pourrait consid\'erer cette id\'ee de Weil comme l'anc\^etre le plus direct de la philosophie des motifs!\footnote{On notera que dans \cite{weil4} et dans ses commentaires sur cet article, Weil ne mentionne pas explicitement la formule des traces de Lefschetz, ni même la (co)homologie singulière: il parle seulement de ``raisonnements de topologie combinatoire classique''. La topologie combinatoire est l'ancien nom de la topologie algébrique, terminologie que Bourbaki a adoptée vers 1944.}
\bigskip

Serre et Grothendieck y mettent plus de substance, comme suit:

\begin{defn}\phantomsection\label{d3.4} Soit $k$ un corps; consid\'erons la cat\'egorie $\V(k)$ des vari\'et\'es projectives lisses (morphismes: les $k$-morphismes). Une \emph{cohomologie de Weil} \index{Cohomologie!de Weil}sur $\V(k)$ est la donn\'ee d'un corps $K$ de caract\'eristique z\'ero et d'un foncteur [contravariant]
\[H^*:\V(k)^\op\to \Vec_K^*\]
de $\V(k)$ vers la cat\'egorie des $K$-espaces vectoriels gradu\'es, v\'erifiant les axiomes suivants:
\begin{thlist}
\item Pour tout $X\in \V(k)$ de dimension $n$, les $H^i(X)$ sont de dimension finie, nuls pour $i\notin [0,2n]$.
\item $H^0(\Spec k)=K$.
\item $\dim H^2(\P^1)=1$; on note cet espace vectoriel $K(-1)$.
\item \emph{Additivit\'e}: si $X,Y\in\V(k)$, le morphisme canonique
\[H^*(X\coprod Y)\to H^*(X)\oplus H^*(Y)\]
est un isomorphisme.
\item \emph{Normalisation}: si $X$ est connexe, de corps des constantes $E$, le morphisme canonique
\[H^0(E)\to H^0(X)\]
est un isomorphisme.
\item \emph{Formule de K\"unneth}: pour $X,Y\in \V(k)$, on a un isomorphisme 
\[\kappa_{X,Y}:H^*(X)\otimes H^*(Y)\iso H^*(X\times Y)\]
naturel en $(X,Y)$, qui v\'erifie des conditions \'evidentes d'associativit\'e, d'unit\'e et de commutativit\'e, cette derni\`ere relativement \`a la r\`egle de Koszul.
\item \emph{Trace et dualit\'e de Poincar\'e}: Pour tout $X\in \V(k)$,  purement de dimension $n$, on a un morphisme canonique
\[\Tr_X:H^{2n}(X)\to K(-n):=K(-1)^{\otimes n}\]
qui est un isomorphisme si $X$ est g\'eom\'etriquement connexe, et tel que $\Tr_{X\times Y} = \Tr_X\otimes \Tr_Y$ modulo la formule de K\"unneth; l'accouplement ``de Poincar\'e''
\begin{equation}\label{eq:dP}
H^i(X)\otimes H^{2n-i}(X)\to H^{2n}(X\times X)\by{\Delta_X^*} H^{2n}(X)\by{\Tr_X} K(-n)
\end{equation}
est non d\'eg\'en\'er\'e.
\item \emph{Classe de cycle}: Pour tout $X\in \V(k)$ et tout $i\ge 0$, on a un homomorphisme
\[\cl_X^i:CH^i(X)\to H^{2i}(X)(i):= \Hom_K(K(-i),H^{2i}(X))\]
o\`u $CH^i(X)$ est le groupe des cycles de codimension $i$ sur $X$, modulo l'\'equivalence rationnelle ($i$-\`eme groupe de Chow de $X$). \index{Groupe de Chow} Ces homomorphismes sont contravariants en $X$ et compatibles \`a la formule de K\"unneth ($\kappa_{X,Y}\left(\cl^i_X(\alpha)\otimes \cl^j_Y(\beta)\right)\allowbreak=\cl^{i+j}_{X\times Y}(\alpha\times \beta)$); de plus, si $\dim X=n$, le diagramme
\[\begin{CD}
CH^n(X)@>\cl^n_X>> H^{2n}(X)(n)\\
@V{\deg}VV @V{\Tr_X}VV\\
\Z@>>> K
\end{CD}\]
est commutatif.
\end{thlist}
\end{defn}

\`A l'aide de la formule de K\"unneth et de la diagonale, on d\'efinit le \emph{cup-produit} sur la cohomologie de $X\in \V(k)$:
\begin{align}\label{eq:cupproduit}
\cdot:H^i(X)\times H^j(X)&\to H^{i+j}(X)\\
x\cdot y &= \delta_X^*(\kappa_{X,X}(x\times y))\notag
\end{align}
o\`u $\delta_X:X\to X\times_k X$ est le morphisme diagonal.

\begin{lemme}\phantomsection\label{l3.5} Si $E$ est une extension finie s\'eparable de $k$ de degr\'e $d$, alors $\dim H^0(E)=d$.
\end{lemme}

\begin{proof} a) Soit $\tilde E$ une cl\^oture galoisienne de $E$. En appliquant (iv) et (vi) \`a $E\otimes_k \tilde E\iso \prod_{E\inj_k \tilde E} \tilde E$, on trouve
\[H^0(E)\otimes H^0(\tilde E)\simeq H^0(\tilde E)^d.\]
\end{proof}

\begin{rque}\phantomsection\label{r3.4} Les axiomes varient suivant les auteurs. J'ai essay\'e d'en prendre un syst\`eme minimal, qui soit suffisant pour les besoins qui suivent. En particulier:
\begin{itemize}
\item J'ai ins\'er\'e les axiomes (iv) et (v), que je ne sais pas d\'eduire des autres.\footnote{L'axiome (v) se d\'eduit de (vii) si $H$ se factorise \`a travers le foncteur $\V(k)\to \V(\bar k)$, o\`u $\bar k$ est une cl\^oture alg\'ebrique de $k$; mais la cohomologie de de Rham\index{Cohomologie!de de Rham}, qui v\'erifie (v), n'a pas cette propri\'et\'e.}
\item Je n'ai pas inclus des axiomes ``Lefschetz faible'' et ``Lefschetz fort'' de \cite{kleiman-dix}, qui ne sont pas n\'ecessaires pour d\'emontrer les conjectures de Weil.
\end{itemize}
Un axiome supplémentaire important est le suivant: si $X$ est projective, lisse et géométriquement connexe, ayant un point rationnel $x_0$, le morphisme d'Albanese $X\to \Alb(X)$ associé à $x_0$ induit un isomorphisme
\[H^1(\Alb(X))\iso H^1(X).\]
Cet axiome est vérifié par toutes les cohomologies de Weil classiques, \ie celles qui apparaissent au \S \ref{3.5.9}.
\end{rque}

%Tirons tout de suite une cons\'equence de l'axiome (vii):

%\begin{lemme}\phantomsection\label{l3.2} Si $X$ est g\'eom\'etriquement connexe de dimension $n$, on a $\dim H^0(X)=\dim H^{2n}(X)=1$.\qed
%\end{lemme}

\subsection{Propri\'et\'es formelles d'une cohomologie de Weil}\index{Cohomologie!de Weil} R\'ef\'erence: Kleiman \cite{kleiman-dix}.

\subsubsection{Image directe} 
Soit $f:X\to Y$ un morphisme de $\V(k)$, avec $X,Y$ irr\'eductibles,  $\dim X=m$, $\dim Y=n$. On d\'efinit
\[f_*:H^i(X)\to H^{i+2n-2m}(Y)(n-m)\]
comme le dual de
\[f^*:H^{2m-i}(Y)(m)\to H^{2m-i}(X)(m)\]
via la dualit\'e de Poincar\'e. On a $\Tr_X = (p_X)_*$ o\`u $p_X$ est le morphisme structural de $X$, et la \emph{formule de projection}:
\[f_*(x\cdot f^*y)=f_* x\cdot y.\]

D\'emontrons-la: par définition, $f_*$ est adjoint de $f^*$ pour l'accouplement de Poincar\'e. Si $z$ est un \'el\'ement de $H^*(Y)(*)$ de dimension et poids compl\'ementaire, on a:
\begin{multline*}
<f_*(x\cdot f^*y),z>_Y = <x\cdot f^* y,f^*z>_X=<x,f^*( y\cdot z)>_X\\
=<f_* x, y\cdot z>_Y = <f_* x\cdot y, z>_Y.
\end{multline*}

\subsubsection{Correspondances cohomologiques}

\begin{prop}\phantomsection\label{p3.1} Soient $X,Y\in\V(k)$, de dimensions $m$ et $n$, et soit $r\in\Z$. On a un isomorphisme canonique
\[\Hom^r(H^*(X),H^*(Y))\simeq H^{2m+r}(X\times Y)(m)\]
o\`u $\Hom^r$ signifie ``homomorphismes homog\`enes de degr\'e $r$''.
\end{prop}

\begin{proof} On a:
\begin{multline*}
\Hom^r(H^*(X),H^*(Y))=\prod_{i\ge 0} \Hom(H^i(X),H^{i+r}(Y))\\
\simeq \prod_{i\ge 0} H^i(X)^*\otimes H^{i+r}(Y)\overset{(*)}{\simeq} \prod_{i\ge 0} H^{2m-i}(X)(m)\otimes H^{i+r}(Y)\\
\simeq H^{2m+r}(X\times Y)(m)
\end{multline*}
o\`u on a utilis\'e successivement la dualit\'e des $K$-espaces vectoriels, la dualit\'e de Poincar\'e (axiome (vii)) et la formule de K\"unneth (axiome (vi)).\footnote{Il y a une ambigu\"\i t\'e de signe dans l'isomorphisme $(*)$: on choisit l'isomorphisme $H^i(X)^*\iso H^{2m-i}(X)(m)$ donn\'e par \eqref{eq:dP}.}
\end{proof}

\begin{defn}\phantomsection Une \emph{correspondance cohomologique de degr\'e $r$} de $X^m$ vers $Y$ est un \'el\'ement de 
$H^{2m+r}(X\times Y)(m)$. Notation: $\Corr^r_H(X,Y)$.
Si $X$ n'est pas \'equidimensionnel, on \'etend cette d\'efinition par additivit\'e.
\end{defn}

La proposition \ref{p3.1} d\'efinit, par transport de structure, la \emph{composition des correspondances cohomologiques}. Elle s'explicite ainsi (m\^eme strat\'egie que pour prouver la formule de projection):
\begin{equation}\label{eq:comp-corr}
\beta\circ \alpha = (p_{13})_*(p_{12}^*\alpha \cdot p_{23}^*\beta)
\end{equation}
pour $\alpha\in \Corr^*_H(X_1,X_2)$ et $\beta\in \Corr^*_H(X_2,X_3)$, o\`u $p_{ij}$ d\'esigne la projection de $X_1\times X_2\times X_3$ sur le facteur $X_i\times X_j$. Elles op\`erent sur la cohomologie par la formule
\begin{equation}\label{eq:act-corr}
\alpha(x) = (p_Y)_*(p_X^*x\cdot \alpha)
\end{equation}
pour $\alpha\in \Corr^r_H(X,Y)$, $x\in H^i(X)$, $\alpha(x)\in H^{i+r}(Y)$.

Ainsi, les correspondances cohomologiques agissent sur la cohomologie de mani\`ere \emph{covariante}. Pour la suite, explicitons \eqref{eq:act-corr} dans le cas o\`u $\alpha$ est donn\'e par un tenseur simple $v\otimes w$, $v\in H^{2m-i}(X)(m)$, $w\in H^{i+r}(Y)$ et o\`u $x\in H^i(X)$:
\begin{equation}\label{eq:expl}
\alpha(x) = <x,v>w
\end{equation}
(\cf \cite[1.3]{kleiman-dix}).

\subsubsection{Transposition} La dualit\'e de Poincar\'e donne un isomorphisme
\[\Hom^r(H^*(X),H^*(Y)) \iso \Hom^{2m-2n+r}(H^*(Y),H^*(X))(m-n)\]
($m=\dim X$, $n=\dim Y$), soit un isomorphisme
\begin{multline*}
H^{2m+r}(X\times Y)(m)=\Corr^r_H(X,Y)\\
\iso \Corr^{2m-2n+r}_H(Y,X)(m-n)=H^{2m+r}(Y\times X)(m)
\end{multline*}
appel\'e \emph{transposition des correspondances}. On note $\phi\mapsto {}^t \phi$ cet isomorphisme. Pour la suite, explicitons le transpos\'e d'un tenseur simple $v\otimes w\in H^{2m-i}(X)(m)\otimes H^{i+r}(Y)$ (r\`egle de Koszul):
\begin{equation}\label{eq:expl2}
(v\otimes w)' = (-1)^{(2m-i)(i+r)}w\otimes v = (-1)^{i(i+r)}w\otimes v.
\end{equation}

\begin{ex}\phantomsection Si $f:Y\to X$ est un morphisme, $f^*:H^*(X)\to H^*(Y)$ d\'efinit une correspondance cohomologique, not\'ee $f^*\in \Corr^0_H(X,Y)$. Sa transpos\'ee est (par d\'efinition) $f_*$. On a  les formules \'evidentes $(g\circ f)^*=f^*\circ g^*$ et $(g\circ f)_*=g_*\circ f_*$. En particulier,
\begin{equation}\label{eq3.2}
\Tr_X = \Tr_Y\circ f_*.
\end{equation}
\end{ex}

\begin{lemme}\phantomsection\label{l:tr} Soit $\sigma_{X,Y}:Y\times X\iso X\times Y$ l'\'echange des facteurs. Alors pour tout $\phi\in \Corr^r_H(X,Y)$, on a ${}^t \phi=\sigma_{X,Y}^*\circ \phi$.
\end{lemme}

\begin{proof} M\^eme esprit que celle de la formule de projection (de nouveau, ${}^t\phi$ est par d\'efinition l'adjoint de $\phi$ pour l'accouplement de Poincar\'e).
\end{proof}

\subsubsection{La formule des traces de Lefschetz} C'est la suivante:

\begin{prop}\phantomsection\label{p3.4} Soient $\alpha\in\Corr^r_H(X,Y)$ et $\beta\in \Corr^{-r}_H(Y,X)$. On a:
\[<\alpha,{}^t\beta>_{X\times Y} = \sum_{i\ge 0} (-1)^i \Tr(\beta\circ \alpha\mid H^i(X))\]
o\`u, dans le membre de droite, $\Tr(\beta\circ \alpha\mid H^i(X))$ d\'esigne la trace de $\beta\circ \alpha$ vu comme  endomorphisme de l'espace vectoriel $H^i(X)$.
\end{prop}

Dans le cas $\dim X=\dim Y=1$, cette formule remonte \`a Weil \cite[no 43, 66 et 68]{weil2bis}.

\begin{proof} C'est une trivialit\'e: par la formule de K\"unneth et par bilin\'earit\'e, on peut supposer $\alpha=v\otimes w$, $\beta=w'\otimes v'$ avec  $v\in H^{2m-i}(X)(m)$, $w\in H^{i+r}(Y)$, $w'\in H^{2n-i-r}(Y)(n)$, $v'\in H^i(X)$. Alors $\alpha$ (\resp $\beta$) s'annule en dehors de $H^i(X)$ (\resp de $H^{i+r}(Y)$)  et, pour $x\in H^i(X)$, $y\in H^{i+r}(Y)$, on a par \eqref{eq:expl} 
\[\alpha(x) = <x,v>w, \quad \beta(y) = <y,w'> v'\]
donc
\[\beta\circ \alpha(x) = <x,v><w,w'>v'\]
et 
\[\Tr(\beta\circ \alpha)= <v',v><w,w'>.\]

D'autre part, en utilisant \eqref{eq:expl2}:
\begin{multline*}<\alpha,{}^t\beta>_{X\times Y} = (-1)^{i(i+r)}<v\otimes w,v'\otimes w'>_{X\times Y} \\
=(-1)^{i(i+r)}\Tr_{X\times Y}( v\otimes w\cdot v'\otimes w')\\
=\Tr_{X\times Y}(v\cdot v'\otimes w\cdot w') \quad \text{(deux signes se compensent)}\\
=<v,v'><w,w'> = (-1)^{i(2m-i)} <v',v><w,w'>\\
 = (-1)^i \Tr(\beta\circ \alpha).
\end{multline*}
\end{proof}

\begin{rque}\phantomsection Nous verrons plus tard une autre d\'emonstration infiniment plus conceptuelle de la formule des traces! (théorème  \ref{t6.5} et lemme \ref{lA.6}).
\end{rque}

\begin{cor}\phantomsection\label{c3.2}
Si $E/k$ est finie, galoisienne de groupe $G$, la repr\'esentation
\[G\to \Aut_K H^0(E)\]
donn\'ee par l'action de Galois est la repr\'esentation r\'eguli\`ere de $G$.
\end{cor}

\begin{proof} En effet, soit $g\in G-\{1\}$: alors le graphe de $g:\Spec E\to \Spec E$ est transverse \`a celui de l'identit\'e, donc $\Tr(g)=\break <\Gamma_g,\Delta_{\Spec E}>=0$. Ceci caract\'erise la repr\'esentation r\'eguli\`ere.
\end{proof}

\subsubsection{Groupes de Chow: codimension et dimension} \index{Groupe de Chow} Pour la notion de groupe de Chow, voir Fulton \cite{fulton}. Si $X\in \V(k)$ et $i\ge 0$, on note $CH^i(X)$ (\resp $CH_i(X)$ le groupe des cycles de codimension (\resp de dimension) $i$ sur $X$, modulo l'\'equivalence rationnelle. Ils ont les propri\'et\'es suivantes:

\begin{enumerate}
\item Les $CH^*$ (\resp les $CH_*$) sont des foncteurs contravariants (\resp covariants) sur $\V(k)$.
\item On a des ``produits d'intersection''
\[CH^i(X)\times CH^j(X)\to CH^{i+j}(X).\]
\item On a un homomorphisme degr\'e
\[\deg:CH_0(X)\to \Z.\]
\item Si $X$ est purement de dimension $n$, on a
\[CH^i(X)=CH_{n-i}(X)\]
pour tout $i\in [0,n]$.
\end{enumerate}

La fonctorialit\'e covariante est facile, ainsi que la fonctorialit\'e contravariante pour les morphismes plats; pour les morphismes quelconques, elle est difficile, ainsi que l'existence des produits d'intersection.

\subsubsection{Correspondances alg\'ebriques} 

Au num\'ero \ref{s:hr1}, nous avons introduit les correspondances alg\'ebriques entre courbes. Cette th\'eorie se g\'en\'eralise aux vari\'et\'es projectives lisses quelconques.

\begin{defn}\phantomsection\label{d3.3} Une \emph{correspondance de Chow de degr\'e $r$} de $X^m$ vers $Y^n$ est un \'el\'ement de $CH^{m+r}(X\times Y)$. Notation: $\Corr^r(X,Y)$. Si $X$ n'est pas \'equidimensionnel, on \'etend par additivit\'e.
\end{defn}

Les correspondances se composent selon la formule sempiternelle \eqref{eq:comp-corr}. Elles op\`erent sur les groupes de Chow par la formule non moins sempiternelle \eqref{eq:act-corr}.

\begin{rque}\phantomsection\label{r3.1} \eqref{eq:act-corr} est le cas particulier de \eqref{eq:comp-corr} o\`u $X_1=\Spec k$.
\end{rque}

\begin{ex}\phantomsection  Si $f:X\to Y$ est un morphisme, son graphe $\Gamma_f$  d\'efinit une correspondance $f^*\in \Corr^0(Y,X)$.
\end{ex}

\begin{prop}\phantomsection\label{p3.2} a) La composition des correspondances est associative, et la classe de la diagonale est un \'el\'ement neutre \`a gauche et \`a droite.\\
b) si $X\by{f}Y\by{g} Z$ sont deux morphismes, on a $(g\circ f)^*=f^*\circ g^*$.\\
c) L'action de la correspondance $f^*$ sur les groupes de Chow \index{Groupe de Chow} est l'image r\'eciproque des cycles.
\end{prop}

\begin{meg}\phantomsection
Pour des raisons de compatibilit\'e avec la litt\'erature, on a chang\'e le sens de la composition des correspondances par rapport \`a \ref{s:hr1}! Plus pr\'ecis\'ement, il y a deux conventions pour la composition des correspondances alg\'ebriques. La convention covariante est celle o\`u le graphe d'un morphisme op\`ere de mani\`ere covariante: c'est celle de Weil, de Fulton \cite[ch. 16]{fulton} et de Voevodsky \cite[\S 2.1]{voetri}. \`A l'inverse, Grothendieck adopte la convention contravariante pour avoir la bonne compatibilit\'e avec la contravariance de la cohomologie. C'est aussi la convention de Kleiman \cite{kleiman-standard} ou Scholl \cite{scholl}, et celle que nous prenons ici. On passe d'une convention \`a l'autre de mani\`ere ``triviale'' mais parfois propice \`a s'embrouiller, \cf \cite[7.1]{kmp} et ce qui suit.
\end{meg}

La fa\c con la plus claire de comprendre la situation est de consid\'erer l'action sur les groupes de Chow: les correspondances ci-dessus sont adapt\'ees \`a la th\'eorie $CH^*$, tandis que les correspondances covariantes sont adapt\'ees \`a la th\'eorie $CH_*$.
Ainsi, on d\'efinit les \emph{correspondances covariantes de degr\'e $r$}
\[\Corr_r(X,Y)=CH_{m+r}(X\times Y)\]
(si $X$ est purement de dimension $m$; on passe de l\`a au cas g\'en\'eral par additivit\'e.)

Leur composition se d\'efinit encore par \eqref{eq:comp-corr}, et elles op\`erent sur $CH_*$ par \eqref{eq:act-corr}, la  remarque \ref{r3.1} restant valable. 

Voici deux mani\`eres de passer des correspondances contravariantes aux correspondances covariantes, et inversement:

\begin{lemme}\phantomsection\label{l3.1} a) L'\'echange des facteurs $\sigma_{X,Y}$ du lemme \ref{l:tr} induit des isomorphismes involutifs 
\[\Corr^r(X,Y)\iso \Corr_{-r}(Y,X)\iso \Corr^r(X,Y)\] 
appel\'es \emph{transposition} (notation: $\alpha\mapsto {}^t \alpha$). On a ${}^t(\beta\circ \alpha)={}^t \alpha \circ {}^t \beta$.\\
b) Si $X$ (\resp $Y$) est purement de dimension $m$ (\resp $n$), on a l'\'egalit\'e
\[\Corr^r(X,Y)=\Corr_{n-m-r}(X,Y).\]
\end{lemme}

\begin{ex}\phantomsection Si $f:X\to Y$ est un morphisme, son graphe $\Gamma_f$  d\'efinit  une correspondance covariante $f_*\in \Corr_0(X,Y)$, qui op\`ere sur les groupes de Chow\index{Groupe de Chow} $CH_*$ par l'image directe. On a ${}^t f^*=f_*$ et ${}^t f_* =f^*$.
\end{ex}

Dans \cite[\S 16.1]{fulton},  Fulton  d\'emontre la version covariante de la proposition \ref{p3.2}. On peut en d\'eduire cette derni\`ere via le lemme \ref{l3.1} a) (qui est facile), ou simplement en reproduisant les calculs de \emph{loc. cit.}

\subsubsection{Correspondances alg\'ebriques et cohomologiques}

La classe de cycle d\'efinit des homomorphismes
\[\cl:\Corr^r(X,Y)\to \Corr^{2r}_H(X,Y)(r)\]

Attention au changement de degré et au twist!

\begin{lemme}\phantomsection a) $\cl$ respecte la composition des correspondances.\\
b) Pour toute correspondance alg\'ebrique $\alpha$, on a
\[\cl({}^t \alpha) = {}^t\cl(\alpha).\]
En particulier, pour tout morphisme $f:X\to Y$, on a $\cl\circ f_* = f_*\circ \cl$ (image directe en cohomologie).
\end{lemme}

\begin{proof} a) d\'ecoule directement de la formule \eqref{eq:comp-corr} et des axiomes d'une cohomologie de Weil, tandis que b) d\'ecoule du lemme \ref{l:tr}.\index{Cohomologie!de Weil}
\end{proof}

\begin{cor}\phantomsection Supposons $X$ et $Y$ \'equidimensionnels, de dimensions respectives $m$ et $n$. Alors pour tout $i\ge 0$, le diagramme
\[\begin{CD}
CH^i(X)@>\cl^i_X>> H^{2i}(X)(i)\\
@V{f_*}VV @V{f_*}VV\\
CH^{i+n-m}(Y)@>\cl^{i+n-m}_Y>> H^{2i+2n-2m}(Y)(i+n-m)
\end{CD}\]
est commutatif.
\end{cor}

\subsubsection{Premi\`eres applications de la formule des traces} En l'appliquant \`a la correspondance identique, on a d'abord:

\begin{prop}\phantomsection\label{p3.3} Pour tout $X\in \V(k)$, on a 
\[\chi(X)=\sum_{i\ge 0} (-1)^i \dim H^i(X)\] 
o\`u le membre de gauche est la caract\'eristique d'Euler-Poincar\'e de $X$, d\'efinie comme le nombre d'auto-intersection de la diagonale.\qed
\end{prop}

\begin{cor}\phantomsection\label{l3.3} Soit $H^*$ une cohomologie de Weil. Alors, pour toute $k$-courbe projective lisse $C$, g\'eom\'etriquement connexe de genre $g$, on a
\[\dim H^1(C)=2g.\]
\end{cor}

\begin{proof} D'apr\`es (la d\'emonstration du) th\'eor\`eme \ref{t:trid}, on a $\chi(C)=2-2g$.
En appliquant la proposition \ref{p3.3}, on trouve donc
\[2-2g=\dim H^0(C) -\dim H^1(C) +\dim H^2(C)\]
d'o\`u le corollaire en utilisant l'axiome (v).
\end{proof}

\subsubsection{Cohomologies de Weil classiques}\label{3.5.9} L'exemple primordial de cohomologie de Weil\index{Cohomologie!de Weil!classique} est la \emph{cohomologie de Betti}\index{Cohomologie!de Betti}, quand $k=\C$:
\[H^i_B(X)=H^i(X(\C),\Q)\]
o\`u le membre de droite est la cohomologie singuli\`ere. \`A l'\'epoque o\`u Weil a fait ses conjectures, c'\'etait la seule connue, en caract\'eristique z\'ero.

En caract\'eristique $p$, on ne peut pas esp\'erer de cohomologie de Weil \`a coefficients dans $\Q$, ni m\^eme dans $\R$ (observation de Serre). L'obstruction est le corollaire \ref{l3.3} appliqu\'e \`a une courbe elliptique \emph{supersinguli\`ere} $E$. 

Par d\'efinition, une courbe elliptique sur $k$ est supersinguli\`ere si $E(\bar k)$ n'a pas de $p$-torsion, o\`u $p=\car k$: d'apr\`es  Deuring \cite[p. 198]{deuring2}, ces courbes existent toujours sur les corps finis, et l'anneau $\End^0(E_{\bar k})$ est un corps de quaternions de centre $\Q$, ramifi\'e exactement en $p$ et $\infty$. Or $\End^0(E)\otimes K$ op\`ere sur $H^1(E)$ via l'action des correspondances. Une alg\`ebre de quaternions de centre $K$ ne peut op\'erer sur un $K$-espace vectoriel de dimension $2$ que si elle est d\'eploy\'ee\dots

La cohomologie \'etale a donn\'e naissance \`a la \emph{cohomologie $l$-adique}\index{Cohomologie!$l$-adique}
\[H^i_l(X)=H^i(X_{\bar k},\Q_l):= \plim H^i_\et(X_{\bar k},\Z/l^n)\otimes_{\Z_l} \Q_l\]
pour tout nombre premier $l\ne \car k$. Ses coefficients sont $\Q_l$. En caract\'eristique z\'ero, on a aussi la \emph{cohomologie de de Rham alg\'ebrique}\index{Cohomologie!de de Rham}
\[H^i_\dR(X)=\bH^i_\Zar(X,\Omega^*_{X/k})\]
hypercohomologie de Zariski du complexe des formes diff\'erentielles de K\"ahler. Ses coefficients sont $k$. Sur un corps parfait $k$ de caract\'eristique $p$, on a la \emph{cohomologie cristalline}\index{Cohomologie!cristalline}
\[H^i_\cris(X)\]
dont les coefficients sont le corps des fractions de $W(k)$, l'anneau des vecteurs de Witt de $k$. Si $k=\F_q$ et si $E$ est une courbe elliptique supersinguli\`ere, alors $\End^0(E)=\End^0(E_{\bar k})$ seulement si $q=p^n$ avec $n$ pair; dans ce cas, le corps des fractions de $W(k)$ est une extension de degr\'e pair de $\Q_p$, qui d\'eploie bien $\End^0(E)$!

Ces cohomologies sont parfois appel\'ees \emph{cohomologies de Weil classiques}.

\subsection{Preuve de certaines conjectures de Weil}\index{Conjectures!de Weil}

 L'application la plus importante de la formule des traces est:

\begin{thm}\phantomsection\label{t3.6} S'il existe une cohomologie de Weil\index{Cohomologie!de Weil} sur $\V(\F_q)$ \`a coefficients dans $K$, alors,  parmi les conjectures de Weil,  sont vraies la rationalit\'e, l'\'equation fonctionnelle et la factorisation (en polyn\^omes \`a coefficients dans $K$).\index{Equation fonctionnelle@\'Equation fonctionnelle}
\end{thm}

\begin{proof} Soit $X\in \V(k)$, g\'eom\'etriquement connexe de dimension $n$. Appliquons la formule des traces aux puissances de l'endomorphisme de Frobenius $\pi_X$, vu comme correspondance cohomologique:
\[<\Delta_X,(\pi_X^r)^*>_{X\times X} = \sum_{i=0}^{2n} (-1)^i\Tr((\pi_X^r)^*\mid H^i(X)). \]
Comme au \S \ref{s2.8.7}, le membre de gauche est \'egal \`a $|X(\F_{q^r})|$. La rationalit\'e et la factorisation de $Z(X,t)$ r\'esultent alors de l'identit\'e suivante:
\begin{equation}\label{eq3.3} 
\exp\left(\sum_{r=0}^\infty \frac{\Tr(f^r)}{r}t^r \right) = \frac{1}{\det(1-ft\mid V)}
\end{equation}
o\`u $f$ est un endomorphisme d'un $K$-espace vectoriel $V$ de dimension finie,
\footnote{Pour d\'emontrer cette identit\'e, on peut se ramener au cas o\`u $K$ est alg\'ebriquement clos, puis mettre $f$ sous forme triangulaire, et enfin se ramener au cas o\`u $\dim V=1$ et $f$ est un scalaire.} d'o\`u
\begin{equation} \label{eq3.4} 
Z(X,t) = \prod_{i=0}^{2n} P_i(t)^{(-1)^{i+1}},\quad P_i(t)=\det(1-\pi_X^*t\mid H^i(X)).
\end{equation}

 On voit ainsi que les ``nombres de Betti''\index{Nombres!de Betti} de Weil sont bien donn\'es par
\[B_i=\dim H^i(X).\]

Ceci donne que $Z(X,t)$ est une fraction rationnelle \emph{a priori} \`a coefficients dans $K$; mais un crit\`ere de rationalit\'e pour les s\'eries formelles, utilisant les d\'eterminants de Hankel \cite[Prop. 5.2.1]{amice}, 
%\cite[p. 276]{weilI}, 
permet de voir que $\Q[[t]]\cap K(t)=\Q(t)$. Par contre il ne permet pas de montrer que les $P_i(t)$ sont individuellement dans $\Q[t]$, faute de savoir s\'eparer leurs racines.

Pour obtenir l'\'equation fonctionnelle, on utilise la dualit\'e de Poincar\'e: l'op\'erateur $\pi_X^*$ est presque unitaire relativement \`a l'accouplement de Poincar\'e. Plus pr\'ecis\'ement, en comparant \eqref{eq3.4} pour $X=\P^1$ avec la formule de la proposition \ref{p4} d), on trouve que $\pi_{\P^1}^*$ op\`ere sur l'espace vectoriel de dimension $1$ $H^2(\P^1)$ par multiplication par $q$. Par l'axiome (vii) des cohomologies de Weil, $\pi_X^*$ op\`ere donc sur $H^{2n}(X)$ par multiplication par $q^n$.\footnote{Pour justifier ceci, choisissons une application rationnelle dominante $f:X\tto (\P^1)^n$. Le graphe de $f$ induit un isomorphisme $f_*:H^{2n}(X)\iso H^{2n}((\P^1)^n)\simeq H^2(\P^1)^{\otimes n}$, et $\pi_X^*,\pi_{(\P^1)^n}^*$ commutent \`a $f_*$ tandis que le dernier isomorphisme transforme $\pi_{(\P^1)^n}^*$ en $(\pi_{\P^1}^*)^{\otimes n}$.  On donnera plus tard (proposition \ref{l6.3}) une d\'emonstration plus raisonnable de ce fait.} On a donc, pour $x\in H^i(X)$ et $y\in H^{2n-i}(X)$:
\begin{equation}\label{eq3.5}
<\pi_X^*x,\pi_X^*y>=\Tr_X(\pi_X^*(x\cdot y)) = q^n\Tr_X(x\cdot y)=q^n<x,y>.
\end{equation}

On d\'eduit de ceci:
\begin{multline}
P_{2n-i}(1/q^nt)= \det(1-\pi_X^*/q^nt\mid H^{2n-i}(X))\\= \det(1-(\pi_X^*)^{-1}t^{-1}\mid H^{i}(X))
=\det(\pi_X^*\mid H^{i}(X))^{-1}(-t)^{-B_{i}} P_{i}(t).\label{eq3.6}
\end{multline}

Mais en r\'eappliquant \eqref{eq3.5}, on trouve
\[\det(\pi_X^*\mid H^{i}(X))\det(\pi_X^*\mid H^{2n-i}(X))=q^{nB_i}.\]

En collectant \eqref{eq3.6} pour $i=0,\dots,2n$, en notant que $B_i=B_{2n-i}$ par dualit\'e de Poincar\'e et en tenant compte de la proposition \ref{p3.3}, on obtient l'\'equation fonctionnelle d\'esir\'ee:
\begin{equation}\label{eq3.7}
Z(X,1/q^nt)=\epsilon q^{n\chi/2} (-t)^\chi Z(X,t)
\end{equation}
 avec un signe $\epsilon$ qui provient du groupe de cohomologie central $H^n(X)$.
\end{proof}

\begin{rque}\phantomsection\label{r3.6}  Ce signe peut \^etre \'egal \`a $-1$! Un exemple simple est l'\'eclat\'e de $\P^2_{\F_q}$ en le point ferm\'e quadratique $\{[0:1:\sqrt{d}],[0:1:-\sqrt{d}]\}$ o\`u $d$ n'est pas un carr\'e dans $\F_q$. \cf \cite[ex. 14.6 c)]{kahn-zeta}. (Comme Serre me l'a fait remarquer, un exemple encore plus simple est celui de la quadrique lisse de dimension $2$ et de discriminant $d$ sur $\F_q$.)
\end{rque}

\begin{exo}\phantomsection\label{exo3.2} Montrer que l'entier $n\chi$ de l'\'equation fonctionnelle\index{Equation fonctionnelle@\'Equation fonctionnelle} \eqref{eq3.7} est toujours pair.
\end{exo}

\subsubsection{Les autres conjectures de Weil}\label{3.6.1}

Deux de ces conjectures ne sont pas touch\'ees par le formalisme des cohomologies de Weil\index{Cohomologie!de Weil}: la valeur des nombres de Betti\index{Nombres!de Betti} quand $X$ provient de caract\'eristique z\'ero, et l'hypoth\`ese de Riemann.\index{Hypothèse de Riemann!pour les variétés projectives lisses}

La premi\`ere a \'et\'e d\'emontr\'ee gr\^ace aux propri\'et\'es de la cohomologie $l$-adique\index{Cohomologie!$l$-adique}:

\begin{description}
\item[Th\'eor\`eme de comparaison Betti-$l$-adique (M. Artin)]\index{Cohomologie!de Betti} Si $X\in \V(\C)$, on a des isomorphismes canoniques et fonctoriels en $X$
\[H^i_B(X)\otimes_\Q \Q_l\simeq H^i_l(X).\]
\item[Invariance par extension s\'eparablement close] Si $K/k$ est une extension de corps s\'eparablement clos de caract\'eristique $\ne l$, alors pour toute $X\in\V(k)$,
\[H^i_l(X)\iso H^i_l(X_K).\]
\item[Changement de base propre et lisse] \label{pel} Soit $A$ un anneau de valuation discr\`ete, de corps des fractions $K$ et de corps r\'esiduel $k$. Soit $\sX$ un $A$-sch\'ema projectif lisse, de fibre g\'en\'erique $X$ et de fibre spéciale $X_0$. Alors on a un isomorphisme canonique et fonctoriel en $\sX$:
\[H^i_l(X)\simeq H^i_l(X_0).\]
\end{description}

Pour la seconde conjecture, la strat\'egie naturelle serait de prouver une g\'en\'eralisation du ``lemme important'' de Weil:
\[\Tr({}^t \pi_X\circ \pi_X)>0.\]

Cela a un sens sur $H^n(X)$ si $\dim X=n$, mais sur les autres $H^i$? Serre donne une r\'eponse \`a l'aide d'une polarisation dans le cadre des vari\'et\'es k\"ahl\'eriennes en utilisant le th\'eor\`eme de l'indice de Hodge (en th\'eorie de Hodge) \cite{serre-kahler}. Ces arguments \'evoquent irr\'esistiblement l'id\'ee de Hilbert et P\'olya pour d\'emontrer l'hypoth\`ese de Riemann classique\dots

L'id\'ee de Serre donne naissance aux \emph{conjectures standard} de Grothendieck et Bombieri \index{Conjectures!standard} \cite{grothendieck-standard}, programme pour d\'emontrer l'hypoth\`ese de Riemann sur les corps finis. Malheureusement c'est une impasse: les conjectures standard ne sont toujours pas d\'emontr\'ees \`a l'heure actuelle!

C'est Deligne qui d\'emontre finalement la derni\`ere conjecture de Weil en 1974 \cite{weilI}. Il en donne une seconde d\'emonstration dans \cite{weilII}. Depuis, au moins deux autres d\'emonstrations ont \'et\'e donn\'ees: celles de Laumon \cite{laumon} et de Kedlaya \cite{kedlaya}.

Remarquons seulement ici:

\begin{lemme}\phantomsection\label{l3.4} Si $X$ v\'erifie la derni\`ere conjecture de Weil, les polyn\^omes $P_i$ sont premiers entre eux deux \`a deux, \`a coefficients entiers et ind\'ependants du choix de $l$.
\end{lemme}

\begin{proof} La premi\`ere affirmation est \'evidente, puisque les racines de $P_i$ et de $P_j$ ont des valeurs absolues complexes diff\'erentes si $i\ne j$. Comme $Z(X,t)\in \Q(t)$, le groupe $\Gal(\bar \Q/\Q)$ permute l'ensemble de ses z\'eros et p\^oles; toujours \`a cause de l'hypoth\`ese de Riemann\index{Hypothèse de Riemann!pour les variétés projectives lisses}, il laisse stable les racines de $P_i$ pour tout $i$, donc $P_i\in \Q[t]$. Une telle factorisation de $Z(X,t)$ ne d\'epend pas du choix de $l$.

Enfin, si les racines inverses des $P_i$ sont des entiers alg\'ebriques, les $P_i$ sont \`a coefficients entiers. 
\end{proof}

\subsection{Le th\'eor\`eme de Dwork}

\begin{thm}[Bernard Dwork, \protect{\cite{dwork}}]\phantomsection Soit $X$ un sch\'ema de type fini sur $\F_q$. Alors $Z(X,t)\in \Q(t)$.
\end{thm}

Une description tr\`es agr\'eable de la d\'emonstration de Dwork est donn\'ee par Serre dans son expos\'e Bourbaki \cite{serre-dwork}. J'en reproduis les points principaux:

a) Par des d\'evissages standard (\cf preuve du th\'eor\`eme  \ref{t:abs}),  Dwork se r\'eduit au cas d'une hypersurface ``torique'' (intersection d'une hypersurface de $\A^n$ et de l'ouvert $\prod t_i\ne 0$). On peut de plus supposer que $q=p$.

b) Il g\'en\'eralise un joli th\'eor\`eme d'\'Emile Borel \cite{eborel}:

\begin{thm}[th\'eor\`eme de Borel-Dwork]\label{tbd} \phantomsection Soit $f(t)=\sum_{n=0}^\infty a_nt^n$ une s\'erie enti\`ere \`a coefficients entiers. Supposons qu'il existe deux nombres r\'eels $r_\infty,r_p>0$ tels que
\begin{thlist}
\item $f$ soit méromorphe dans le disque $|t|<r_\infty$ de $\C$. 
\item $f$ soit méromorphe dans le disque $|t|<r_p$ de $\C_p$.
\item $r_\infty r_p>1$. 
\end{thlist}
Alors $f\in \Q(t)$.
\end{thm}

Ici, $\C_p$ est le compl\'et\'e d'une cl\^oture alg\'ebrique de $\Q_p$ et ``m\'eromorphe'' a essentiellement le m\^eme sens que sur $\C$. La preuve n'est pas tr\`es difficile: \cf \cite[th. 5.3.2]{amice}.

c) On applique le théorème \ref{tbd} à $f(t)=Z(X,t)$. Comme cette fonction est domin\'ee par $Z(\A^n,t)= 1/(1-p^nt)$, on peut prendre $r_\infty = p^{-n}$ (elle est m\^eme holomorphe dans ce domaine), et il faut voir qu'on peut prendre $r_p>p^n$. En fait, on peut prendre $r_p=+\infty$: Dwork établit une factorisation (\cf \cite[(48)]{serre-dwork})
\begin{equation}\label{eq:dwork}
Z(X,pt)= \prod_{i=0}^n(1-p^{n-i}t)^{(-1)^{i+1}\binom{n}{i}}\prod_{i=0}^{n+1} \Delta(p^{n+1-i}t)^{(-1)^{i+1}\binom{n+1}{i}}
\end{equation}
o\`u 
\[\Delta(t)=\det(1-t\Psi)\]
a un rayon de convergence infini dans $\C_p$. Ici, $\Psi$ d\'esigne un endomorphisme d'un $\C_p$-espace vectoriel de dimension infinie (plus pr\'ecis\'ement, la limite d'une suite de tels op\'erateurs), dont la nature est clarifi\'ee plus tard par Serre dans \cite{serre-fredholm} (th\'eorie de Fredholm $p$-adique). La formule \eqref{eq:dwork} y est interpr\'et\'ee comme provenant d'une r\'esolution de Koszul (\emph{loc. cit.}, p. 85).

La preuve de Dwork sera interpr\'et\'ee un peu plus plus tard \`a l'aide d'une th\'eorie cohomologique, la cohomologie de Monsky-Washnitzer, qui donnera naissance (nettement plus tard) \`a une ``cohomologie $p$-adique'' ou ``cohomologie rigide'' (Kashiwara-Mebkhout, Berthelot.)

\begin{exo}\phantomsection\label{exo3.3} Soit $K$ un corps commutatif. On se donne une s\'erie formelle $f=\sum_{n=0}^\infty a_n t^n$ \`a coefficients dans $K$.
\begin{enumerate}[label=(\alph*)]
\item \label{exo3.3a}Montrer que, pour tout $k\ge 0$,  les conditions suivantes sont \'equivalentes:
\begin{enumerate}[label=(\roman*)]
\item Il existe un polyn\^ome $P\in K[t]$ de degr\'e $k$ tel que $Pf\in K[t]$. 
\item la suite $(a_n)$ satisfait une relation de r\'ecurrence lin\'eaire
\[a_{n+k}+b_1a_{n+k-1}+\dots +b_ka_n = 0, \quad b_1,\dots, b_k\in K\]
pour $n\ge n_0$, $n_0$ convenable.
\end{enumerate}
\item Si $k$ est minimal dans \ref{exo3.3a}, montrer que les $b_i$ sont uniques.
\item Soit $L$ un surcorps de $K$. Supposons que $f\in L(t)$. Montrer que $f\in K(t)$. (Choisir une base du $K$-espace vectoriel $L$ contenant $1$. Le lecteur n'aimant pas l'axiome du choix pourra observer qu'on peut supposer $L$ de type fini sur $K$, et faire un dévissage.)
\end{enumerate}
(Cette m\'ethode est moins calculatoire que celle utilisant les d\'eterminants de Hankel.)
%\item (Fonction z\^eta d'une alg\`ebre simple; lien avec celle de sa vari\'et\'e de Severi-Brauer).
\end{exo}

\section{Les fonctions $L$ de la th\'eorie des nombres}\label{s4}

\subsection{Fonctions $L$ de Dirichlet}

\subsubsection{Deux approches des produits eul\'eriens}\label{appr1} Ce sont les suivantes:

\begin{enumerate}
\item On travaille \`a un nombre fini de facteurs pr\`es.
\item On tient \`a avoir des r\'esultats pr\'ecis et (donc) tous les facteurs eul\'eriens en jeu.
\end{enumerate}

La premi\`ere est souvent suffisante en premi\`ere approximation, mais on souhaite ensuite passer \`a la seconde. Illustrons ceci avec les caract\`eres de Dirichlet.

\subsubsection{Caract\`eres de Dirichlet et caract\`eres modulaires}\label{section2.5}  \index{Caractère!de Dirichlet}

\begin{comment}
Soit $m  \in \N-\{ 0\}$. On consid\`ere le groupe $(\Z/m\Z)^*$ des
\'el\'ements inversibles de l'anneau $\Z/m\Z$ . C'est
 un groupe ab\'elien fini d'ordre $\varphi(m)$, o\`u $\varphi(m)$ est
 l'indicateur d'Euler (c'est la d\'efinition de l'indicateur d'Euler).  

\emph{Structure de} $G$. De mani\`ere classique, on a \cite[ch. 4, th.
3]{ire-ros}:

\begin{itemize}
\item $(\Z/2\Z)^*  =  \{1\}$, $(\Z/4\Z)^*=\{\pm 1\}\simeq
\Z/2\Z$, 

\item $ \forall \alpha \geq  2, (\Z/2^{\alpha}\Z)^*=\{\pm 1\}\times
(1+4\Z)/2^\alpha\Z \simeq\Z/2\Z \times\Z/2^{\alpha-2}\Z$.

\item $ \forall  p  \in\mathcal{P}-\{2\},  \forall\alpha \geq 1,
(\Z/p^{\alpha}\Z)^*=\mu_{p-1}\times (1+p\Z)/p^\alpha \Z \simeq
\Z/p^{\alpha-1}(p-1)\Z$
 
\item Si $\displaystyle m =  \prod^r_{i=1}p^{\alpha_{i}}_{i} $ o\`u $
\alpha_{i}
 \geq  1$, $i  \neq  j \Rightarrow  p_{i} \neq  p_{j} $, alors
$\displaystyle (\Z/m\Z)^*   \simeq
\prod_{i=1}^r(\Z/p^{\alpha_{i}}_{i}\Z)^*$.
\end{itemize}

C'est donc en g\'en\'eral un groupe assez compliqu\'e.   
\end{comment}

\begin{defn}\phantomsection\label{d2.5.1} a) Un \emph{caract\`ere de Dirichlet modulo $m$}\index{Caractère!de Dirichlet} est une
fonction  $\chi:\Z\rightarrow\C$ telle que
\begin{thlist}
\item    $ \chi \neq  0$.
\item   $ \chi(x+m) =  \chi(x)$ $\forall x\in\Z$ ($\chi$ est p\'eriodique de
p\'eriode $m$). 
\item  $ \chi(xy) =  \chi(x)\chi(y)$ $\forall x,y\in\Z$. 
\item   $ \chi(x) = 0$   si  $(x,m)  \ne  1$. 
\end{thlist}
b) Un \emph{caract\`ere modulaire modulo $m$} est un caract\`ere du groupe ab\'elien $(\Z/m)^*$ (homomorphisme de $(\Z/m)^*$ vers $\C^*$).
\end{defn}
 
\begin{rque}\phantomsection Un caract\`ere de Dirichlet modulo $m$ est un caract\`ere de
$(\Z/m)^*$, prolong\'e par $0$ \`a $ \Z/m$, et relev\'e \`a
$\Z$. Si $m'|m$,  tout caract\`ere de $(\Z/m')^*$ d\'efinit un caract\`ere de $ 
(\Z/m)^*$ via $(\Z/m)^*\to (\Z/m')^*$. Mais les caract\`eres de Dirichlet associ\'es sont
diff\'erents en g\'en\'eral. Par exemple, $ \chi  = 1 \pmod m$ correspond \`a 
\[\begin{array}{ccc}
1_{m}:\Z&\longrightarrow& \C\\
x &\longmapsto&\left\{\begin{array}{cl}
1 &\mbox{si }(x, m) = 1;\\
0 &\mbox{sinon.}
\end{array}\right.
\end{array}\]
\end{rque}

On est donc conduit \`a:

\begin{defn}\phantomsection\label{d2.5.2} Soit $ \chi $ un caract\`ere de Dirichlet modulo $m$, vu
comme fonction de $ \Z/m  $ vers $ \C$. Le \emph{conducteur} $f$ de
$\chi$ est le plus petit diviseur de $m$ tel que $ \chi|_{(\Z/m)^*} $ se
factorise par $(\Z/f)^*$. On dit que $\chi$ est \emph{primitif} si son
conducteur est \'egal \`a $m$.
\end{defn}
     
\subsubsection{Les fonctions $L$ de Dirichlet}\label{section2.6}\index{Fonction $L$!de Dirichlet}

\begin{defn}\phantomsection{}\label{d2.6.1} Soient $m\in\N-\{0\}$ et $\chi$ un caract\`ere de
Dirichlet\index{Caractère!de Dirichlet} modulo $m$. On appelle \emph{fonction $L$ de Dirichlet} (attach\'ee \`a
$\chi$) la s\'erie
$$L(\chi,s)=\sum_{n=1}^\infty\frac{\chi(n) }{ n^s}$$
\end{defn}

Comme $\chi$ est multiplicative, on a (\cf proposition \ref{p1}):

\begin{prop}\phantomsection{}\label{p2.6.1} Pour tout $\chi$, on a l'identit\'e (de s\'eries
formelles de Dirichlet):
\[L(\chi,s) =\prod_{p\in\sP}\frac{1}{ 1-\chi(p)p^{-s} }\]
o\`u $\sP$ d\'esigne l'ensemble des nombres premiers.
\end{prop}

Comparons la fonction $L$ attach\'ee \`a un caract\`ere de
Dirichlet et celle attach\'ee au caract\`ere primitif associ\'e: 

\begin{prop}\phantomsection\label{p2.6.2} Soient $\chi$ un caract\`ere de Dirichlet modulo $m$, $f$
son conducteur et $\chi'$ le caract\`ere primitif modulo $f$ associ\'e. Alors:
$$L(\chi,s)=L(\chi',s)\prod_{p\in S}(1-\frac{\chi'(p)}{ p^s}),$$
o\`u $S$ d\'esigne l'ensemble des facteurs premiers de $m$ qui ne divisent pas
$f$. En particulier,\\
a) Pour $\chi=1_m$, on a : 
$$L(1_m,s) = F(s)\zeta(s)$$
o\`u $\displaystyle F(s) =\prod_{p|m}(1- \frac{1}{ p^s})$.\\
b) Si $m$ et $f$ ont les m\^emes facteurs premiers,
$L(\chi,s)=L(\chi',s)$. 
\end{prop}  

\subsubsection{Fonctions $L$ et fonctions z\^eta}\index{Fonction zêta!de Dedekind} Soit $K=\Q(\mu_m)$: c'est une extension galoisienne de $\Q$, de groupe de Galois $G$ canoniquement isomorphe \`a $(\Z/m)^*$. Rappelons comment: le caract\`ere cyclotomique\index{Caractère!cyclotomique}
\[\kappa:G\to (\Z/m)^*,\]
d\'efini par $g\zeta = \zeta^{\kappa(g)}$ pour tout $\zeta\in \mu_m$, est \'evidemment injectif et un th\'eor\`eme de Gauss implique qu'il est surjectif (irr\'eductibilit\'e des polyn\^omes cyclotomiques: voir une jolie preuve dans Cassels-Fr\"ohlich \cite[ch. III, lemma 1]{cf}). Son degr\'e est donc $\phi(m)$, indicateur d'Euler de $m$.

 Les fonctions $L$ de Dirichlet permettent de factoriser la fonction z\^eta de $K$:

\begin{prop}\phantomsection\label{d2.6.2} On a
$$\zeta_K(s) = G(s)\prod_\chi L(\chi,s)$$
o\`u le produit est \'etendu \`a tous les caract\`eres $ \chi $ de
$(\Z/m\Z)^*$ et
\[G(s) = \prod_{\fP\mid m} \left(1-\frac{1}{N(\fP)^s}\right)^{-1}\]
(produit fini de facteurs eul\'eriens).
\end{prop}

\begin{proof} Si $p\nmid m$, notons simplement $p$ son image dans $(\Z/m\Z)^*$. Notons
\'egalement $f(p)$ l'ordre de $p$ dans ce groupe, et $\displaystyle g(p)
=\frac{\varphi(m)}{ f(p)}$ l'ordre de $(\Z/m\Z)^*/<p>$. Alors $f(p)$ est le degr\'e r\'esiduel de $p$ dans l'extension $K/\Q$, et $g(p)$ est le nombre d'id\'eaux premiers de $K$ au-dessus de $p$. D'o\`u
\[\prod_{\fP\mid p} \left(1-\frac{1}{N(\fP)^s}\right)^{-1}= \frac{1}{ (1-\frac{ 1}{
{\displaystyle p^{f(p)s}}})^{g(p)}}.\]

L'\'enonc\'e r\'esulte alors de l'identit\'e facile
\[\prod_\chi(1-\chi(p)t)=
(1-t^{f(p)})^{g(p)}\]
o\`u $\chi$ d\'ecrit les caract\`eres de $(\Z/m)^*$.
\end{proof}

On voit ainsi qu'en consid\'erant les caract\`eres de Dirichlet\index{Caractère!de Dirichlet} de mani\`ere na\"\i ve, on n'arrive qu'\`a une expression approximative de $\zeta_K(s)$, qui \'evite pr\'ecis\'ement les facteurs locaux correspondant aux id\'eaux premiers ramifi\'es. Pour obtenir une valeur exacte, il faut remplacer chaque caract\`ere par le caract\`ere primitif associ\'e, et on obtient:

\begin{thm}\phantomsection\label{t4.0} Pour tout entier $n>1$, notons $\Prim(n)$ l'ensemble des caract\`eres modulaires primitifs modulo $n$. Alors si $K=\Q(\mu_m)$, on a
\[\zeta_K(s) = \prod_{d\mid m} \prod_{\chi\in \Prim(d)} L(\chi,s).\]
\end{thm}

(La d\'emonstration est essentiellement une extension de celle de la proposition \ref{d2.6.2} au cas des nombres premiers ramifi\'es: rappelons que $p$ est ramifi\'e dans $K/\Q$ si et seulement s'il divise $m$, et qu'alors son indice de ramification est $\phi(p^r)$ o\`u $p^r$ est la plus grande puissance de $p$ divisant $m$. \cf \cite[ch. III]{cf}.)

\subsection{Les th\'eor\`emes de Dirichlet}\index{Théorème!de Dirichlet}

\subsubsection{Convergence des s\'eries de Dirichlet g\'en\'eralis\'ees}\label{section2.7}

Ce pa\-ra\-gra\-phe est repris de \cite[ch. VI, \S 2.2]{serre-arith}. Son but est de rassembler des r\'esultats g\'en\'eraux sur la convergence
des s\'eries de Dirichlet.

\begin{lemme}\phantomsection\label{l2.7.2} Soient $(\alpha,\beta) \in\R^{2}$ tels que $0 <
\alpha <\beta$. Soit $z\in\C$ avec $\hbox{\rm Re}(z) > 0$. Alors on a
$$\big| e^{-\alpha z}- e^{-\beta z}\big| \leq \big|\frac{z}{x}\big|
 (e^{-\alpha x}- e^{-\beta x})$$
o\`u $x = \hbox{\rm Re}(z)$.
\end{lemme}
 
\begin{proof} On \'ecrit $\displaystyle e^{-\alpha z}- e^{-\beta z}=
z\int^{\beta}_{\alpha} e^{-tz}dt$, d'o\`u en passant aux valeurs absolues:
$$\big| e^{-\alpha z}-e^{-\beta z}\big| \leq
|z|\int^{\beta}_{\alpha}  e^{-tx} dt =\big|\frac{z}{x}\big|(e^{-\alpha
x}-e^{-\beta x})$$
 \end{proof}

\begin{thm}\phantomsection\label{t2.7.1} Soit $(\lambda_n)_{n\geq 1}$ une suite strictement
croissante de nombres r\'eels tendant vers $+\infty$, et $(a_n)_{n\geq 1}$ une
suite de nombres complexes. Si la s\'erie $f(s) = \sum  a_{n}
e^{-\lambda_{n}s}$ converge
 en $s = z_0\in\C$, elle converge uniform\'ement dans les domaines 
$\Re(s-z_0)\geq 0$, $\left|\text{\rm Arg}(s-z_0)\right|  \leq \alpha$ o\`u  $0 < \alpha
<\pi/2$.  On a $f(s)\longrightarrow  f(z_0)$  si  $s\longrightarrow  z_0 $ dans
un tel domaine. En particulier, $f$ converge dans le domaine $\Re(z)>\Re(z_0)$,
o\`u elle d\'efinit une fonction holomorphe.
\end{thm}    

(Notons tout de suite les deux cas particuliers les plus int\'eressants:

\begin{itemize}
\item $\lambda_n = n$: au changement de variables $z=e^{-s}$ pr\`es, on retrouve les s\'eries enti\`eres.
\item $\lambda_n=\log n$: on retrouve les s\'eries de Dirichlet habituelles.)
\end{itemize}

\begin{proof} Quitte \`a faire une translation sur $s$, on peut supposer $z_0 =
 0$. L'hypo\-th\`ese signifie alors que la s\'erie  $ \sum  a_{n}$ converge. Il
faut d'abord montrer qu'il y a convergence uniforme dans tout domaine de la
forme $\displaystyle\Re(s)  \geq  0,\big|\frac{s}{\Re (s)}\big|\leq  k$.

Soit $\varepsilon
>0$ . Puisque la s\'erie $ \sum  a_{n} $ converge, il existe  $N  \in \N$
tel que, pour tous $m,m'  \geq  N$,  $| A_{m,m'}| \leq \varepsilon
$, avec
$A_{m,m'}=\displaystyle\sum_{n=m}^{m'}a_n$. Si on applique le lemme
\ref{l2.7.1} avec $b_{n} = e^{-\lambda_{n}s}$ , on trouve :  
$$ S_{m,m'}= \sum^{m'-1}_{n=m} A_{m,n}(e^{-\lambda_{n}s}- e^{-\lambda_{n+1}s}) +
A_{m,m'}e^{-\lambda_{m'}s}.$$

Posons $s = x+iy$ et appliquons le lemme \ref{l2.7.2}. Il vient 
$$\big|S_{m,m'}\big| \leq \varepsilon(1 +\frac{\big| s\big|
 }{ x }\sum_{n=m}^{m'-1}(e^{-\lambda_{n}x}- e^{-\lambda_{n+1}x}))$$
d'o\`u $|S_{m,m'} | \leq \varepsilon(1+k(e^{-\lambda_{m}x} -
e^{-\lambda_{m'}x}))\leq \varepsilon(1+k)$, d'o\`u la convergence uniforme. Le
fait que $f(z)\rightarrow f(z_0)$ en r\'esulte. Pour montrer que $f$ est
holomorphe, on applique le lemme \ref{l2.7.3} aux sommes partielles $\displaystyle f_n(s)=\sum_{i\leq
n}  a_{i} e^{-\lambda_{i}s}$.
\end{proof}

\begin{cor}\phantomsection\label{c2.7.1} Avec les notations du th\'eor\`eme \ref{t2.7.1},
l'ensemble de convergence de $f$ contient un demi-plan ouvert
maximal.\qed
\end{cor}

\begin{cor}\phantomsection\label{c2.7.2} Si $f$ est identiquement nulle, on a $a_{n}=0$ pour tout
$n>0$. En particulier, une s\'erie de Dirichlet formelle non nulle donne
naissance \`a une fonction holomorphe non identiquement nulle dans son domaine
de convergence.
\end{cor}    

\begin{proof} Supposons le contraire et soit $n$ le plus petit entier tel que $a_n \neq
0$. On multiplie $f$ par $e^{\lambda_n s} $, et on fait tendre $s\in\R$ vers
$+\infty $. La convergence uniforme montre que $e^{\lambda_n s}f(s)$ tend vers
$a_n$, ce qui entra\^\i ne que $a_n=0$, contradiction.\end{proof}

\begin{defn}\phantomsection\label{d2.7.1} On appelle \emph{demi-plan de convergence} de $f$ le
demi-plan du corollaire \ref{c2.7.1} (par abus de langage, $\emptyset$ et
$\C$ sont consid\'er\'es comme des demi-plans ouverts). Si le demi-plan de
convergence est donn\'e par $\Re(s)>\rho$, on dit que $\rho=\rho(f)$ est l'\emph{abscisse de convergence} de $f$. (Les cas $\emptyset$ et $\C$ correspondent
respectivement \`a $\rho=+\infty$ et $\rho=-\infty$.) L'\emph{abscisse de
convergence absolue} de $f$ est l'abscisse de convergence de $\sum
|a_{n}| e^{-\lambda_{n}s}$ , not\'ee $\rho^{+}(f)$.
\end{defn}

Le th\'eor\`eme suivant est une sorte de r\'eciproque du
th\'eor\`eme~\ref{t2.7.1}:

\begin{thm}[Landau]\phantomsection\label{t2.7.2} On suppose que $a_{n} \geq  0$ pour tout $n\geq 1$,
 que $f$ converge pour $\Re(s)  > \rho $ et que $f$ se prolonge analytiquement
 en une fonction holomorphe au voisinage de $s =  \rho$. Alors il existe
 $ \varepsilon >  0$ tel que $f$ converge pour $\Re(s)>\rho-\varepsilon$.
\end{thm}
  
\begin{rque}\phantomsection Ceci signifie que, si $f$ est \`a coefficients $\geq 0$, son domaine de
convergence est limit\'e par une singularit\'e de $f$ situ\'ee sur l'axe
r\'eel.
\end{rque}  

\begin{proof} Quitte \`a remplacer $s$ par $s-\rho$, on peut supposer $ \rho  = 0$.
 Puisque $f$ est holomorphe \`a la fois pour $\Re(s)  >  0$ et dans un
 voisinage de $0$, elle est holomorphe dans un disque $ |s-1|\leq
1+\varepsilon$ o\`u $ \varepsilon >0$. En particulier sa s\'erie de Taylor
converge dans ce disque. \medskip

D'apr\`es le lemme \ref{l2.7.3}, on a:
$$f^{(p)}(s)=\sum_{n\geq1}a_n(-\lambda_n)^p e^{-\lambda_n s}\text{ si }\Re(s)>0,$$
d'o\`u
$$(-1)^p f^{(p)}(1) =\sum_{n\geq 1}\lambda^p_n a_n e^{-\lambda_n}.$$ 

La s\'erie de Taylor en question s'\'ecrit, pour $ |s-1| \leq  1+\varepsilon$:  
$$f(s) = \sum^{+\infty}_{p=0}\frac{1}{ p!}(s-1)^{p}f^{(p)}(1).$$

En particulier en $s = -\varepsilon$, on a : 
$$f(-\varepsilon) =
\sum^{+\infty}_{p=0}\frac{1}{ p!}(1+\varepsilon)^{p}(-1)^{p}f^{(p)}(1),$$
 la s\'erie \'etant convergente. 

Mais $\displaystyle (-1)^{p}f^{(p)} (1) =  \sum_{n\geq
1}\lambda^{p}_{n}a_{n}e^{-\lambda_{n}}$ est une s\'erie convergente \`a termes
$\geq  0$. Par cons\'equent, la s\'erie double \`a termes positifs 
$$f(-\varepsilon) =
\sum_{p,n}a_{n}\frac{1}{ p!}(1+\varepsilon)^p\lambda^p_ne^{-\lambda_{n}}$$
converge. On peut donc regrouper ses termes et l'\'ecrire sous la forme : 
$$\begin{array}{ccc}
f(-\varepsilon)& =& \sum_{n\geq1}a_{n} e^{-\lambda_{n}}
\sum_{p=0}^\infty\frac{1 }{  p!}(1+\varepsilon)^{p}\lambda^{p}_{n}\\
&=&  \sum_{n\geq1}a_{n}e^{-\lambda_{n}}e^{\lambda_{n}(1+\varepsilon)}\\
&= & \sum_{n\geq1}  a_{n}e^{\lambda_{n}\varepsilon}
\end{array}$$  

Cela montre que la s\'erie de Dirichlet donn\'ee converge pour $s =
-\varepsilon$, donc aussi pour $\Re(s)>-\varepsilon$ d'apr\`es le
th\'eor\`eme \ref{t2.7.1}.\end{proof}  

\begin{rques}\phantomsection\
\begin{itemize}
\item Comme vu plus haut, ces r\'esultats s'appliquent aussi bien aux s\'eries de
Dirichlet qu'aux s\'eries enti\`eres.
\item On a $\rho^{+}\geq\rho$, et en g\'en\'eral $\rho^{+}>\rho$ pour les
s\'eries de Dirichlet, contrairement au cas des s\'eries enti\`eres.
\end{itemize}
\end{rques}

\begin{thm}\phantomsection\label{t2.7.3} Soit $\displaystyle f(s)=\sum \frac{a_n}{ n^s}$ une s\'erie
de Dirichlet.\\
a) Supposons $a_n=O((\log n)^\alpha)$ pour $\alpha >0$ convenable. Alors
$\rho^+(f)\leq 1$.\\
b) Supposons qu'il existe $\alpha >0$ et $M>0$ tels que les sommes partielles
$\sum_{m\leq n\leq m'}a_n$ v\'erifient $\displaystyle\big|\sum_{m\leq n\leq
m'}a_n\big|\leq M(\log m)^\alpha$ pour tous $m\leq m'$. Alors $\rho(f)\leq
0$.
\end{thm}

\begin{proof} a) Supposons qu'il existe $M>0$ tel que pour tout $n\geq 1$,
$|a_{n}|\leq M(\log n)^\alpha$. Alors, $\displaystyle \sum
\frac{|a_n|}{ n^s} \leq  M  \sum_{n\geq1}\frac{(\log n)^\alpha }{ n^{\sigma}}$ et
il est connu que $\displaystyle\sum\frac{(\log n)^\alpha }{ n^{\sigma}}$ converge
pour $\sigma>1$.\smallskip

b) Soit $s$ tel que $\Re(s)>0$: montrons que $f(s)$ converge. Le th\'eor\`eme
\ref{t2.7.1} nous autorise m\^eme \`a supposer $s \in \R$.  En appliquant
alors le lemme \ref{l2.7.1}, il vient alors:
$$\big|\sum_{m\leq n\leq m'}\frac{a_n}{n^s}\big|\leq \frac{M(\log n)^\alpha 
}{ m^s}$$
qui tend vers $0$ quand $m$ tend vers l'infini. \end{proof}  

\begin{comment}
\begin{cor}\phantomsection\label{c2.7.3} Soit $f$ comme dans le th\'eor\`eme \ref{t2.7.3}.\\
a) Supposons $a_n=O((\log n)^\alpha n^r)$ pour $\alpha >0$ convenable et $r\in
\R$. Alors $\rho^+(f)\leq r+1$.\\
b) Supposons qu'il existe $r\in\R$, $\alpha >0$ et $M>0$ tels que les sommes
partielles v\'erifient $\displaystyle\left|\sum_{m\leq n\leq m'}a_n\right |\leq
M(\log m)^\alpha m^r$ pour tous $m\leq m'$. Alors $\rho(f)\leq r$.
\end{cor}

\begin{proof} Cela r\'esulte du th\'eor\`eme \ref{t2.7.2} en rempla\c cant $s$ par
$s-r$.\end{proof}
\end{comment}

\subsubsection{Convergence des fonctions $L$ de Dirichlet}

\begin{lemme}\phantomsection\label{l4.3} Soit $\chi$ un caract\`ere de Dirichlet\index{Caractère!de Dirichlet} non trivial, et soit $f=L(\chi,s)$. Alors $\rho^+(f)= 1$ et  $\rho(f)\le 0$.
\end{lemme}

\begin{proof} Comme on a $|\chi(n)|\le 1$ pour tout $n\ge 1$,  l'in\'egalit\'e $\rho^+(f)\le 1$ suit du th\'eor\`eme \ref{t2.7.3} a). Mais il existe $m>0$ tel que $\chi(n)=1$ pour tout $n\equiv 1\pmod{m}$; comme la s\'erie $\sum\limits_{n\equiv 1\pmod{m}} 1/n$ diverge, on a \'egalit\'e.

 La seconde in\'egalit\'e suit du th\'eor\`eme \ref{t2.7.3} b)  en remarquant que les sommes partielles $\sum_{n=m_0}^{m_1} \chi(n)$ sont born\'ees gr\^ace \`a la relation d'``or\-tho\-go\-na\-li\-t\'e'' $\sum_{n=0}^{m-1} \chi(n)=0$. (Rappelons-en la preuve: soit $x=\sum_{n=0}^{m-1} \chi(n)$. Pour tout $n$ premier \`a $m$, on a $\chi(n)x=x$, donc $x=0$ puisque $\chi\ne 1$.)
\end{proof}

\begin{thm}\phantomsection\label{t4.1} Soit $\chi$ un caract\`ere de Dirichlet non trivial. Alors  $L(\chi,1)\ne 0$.
\end{thm}

\begin{cor}\phantomsection\label{c4.1} Pour $K=\Q(\mu_m)$, la fonction $\zeta_K(s)$ a un p\^ole simple en $s=1$.
\end{cor}

Le th\'eor\`eme \ref{t4.1} est d\^u \`a Gustav Lejeune-Dirichlet (\cite{dirichlet}, 1837). La d\'emonstration ci-dessous remonte sans doute \`a Landau \cite{landau}: c'est essentiellement celle donn\'ee dans \cite[ch. VI, \S 3.4]{serre-arith}. \index{Théorème!de Dirichlet}

\begin{proof}[D\'emonstration du th\'eor\`eme \ref{t4.1}] Vu le lemme \ref{l4.3}, la proposition \ref{p2}, la proposition \ref{p2.6.2} et le th\'eor\`eme \ref{t4.0},  le th\'eor\`eme  \ref{t4.1} est \'equivalent au corollaire \ref{c4.1}. Si ce dernier est faux, $\zeta_K(s)$ est holomorphe pour $\Re(s)>0$. Comme c'est une s\'erie de Dirichlet \`a coefficients r\'eels positifs, le th\'eor\`eme \ref{t2.7.2} implique qu'elle converge dans ce domaine. Montrons que c'est absurde: si $\zeta_K(s) = \sum a_n n^{-s}$, on a $a_n\ne 0$ s'il existe un id\'eal $\fA\subset O_K$ de norme $n$, ce qui se produit certainement si $n$ est de la forme $r^{[K:\Q]}=r^{\phi(m)}$ ($N_{K/\Q}(rO_K) = r^{[K:\Q]}\Z$). Ceci montre que la s\'erie de Dirichlet formelle $\zeta_K(s)$ est minor\'ee par $\zeta(\phi(m)s)$, qui diverge pour $s=1/\phi(m)$ (proposition \ref{p5}).
\end{proof}

\begin{rques}\phantomsection\label{r4.2} 1) Pour d\'emontrer le th\'eor\`eme \ref{t4.1}, il suffirait d'utiliser la proposition \ref{d2.6.2} qui est un peu plus \'el\'ementaire que le th\'eor\`eme \ref{t4.0}: c'est ainsi que proc\`ede Serre dans \cite{serre-arith}.

2) Ne disposant pas du théorème de Landau, Dirichlet procédait différemment. Il distinguait deux cas: $\chi$ r\'eel (c'est-à-dire $\chi^2=1$) et $\chi$ non réel. Dans le second cas, $L(\chi,1)=0$ $\Rightarrow$ $L(\bar\chi,1)=0$ ce qui impliquerait que $\zeta_K(s)\to 0$ quand $s\to 1$, or on voit facilement que $\zeta_K(s)>1$ pour $s$ réel $>1$. Dans le premier cas, on a $\chi(n) = \leg{d}{n}$ (symbole de Jacobi) pour un certain $d$ sans facteurs carrés. Dirichlet calculait explicitement 
\[L(\chi,1) = 
\begin{cases}
\displaystyle\frac{2\pi h(d)}{w\sqrt{d}} &\text{si $d<0$}\\
\displaystyle\frac{\log(\epsilon) h(d)}{\sqrt{d}} &\text{si $d>0$}
\end{cases}
\]
où $h(d)$ (resp. $w$) est le nombre de classes (resp. de racines de l'unité) de $\Q(\sqrt{d})$ et $\epsilon$ est l'unité fondamentale $>1$ quand $d>0$. C'est un cas particulier de la formule analytique du nombre de classes du \S \ref{s:dedekind}. 

Voir \cite[ch. 6]{davenport} pour des détails sur ce calcul, et des compléments.
\end{rques}

\subsubsection{Application: le th\'eor\`eme de la progression arithm\'etique}\index{Théorème!de Dirichlet}

Introduisons les s\'eries de Dirichlet suivantes:
$$g_{a}(s) = \sum_{\begin{smallmatrix}p\in\sP\\ p\equiv a \pmod m\end{smallmatrix}}\frac{1}{ p^{s}}$$
o\`u $a$ est un entier premier \`a $m$;
$$f_{\chi}(s) =\sum_{p\in \sP} \frac{\chi(p)}{ p^s}$$
o\`u $\chi$ est un caract\`ere de Dirichlet modulo $m$.

Un calcul \'el\'ementaire donne:

\begin{lemme}[``Transformation de Fourier'']\phantomsection\label{l2.6.2} On a 
$$g_{a}(s) = \frac{1}{\varphi(m)}\sum_{\chi}\frac{f_{\chi}(s)}{ \chi(a)},$$
o\`u la somme est \'etendue \`a tous les caract\`eres $\chi$
de $(\Z/m)^*$.\qed
\end{lemme}

\begin{prop}\phantomsection\label{p2.10.1} Pour tout caract\`ere de Dirichlet $\chi$ modulo $m$, la fonction
$|\log L(\chi,s)-f_{\chi}(s)|$ est born\'ee quand $s\rightarrow 1$.
\end{prop}

\begin{proof} La proposition \ref{p2.6.1} donne
\[\log L(\chi,s)=\sum_p\log\frac{1}{ 1-\chi(p)p^{-s}} = \sum_{n,p}
\chi(p)^n/np^{ns}\]
cette \'egalit\'e \'etant valable pour tout $s\in\C$ tel que
$\Re(s)>1$.

Donc $\log L(\chi,s)=f_{\chi}(s)+F_{\chi}(s)$ o\`u $\displaystyle F_{\chi}(s)
=\sum_{p,n\geq 2}\frac{\chi(p)^{n} }{ np^{ns}}$, et on voit facilement  que $F_{\chi}(s)$
reste born\'e quand $s\rightarrow 1$.\end{proof}

\begin{thm}[Dirichlet \protect{\cite{dirichlet}}]\phantomsection\label{tdirichlet} Pour tout entier $a$ premier \`a $m$, on a
\[g_a(s)\underset{s\rightarrow 1}{\sim}^{}\frac{1}{ \varphi(m)} \log \frac{1}{ s-1}.\]
\end{thm}

\begin{proof} \'Etant donn\'e la proposition \ref{p2.10.1}, la proposition
\ref{p2}  implique que $\displaystyle
f_{1}(s)\underset{s\rightarrow 1}{\sim}^{} \log \frac{1}{ s-1}$ tandis que le th\'eor\`eme \ref{t4.1} implique que
$f_{\chi}(s)$ reste born\'e quand
$s\rightarrow 1$  pour $\chi\neq 1$. On conclut par le lemme \ref{l2.6.2}.
\end{proof}

\begin{cor}\phantomsection\label{c4.2} Pour tout entier $a$ premier \`a $m$, il existe une infinit\'e de nombres premiers $p\equiv a\pmod{m}$.\qed
\end{cor}

C'est pour d\'emontrer ce corollaire que Dirichlet a introduit ses s\'eries $L$.

\subsubsection{Remarque: densit\'e naturelle et densit\'e analytique} \label{section2.2}

\begin{defn}\phantomsection\label{d2.2.1} Soit $T$ un sous-ensemble de $\sP$ . On dit que $T$
a pour \emph{densit\'e naturelle} $k$ si le rapport 
$$\frac{|\{p\in T\mid p<n\}|}{ |\{p\in \sP\mid p<n\}|}$$
tend vers $k$ quand $n\longrightarrow\infty$.
\end{defn}

On a une seconde notion de densit\'e moins na\"\i ve:

\begin{defn}\phantomsection\label{d2.2.2} Soient $T$ un sous-ensemble de $\sP$ et $k\in \R$.
On dit que $T$ a pour {\sl densit\'e analytique} $k$ si le rapport 
$$\frac{\displaystyle\sum_{p\in T}{1/p^s}}{ \displaystyle \sum_{p\in\sP}{1/p^s}}$$ tend vers $k$ lorsque $s>1$ tend vers $1$.
\end{defn}

Notons que pour les deux notions de densit\'e on a n\'ecessairement $0\leq k\leq
1$ et qu'un ensemble fini a une densit\'e nulle: pour la densit\'e analytique,
on l'a vu au cours de la d\'emonstration du th\'eor\`eme \ref{tdirichlet}. Ce dernier \'enonce donc que, 
pour $(a, m)=1$, l'ensemble
\[\sP_{a\pmod m}=\{p \in\sP \mid p \equiv  a\pmod m\}\] 
a
pour densit\'e analytique $\displaystyle\frac{1}{\varphi(m)}$, o\`u $\varphi$
est l'indicateur d'Euler.

Si un ensemble de nombres premiers a une densit\'e
naturelle, il a une densit\'e analytique et ces deux densit\'es
co\"{\i}ncident (exercice \ref{exo4.1}). On peut montrer que $\sP_{a\pmod m}$ a une densit\'e naturelle
(donc \'egale \`a 
$\displaystyle\frac{1}{\varphi(m)}$).  Par contre, on peut exhiber des $T\subset\sP$ ayant
une densit\'e analytique, mais pas de densit\'e naturelle (Bombieri, \emph{cf}
\cite[ch. 7, exercices 14 et 15]{ellison}). 

\subsubsection{Autres applications du th\'eor\`eme de Dirichlet}\index{Théorème!de Dirichlet} Le th\'eor\`eme \ref{t4.1} a des applications arithm\'etiques importantes \`a des \emph{formules analytiques du nombre de classes} (pour les corps quadratiques et les corps cyclotomiques $\Q(\mu_p)$ avec $p$ premier): voir remarque \ref{r4.2} 2) et Borevi\v c-\v Safarevi\v c \cite[ch. VI]{borsaf}.

\begin{exo}\phantomsection\label{exo4.1} Soit $T$ comme dans la définition \ref{d2.2.2}.  On note
\begin{align*}
\pi_T(n) = |\{p\in T\mid p<n\}|, \quad& \pi(n) =\pi_\sP(n) \quad (n\in \N)\\
 F_T(s) = \sum_{p\in T}{1/p^s}, \quad& F(s)= F_\sP(s)\quad  (s>1).
 \end{align*}
 \begin{enumerate}[label=(\alph*)]
\item V\'erifier l'identit\'e
\[F_T(s) = \sum_{n= 1}^\infty \frac{\pi_T(n+1)-\pi_T(n)}{n^s}.\]
\item En d\'eduire l'identit\'e
\[F_T(s) = \sum_{n= 1}^\infty \pi_T(n+1)\left(\frac{1}{n^s}-\frac{1}{(n+1)^s}\right).\]
\item Supposons que $T$ ait pour densit\'e naturelle $\delta$. Soit $\epsilon>0$: il existe donc $n_0\gg 0$ tel que $|\pi_T(n)-\delta\pi(n)|<\epsilon\pi(n)$ d\`es que $n> n_0$. 
En consid\'erant $F_T(s)-\delta F(s)$, montrer que $T$ a pour densit\'e analytique $\delta$. (On s\'eparera la s\'erie en deux morceaux: $\sum_{n=1}^{n_0-1}$ et $\sum_{n=n_0}^\infty$, et on utilisera le fait que $\lim_{s\to 1} F(s)=+\infty$, \cf preuve du th\'eor\`eme \ref{tdirichlet}.)
\end{enumerate}
\end{exo}

\begin{exo}\phantomsection\label{exo4.2} Soit $a\in\Z-\{0\}$.
\begin{enumerate}[label=(\alph*)]
\item Si $p$ est un nombre premier ne divisant pas $2a$, montrer que $a^{\frac{p-1}{2}}\equiv \pm 1\pmod{p}$. On note $\leg{a}{p}$ l'\'el\'ement de $\{\pm 1\}\subset \Z$ qui s'envoie sur $a^{\frac{p-1}{2}}$ modulo $p$: c'est le \emph{symbole de Legendre de $a$ modulo $p$}.
\item Montrer que, $p$ impair \'etant donn\'e, il existe $b$ tel que $\leg{b}{p}=-1$.\\
Pour la suite, on admettra la \emph{loi de r\'eciprocit\'e quadratique}
\[\leg{p}{q}\leg{q}{p} = (-1)^{\frac{p-1}{2}\cdot \frac{q-1}{2}}\]
pour deux nombres premiers impairs $p,q$, et les formules compl\'ementaires
\[\leg{-1}{p} = (-1)^{\frac{p-1}{2}}, \quad \leg{2}{p} = (-1)^{\frac{p^2-1}{8}}.\]
\item Supposons $a$ premier impair, et soit $m$ un entier premier \`a $2a$:  en utilisant la loi de r\'eciprocit\'e quadratique, donner une expression du produit $\prod\limits_{l\mid m} \leg{a}{l}^{v_l(m)}$ en fonction de $\leg{m}{a}$.
\item Donner une expression du m\^eme produit quand $a=-1$ et $a=2$, en utilisant les formules compl\'ementaires.
\item  Montrer qu'il existe un unique caract\`ere de Dirichlet\index{Caractère!de Dirichlet}
$\chi_a$ modulo $4a$ tel que $\chi_a(p)=\leg{a}{p}$ pour tout nombre
premier $p$ ne divisant pas $2a$. (Se ramener \`a $a$ premier ou $a=-1$, et utiliser (c), (d).)
\item Si $a$ est premier ou $a=-1$, montrer que $\chi_a\ne 1$.
\item Si $a$ n'est pas un carr\'e, montrer que $\chi_a\ne 1$. (\'Ecrire $a=l^na_1$ avec $l$ premier, $n$ impair et $(a_1,l)=1$; utiliser (f) et le lemme chinois.) 
\item Supposons que $a$ soit un carr\'e modulo $p$ pour tout
$p\in\mathcal{P}$ sauf un nombre fini. Montrer que $a$ est un carr\'e. (Utiliser (g) et le th\'eor\`eme de la progression arithm\'etique.)
\end{enumerate}
\end{exo}

\subsection{Premi\`ere g\'en\'eralisation: fonctions $L$ de Hecke}

\subsubsection{Produits directs restreints}\label{resprod}

\begin{defn}\phantomsection Soit $(G_i)_{i\in I}$ une famille d'ensembles et, pour tout $i\in I$, soit $U_i$ un sous-ensemble de $G_i$. On note
\[\resprod_{i\in I} (G_i,U_i) = \{(g_i)\in \prod_i G_i\mid \{j\in I\mid g_j\notin U_j\} \text{ est fini}\}.\]
C'est le \emph{produit direct des $G_i$ restreint relativement aux $U_i$}.
\end{defn}

On a \'evidemment
\[\resprod_{i\in I} (G_i,U_i) =\colim_{J\subset I \text{ fini}} \prod_{i\in J} G_i\times \prod_{i\notin J} U_i.\]

\subsubsection{Places des corps globaux} 
 Soit $K$ un corps global: corps de nombres ou corps de fonctions d'une variable sur un corps fini. Notons $\Sigma_K$ l'ensemble des places de $K$ (classes d'\'equivalences de valeurs absolues non triviales). Notons aussi $\Sigma_K^f$ l'ensemble des places non archim\'ediennes et $\Sigma_K^\infty$ l'ensemble des places archim\'ediennes. Si $\car K=0$, les $v\in \Sigma_K^f$ correspondent bijectivement aux id\'eaux premiers de $O_K$ et les $v\in \Sigma_K^\infty$ aux plongements $K\inj \C$; si $\car K>0$, on a $\Sigma_K^\infty=\emptyset$ et $\Sigma_K=\Sigma_K^f$ est en bijection avec les points ferm\'es d'une courbe $C$, l'unique mod\`ele projectif lisse de $K$ sur son corps des constantes.

Pour $v\in \Sigma_K$, notons $K_v$ le compl\'et\'e de $K$ par rapport \`a $v$. Ainsi:

\begin{enumerate}
\item Si $K$ est un corps de nombres, $K_v$ est un corps $p$-adique si $v$ est non archim\'edienne, correspondant \`a un id\'eal $\fp\subset O_K$ au-dessus de $p$, $K_v=\R$ ou $\C$ si $v$ est archm\'edienne.
\item Si $K$ est un corps de fonctions de corps des constantes $k=\F_q$, $K_v\simeq k_v((t))$ o\`u $k_v$ est une extension finie de $k$ (le corps r\'esiduel de $C$ en $v$, si $C$ est le mod\`ele projectif lisse de $K/k$).
\end{enumerate}

\subsubsection{Ad\`eles et id\`eles} Soit $K$ un corps global.

\begin{defn}\phantomsection a) L'\emph{anneau des ad\`eles de $K$} est l'anneau
\[\bA_K = \resprod_{v\in \Sigma_K} (K_v,O_v)\]
o\`u $O_v$ est l'anneau de valuation de $v$ si $v$ est non archim\'edienne et $O_v=K_v$ si $v$ est archim\'edienne.\\
b) Le \emph{groupe des id\`eles} de $K$ est le groupe des unit\'es $\bI_K$ de $\bA_K$:
\[\bI_K = \resprod_{v\in \Sigma_K} (K_v^*,O_v^*).\]
c) Si $K$ est un corps de nombres, on note aussi
\[\bA^f_K=\resprod_{v \in \Sigma_K^f} (K_v,O_v)\]
(ad\`eles finies) et $\bI^f_K$ le groupe des id\`eles correspondant.
\end{defn}

\begin{rque}\phantomsection Voici une description diff\'erente de $\bA_K$ si $\car K=0$: $\bA_K=\bA_\Z\otimes_\Z K$, avec
\[\bA_\Z=\R\times \prod_p \Z_p.\]
En caract\'eristique $p$, on peut faire la m\^eme chose en partant de l'anneau
\[\bA = \prod_v O_v\]
o\`u $v$ d\'ecrit toutes les places de $K$.
\end{rque}

Les id\`eles ont \'et\'e introduites par Claude Chevalley en 1936 \cite{chevalley}. Les ad\`eles ont \'et\'e introduites un peu plus tard par Artin et Whaples sous le nom de \emph{valuation vectors} \cite{artin-whaples}, nom remplac\'e par ad\`ele par Weil (dans \cite{weil-classes}?)  Le but \'etait de reformuler la th\'eorie du corps de classes sous une forme plus conceptuelle. Une notion plus ancienne d'``anneau des r\'epartitions'' (sans compl\'etions) est due à Weil.

\subsubsection{Topologie des ad\`eles et des id\`eles} R\'ef\'erences: Cassels-Fr\"ohlich \cite[ch. II]{cf},  Lang \cite[ch. VII]{lang}. 

Partons d'une famille $(G_i,U_i)_{i\in I}$ d'\emph{espaces topologiques} comme au num\'ero \ref{resprod}. Supposons $G_i$ localement compact et $U_i$ ouvert dans $G_i$ pour tout $i$; supposons de plus $U_i$ compact pour presque tout $i$ (\ie tout $i$ sauf un nombre fini). Alors les produits $\prod_{i\in J} G_i\times \prod_{i\notin J} U_i$ sont localement compacts, et leur limite inductive $\resprod_{i\in I}(G_i,U_i)$ est \emph{localement compacte} pour la topologie limite inductive.

Ceci s'applique \`a $\bA_K$, $\bI_K$, $\bA_K^f$ et $\bI_K^f$.

\begin{meg}\phantomsection\label{cav2} L'inclusion $\bI_K\inj \bA_K$ est continue, mais la topologie de $\bI_K$ est plus forte que la topologie induite par celle de $\bA_K$ (pour laquelle $x\mapsto x^{-1}$ n'est pas continue). La topologie de $\bI_K$ est celle induite par le plongement
\begin{align*}
\bI_K&\to \bA_K\times \bA_K\\
x&\mapsto (x,x^{-1})
\end{align*}
d'image ferm\'ee.
\end{meg}

On a le plongement diagonal $K\inj \bA_K$, et

\begin{thm}\phantomsection\label{tad} a) $K$ est discret dans $\bA_K$; le quotient $\bA_K/K$ est compact.\\
b) {\bf Th\'eor\`eme d'approximation forte}. Soit $v_0\in \Sigma_K$. Alors le plongement diagonal
\[K\inj \resprod_{v\ne v_0} (K_v,O_v)\]
est d'image dense.
\end{thm}

\begin{proof}[Esquisse] Pour a), soit $K_0=\Q$ ou $\F_q(t)$. Le choix d'une base de $K$ sur $K_0$ fournit un isomorphisme $\bA_K\simeq \bA_{K_0}^{[K:K_0]}$ compatible aux plongements diagonaux. On se ram\`ene ainsi \`a $K=K_0$, et cela r\'esulte d'un calcul facile \cite[p. 65]{cf}. Pour b), on utilise une forme ad\'elique du th\'eor\`eme de Minkowski sur les points des r\'eseaux de $\R^N$, \cf \cite[ch. II, \S 15]{cf}.
\end{proof}

Pour le cas des id\`eles, il faut introduire la \emph{norme}
\begin{align}\label{eq:norme}
||\,||:\bI_K&\to \R^{+*}\\
||(a_v)_{v\in \Sigma_K}||&=\prod_v ||a_v||_v\notag
\end{align}
o\`u, pour tout $v$, $||\,||_v$ est la ``valeur absolue'' normalis\'ee sur $K_v$. Rappelons cette derni\`ere:

\begin{description}
\item[$K_v=\R$] $||a||_v=|a|$.
\item[$K_v=\C$] $||a||_v=|a|^2$.
\item[$K_v$ est non archim\'edien] $||a||_v= N(v)^{-v(a)}$, o\`u $N(v)$ est le cardinal du corps r\'esiduel de $K_v$ et $v(a)$ est la valuation normalis\'ee de $a$ ($v(\pi)=1$ si $\pi$ est une uniformisante).
\end{description}

\begin{thm}\phantomsection\label{tid} Le plongement diagonal
\[K^*\inj \bI_K\]
est d'image discr\`ete. Cette image est contenue dans $\bI_K^0=\Ker(||\,||)$, et $\bI_K^0/K^*$ est compact.
\end{thm}

\begin{proof} La premi\`ere assertion r\'esulte facilement du th\'eor\`eme \ref{tad} a). Le fait que $||a||=1$ pour tout $a\in K^*$ s'appelle la \emph{formule du produit}. Pour la d\'emontrer, on v\'erifie que si $K_0\subset K$, on a $||a|| = ||N_{K/K_0}(a)||$ (ceci s'\'etend \`a $\bI_K$ et $\bI_{K_0}$ pour la d\'efinition convenable de la norme). On est alors ramen\'e \`a $K=\Q$ ou $\F_q(t)$; dans les deux cas on a un anneau d'entiers qui est principal, et la v\'erification devient facile. (Pour une autre d\'emonstration, voir \cite[th. 4.3.1]{tate}.)

Pour la compacit\'e du quotient, on utilise le th\'eor\`eme \ref{tad} b), en montrant d'abord que la topologie de $\bI_K^0$ est induite par celle de $\bA_K$, \cf \ref{cav2} \cite[ch. II, \S 16]{cf}.
\end{proof}

\begin{defn}\phantomsection Le quotient $C_K=\bI_K/K^*$ s'appelle le \emph{groupe des classes d'id\`eles}.
\end{defn}

On a donc une suite exacte
\[0\to C_K^0\to C_K\by{||\,||} \R^{+*}\to 0\]
o\`u $C_K^0$ est le groupe des classes d'id\`eles de norme $1$.

\subsubsection{Raccord avec la th\'eorie alg\'ebrique des nombres ``classique''} Les id\`eles ont \'et\'e introduites pour mettre ensemble les groupes de classes g\'en\'eralis\'es:

Si $A$ est un anneau de Dedekind, le groupe des classes d'id\'eaux $\Cl(A)$ est le quotient du groupe des id\'eaux fractionnaires par le groupe des id\'eaux fractionnaires principaux: il s'identifie au groupe de Picard \allowbreak $\Pic(\Spec A)$.

On g\'en\'eralise cette d\'efinition ainsi: soit $K$ un corps global.

\begin{defn}\phantomsection\label{d4.1} On note $I_K$ le groupe des diviseurs de $K$:
\begin{thlist}
\item Si $\car K=0$, c'est le groupe des id\'eaux fractionnaires de l'anneau des entiers $O_K$.
\item Si $\car K=p>0$, on a $K=\F_q(C)$ pour une courbe projective lisse $C$ g\'eom\'etrquement connexe, o\`u $q$ est une puissance de $p$. On pose $I_K=\Div(C)=Z_0(C)$.
\end{thlist}
On a un homomorphisme surjectif
\begin{align*}
\phi:\bI_K&\to I_K\\
(a_v) &\mapsto \sum_{v\in \Sigma_K^f} v(a_v) v
\end{align*}
continu (pour la topologie discr\`ete sur $I_K$) et tel que  $\phi(f)=(f)$ (diviseur de $f$) pour tout $f\in K^*$.
\end{defn}

\begin{lemme}\phantomsection\label{l4.1} Si $\car K=0$, la restriction de $\phi$ \`a $\bI_K^0$ est surjective.
\end{lemme}

\begin{proof} Si $\fp$ est un id\'eal premier de $O_K$, l'id\`ele $(a_v)$ avec $a_v=1$ pour $v\ne \fp$, $v_\fp(a_\fp)=1$, $||a_v||_v=||a_\fp||_\fp^{-1}$ pour une place archim\'edienne $v$ et $a_v=1$ pour les autres places archim\'ediennes est dans $\bI_K^0$ et s'envoie sur $\fp$.
\end{proof}

\begin{defn}\phantomsection  Un \emph{module} est une combinaison lin\'eaire formelle
\[\fm = \sum_{v\in \Sigma_K} n_v v\]
o\`u $n_v\in \N$ est nul pour presque tout $v$, et $n_v=0$ ou $1$ si $v$ est archim\'edienne. L'ensemble $S=\{v\mid n_v\ne 0\}$ est le \emph{support de $\fm$}.
\end{defn}

\begin{defn}\phantomsection\label{d4.3} Soit $\fm=\sum n_v v$ un module de support $S$.\\
a) On note $I_K^S$ le sous-groupe de $I_K$ form\'e des $\fa$ premiers \`a $S$: $v(\fa)=0$ si $v$ est une place finie de $S$.\\
b) Soit $f\in K^*$. On note $f\equiv 1 \modmult{\fm}$ si
\begin{thlist}
\item $v(f-1)\ge n_v$ lorsque $v$ est finie et $n_v>0$;
\item $f>0$ si $v$ est r\'eelle et $n_v=1$.
\end{thlist}
c) On note 
\[K^\fm=\{ f\in K^*\mid f\equiv 1\modmult{\fm}\}\]
et $\Cl^\fm(K)$ le quotient $I_K^S/(K^\fm)$. C'est le \emph{groupe des classes g\'en\'eralis\'ees} relatives au module $\fm$. (On dit souvent: corps de classes de rayon ou \emph{ray class field}.)
\end{defn}

\begin{exs}\phantomsection\label{ex4.1} 1) $\fm=0$: on retrouve le groupe des classes (de $O_K$ en caract\'eristique z\'ero, de $C$ en caract\'eristique $p$).\\
2) $K=\Q$, $\fm=m$. On trouve un isomorphisme $(\Z/m)^*\iso \Cl^m(\Q)$.
\end{exs}

Revenant aux id\`eles, si $\fm$ est un module, on note 
\[\bI_K^\fm = \{a\in \bI_K\mid a\equiv 1\modmult{\fm}\}\]
(g\'en\'eralisation directe de la d\'efinition \ref{d4.3}), de sorte que  $K^\fm= K^*\cap \bI_K^\fm$. L'homomorphisme $\phi$ de la d\'efinition \ref{d4.1} envoie $\bI_K^\fm$ dans $I_K^S$, et

\begin{prop}\phantomsection\label{p4.1} On a un diagramme
\[\begin{CD}
\bI_K^\fm/K^\fm@>i>> C_K\\
@V{\bar \phi}VV \\
\Cl^\fm(K)
\end{CD}\]
o\`u $i$, induit par l'inclusion $\bI_K^\fm\subset \bI_K$, est bijectif et $\bar \phi$, induit par $\phi$, est surjectif et continu (pour la topologie discr\`ete de $\Cl^\fm(K)$). Si $\car K=0$, $\Cl^\fm(K)$ est fini.
\end{prop}

\begin{proof}[Esquisse] La surjectivit\'e de $i$ r\'esulte facilement du th\'eor\`eme d'approximation forte, et m\^eme de moins. Pour la finitude, le lemme \ref{l4.1} montre que la restriction de $\bar \phi$ \`a $(\bI_K^\fm\cap \bI_K^0)/K^\fm$ est surjective. Donc $\Cl^\fm(K)$ est un quotient discret d'un groupe compact, et est fini.
\end{proof}

\begin{rque}\phantomsection La finitude de $\Cl^\fm(K)$ est fausse si $\car K>0$: d'une part le lemme \ref{l4.1} est faux dans ce cas, et on a m\^eme un diagramme commutatif
\[\begin{CD}
\bI_K^\fm/K^\fm@>||\,||>> \R^*\\
@V{\bar \phi}VV @A\rho AA\\
\Cl^\fm(K) @>\deg>> \Z
\end{CD}\]
o\`u $\rho(1)= q^{-1}$, si $\F_q$ est le corps des constantes de $K$ et o\`u l'homomorphisme $\deg$ est surjectif. On en d\'eduit que $\Cl^\fm(K)^0:=\Ker\deg$ est fini, ce qui est l'\'enonc\'e correct.

Pour avoir un parall\`ele parfait entre caract\'eristique z\'ero et caract\'eristique $p$, il faudrait introduire en caract\'eristique z\'ero des groupes de diviseurs compactifi\'es (munis de m\'etriques \`a l'infini), \cf Szpiro \cite[\S 1]{szpiro}. Cette id\'ee remonte \`a Weil et a donn\'e naissance \`a la th\'eorie d'Arakelov.
\end{rque}

\subsubsection{Caract\`eres de Hecke ou Gr\"o\ss encharaktere: version id\'elique}\index{Caractère!de Hecke}

\begin{defn}\phantomsection
Un \emph{quasicaract\`ere} d'un groupe localement compact $G$ est un homomorphisme continu $\chi:G\to \C^*$. C'est un \emph{caract\`ere} s'il prend ses valeurs dans  le cercle unit\'e $S^1\subset \C^*$.
\end{defn}

\begin{defn}\phantomsection  Un quasicaract\`ere $\chi:\bI_K\to \C^*$ est \emph{non ramifi\'e en $v\in \Sigma_K^f$} si $\chi(O_v^*)= 1$.
\end{defn}

\begin{defn}\phantomsection Un \emph{caract\`ere de Hecke} est un quasicaract\`ere de $C_K$.
\end{defn}

\begin{ex}\phantomsection La norme $||\,||$ \eqref{eq:norme} est un caract\`ere de Hecke, appel\'e \emph{principal}.
\end{ex}

\begin{lemme}\phantomsection Soit $\chi$ un caract\`ere de Hecke. On peut \'ecrire
\[\chi=\prod_v \chi_v\]
o\`u $\chi_v = \chi_{|K_v^*}$; de plus, $\chi_v$ est non ramifi\'e pour presque tout $v$, \ie $\chi_v(O_v^*)=1$ pour tout $v\in \Sigma_K^f$ sauf un nombre fini.
\end{lemme}

\begin{proof} L'assertion ``non ramifi\'e'' provient de la continuit\'e de $\chi$.
\end{proof}

Le lemme suivant permet de se ramener des quasicaract\`eres aux caract\`eres:

\begin{lemme}\phantomsection\label{l4.2} Soit $\chi$ un caract\`ere de Hecke. Il existe un unique $\sigma\in \R$ tel que le caract\`ere de Hecke
\[a\mapsto \chi(a)||a||^{-\sigma}\]
soit \`a valeurs dans $S^1$ (\ie soit un ``vrai'' caract\`ere). On appelle $\sigma$ l'\emph{ex\-po\-sant} de $\chi$. Si $\chi$ est d'exposant $0$, on dit qu'il est \emph{unitaire}.
\end{lemme}

\begin{proof} La restriction de $|\chi|$ au groupe compact $C_K^0$ est triviale, puisque $\R^{+*}$ n'a pas de sous-groupe compact non trivial. Donc $|\chi|$ se factorise par $|| \, ||$. Mais tout endomorphisme continu de $\R^{+*}$ est de la forme $x\mapsto x^\sigma$ pour un $\sigma\in \R$, comme on le voit en passant par le logarithme.
\end{proof}

\subsubsection{Passage aux caract\`eres admissibles} Cette exposition est extraite de l'expos\'e de Tate \cite[ch. VII, \S\S 3,4]{cf}.

\begin{defn}\phantomsection\label{d4.2} Soit $S\subset \Sigma_K$ un ensemble fini contenant $\Sigma_K^\infty$; notons $I_K^S$ l'ensemble des diviseurs premiers \`a (l'ensemble des places finies de) $S$. Si $f\in K^*$, on note $(f)^S$ la partie du diviseur de $f$ qui appartient \`a $I^S$, soit $(f)^S=\sum_{v\in \Sigma_K^f- S} v(f)v$.
\end{defn}

\begin{defn}\phantomsection \label{d4.4} Un \emph{couple admissible} est un couple $(S,\phi)$ o\`u $S$ est comme ci-dessus et   $\phi$ est un homomorphisme de $I_K^S$ vers un groupe topologique commutatif $G$ v\'erifiant:
\begin{quote} Pour tout voisinage $U$ de l'\'el\'ement neutre $1\in G$, il existe $\epsilon>0$ tel que $\phi((f)^S)\in U$ d\`es que $|f-1|_v<\epsilon$ pour tout $v\in S$.
\end{quote}
\end{defn}

\begin{prop}[\protect{\cite[ch. VII, prop. 4.1]{cf}}] \phantomsection\label{p4.2} Via l'isomorphisme $i$ de la proposition \ref{p4.1}, tout couple admissible $(S,\phi)$ induit un homomorphisme continu $\psi:C_K\to G$ tel que $\psi(a)=\phi((a)^S)$ pour tout $a\in \bI_K^S$, o\`u
\[\bI_K^S=\{(a_v)\in \bI_K\mid |a_v|=1 \text{ pour } v\in S\}\]
et $(a)^S$ est d\'efini comme en \ref{d4.2}.\\
R\'eciproquement, s'il existe un voisinage $U$ de $1\in G$ tel que le seul sous-groupe de $G$ contenu dans $U$ soit \'egal \`a $\{1\}$, alors tout homomorphisme continu $\psi:C_K\to G$ provient d'un couple admissible.
\end{prop}

On notera que, dans la proposition \ref{p4.2}, le couple admissible $(S,\phi)$ n'est pas uniquement d\'etermin\'e en fonction de $\psi$: ceci correspond \`a l'ambigu\"\i t\'e des caract\`eres de Dirichlet non primitifs et conduit \`a la g\'en\'eralisation du conducteur:

\begin{defn}\phantomsection Soit $\chi:C_K\to \C^*$ un caract\`ere de Hecke\index{Caractère!de Hecke}. On appelle \emph{conducteur} de $\chi$ le diviseur  $\ff(\chi)=\sum_{v\in \Sigma_K^f} n_v v$ o\`u $n_v$ est le plus petit entier $n$ tel que $\chi_v(1+\fp_v^n)=1$, o\`u $\fp_v$ est l'id\'eal maximal de $O_v$ ($n_v=0$ si $\chi_v(O_v^*)=1$).
 \end{defn}

On a aussi la notion de \emph{type \`a l'infini}:  On peut consulter \cite[ch. VIII, \S 1]{cf}, et les notes de Jerry Shurman \cite{jerry} pour une description plus d\'etaill\'ee de la situation. Bri\`evement, en appliquant la proposition \ref{p4.2} \`a un caract\`ere de Hecke $\chi$, on peut voir $\chi$ comme un homomorphisme
\[\chi:I_K^S\to \C^*\]
tel que, si $f\equiv 1\modmult{\ff(\chi)}$, $\chi((f))$ est d\'etermin\'e par le type \`a l'infini de $\chi$: c'est la d\'efinition originelle de Hecke \cite{hecke}. Donnons seulement deux exemples:

\begin{exs}\phantomsection \label{ex4.2} 1) En partant de l'exemple \ref{ex4.1}, on voit que les caract\`eres de Dirichlet sont des cas particuliers de caract\`eres de Hecke unitaires.\\\index{Caractère!de Hecke}\index{Caractère!de Dirichlet}
2) Pour tout corps global $K$ et tout $s\in \C$, la fonction $\fA\mapsto N(\fA)^{-s}$ d\'efinit un caract\`ere de Hecke (c'est celui correspondant \`a $||\,||^{s}:C_K\to \C^*$).
\end{exs}

\subsubsection{Fonctions $L$ de Hecke} \index{Fonction $L$!de Hecke}

\begin{defn}\phantomsection\label{d4.5} Soit $\chi$ un caract\`ere de Hecke, de conducteur $\ff$. On pose
\[L(\chi,s)=\sum_{\fa\in I_K^+} \frac{\chi(\fa)}{N(\fa)^s}\]
o\`u $I_K^+$ d\'esigne le mono\"\i de des diviseurs effectifs de $K$ et o\`u $\chi(\fa):=0$ si $(\fa,\ff)\ne 1$.
\end{defn}

\begin{prop}\phantomsection \label{p4.3} a) On a l'identit\'e
\[L(\chi\cdot ||\,||^t,s)=L(\chi,s+t).\]
b) On a $\rho^+(L(\chi,s))\le1+ \sigma$, o\`u $\sigma$ est l'exposant de $\chi$ (lemme \ref{l4.2}).\\
c) On a le produit eul\'erien habituel
\[L(\chi,s)=\prod_{\fp\nmid \fm} \left(1-\frac{\chi(\fa)}{N(\fa)^s}\right)^{-1}.\]
\end{prop}

\begin{proof} a) est \'evident (\cf exemple \ref{ex4.2} 2)). On en d\'eduit b) \`a partir du th\'eor\`eme \ref{t2.7.3} (ou du cas de la fonction z\^eta de Riemann)\index{Fonction zêta!de Riemann} en se ramenant \`a $\chi$ unitaire. Enfin, c) se d\'emontre comme d'habitude.
\end{proof}

\enlargethispage*{20pt}

\begin{rque}\phantomsection L'identit\'e de la proposition \ref{p4.3} a) peut aussi s'\'ecrire
 \[L(\chi\cdot ||\,||^t,s)=L(\chi\cdot ||\,||^{t+s},1)\]
ce qui permet de voir une fonction $L$ de Hecke comme une fonction continue sur l'espace des caract\`eres de Hecke. C'est cette interpr\'etation qui est \`a la base de la preuve que Tate a donn\'ee du th\'eor\`eme suivant.\index{Caractère!de Hecke}
\end{rque}

\begin{thm}[Hecke, 1920 \phantomsection\cite{hecke}] \label{t:hecke}Soit $\chi$ un caract\`ere de Hecke unitaire. Posons 
\[\Lambda(\chi,s) = \big(|d_K|N(\ff(\chi))\big)^{s/2}\prod_{v\in \Sigma_K^\R} \Gamma_\R(s+n_v)\prod_{v\in \Sigma_K^\C} \Gamma_\C(s+\frac{|n_v|}{2}) L(\chi,s)\]
o\`u $\Gamma_\R, \Gamma_\C$ ont été définis en \eqref{eq:gammar}, $d_K$ est le discriminant absolu de $K$ si $\car K=0$, $d_K=q^{2g-2}$ si $K=\F_q(C)$ o\`u $C$ est une courbe projective, lisse et g\'eom\'etriquement int\`egre sur $\F_q$, de genre $g$, et
\begin{description}
\item[$v$ r\'eelle] $n_v=1$ si $\chi_v\ne 1$ et $n_v=0$ si $\chi_v=1$.
\item[$v$ complexe] $n_v$ est l'entier $n$ tel que $\chi_v(z)=z^n$ pour $z$ de module $1$. En particulier, $n_v=0$ si $\chi$ est d'ordre fini.
\end{description}
Alors $\Lambda(\chi,s)$ admet un prolongement m\'eromorphe \`a tout le plan complexe, et v\'erifie l'\'equation fonctionnelle\index{Equation fonctionnelle@\'Equation fonctionnelle}
\[\Lambda(\bar \chi,1-s) = W(\chi) \Lambda(\chi,s)\]
o\`u $(\bar \chi)(x):= \overline{\chi(x)}$ et $W(\chi)\in \C^*$  poss\`ede une d\'ecomposition canonique
\[W(\chi)=\prod_{v\in \Sigma_K} W(\chi_v)\]
o\`u $|W(\chi_v)|=1$, $W(\chi_v)=1$ si $\chi_v$ est non ramifi\'e (en g\'en\'eral, $W(\chi_v)$ est une somme de Gauss normalis\'ee).\\
Si $\chi$ est  de la forme $||\,||^{it}$, $t\in\R$, $L(\chi,s)$ a un p\^ole simple en $s=1+it$ et est holomorphe ailleurs. Sinon, $L(\chi,s)$ est holomorphe dans le plan complexe.
\end{thm}

La preuve de Tate \cite{tate} consiste \`a faire de l'analyse de Fourier sur l'espace des ad\`eles; le point-clef est la formule de Poisson, qui est un analogue analytique du th\'eor\`eme de Riemann-Roch\index{Théorème!de Riemann-Roch}. Cela \'evite le recours qu'avait Hecke \`a des fonctions th\^eta, g\'en\'eralisant la d\'emonstration du th\'eor\`eme \ref{triemann}.

\subsection{Seconde g\'en\'eralisation: fonctions $L$ d'Artin} Pour plus de d\'etails, voir Martinet \cite{martinet}.

\subsubsection{Rappels: th\'eorie de la ramification} R\'ef\'erence: tout bon livre de th\'eorie alg\'ebrique des nombres.

Soient $A$ un anneau de Dedekind de corps des fractions $K$: si $\fp$ est un id\'eal premier de $A$, on note $\kappa(\fp)=A/\fp$ son \emph{corps r\'esiduel}. 

Soit $L$ une extension finie s\'eparable de $K$ et $B$ la cl\^oture int\'egrale de $A$ dans $L$: c'est un anneau de Dedekind fini sur $A$. Si $\fP$ est un id\'eal premier de $B$, $\fp=A\cap \sP$ est un id\'eal premier de $A$: on dit que $\fP$ est \emph{au-dessus de $\fp$} ou que $\fP$ divise $\fp$ (notation: $\fP\mid \fp$). Partant de $\fp$, on a
\[\fp B = \prod_{\fP\mid \fp} \fP^{e_\fP}\]
o\`u $e_\fP$ est (par d\'efinition) l'\emph{indice de ramification} de $\fP$. On a la formule
\begin{equation}\label{eq4.1}
[L:K]= \sum_{\fP\mid \fp} e_\fP f_\fP
\end{equation}
o\`u $f_\fP=[\kappa(\fP):\kappa(\fp)]$ est le \emph{degr\'e r\'esiduel} de $\fP$.

Supposons $L/K$ galoisienne de groupe $G$. Alors tous les $\fP\mid \fp$ sont conjugu\'es sous l'action de $G$. On peut donc noter $e_\fP=e_\fp$, $f_\fP=f_\fp$ et \eqref{eq4.1} se r\'eduit \`a
\begin{equation}\label{eq4.2}
[L:K]= e_\fp f_\fp g_\fp
\end{equation}
o\`u $g_\fp$ est le nombre d'id\'eaux premiers au-dessus de $\fp$.
Choisissons $\fP$ au-dessus de $\fp$: on note 
\[D_\fP=\{g\in G\mid g\fP=\fP\}\]
le \emph{groupe de d\'ecomposition de $\fP$}. Il op\`ere sur le corps r\'esiduel $\kappa(\fP)$; l'extension $\kappa(\fP)/\kappa(\fp)$ est galoisienne et l'homomorphisme induit
\[D_\fP\to \Gal(\kappa(\fP)/\kappa(\fp))\]
est \emph{surjectif}. Son noyau $I_\fP$ est le \emph{groupe d'inertie en $\fP$}.  On a donc
\[(G:D_\fP)=g_\fp,\quad (D_\fP:I_\fP)=f_\fp,\quad |I_\fP| = e_\fp.\]

Soit $\fP'\mid \fp$: il existe $g\in G$ tel que $\fP'=g\fP$, et on a
\[D_{\fP'} =gD_\fP g^{-1},\quad I_{\fP'} =gI_\fP g^{-1}.\]

\subsubsection{Cas des courbes}\label{4.4.2} Supposons $K=k(C)$, o\`u $C$ est une $k$-courbe projective, lisse et g\'eom\'etriquement int\`egre. En consid\'erant les anneaux locaux de $C$ on obtient une th\'eorie locale de la ramification, qu'on peut globaliser.

Soit $L$ une extension finie et s\'eparable de $K$, de corps des constantes $l$. Pour simplifier, je suppose que $l=k$: on peut toujours s'y ramener en rempla\c cant $K$ par $Kl$.

Soit $C'$ le mod\`ele projectif lisse de $L$. Le morphisme $C'\to C$ est fini et plat: c'est un ``rev\^etement ramifi\'e''. SI $x$ est un point ferm\'e de $C'$, on note $e_x$ l'indice de ramification de l'id\'eal premier de $\sO_{C',x}$. L'information globale fournie par la ramification est la \emph{formule de Riemann-Hurwitz}
\[2g'-2=[L:K](2g-2) +\sum_{x\in C'_{(0)}} (e_x-1)\]
o\`u $g$ est le genre de $C$ et $g'$ celui de $C'$.

(Cette formule est d\'emontr\'ee dans Hartshorne \cite[ch. IV, cor. 2.4]{hartshorne} quand $k$ est alg\'ebriquement clos: dans le cas g\'en\'eral, on peut soit recopier la d\'emonstration soit se ramener \`a ce cas en constatant que $g,g'$ et les $e_x$ sont invariants par extension des scalaires.)

Supposons $L/K$ galoisienne de groupe $G$: alors l'extension $l/k$ des corps de constantes est galoisienne de groupe $g$, et l'homomorphisme $G\to g$ est \emph{surjectif}.

\subsubsection{Les automorphismes de Frobenius} Soit $K$ un corps global, et soit $L/K$ une extension finie, galoisienne de groupe $G$. Soit $v$ une place finie de $K$ et $w$ une place finie de $L$ divisant $v$. Le corps r\'esiduel $\kappa(w)$ poss\`ede un automorphisme canonique: l'\emph{automorphisme de Frobenius} $\phi_{w/v}$. Si $\kappa(v)=\F_q$, on a
\[\phi_{w/v}(x)=x^q\]
pour $x\in \kappa(w)$.

\begin{defn}\phantomsection L'\emph{automorphisme de Frobenius de $G$ en $w$} est l'\'el\'ement de $D_w/I_w$ correspondant \`a $\phi_{w/v}$: on garde la notation $\phi_{w/v}$.
\end{defn}

Supposons $v$ non ramifi\'e, de sorte que $I_w=1$ pour tout $w\mid v$. Alors l'ensemble des $\phi_{w/v}$ forme une classe de conjugaison de $G$, not\'ee $\phi_v$.

\subsubsection{Les fonctions $L$ d'Artin}\index{Fonction $L$!d'Artin} Soit $\rho:G\to GL(V)$ une repr\'esentation lin\'eaire sur un $\C$-espace vectoriel $V$ de dimension finie.%, de caract\`ere $\chi$: rappelons que
%\[\chi(g)= \Tr \rho(g)\]
%pour $g\in G$. En particulier, $\chi:G\to \C$ est une \emph{fonction centrale} (invariante par conjugaison).

\begin{defn}[Artin, 1923-1930 \protect{\cite{artinL,artinL2}}]\phantomsection La \emph{fonction $L$ de $\rho$} est
\[L(\rho,s)=\prod_{v\in \Sigma_K^f} L_v(\rho,s)\]
avec
\[L_v(\rho,s)=\det(1-\rho(\phi_{w/v})N(v)^{-s}\mid V^{I_w})^{-1}.\]
\end{defn}

Expliquons cette d\'efinition: le groupe quotient $D_w/I_w$ op\`ere sur $V^{I_w}$; le polyn\^ome caract\'eristique de $\rho(\phi_{w/v})$ sur cet espace ne d\'epend pas du choix de $w$. Artin a commenc\'e par d\'efinir ses fonctions $L$ aux facteurs ramifi\'es pr\`es \cite{artinL}, puis a compris la formule pour ces derniers dans \cite{artinL2}.

\begin{prop}\phantomsection\label{p4.4} Les fonctions $L$ d'Artin convergent absolument pour $\Re(s)>1$ et ont les propri\'et\'es suivantes:
\begin{thlist}
\item $L(\rho\oplus \rho',s)=L(\rho,s)L(\rho',s)$.
\item Soit $L'/K$ une extension galoisienne de groupe $G'$, contenant $L/K$. Notons $\tilde\rho$ la repr\'esentation de $G'$ d\'eduite de $\rho$ (par ``inflation''). Alors $L(\tilde\rho,s)=L(\rho,s)$.
\item Soit $M/K$ une sous-extension de $L/K$, et soit $H=\Gal(L/M)$. Soit $\theta$ une repr\'esentation lin\'eaire de $H$. Alors $L(\Ind_H^G \theta,s)=L(\theta,s)$.
\end{thlist}
\end{prop}

\begin{proof} La convergence absolue se montre en majorant par la fonction z\^eta de Dedekind de $K$\index{Fonction zêta!de Dedekind}. Les propri\'et\'es (i) et (ii) sont triviales\footnote{Pour (ii), le point est que si $w'$ est une place de $L'$ au-dessus de $w$, l'homomorphisme $I_{w'}\to I_w$ est surjectif.}; montrons (iii) qui est la plus importante.

Soit $w\in \Sigma_L^f$ et soient $w_1,v$ les places induites sur $M$ et $K$ respectivement. Pour plus de pr\'ecision, notons $I_{w/w_1}$ et $I_{w/v}$ les groupes d'inertie correspondants. Soit $V$ l'espace de repr\'esentation de $\theta$, et soit $W$ l'espace de repr\'esentation de $\Ind_H^G\theta$. Il suffit de montrer l'\'egalit\'e
\[\det(1-(\Ind_H^G \theta)(\phi_{w/v})N(v)^{-s}\mid W^{I_{w/v}})^{-1}=  \det(1-\theta(\phi_{w/w_1})N(w_1)^{-s}\mid V^{I_{w/w_1}})^{-1}.\]

Posons $f=f_{w_1/w} = [\kappa(w_1):\kappa(v)]$. Comme $N(w_1)=N(v)^f$, il s'agit par \eqref{eq3.3} de d\'emonter l'identit\'e
\[
\sum_{n=1}^\infty\Tr ((\Ind_H^G \theta)(\phi_{w/v}^n)\mid W^{I_{w/v}})\frac{t^n}{n}
= \sum_{n=1}^\infty\Tr (\theta(\phi_{w/w_1}^n)\mid V^{I_{w/w_1}})\frac{t^{nf}}{n}
\]
soit
\begin{equation}\label{eq3.8}\Tr \left((\Ind_H^G \theta)(\phi_{w/v}^n)\mid W^{I_{w/v}}\right)=
\begin{cases}
0&\text{si $f\nmid n$}\\
f\Tr \left(\theta(\phi_{w/w_1}^{n/f})\mid V^{I_{w/w_1}}\right) &\text{si $f\mid n$}
\end{cases}
\end{equation}
ce qui est un petit exercice amusant \cite[ch. XII, \S 3]{lang}.
\end{proof}

La propri\'et\'e (ii) permet de parler de $L(\rho,s)$ pour une repr\'esentation complexe (continue) du \emph{groupe de Galois absolu} de $K$, sans faire le choix d'une extension finie galoisienne $L/K$.

\subsubsection{La loi de r\'eciprocit\'e d'Artin}

Soit $L/K$ une extension \emph{ab\'elienne} de groupe $G$. Soit $S\subset \Sigma_K$ un ensemble fini contenant $\Sigma_K^\infty$ et l'ensemble des places ramifi\'ees dans $L$. Alors, pour tout $v\in \Sigma_K-S$, l'automorphisme de Frobenius $\phi_v$ est un \'el\'ement bien d\'efini de $G$. D'o\`u un homomorphisme
\begin{align}\label{eq:artin}
F_{L/K}:I_K^S&\to G\\
v&\mapsto \phi_v.\notag
\end{align}

C'est l'\emph{application de r\'eciprocit\'e d'Artin}.

\begin{thm}[Artin, 1927 \protect{\cite{artin-rec}}] \phantomsection\label{t:artinrec}Le couple $(S,F_{L/K})$ de \eqref{eq:artin} est admissible au sens de la d\'efinition \ref{d4.4}. De plus, $F_{L/K}$ est surjectif. Il induit un isomorphisme
\[F_{L/K}:\Cl^\fm(K)\iso G\]
pour un module $\fm$ convenable.
\end{thm}

\begin{proof} Voir n'importe quel bon livre de th\'eorie du corps de classes. La surjectivit\'e r\'esulte d'une g\'en\'eralisation du th\'eor\`eme de la progression arithm\'etique: le \emph{th\'eor\`eme de densit\'e de \v Cebotarev}. L'injectivit\'e (\ie l'existence de $\fm$) est beaucoup plus difficile.
\end{proof}

\begin{rque}\phantomsection L'isomorphisme ``abstrait'' du th\'eor\`eme \ref{t:artinrec} est d\^u \`a Takagi (1922, \cite{takagi}); c'est Artin qui a construit cet isomorphisme explicite. Citons Tate \cite[p. 315]{tate-hilbert}:
\begin{quote}\it
How did Artin guess his reciprocity law? He was not looking for it, not trying to solve a Hilbert problem.  Neither was he, as would seem so natural to us today, seeking a natural isomorphism, to make Takagi's theory more functorial. He was led to the law trying to show that a new kind of $L$-series which he introduced really was a generalization of the usual $L$-series. 
\end{quote}
\end{rque}

C'est le contenu du

\begin{cor}\phantomsection\label{c4.3} Supposons que la repr\'esentation $\rho$ de $G$ soit irr\'eductible (\ie de degr\'e $1$). Alors 
\[L(\rho,s)=L(\rho \circ F_{L/K},s)\]
est une fonction $L$ de Hecke.\index{Fonction $L$!de Hecke}
\end{cor}

\subsubsection{Les facteurs locaux archim\'ediens et le conducteur} Pour \'ecrire l'\'equation fonctionnelle que va v\'erifier $L(\rho,s)$, il est pratique d'introduire tout de suite la ``fonction $L$ compl\'et\'ee'' correspondante. 

Commen\c cons par les facteurs \`a l'infini lorsque $\car K=0$. Si $v$ est une place r\'eelle et si $w\mid v$, l'extension $L_w/K_v$ est de degr\'e $1$ ou $2$. Soit $G_w$ son groupe de Galois, et soit $\phi_w$ son g\'en\'erateur (le Frobenius \`a l'infini!): pour $\epsilon=\pm 1$, on pose
\[V^\epsilon =\{v\in V\mid \phi_wv=\epsilon v\}\]
(donc $V^+=V$, $V^-=0$ si $G_w=1$).

\begin{defn}\phantomsection\label{d:gamma} Pour $v\in \Sigma_K^\infty$,
\[\Gamma_v(\rho,s)=
\begin{cases}
\Gamma_\R(s)^{\dim V^+}\Gamma_\R(s + 1)^{\dim V^-} &\text{si $v$ est r\'eelle}\\
\Gamma_\C(s)^{\dim V} &\text{si $v$ est complexe}.
\end{cases}
\]
\end{defn}

La d\'efinition du \emph{conducteur d'Artin} $\ff(\rho)$ \cite{artin-cond} est beaucoup plus d\'elicate: je renvoie par exemple \`a Serre \cite[ch. VI]{serre-cl} o\`u elle est donn\'ee localement en termes des groupes de ramification sup\'erieurs (cette d\'efinition d\'epend du \emph{th\'eor\`eme de Hasse-Arf}).

On pose maintenant
\[\Lambda(\rho,s)=\left(|d_K|^{\dim V}N(\ff(\rho))\right)^{s/2} \prod_{v\in \Sigma_K^\infty}\Gamma_v(\rho,s)\cdot L(\rho,s).\]

\subsubsection{Prolongement analytique et \'equation fonctionnelle}\index{Equation fonctionnelle@\'Equation fonctionnelle}

\begin{thm}[Artin-Brauer, (1947) \phantomsection\cite{brauer}]\label{t4.2} Pour toute repr\'esentation galoisienne $\rho$, la fonction $L$ d'Artin $L(\rho,s)$ se prolonge analytiquement \`a tout le plan complexe et admet une \'equation fonctionnelle de la forme
\[\Lambda(\rho^\vee,1-s)=W(\rho)\Lambda(\rho,s)\]
o\`u $\rho^\vee$ est la repr\'esentation duale de $\rho$ et $W(\rho)$ est un nombre complexe de module $1$.
\end{thm}

\begin{proof} Elle repose sur le \emph{th\'eor\`eme de Brauer} \cite{brauer}: toute repr\'esentation lin\'eaire de $G$ est combinaison lin\'eaire \`a coefficients dans $\Z$ d'induites de repr\'esentations de degr\'e $1$. Une version de ce th\'eor\`eme avait \'et\'e prouv\'ee ant\'erieurement par Artin, en rempla\c cant $\Z$ par $\Q$. Cela d\'emontrait le th\'eor\`eme \ref{t4.2} apr\`es \'el\'evation \`a une puissance enti\`ere.

Pour faire la d\'emonstration, il faut v\'erifier que les fonctions $L$ compl\'et\'ees continuent \`a avoir la propri\'et\'e d'induction (iii) de la proposition \ref{p4.4}, ce qui se fait facteur par facteur. On est alors ramen\'e par (une extension du) corollaire \ref{c4.3} au cas d'une fonction $L$ de Hecke, et on applique le th\'eor\`eme \ref{t:hecke}.
\end{proof}

\subsubsection{La conjecture d'Artin}\index{Conjecture!d'Artin}

Le th\'eor\`eme \ref{t4.2} donne la m\'eromorphie des $L(\rho,s)$, mais pas la

\begin{conj}[Artin]\phantomsection Si $\rho$ est une repr\'esentation irr\'eductible diff\'erente de $1$, $L(\rho,s)$ est holomorphe dans le plan complexe.
\end{conj}

Cette conjecture est toujours ouverte, sauf dans les cas o\`u on sait prendre les coefficients positifs dans le th\'eor\`eme de Brauer et des cas de dimension $2$ ``automorphes'' \'etablis par Langlands et ses successeurs.\footnote{Il y a aussi le th\'eor\`eme d'Aramata-Brauer: si $L/K$ est finie galoisienne, $\zeta(L,s)/\zeta(K,s)$ est enti\`ere. Sa preuve repose sur le th\'eor\`eme d'Artin plut\^ot que sur celui de Brauer.} Toutefois, on a:

\begin{thm}[Weil]\phantomsection\label{t4.3} La conjecture d'Artin est vraie sur les corps de fonctions pour les repr\'esentations irr\'eductibles non triviales $\rho:G\to GL_n(\C)$ du groupe de Galois $G$ d'une extension $L/K$ \emph{r\'eguli\`ere}, et l'hypoth\`ese de Riemann\index{Hypothèse de Riemann!pour les courbes} est v\'erifi\'ee.
\end{thm}

La preuve est dans  \cite[\S V]{weil2}. En fait, $Z(\rho,t)$ est un polyn\^ome $P(t)\in \Z[t]$ dont les racines inverses sont de valeurs absolues complexes $\sqrt{q}$, si $q$ est le nombre d'\'el\'ements du corps des constantes de $K$. L'id\'ee de la d\'emonstration est de plonger l'alg\`ebre du groupe $\Z[G]$ dans l'anneau des correspondances $\Corr_\equiv(C,C)$ o\`u $C$ est le mod\`ele projectif lisse de $K$, et d'utiliser le th\'eor\`eme de positivit\'e \ref{t:lemmeimportant}.

\begin{rques} a) Le th\'eor\`eme \ref{t4.3} est faux si on enl\`eve l'hypoth\`ese ``r\'eguli\`ere'' (c'est-\`a-dire si $K$ et $L$ n'ont pas le m\^eme corps des constantes): prenons par exemple $K=\F_q(x)$, $L=\F_{q^2}(x)$ et pour $\rho$ le caract\`ere de valeur $-1$ de $G=\Gal(L/K)=\Z/2$. Alors 
\[Z(\rho,t)=\frac{Z(\P^1_{\F_{q^2}},t)}{Z(\P^1_{\F_q},t)}= \frac{(1-t^2)^{-1}(1-q^2t^2)^{-1}}{(1-t)^{-1}(1-qt)^{-1}}=\frac{1}{(1+t)(1+qt)}.\]
Ceci contraste fortement avec le cas des corps de nombres, o\`u on ne demande pas que $K$ et $L$ aient le m\^eme nombre de racines de l'unit\'e.\\
b) Nous donnerons une d\'emonstration motivique du th\'eor\`eme \ref{t4.3} au \S \ref{cor-loco}.
\end{rques}

\subsection{Le mariage d'Artin et de Hecke} 
Il est r\'ealis\'e par Weil qui introduit ses $W$-groupes (appel\'es commun\'ement \emph{groupes de Weil}), modifications d'un groupe de Galois par l'interm\'ediaire du corps de classes \cite{weil-classes}: c'est un groupe localement compact. Dans le cas local non archimédien, ou global de caractéristique $>0$, le groupe de Weil $W(K)$ est l'image réciproque de $\Z$ (engendré par Frobenius) par l'homomorphisme continu $\pi:G_K\to \hat \Z$, où $G_K$ est le groupe de Galois absolu de $K$: si $K$ est local (resp. global), $\pi$ est donné par la restriction à l'extension non ramifiée maximale de $K$ (resp. au sous-corps des constantes). Le groupe de Weil de $\C$ (resp. de $\R$) est $\C^*$ (resp. l'extension non triviale de $\mu_2$ par $\C^*$ réalisée par $\langle \C^*,j\rangle$ dans les unités du corps des quaternions réels).  Dans le cas des corps de nombres, le groupe de Weil $W(L/K)$ d'une extension galoisienne finite de groupe $G$ est une certaine extension de $G$ par le groupe des classes d'id\`eles $C_L$; il y a ensuite un travail hautement non trivial pour montrer que ces extensions forment un système cohérent ``convergeant'' vers un groupe de Weil $W(K)$, extension de $G_K$. Voir aussi Artin-Tate \cite{at}.

Ceci permet à Weil de d\'efinir des fonctions $L$ qui g\'en\'eralisent \`a la fois celles de Hecke et d'Artin, et d'étendre la factorisation d'Artin des fonctions zêta à toutes les fonctions $L$ de Hecke.  
Voir aussi le commentaire de Tate \`a la fin de \cite{tate-hilbert}.

\subsection{La constante de l'équation fonctionnelle}\index{Equation fonctionnelle@\'Equation fonctionnelle} Ce numéro est consacré à
 la constante $W(\rho)$ pour la fonction $L$ d'Artin d'une représentation galoisienne complexe $\rho$,  apparaissant dans l'équation fonctionnelle du théorème \ref{t4.2}. Deux excellentes références sont les articles de Martinet \cite{martinet} et de Tate \cite{tate-local}, dont nous résumons ici les points principaux.
 
Dans sa thèse \cite[4.5]{tate}, Tate donne une formule explicite pour $W(\rho)$ quand $\rho$ est de degré $1$\footnote{Plus correctement, il donne une telle formule pour la fonction $L$ de Hecke du caractère\index{Caractère!de Hecke} correspondant à $\rho$ par la théorie du corps de classes.}; par le théorème de Brauer, cela détermine la valeur de $W(\rho)$ pour $\rho$ quelconque. Mais Tate fait mieux: pour $\rho$ de degré $1$ il décompose $W(\rho)$ en un produit de facteurs locaux $W(\rho_v)$, ne dépendant que de la restriction $\rho_v$ de $\rho$ au groupe de Galois local $G_v$ en la place $v$. Plus précisément, $W(\rho_v)$ est une puissance de $i$ si $v$ est réelle et est une certaine somme de Gauss normalisée si $v$ est non archimédienne, égale à $1$ si $\rho$ est non ramifiée en $v$. Voir \emph{loc. cit.}, 2.5 pour les formules précises.

Peut-on étendre cette décomposition au cas général? La réponse est oui:

\begin{thm}[Langlands, Deligne]\footnote{Langlands et Deligne préfèrent exprimer ce résultat en termes de \emph{facteurs epsilon}.}\phantomsection\label{tld} Soit $K$ un corps local. Il existe une et une seule famille d'homomorphismes
\[W:R(G_E)\to S^1\]
où $E$ décrit les extensions finies de $K$, $G_E$ est le groupe de Galois absolu de $E$, $R(G_E)$ est le groupe des représentations complexes continues de $G_E$, et $S^1\subset \C^*$ est le cercle unité, ayant les propriétés suivantes:
\begin{thlist}
\item Si $\deg\rho=1$, $W(\rho)$ est le facteur local déterminé par Tate.
\item Si $E'/E$ est une extension finie et $\rho\in R(G_{E'})$ avec $\deg\rho=0$, on a $W(\Ind_{G_{E'}}^{G_E} \rho) = W(\rho)$.
\end{thlist}
On a
\begin{equation}\label{eq4.3}
W(\rho)W(\rho^*) = \det\nolimits_\rho(-1) 
\end{equation}
où $\rho^*$ est la représentation duale de $\rho$, et $\det_\rho$ est son déterminant (représentation de degré $1$).
\end{thm}

Cet énoncé s'étend sans difficulté aux représentations continues du groupe de Weil (voir numéro précédent). L'unicité est une conséquence presque immédiate du théorème de Brau\-er. L'existence est une toute autre affaire: Langlands en a donné une démonstration non publiée d'environ 400 pages \cite{langlands} avant que Deligne n'en trouve une démonstration beaucoup plus courte, par voie globale (voir \cite{deligne-constantes} et \cite{tate-local}).

Voici quelques propriétés remarquables de la constante $W(\rho)$:

\begin{thm} a) (Dwork, Deligne, \cite[cor. 4]{tate-local}) Si $\rho$ est irréductible et sauvagement ramifiée, $W(\rho)$ est une racine de l'unité. (C'est aussi le cas si $\rho$ est non ramifiée, par réduction au cas abélien.)\\
b) (Deligne \cite{deligne-constantes2}) Si $\rho$ est une représentation orthogonale de $G_K$ virtuelle de degré $0$ et de déterminant $1$, on a
\[W(\rho)=(-1)^{w_2(\rho)}\]
où $w_2(\rho)\in H^2(K,\Z/2)\simeq \Z/2$ est la \emph{seconde classe de Stiefel-Whitney} de $\rho$. 
\end{thm}

Notons qu'en vertu de \eqref{eq4.3}, $W(\rho)=\pm 1$ si $\rho$ est \emph{symplectique};  ce signe est intimement lié à la structure de module galoisien de l'anneau des entiers de $L$, où $\rho$ se factorise par $Gal(L/K)$ (Fröhlich, Taylor\dots, voir par exemple \cite{cassou}).

Signalons aussi les travaux de Deligne-Henniart et Henniart sur la variation de $W(\rho)$ quand on tord $\rho$ par un caractère pas trop ramifié \cite{deligne-henniart,henniart}. Enfin, comme $W(\rho)$ appartient à l'extension abélienne maximale $A$ de $\Q$, le groupe de Galois $\Gamma=Gal(A/\Q)$ \emph{opère sur les $W(\rho)$}, et son action peut être comparée à celle de $\Gamma$ sur le caractère de $\rho$ (Fröhlich, \cite[cor. 5.2]{martinet}); ceci permet de majorer le corps cyclotomique contenant $W(\rho)$ \cite{kahn-constantes}.

\section{Les fonctions $L$ de la g\'eom\'etrie}

\subsection{Fonctions zêta ``de Hasse-Weil''} \index{Fonction zêta!de Hasse-Weil}

\subsubsection{Un exemple de Weil \cite{weil-jacobi}} Il y d\'emontre une conjecture de Hasse: sur un corps $k$, soit $C$ la $k$-courbe alg\'ebrique plane d'\'equation
\[Y^e=\gamma X^f+\delta\]
$2\le e\le f$, $\gamma,\delta\in k^*$, $ef\in k^*$. Si $k$ est un corps de nombres, on peut r\'eduire $C$ modulo $\fp$ pour presque tout $\fp\subset O_K$ en gardant les hypoth\`eses, soit $C_\fp$. Soit $S$ l'ensemble des ``mauvais'' $\fp$. 

\begin{defn}[Hasse]\phantomsection
\[\sZ(C,s)=\prod_{\fp\notin S} \zeta(C_\fp,s).\]
\end{defn}

\begin{thm}[Weil \protect{\cite[p. 495]{weil-jacobi}}]\phantomsection\label{twh} $\sZ(C,s)$ admet un prolongement m\'eromorphe \`a $\C$ et satisfait une \'equation fonctionnelle de la forme suivante: $\sZ(C,2-s)\sZ(C,s)^{-1}$ peut s'exprimer comme produit fini de ``facteurs \'el\'ementaires'' (incluant bien s\^ur des fonctions gamma) qui pourraient \^etre facilement \'ecrits explicitement.\index{Equation fonctionnelle@\'Equation fonctionnelle}
\end{thm}

Citons Weil (commentaires sur \cite{weil-jacobi}):

\begin{quote}\it
Peu avant la guerre, si mes souvenirs sont exacts, G. de Rham me raconta qu'un de ses \'etudiants de Gen\`eve, Pierre Humbert, \'etait all\'e \`a G\"ottingen avec l'intention de travailler sous la direction de Hasse, et que celui-ci lui avait propos\'e un probl\`eme sur lequel de Rham d\'esirait mon avis. Une courbe elliptique $C$ \'etant donn\'ee sur le corps des rationnels, il s'agissait principalement, il me semble, d'\'etudier le produit infini des fonctions z\^eta des courbes $C_p$ obtenues en r\'eduisant $C$ modulo $p$ pour tout nombre premier $p$ pour lequel $C_p$ est de genre $1$; plus pr\'ecis\'ement, il fallait rechercher si ce produit poss\`ede un prolongement analytique et une \'equation fonctionnelle. (...) J'avoue avoir pens\'e que Hasse avait \'et\'e par trop optimiste. (...)  

Au Br\'esil, cherchant des applications de l'hypoth\`ese de Riemann\index{Hypothèse de Riemann!pour les variétés projectives lisses}, je fis des calculs sur les fonctions z\^eta des courbes hyperelliptiques sur un corps fini (...) Comme l'indiquent ces notes, il s'ensuit imm\'ediatement que le produit infini de Hasse pour la courbe $Y^2=X^4+1$ est une fonction $L$ de Hecke\index{Fonction $L$!de Hecke} relative au corps $\Q(i)$. La conjecture de Hasse commen\c cait \`a prendre forme \`a mes yeux. (...)  

(...) (Chevalley) ne croyait pas que les fonctions $L$ de Hecke eussent un r\^ole \`a jouer en th\'eorie des nombres. Je soutenais le contraire (...) et citai en exemple le produit de Hasse pour la courbe $Y^2=X^4+1$, peut-\^etre aussi pour $Y^2=X^3-1$, en ajoutant qu'il devait s'agir l\`a d'un ph\'enom\`ene bien plus g\'en\'eral. L\`a je m'avan\c cais beaucoup; ce qui finalement me donna confiance dans cette id\'ee, ce fut une fois de plus l'analogie entre corps de nombres et corps de fonctions; j'observai (...) que la conjecture de Hasse pour les courbes sur les corps de nombres correspond exactement \`a mes conjectures de 1949 pour les surfaces sur un corps fini, au sujet desquelles il ne me restait plus aucun doute.
\end{quote}

\subsubsection{Id\'ee de la d\'emonstration} Les calculs de \cite{weil4} expriment les fonctions $\zeta(C_\fp,s)$ en termes de sommes de Jacobi: plus pr\'ecis\'ement, $Z(C_\fp,t)$ est de la forme
\[Z(C_\fp,t) = \prod_{a,b} L_{a,b}(t)\]
o\`u $(a,b)$ d\'ecrit le quotient de $\Z/f\times \Z/e$ par l'action diagonale de $(\Z/m)^*$,  $m$ \'etant le ppcm de $e$ et $f$. Hormis les cas triviaux, on a 
\[L_{a,b}(t)=1+\lambda t^d\]
o\`u $d$ est le degr\'e d'une certaine extension cyclotomique $\kappa'$ de $\kappa(\fp)$ et $\lambda$ est essentiellement une \emph{somme de Jacobi}
\[j_\fp^{a,b}=\sum_{x=y=-1}\chi(x)^{a_0}\chi(y)^{b_0}.\]

Cette somme est calcul\'ee dans $\kappa'$, et $\chi$ est un caract\`ere multiplicatif de $\kappa'$, c'est-\`a-dire un homomorphisme $(\kappa')^*\to \C^*$. Le point-cl\'e est alors, toujours en simplifiant un peu, que la loi $\fp\mapsto j_\fp^{a,b}$ \emph{d\'efinit un caract\`ere de Hecke} $\chi^{a,b}$. La fonction $\sZ(C,s)$ est alors essentiellement le produit des fonctions $L$ des $\chi^{a,b}$.

\subsubsection{L'intuition de Weil} Au lieu d'un corps de nombres, choisissons comme base un corps global $k$ de caract\'eristique $p$, soit $k=\F_q(U)$ o\`u $U$ est une courbe lisse. Soit $C$ une courbe lisse sur $k$. Quitte \`a remplacer $U$ par un ouvert non vide, le morphisme $C\to \Spec k$ se prolonge en un $\F_q$-morphisme lisse de type fini
\[f:C_U\to U.\]

La fonction z\^eta de $C_U$\index{Fonction zêta!d'un schéma de type fini sur $\Z$} s'\'ecrit alors (\cf proposition \ref{p4} b)):
\[\zeta(C_U,s) = \prod_{u\in U_{(0)}} \zeta(C_u,s)\]
o\`u $C_u$ est la fibre de $f$ en $u$. \`A droite on reconnait une fonction z\^eta ``\`a la Hasse'', et \`a gauche la fonction z\^eta de la surface (lisse mais pas projective) $C_U$.\index{Fonction zêta!de Hasse-Weil}

(Ce raisonnement typique \`a un nombre fini de facteurs pr\`es, analogue \`a celui indiqu\'e au  num\'ero \ref{appr1}, ne r\'eduit pas tout \`a fait la conjecture de Hasse \`a celle de Weil pour les surfaces, qui n'\'etait \'enonc\'ee que dans le cas projectif lisse: Shreeram Abhyankar ne d\'emontre la r\'esolution des singularit\'es pour les surfaces sur un corps parfait de caract\'eristique $p$ qu'en 1957 \cite{abhy}.) %On peut imaginer que Weil a soutenu son intuition en v\'erifiant ``\`a la main'' que $C_U$ admet un mod\`ele projectif lisse pour la courbe $C$ consid\'er\'ee, disons sur $k=\F_q(t)$ et pour des valeurs convenables de $\gamma$ et $\delta$.

\subsubsection{D\'efinition g\'en\'erale des fonctions $L$ de Hasse-Weil: premi\`ere approximation}\label{5.1.4}

Soit $k$ un corps global, et soit $C$ une courbe projective lisse sur $k$. Soit $U$ un mod\`ele de $k$ r\'egulier de type fini sur $\Z$: si $k$ est un corps de nombres, $U$ est de la forme $\Spec O_S$ o\`u $O_S$ est une localisation de l'anneau des entiers de $k$; si $k=\F_q(C)$ pour une courbe projective lisse $C$, $U$ est un ouvert non vide de $C$. Si $U$ est assez petit, le morphisme $C\to \Spec k$ se prolonge en un morphisme projectif lisse
\[f:C_U\to U\]
comme ci-dessus. On peut alors d\'efinir, en premi\`ere approximation:
\[L(C,s) = \zeta(C_U,s)= \prod_{u\in U_{(0)}} \zeta(C_u,s).\]

Supposons d'abord $k=\F_q(C)$. En utilisant le th\'eor\`eme d'Abhyankar \cite{abhy}, on peut compl\'eter $C_U$ en une surface projective lisse $S/\F_q$; alors $\zeta(C_U,s) =\zeta(S,s)\zeta(S-C_U,s)^{-1}$. Le facteur $\zeta(S,s)$ est une fonction rationnelle en $q^{-s}$ et v\'erifie une \'equation fonctionnelle reliant $\zeta(S,s)$ et $\zeta(S,2-s)$. Le ferm\'e $S-C_U$ est de dimension $1$; en compl\'etant et normalisant, on se ram\`ene au cas de courbes projectives lisses et on obtient ainsi une \'equation fonctionnelle (laide) pour $L(C,s)$.\index{Equation fonctionnelle@\'Equation fonctionnelle}

Supposons maintenant que $k$ soit un corps de nombres. Quelle est la g\'en\'eralit\'e de l'argument de Weil?

Le point est que la courbe consid\'er\'ee par Weil (ou plus exactement sa compl\'et\'ee projective) a une jacobienne $J$ \emph{\`a multiplication complexe}: $\End^0(J)$ est une $\Q$-alg\`ebre commutative de rang $2g$ o\`u $g$ est le genre (g\'eom\'etrique) de $C$. Dans le cas que consid\`ere Weil et sur lequel il revient dans \cite{weil-jacobi2},  $\End^0(J)=\prod_{i=1}^r K_i$ o\`u les $K_i$ sont des sous-corps de $\Q(\mu_m)$ pour $m$ assez grand.\footnote{Weil indique dans  \cite{weil-jacobi2} qu'un cas particulier de ses calculs de \cite{weil-jacobi} remonte \`a Eisenstein en 1850! \cite[pp. 192--193]{eis}.}
Cela a \'et\'e g\'en\'eralis\'e par Yutaka Taniyama et Weil au cas de toutes les vari\'et\'es ab\'eliennes \`a multiplication complexe \cite[p. 6]{weil-nikko}, \cite{shimura-taniyama}. Voir \cite[Sommes trig., \S 5]{SGA412} pour une reformulation de \cite{weil-jacobi,weil-jacobi2} en termes de faisceaux $l$-adiques.\index{Faisceau $l$-adique}

Ces r\'esultats contiennent en particulier le cas des courbes elliptiques \`a multiplication complexe, qui remonte \`a Deuring \cite{deuring3}. Le cas des courbes elliptiques sans multiplication complexe d\'efinies sur $\Q$ est beaucoup plus r\'ecent: ici les caract\`eres de Hecke\index{Caractère!de Hecke} ne suffisent plus, et on utilise les \emph{formes modulaires} pour prouver le prolongement analytique et l'\'equation fonctionnelle. C'est la d\'emonstration de la conjecture de Shimura-Taniyama-Weil: travaux d'Andrew Wiles, Richard Taylor, Christophe Breuil, Brian Conrad, Fred Diamond (voir entre autres \cite{wiles,taylor-wiles,bcdt}).

\subsection{Bonne réduction}\label{s5.bonnered} Nous verrons au \S \ref{s5.5} comment Serre propose une définition précise des facteurs locaux de la fonction zêta d'une variété projective lisse $X$ définie sur un corps global. Serre est plus exigeant: il propose une formule non seulement pour ces facteurs locaux, mais pour leur factorisation en termes de la cohomologie de $X$ dans le style du \S \ref{s3.3}. 
%(voir Tate \cite[\S 4]{Tate-purdue} pour une origine de cette problématique). 
\`A l'époque de \cite{serre-dpp}, cette formule  reposait sur un certain nombre de conjectures; parmi celles-ci, les conjectures de Weil sont maintenant démontrées,  ce qui permet d'avoir une définition inconditionnelle des facteurs locaux avec leur factorisation au moins aux places de bonne réduction.  Le but de ce paragraphe est de remarquer que, si on ne demande pas une telle factorisation, la définition de ces facteurs locaux est \emph{élémentaire}.

Commençons par préciser la notion de bonne réduction:

\begin{defn}\phantomsection\label{d5.3}
Soit $K$ un corps muni d'une valuation discrète $v$ de rang $1$, d'anneau de valuation $\sO$ et de corps résiduel $k$. Une $K$-variété projective lisse $X$ a \emph{bonne réduction} (relativement à $v$) s'il existe un $\sO$-schéma projectif \emph{lisse} $\sX$, de fibre générique $\sX_\eta=\sX\otimes_\sO K$ isomorphe à $X$. On dit que $\sX$ est un \emph{modèle lisse} de $X$, et on appelle $\sX\otimes_\sO k$ la \emph{fibre spéciale} associée à $\sX$: c'est une $k$-variété projective lisse.
\end{defn}

Dans certains cas (surfaces K3), on affaiblit la condition que $\sX$ soit un schéma en demandant seulement qu'il soit un \emph{espace algébrique} -- mais la fibre spéciale est toujours supposée être un schéma.

\begin{prop}\phantomsection\label{p5.3} Dans la définition \ref{d5.3}, supposons que $K$ soit un corps global. Alors $X$ a bonne réduction en toutes les places de $K$, sauf (au plus) un nombre fini d'entre elles.
\end{prop}

\begin{proof} Donnons-nous un plongement $X\subset \P^N_K$. Soit $C= \Spec O_K$ si $K$ est un corps de nombres, et soit $C$ le modèle projectif lisse de $K$ si $K$ est un corps de fonctions. Notons $\sX$ l'adhérence de $X$ dans $\P^N_C$: c'est un $C$-schéma projectif, de fibre générique $X$. Le morphisme $p:\sX\to C$ est lisse en dehors d'un fermé $Z\subset \sX$; comme $p$ est propre, $p(Z)$ est fermé dans $C$, et \emph{différent de $C$} puisque $p$ est lisse au-dessus du point générique de $C$. Alors $p$ est lisse au-dessus de l'ouvert non vide $U=C-p(Z)$. 
\end{proof}

\begin{rques}\phantomsection\label{r5.1} 1) Toute $K$-variété projective lisse n'a pas bonne réduction! Le contre-exemple le plus simple est celui d'une courbe elliptique dont l'invariant $j$ vérifie $v(j)<0$. (On peut aussi prendre une conique sans point rationnel, mais cet exemple est moins pertinent, cf. remarque \ref{r6.2} b).) \\
2) Si $X$ a un modèle lisse, elle en a en général bien d'autres; étant donné deux tels modèles $\sX,\sX'$, leurs fibres spéciales $Y,Y'$ ne sont en général pas isomorphes. Toutefois:
\end{rques}

\begin{prop}\phantomsection\label{p5.2} Dans la situation de la remarque \ref{r5.1} 2), on a $\zeta(Y,s)=\zeta(Y',s)$.
\end{prop}

\begin{proof} D'après le théorème \ref{t6.8}, $Y$ et $Y'$ ont des motifs de Chow isomorphes, et d'après le lemme \ref{l6.4}, la fonction zêta d'une $k$-variété projective lisse ne dépend que de son motif de Chow.\index{Fonction zêta!d'un motif (pur)}
\end{proof}

\begin{thm}\phantomsection\label{t5.6} Soit $X$ une $K$-variété projective lisse; notons $\Sigma_K(X)$ l'ensemble des places finies de $K$ où $X$ a bonne réduction. Pour toute place $v\in \Sigma_K(X)$, notons $\zeta_v(X,s)$ la fonction $\zeta(Y,s)$, où $Y$ est la fibre spéciale d'un modèle lisse quelconque de $X$ en $v$ (proposition \ref{p5.2}). Alors le produit
\[\prod_{v\in \Sigma_K(X)} \zeta_v(X,s)\]
converge absolument pour $\Re(s)>\dim X+1$.
\end{thm}

\begin{proof} Reprenons les notations de la preuve de la proposition \ref{p5.3}: l'ensemble $U_{(0)}$ des points fermés de l'ouvert $U$ est contenu dans $\Sigma_K(X)$, et le complémentaire $\Sigma_K(X)-U_{(0)}$ est fini. Comme $\zeta_v(X,s)$ converge absolument pour $\Re(s)>\dim X$ (théorème \ref{t:abs}), il suffit de démontrer la convergence en remplaçant $\Sigma_K(X)$ par $U_{(0)}$. Mais on a
\[\prod_{v\in U_{(0)}} \zeta_v(X,s)=\zeta(p^{-1}(U),s)\]
(proposition \ref{p4}), et le deuxième membre converge absolument pour $\Re(s)>\dim p^{-1}(U)=\dim X+1$ en réappliquant le théorème \ref{t:abs}.
\end{proof}

Le théorème \ref{t5.6} fournit une définition de la ``partie non ramifiée'' de la fonction zêta de Hasse-Weil de $X$.\index{Fonction zêta!de Hasse-Weil}

\subsection{Fonctions $L$ de faisceaux $l$-adiques}\label{s5.2}\index{Faisceau $l$-adique}
R\'ef\'erences: Grothendieck \cite{groth-lefsch}, \cite[exp. XV]{SGA5}, et Illusie \cite{illusie} pour des commentaires historiques tr\`es \'eclairants (d\'epassant largement le cadre des fonctions $L$).

\subsubsection{Faisceaux $l$-adiques}\label{5.2.1}  Citons Serre \cite[fin lettre 26 oct. 61]{corresp}:

\begin{quote} {\it 
 (\`A propos des vari\'et\'es affines: peux-tu d\'ecomposer (filtrer, ou n'importe quoi) leur cohomologie de fa\c con \`a mettre en \'evidence les morceaux qui proviennent de la vari\'et\'e compl\`ete? Je n'arrive pas \`a m'exprimer (*), mais tu comprendras ce que je veux si je le mets sous la forme: quelles conjectures doit-on faire pour la fonction z\^eta d'une vari\'et\'e affine non singuli\`ere?\index{Fonction zêta!d'un schéma de type fini sur $\Z$} Il me para\^\i t scandaleux qu'il soit n\'ecessaire de la plonger (si tant est que ce soit possible!) dans une vari\'et\'e projective non singuli\`ere, pas du tout unique, et d\'egueulasse; d'autre part, je n'arrive \`a rien formuler. As-tu en main une autre homologie que l'usuelle (par exemple \og \`a supports ferm\'es\fg\ ou Dieu sait quoi) qui pourrait servir \`a quelque chose?)
 
 (*)  Il me manquait le langage des motifs.} (note ajout\'ee en 2001).
\end{quote}

puis Grothendieck \cite{groth-lefsch}:

\begin{quote}\it 
Le but de cet expos\'e est de prouver le r\'esultat de rationalit\'e pour toutes les fonctions $L$ envisag\'ees au \S 1, et m\^eme pour un type de fonctions $L$ beaucoup plus g\'en\'eral, associ\'e \`a des \emph{faisceaux $l$-adiques} sur $X$. Nous donnerons en effet une formule explicite du style Lefschetz-Weil de ces fonctions. Les outils essentiels sont de deux sortes:

(a) Le formalisme de la ``\emph{cohomologie \`a supports compacts}'' (...)

(b) Une \emph{formule de Lefschetz g\'en\'eralis\'ee}, due \`a \emph{J. L. Verdier} (...)
\end{quote}

\begin{defn}[\protect{\cite[exp. VI, d\'ef. 1.1.1 et 1.2.1]{SGA5}}]\phantomsection\label{dlad} Soit $X$ un sch\'ema, et soit $l$ un nombre premier. \\
a) Un \emph{faisceau $l$-adique} \index{Faisceau $l$-adique}sur $X$ est un syst\`eme projectif  $F=(F_n)_{n\ge 0}$ de faisceaux \'etales sur $X$ tels que, pour tout $n$, le morphisme de projection $F_{n+1}\to F_n$ soit isomorphe au morphisme canonique $F_{n+1}\to F_{n+1}\otimes_\Z {\Z/l^{n+1}}$, ce qui implique que chaque $F_n$ est annul\'e par $l^{n+1}$.\\
b) $F$ est \emph{constructible} si chaque $F_n$ est constructible (pour cela, il suffit que $F_0$ le soit).\\
c) $F$ est \emph{constant tordu} [constructible]\footnote{Plus tard, on dira aussi \emph{lisse}.} Si les $F_n$ sont localement cons\-tants et constructibles.
\end{defn}

Lorsque $X$ est connexe, si $a\to X$ est un point g\'eom\'etrique, on a une \'equivalence de cat\'egories \cite[exp. VI, prop. 1.2.5]{SGA5}
\begin{multline*}
\{\text{faisceaux $l$-adiques constants tordus}\} \leftrightarrows\\
 \{\text{repr\'esentations continues de $\pi_1(X,a)$ \`a valeurs}\\ \text{sur les $\Z_l$-modules de type fini}\}.
\end{multline*}

\begin{meg}\phantomsection Une repr\'esentation continue d'un groupe profini sur un $\C$-espace vectoriel de dimension finie se factorise par un quotient fini. C'est totalement faux si on remplace $\C$-espace vectoriel par $\Z_l$-module, ou m\^eme par $\Q_l$-espace vectoriel (exemple: $\Z_l$ op\`ere contin\^ument  et fid\`element sur $\Q_l$ par $x\mapsto \exp(lx)$).
\end{meg}

\begin{exo}\phantomsection\label{exo5.1} Soit $k$ un corps commutatif, de cl\^oture s\'eparable $k_s$. Si $G=\Gal(k_s/k)$ et $n$ est un entier inversible dans $k$, l'action de $G$ sur les racines de l'unit\'e $\mu_n$ d\'efinit un homomorphisme
\[\kappa^{(m)}:G\to \Aut(\mu_m)\simeq (\Z/m)^*\]
c'est le \emph{caract\`ere cyclotomique modulo $m$}\index{Caractère!cyclotomique}. En passant \`a la limite projective, on obtient le caract\`ere cyclotomique
\[\kappa:G\to \plim (\Z/m)^*=\prod_{l\ne p} \Z_l^*\]
o\`u $p=\car k$. Par projection sur $\Z_l^*$, on obtient sa composante $l$-primaire $\kappa_l$.
\begin{enumerate}[label=(\alph*)]
\item Supposons $k=\F_q$, et soit $\phi_q\in G$ le Frobenius arithm\'etique ($\phi_q(x)=x^q$). Montrer que $\kappa_l(\phi_q)=q$ pour tout $l\ne p$. En d\'eduire que l'image de $\kappa_l$ est ouverte dans $\Z_l^*$.%\footnote{Rappel: dans un groupe compact, un sous-groupe est ouvert si et seulement s'il est ferm\'e et d'indice fini.}
\item Supposons que $k=\Q$. Montrer que $\kappa_l$ est surjectif pour tout $l$. (Utiliser l'irr\'eductibilit\'e des polyn\^omes cyclotomiques.)
\item Supposons $k$ de type fini sur son sous-corps premier. Montrer que l'image de $\kappa_l$ est ouverte.
\end{enumerate}
\end{exo}

\begin{exo}\phantomsection\label{exo5.2} Soit $l$ un nombre premier; soit $K$ une extension finie de $\Q_l$ et $O_K$ son anneau d'entiers.
\begin{enumerate}[label=(\alph*)]
\item Soit $x\in 1+l^nO_K$.   Montrer que $x^l\in 1+l^{n+1}O_K$.
\item Soient $x\in 1+lO_K$ et $a\in \Z_l$. Soit $(a_r)$ une suite d'entiers naturels convergeant $l$-adiquement vers $a$. Montrer que la suite $x^{a_n}$ 
converge dans $1+lO_K$ vers une limite $x^a$ qui ne d\'epend que de $a$.
\item Montrer les identit\'es
\[(xy)^a = x^ay^a,\quad (x^a)^b=x^{ab}.\]
\item Soit  $m$ un entier premier \`a $l$. Soit $k$ le corps r\'esiduel de $K$. Montrer que l'homomorphisme de r\'eduction $\mu_m(K)\to \mu_m(k)$ est  bijectif. (Utiliser le lemme de Hensel.)
\item *Soit $\zeta$ une racine de l'unit\'e de $K$, diff\'erente de $1$. Montrer que $\zeta\notin 1+lO_K$ si $l>2$, et que $\zeta\notin 1+4O_K$ si $l=2$. (Traiter d'abord le cas o\`u $\zeta$ est d'ordre premier \`a $l$, en utilisant (d). Sinon, se ramener \`a $\zeta$ d'ordre $l$ si $l>2$ et d'ordre $4$ si $l=2$; traiter s\'epar\'ement le cas $l=2,\zeta=-1$.)
\item Soit $n\ge 1$. Montrer que $U=1+lM_n(\Z_l)$ est un sous-groupe de $GL_n(\Z_l)$, et un pro-$l$-groupe. (Filtrer $U$ par les $U^{(r)}=1+l^rM_n(\Z_l)$, $r\ge 1$, et construire un isomorphisme de groupes $U^{(r)}/U^{(r+1)}\iso M_n(\F_l)$ pour tout $r$.)
\end{enumerate}
\end{exo}

\subsubsection{Images directes, inverses\dots} Soit $f:X\to Y$ un morphisme de sch\'emas. Je renvoie \`a \cite{SGA4} ou Milne \cite[ch. III]{milne} pour la d\'efinition de l'image r\'eciproque des faisceaux $f^*$ et des images directes sup\'erieures $R^qf_*$, et surtout des images directes sup\'erieures \`a support propre $R^qf_!$ (ces derni\`eres si $f$ est un morphisme s\'epar\'e de type fini de sch\'emas localement noeth\'eriens), dont la construction compl\`ete est faite dans l'expos\'e XVII de \cite{SGA4} (voir aussi \cite[ch. VI, \S 3]{milne}).

D'apr\`es \cite[exp. VI]{SGA5}, on peut d\'efinir ces foncteurs sur les faisceaux\index{Faisceau $l$-adique} $l$-adiques\footnote{Ce point est un peu d\'elicat, car en g\'en\'eral l'image directe d'un faisceau $l$-adique ne v\'erifie  plus tout \`a fait la condition de la d\'efinition \ref{dlad} a). La solution est de le ``rectifier'' en utilisant la technique du th\'eor\`eme \ref{tA.1} avec le syst\`eme multiplicatif des morphismes de d\'ecalages $(F_n)\mapsto (F_{n+r})$, \cf \cite[exp. V, \S\S 2.4 et 2.5]{SGA5}.}, et

\begin{thm}[\protect{\cite[exp. VI, lemme 2.2.2 et 2.2.3 A)]{SGA5}}]\phantomsection Si $F$ est constructible, $R^qf_!F$ est constructible pour tout $q\ge 0$ et nul pour $q>2d$, o\`u $d$ est la dimension relative de $f$.
\end{thm}

\begin{thm}[\protect{\cite[exp. VI, 2.2.3 B)]{SGA5}}]\phantomsection\label{t5.2} La formation des \goodbreak $R^qf_! F$ commute \`a tout changement de base.
\end{thm}

Pr\'ecisons ce dernier point en utilisant le langage des cat\'egories d\'eriv\'ees (\S \ref{catder}). En g\'en\'eral, soit 
\[\begin{CD}
\sC'@>g'_*>> \sC\\
@V{f'_*}VV @V{f_*}VV\\
\sD'@>g_*>> \sD
\end{CD}\]
un carr\'e naturellement commutatif de cat\'egories et de foncteurs: naturellement commutatif signifie qu'on s'est donn\'e un isomorphisme naturel 
\begin{equation}\label{eq5.4}
 f_* g'_*\simeq g_*f'_*.
\end{equation}
 
 Supposons que $g_*$ et $g'_*$ aient des adjoints \`a gauche $g^*,{g'}^*$. De \eqref{eq5.4}, on tire une composition
 \[ f_*\by{f_**\eta'} f_* g'_*{g'}^*\simeq g_*f'_*{g'}^*\]
 o\`u $\eta'$ est l'unit\'e de l'adjonction $({g'}^*,g'_*)$, puis
 \begin{equation}\label{eq5.5}
  g^*f_*\to g^*g_*f'_*{g'}^*\by{\epsilon*f'_*{g'}^*}f'_*{g'}^*
  \end{equation}
 o\`u $\epsilon$ est la co\"unit\'e de l'adjonction $(g^*,g_*)$: c'est le \emph{morphisme de changement de base}. Lorsque cette situation \'emane d'un carr\'e cart\'esien de $\F_q$-sch\'emas de type fini
 \begin{equation}\label{eq5.7}
 \begin{CD}
X'@>g'>> X\\
@V{f'}VV @V{f}VV\\
Y''@>g>> Y
\end{CD}
\end{equation}
avec  $\sC=D^b_c(X,\Z_l)$, etc. et $f_*$ signifie $Rf_*$, etc., le th\'eor\`eme de changement de base propre implique que \eqref{eq5.5} est un isomorphisme quand $f$ est propre et $g$ quelconque.

Si $f:X\to Y$ est maintenant un morphisme quelconque de $\Sch(\F_q)$, on peut le factoriser en
\[X\by{j} \bar X\by{\bar f} Y\]
o\`u $j$ est une immersion ouverte et $\bar f$ est propre (gr\^ace au th\'eor\`eme de compactification de Nagata\footnote{Si $X\subset \bar X_0$ est une compactification de $X$ au-dessus de $\F_q$, on peut prendre pour $\bar X$ l'adh\'erence dans $\bar X_0\times Y$ du graphe de $f$ dans $X\times Y$.}). Par d\'efinition,
\[Rf_! = R\bar f_*\circ j_!,\]
le th\'eor\`eme de changement de base propre assurant que cette d\'efinition est ind\'ependante du choix de $(j,\bar f)$. On d\'efinit alors l'isomorphisme de changement de base
\begin{equation}\label{eq5.6}
g^*Rf_!\iso Rf'_! {g'}^*
\end{equation}
comme le compos\'e de \eqref{eq5.5} pour $\bar f$ et de l'isomorphisme trivial ${g''}^*j_!\iso j'_! {g'}^*$. (Ici, $g''$ est le changement de base de $g$ par $\bar f$ et $j'$ le changement de base de $j$ par $g''$.)

\subsubsection{La fonction $L$ d'un faisceau $l$-adique constructible}\index{Fonction $L$!d'un faisceau $l$-adique}\index{Faisceau $l$-adique}

\begin{defn}[\protect{\cite[exp. XV, \S 3 no 1)]{SGA5}}]\phantomsection a) Soit $F$ un faisceau $l$-adique sur $\Spec \F_q$ ($l\nmid q$), c'est-\`a-dire une repr\'esentation $l$-adique de $\Gal(\bar \F_q/\F_q)\simeq \hat\Z$. On d\'efinit 
\[L(\F_q,F)=\det(1-\phi_q^{-1} t^n\mid F\otimes \Q_l)^{-1}\in \Q_l(t)\subset \Q_l[[t]]\]
o\`u $\phi_q$ est l'automorphisme de Frobenius $x\mapsto x^q$ de $\bar \F_q$ et $q=p^n$.\\
\\
b) Soit $X$ un sch\'ema de type fini sur $\F_q$, et soit $F$ un faisceau $l$-adique constructible sur $X$. On d\'efinit:
\[L(X,F) = \prod_{x\in X_{(0)}} L(x,F_x)\]
o\`u $F_x$ est la fibre de $F$ en $x$ (soit $i_x^*F$, o\`u $i_x:x\to X$ est l'immersion ferm\'ee canonique).
\end{defn}

\subsubsection{Explication de $\phi_q^{-1}$} \label{5.2.3} Soit $\pi_X:X\to X$ l'endomorphisme de Frobenius au-dessus de $\F_q$: il induit un isomorphisme de faisceaux $l$-adiques $(\pi_X)_*F\iso F$,  dont l'inverse d\'efinit apr\`es adjonction un morphisme:
\[Fr_{F/X}^*:\pi_X^* F\to F\]
la ``correspondance de Frobenius''. On v\'erifie ``trivialement'' que la composition
\[H^*(X,F)\to H^*(X,\pi_X^*F)\to H^*(X,F)\]
o\`u la premi\`ere fl\`eche est l'application canonique (d\'efinie par le morphisme $F\to (\pi_X)_*\pi_X^*F$)  et la seconde fl\`eche est $H^*(X,Fr_{F/X}^*)$, est \'egale \`a l'identit\'e. En appliquant ceci \`a $\bar X=X\times_{\F_q} \bar \F_q$, on en d\'eduit que l'action de $\pi_X\times 1$ sur $H^*(\bar X,Fr_{F/X}^*)$ est \emph{l'inverse} de celle de $1\times \pi_{\F_q}$, soit l'action de $\phi_q$: ``le Frobenius arithm\'etique est l'inverse du Frobenius g\'eom\'etrique''.

Pour une excellente exposition de ce sorite incompr\'ehensible, voir  l'exposé de Houzel \cite[exp. XV, \S\S 1 et  2]{SGA5}.

\subsubsection{Convergence; fonctorialit\'e en $F$}

\begin{lemme}\phantomsection Le produit infini $L(X,F)$ est convergent dans $\Q_l[[t]]$ (comme s\'erie formelle).
\end{lemme}

\begin{proof} Cela revient \`a dire que, pour tout $n>0$, l'ensemble
\[\{x\in X_{(0)}\mid v_t(L(x,F_x)-1)\le n\}\]
est fini, ce qui r\'esulte de la d\'efinition.
\end{proof}

\begin{lemme}\phantomsection\label{l5.1} Si $0\to F'\to F\to F''\to 0$ est une suite exacte de faisceaux $l$-adiques constructibles, on a\index{Faisceau $l$-adique}
\[L(X,F)=L(X,F')L(X,F'').\]
\end{lemme}

\begin{proof} C'est \'evident en se ramenant au cas des corps finis.
\end{proof}

\subsubsection{Le th\'eor\`eme}

\begin{thm}[\protect{\cite[exp. XV, \S 3 no 2)]{SGA5}}]\phantomsection \label{t5.1}Soit $f:X\to Y$ un $\F_q$-morphisme de type fini. Pour tout faisceau $l$-adique $F$ sur $X$, on a
\[L(X,F)=\prod_{i=0}^{2n} L(Y,R^if_! F)^{(-1)^i}.\]
\end{thm}

En prenant $Y=\Spec \F_q$, on obtient:

\begin{cor}\phantomsection\label{c5.1} $L(X,F)=\prod\limits_{i=0}^{2n} \det(1-\phi_q^{-1}t\mid H^i_c(\bar X,F)\otimes_{\Z_l}\Q_l)^{(-1)^{i+1}}\allowbreak\in \Q_l(t)$.
\end{cor}

\begin{cor}\phantomsection $Z(X,t)= \prod\limits_{i=0}^{2n} \det(1-\phi_q^{-1}t\mid H^i_c(\bar X,\Q_l))^{(-1)^{i+1}}$; $Z(X,t)\in \Q_l(t)\cap \Q[[t]] = \Q(t)$.
\end{cor}

En effet, $Z(X,t)=L(X,\Z_l)$.

\subsubsection{R\'eduction \`a $X=$ un ouvert de $\P^1$ (\loccit).}\label{5.2.6} En quatre \'etapes:

\begin{enumerate}
\item Si le th\'eor\`eme \ref{t5.1} est vrai pour chaque fibre de $f$ au-dessus d'un point ferm\'e de $Y$, alors il est vrai pour $f$. (R\'esulte du th\'eor\`eme \ref{t5.2}.)
\item Le th\'eor\`eme est vrai si $f$ est fini. (Par 1, on se ram\`ene \`a une extension de corps finis $\F_{q'}/\F_q$, et alors
\[L(\F_{q'},F)=L(\F_q,\Ind_{G_{\F_{q'}}}^{G_{\F_q}} F)\]
r\'esulte du m\^eme calcul  que dans la d\'emonstration de la proposition \ref{p4.4} (iii).)
\item Le th\'eor\`eme est stable par composition des morphismes (r\'esulte de la suite spectrale de Leray
\[R^pg_!R^qf_! F\Rightarrow R^{p+q}(g\circ f)_!F\]
et du lemme \ref{l5.1}.)\footnote{Le point ici est: ``les suites spectrales pr\'eservent les caract\'eristiques d'Euler-Poincar\'e''; pour ce genre de raisonnements, voir par exemple Serre \cite[p. V-24]{serre-locale}.}
\item Supposons $Y=\Spec \F_q$. Si $Z$ est un fermé de $X$ d'ouvert compl\'ementaire $U$ et si le th\'eor\`eme est vrai pour, $f_{|Z}$, alors il l'est pour $f$ si et seulement s'il l'est pour $f_{|U}$. (R\'esulte du lemme \ref{l5.1} et de la suite exacte de cohomologie \`a supports propres
\[\dots \to R^q(f_{|U})_!F\to R^qf_!F\to R^q(f_{|Z})_! F\to R^{q+1}(f_{|U})_!F\to \dots)\]
\end{enumerate}

Par le point 4 et par r\'ecurrence noeth\'erienne, on se ram\`ene \`a $X$ affine, puis par le point 2 et le lemme de normalisation  de Noether au cas $X=\A^n$, puis par les points 1 et 3 \`a $n=1$. On voit facilement que le corollaire \ref{c5.1} est vrai pour $\dim X=0$: en r\'eutilisant le point 4, on peut remplacer $\A^1$ par n'importe quel ouvert non vide $U$ de $\P^1$.

\subsubsection{D\'emonstration pour $X=U\subset \P^1$}\label{s5.2.3} En prenant les logarithmes, le corollaire \ref{c5.1} est \'equivalent via \eqref{eq3.3} \`a la \emph{formule des traces}
\begin{equation} \label{eq5.1}
\sum_{x\in U_{(0)}} \Tr(\phi_q^{-n}\mid F_x) = \sum_{i=0}^2 (-1)^i \Tr(\phi_q^{-n}\mid H^i_c(\bar U,F)\otimes \Q_l).
\end{equation}

Comme $\phi_q^n = \phi_{q^n}$, quitte \`a remplacer $\F_q$ par $\F_{q^n}$ on peut supposer $n=1$. De plus, quitte \`a restreindre $U$, on peut supposer $F$ \emph{localement constant}.

Par un argument tr\`es d\'elicat, on se ram\`ene \`a d\'emontrer \eqref{eq5.1} pour un \emph{faisceau $F$ localement libre de $\Z/l^n$-Modules}. Le premier point est qu'alors, le second membre n'est pas \emph{a priori} d\'efini parce que les $\Z/l^n$-modules $H^i_c(\bar U,F)$ ne sont pas libres en g\'en\'eral. De plus, on ne peut pas en g\'en\'eral les remplacer par une r\'esolution projective finie, parce que l'anneau $\Z/l^n$ est de dimension homologique infinie pour $n>1$. Pour r\'esoudre ce probl\`eme, on fait appel de mani\`ere essentielle aux cat\'egories d\'eriv\'ees $D^b_c(X,\Z/l^n)$ et aux \emph{complexes parfaits} $R\Gamma_c(X,F)\in D^b_c(X,\Z/l^n)$ dont la cohomologie est $H^*_c(\bar X,F)$ (voir d\'efinition \ref{dparf} pour les complexes parfaits): on peut alors donner un sens aux traces $\Tr(\phi^{-1}\mid R\Gamma_c(X,F))$, et donc \`a \eqref{eq5.1} dans le cas d'un faisceau de torsion.

L'argument de \cite[exp. XV, \S 3 no 3]{SGA5} utilise de plus des lemmes subtils d'alg\`ebre homologique sur les anneaux $\Z/l^n$. 

Il y a alors deux d\'emonstrations:

\begin{itemize}
\item Une ``formule de Lefschetz-Verdier'' qui calcule le second membre de \eqref{eq5.1} comme somme de termes locaux, qu'il faut ensuite reconna\^\i tre comme \'etant les $\Tr(\phi_q^{-n}\mid F_x)$. Une exposition courte mais d\'etaill\'ee s'en trouve dans \cite{pepin}\footnote{Je remercie Oussama Ouriachi de m'avoir signal\'e cette r\'ef\'erence.}.
\item Une r\'eduction \`a la formule des traces de la proposition \ref{p3.4} inspir\'ee de travaux de Nielsen et Wecken \cite{wecken} que nous allons r\'esumer maintenant. (R\'ef\'erences: \cite[exp. XII]{SGA5}, \cite[Rapport]{SGA412}, \cite[ch. VI, \S 13]{milne}.)
\end{itemize}

Comme $F$ est localement constant constructible, c'est un $\pi_1(U,u)$-module pour un point g\'eom\'etrique $u\to U$; l'action se factorise par le groupe de Galois $G$ d'un rev\^etement \'etale fini (connexe) $\pi:U'\to U$. On veut montrer la formule
\begin{equation}\label{eq5.2}
\sum_{x\in U_{(0)}} \Tr(\phi_q^{-1}\mid F_x) = \Tr(\phi_q^{-1}\mid R\Gamma_c(\bar U,F)).
\end{equation}

Soit $K$ (\resp $K'$) le corps des fonctions de $U$ (\resp de $U'$): l'extension $K'/K$ est galoisienne de groupe $G$. Soit $k=\F_q$ (\resp $k'$) la fermeture alg\'ebrique de $\F_p$ dans $K$ (\resp dans $K'$): on a un diagramme d'extensions galoisiennes
\[\begin{CD}
k'@>>> K'@>>> K'\bar k\\
@A{g}AA @A{G}AA @A{H}AA\\
k@>>> K @>>> K\bar k
\end{CD}\]
o\`u le groupe $H$ est le groupe de Galois du rev\^etement \'etale g\'eom\'etrique $\bar U'=U'\times_{k'}\bar k\to U\times_k \bar k=\bar U$. La partie droite du diagramme ci-dessus est la fibre g\'en\'erique du carr\'e (non n\'ecessairement cart\'esien!)
\[\begin{CD}
U'@<<< \bar U'\\
@V{\pi}V{G}V @V{\bar \pi}V{H}V\\
U@<<\Gamma< \bar U
\end{CD}\]
o\`u $\Gamma=\Gal(\bar k/k)$. On a un diagramme commutatif de suites exactes de groupes profinis
\[\begin{CD}
1@>>> H@>>> \tilde G@>p>> \Gamma@>>> 1\\
&&||&& @VVV @VVV\\
1@>>> H@>>> G@>>> g@>>> 1
\end{CD}\]
o\`u $\tilde G=\Gal(\bar U'/U)$ (les fl\`eches verticales sont surjectives).

Le groupe $\Gamma$ est isomorphe \`a $\hat \Z$, de g\'en\'erateur l'inverse $\phi^{-1}$ du Frobenius ``arithm\'etrique'' $\phi:=\phi_q$. Notons $H_1=p^{-1}(\phi^{-1})\subset \tilde G$: c'est un $H$-torseur dont les \'el\'ements op\`erent sur $\bar U'$. 

Notons d'autre part $\pi_X$ le Frobenius ``g\'eom\'etrique'' de $\bar X$ pour tout $\F_q$-sch\'ema $X$ de type fini, de sorte que l'action de $\pi_X$ sur $R\Gamma_c(\bar X,F)$ est \'egale \`a celle de $\phi^{-1}$ (\cf \S \ref{4.4.2}). Sur $\bar U'$, on a donc l'action de $H_1$ et de $\pi_{U'}$: la premi\`ere est au-dessus de $U$ et la seconde au-dessus de $\bar k$. 

\begin{lemme}\phantomsection Soit $E=\{\sigma\in \Aut_{\bar k}(\bar U')\mid \bar\pi\sigma = \pi_U\bar \pi\}$. Alors
\[H_1=E=\pi_{U'}H.\qed\]
\end{lemme}

Le groupe $H$ op\`ere sur $H_1$ par conjugaison, donc aussi sur $H$ par transport de structure: soit $S$ l'ensemble des orbites de cette action. Pour $h\in H$, notons $Z_h$ son stabilisateur: \`a conjugaison pr\`es, il ne d\'epend que de l'image de $h$ dans $S$. La formule fondamentale est alors (\cite[Rapport, (5.12.1)]{SGA412}, \cite[ch. VI, lemma 13.15]{milne})
\begin{multline}\label{eq5.3}
\Tr(\phi_q^{-1}\mid R\Gamma_c(\bar U,F))=\Tr(\pi_U\mid R\Gamma_c(\bar U,F))\\= \sum_{h\in S}\frac{1}{|Z_h|} \Tr(\pi_{U'} h^{-1}\mid H^*_c(\bar U',\Q_l))\cdot \Tr(h\mid \pi^*F)
\end{multline}
(voir \S \ref{5.2.3} pour la premi\`ere \'egalit\'e), o\`u les nombres $l$-adiques 
\[\frac{1}{|Z_h|} \Tr(\pi_{U'} h^{-1}\mid H^*_c(\bar U',\Q_l))\]
appartiennent \`a $\Z_l$ et  $\pi^*F$ est consid\'er\'e comme faisceau \emph{constant} sur $U'$, muni d'une action de $G$ (donc de $H$).

Soit $C'$ la compl\'etion projective lisse de $U'$: l'action de $G$ se prolonge canoniquement \`a $C'$; notons que les points fixes de $\pi_{C'}h^{-1}$ sur $\bar C'$ sont tous de multiplicit\'e $1$. En leur appliquant la formule des traces de la proposition \ref{p3.4}, un petit calcul fournit le th\'eor\`eme \ref{t5.1} (en passant encore par le point no 4 du \S \ref{5.2.6}).

Pour d\'emontrer la formule \eqref{eq5.3}, on a besoin de la notion de \emph{traces non commutatives} (Stallings \cite{stallings}, Hattori). Soit $R$ un anneau unitaire, non n\'ecessairement commutatif. On note
\[H_0(R) = R/[R,R]\]
le quotient de $R$ par le sous-groupe engendr\'e par les commutateurs $ab-ba$: c'est le $0$-\`eme groupe d'homologie de Hochschild de $R$.  \'Etant donn\'e une matrice $M= (m_{ij})\in M_n(R)$, on d\'efinit sa trace comme
\[\Tr_R(M)=\sum m_{ii} \in H_0(R).\]

Cette d\'efinition s'\'etend aux endomorphismes des $R$-modules \`a gauche, projectifs de type fini. On montre que, si $C$ est un complexe parfait sur $R$ et que $f$ est un endomorphisme de $C$ dans $D(R)$, la trace de $f$ est bien d\'efinie \cite[Rapport, 4.3]{SGA412}. Alors \eqref{eq5.3} r\'esulte de propri\'et\'es formelles de la trace non commutative, la plus importante \'etant:

\begin{lemme}[\protect{\cite[Rapport, prop. 5.6]{SGA412}}]\phantomsection Soit $H$ un groupe fini, et soit $P$ un $\Lambda[H]$-module projectif, o\`u $\Lambda$ est un anneau (par exemple, $\Lambda=\Z/l^n$). Soit $u$ un endomorphisme de $P$. Alors
\[\Tr_\Lambda(u) = |H|\Tr_{\Lambda[H]}(u).\]
\end{lemme}

\subsubsection{G\'en\'eralisations: complexes parfaits et faisceaux de Weil} Tout d'abord, on peut d\'efinir la fonction $L$ d'un complexe parfait $C$ de faisceaux $l$-adiques\index{Faisceau $l$-adique} sur $X$:\index{Fonction $L$!d'un faisceau $l$-adique}
\[L(X,C)=\prod_{i\in \Z} L(X,H^i(C))^{(-1)^i}.\]

Cette d\'efinition ne d\'epend que de la classe de $C$ dans $D^b_c(X,\Q_l)$ (voir \S \ref{5.3.3}). Le th\'eor\`eme \ref{t5.1} se reformule alors de mani\`ere plus agr\'eable:
\begin{equation}\label{eq5.12}
L(X,C)=L(Y,Rf_! C).
\end{equation}

Une autre g\'en\'eralisation est aux \emph{faisceaux de Weil}. Pour $X$ de type fini sur $\F_p$, avec $\bar X=X\times_{\F_p}\bar \F_p$, un faisceau de Weil sur $X$ est un faisceau \'etale $F$ sur $\bar X$ muni d'une action de Frobenius, c'est-\`a-dire d'un morphisme $\pi_F:\pi_X^*F\to F$. Un faisceau \'etale sur $X$ est un cas particulier de faisceau de Weil. On dit constructible, $l$-adique, etc si ces adjectifs s'appliquent \`a $F$ sur $\bar X$.

On peut alors d\'efinir la fonction $L$ d'un faisceau de Weil, voire d'un complexe parfait de tels faisceaux, puis \'enoncer le m\^eme th\'eor\`eme, avec la m\^eme d\'emonstration \cite[\S 10]{deligne-constantes}.

L'id\'ee des faisceaux de Weil a \'et\'e reprise par Lichtenbaum pour d\'efinir sa \emph{topologie Weil-\'etale}, et la cohomologie correspondante \cite{licht}.

\subsection{L'\'equation fonctionnelle en caract\'eristique $p$}\index{Equation fonctionnelle@\'Equation fonctionnelle}

\subsubsection{Le formalisme des six op\'erations} Pour exprimer une \'equation fonctionnelle que v\'erifie la fonction $L$ d'un faisceau $l$-adique, on a besoin du formalisme des six op\'erations
\[f^*,f_*,f_!,f^!,\otimes,\uHom\]
invent\'e par Grothendieck. Risquons la

\begin{defn}\phantomsection\label{d5.1}  Soit $S$ un sch\'ema, et soit $\Sch(S)$ la cat\'egorie des $S$-sch\'emas de type fini (les morphismes sont les morphismes de $S$-sch\'emas). Un \emph{formalisme des $4$ op\'erations} sur $S$ est la donn\'ee:
\begin{enumerate}
\item Pour tout $X\in \Sch(S)$, d'une cat\'egorie $T(X)$.
\item Pour tout morphisme $f:X\to Y$ dans $\Sch(S)$, de $4$ foncteurs
\[f^*,f^!:T(Y)\to T(X),\quad f_*,f_!:T(X)\to T(Y)\]
munis d'isomorphismes naturels coh\'erents $(g\circ f)^*\simeq f^*\circ g^*$, etc., poss\'edant les propri\'et\'es suivantes:
\begin{enumerate}
\item $f_*$ est adjoint \`a droite de $f^*$, $f_!$ est adjoint \`a gauche de $f^!$.
\item {\bf Recollement faible}:\label{essai2} Soit $Z\by{i} X\yb{j} U$ une d\'ecomposition de $X$ en ferm\'e ($Z$) et ouvert compl\'ementaire ($U$); alors le couple $(i^*;j^*)$ est conservatif: si $\phi:x\to y$ est une fl\`eche de $T(X)$ telle que $i^*\phi$ et $j^*\phi$ soient des isomorphismes,  $\phi$ est un isomorphisme. De plus, $i_*$ et $j_*$ sont pleinement fidèles.
\item On a un morphisme de foncteurs $f_!\to f_*$, qui est un isomorphisme si $f$ est propre.
\item\label{essai} si $f$ est lisse, il existe une auto\'equivalence $Th(f)$ de $T(Y)$ telle que $f^!\simeq Th(f)\circ f^*$, avec des compatibilit\'es ``\'evidentes'' pour la composition des morphismes lisses. Si $f$ est \'etale, on a $Th(f)=1$.
%\item {\bf Changement de base propre}: pour tout diagramme cart\'esien \eqref{eq5.7} avec $f$ propre, le morphisme de changement de base \eqref{eq5.6} est un isomorphisme.
%\item {\bf Changement de base lisse}: m\^eme \'enonc\'e en supposant $g$ lisse.
\end{enumerate}
\end{enumerate}
\end{defn}

Pour les deux derni\`eres op\'erations, on suppose que pour tout $X$, $T(X)$ soit munie d'une structure de cat\'egorie mono\"\i dale ferm\'ee, en g\'en\'eral unitaire et sym\'etrique\footnote{Mais Ayoub \'evite cette hypoth\`ese dans \cite[ch. 2]{ayoub-ast}.}  (\S \ref{sA.2}, d\'efinition \ref{dA.10}).  On d\'esire alors les propri\'et\'es suivantes:
\begin{description}
\item[M1] Les $f^*$ sont des foncteurs mono\"\i daux.
\item[M2] On a des accouplements canoniques ``associatifs et unitaires''
\[f^*a\otimes f^! b\to f^!(a\otimes b)\]
qui sont des isomorphismes pour $f$ lisse; dans ce cas,  l'auto\'equivalence $Th(f)$ de la d\'efinition \ref{d5.1} \ref{essai} est obtenue en faisant $b=\un$ ci-dessus.
\item[M3] Le morphisme de foncteurs
\[f_!(f^*a\otimes c)\to a\otimes f_! c\]
d\'eduit de M2 par adjonction est un isomorphisme (formule de projection).
\end{description}

On en d\'eduit alors formellement un isomorphisme de foncteurs

\begin{description}
\item[M4] $\uHom(f_!a,b)\iso f_*\uHom(a,f^!b)$.
\end{description}

Soit $\sM$ une cat\'egorie mono\"\i dale sym\'etrique unitaire ferm\'ee, et soit $K$ un objet de $\sM$. Pour tout $X\in \sM$, le morphisme d'\'evaluation (co\"unit\'e de l'adjonction $\otimes-\uHom$)
\[X\otimes \uHom(X,K)\to K\]
donne par adjonction une \emph{fl\`eche de bidualit\'e}
\begin{equation}\label{eq5.8}
X\to \uHom(\uHom(X,K),K).
\end{equation}

\begin{defn}\phantomsection On dit que $K$ est un \emph{objet dualisant} si \eqref{eq5.8} est un isomorphisme pour tout $X\in \sM$.
\end{defn}

Un objet dualisant n'est pas uniquement d\'etermin\'e: si $K$ est dualisant, $L\otimes K$ est dualisant pour tout objet inversible $L$ de $\sM$ (\ie tel qu'il existe $L^{-1}$ et un isomorphisme $\un\iso L\otimes L^{-1}$). La r\'eciproque est vraie: si $K'$ est dualisant, $L=\uHom(K,K')$ est inversible, d'inverse $\uHom(K',K)$, et $K\otimes L\iso K'$.

On d\'esire alors les propri\'et\'es suivantes:
\begin{description}
\item[M5] Si $K\in T(Y)$ est dualisant, alors $f^!K\in T(X)$ est dualisant pour tout $f:X\to Y$.
\item[M6]\ Si $X$ est r\'egulier, $\un\in T(X)$ est dualisant.
\end{description}

Supposons $S$ r\'egulier et v\'erifiant {\bf M6}. Soit $f:X\to S$ un morphisme lisse: d'apr\`es {\bf M2}, $f^*\un$ est inversible et les conditions {\bf M5} et {\bf M6} sont \'equivalentes pour $X$.

Supposons toujours $S$ r\'egulier: on pose alors
\[K_X=f_X^!\un\]
pour tout $X\in \Sch(S)$, o\`u $f_X:X\to S$ est le morphisme structural, et
\[D_X(a) = \uHom(a,K_X)\]
pour $a\in T(X)$. Supposant {\bf M1},\dots, {\bf M6} v\'erifi\'es, $D_X$ est une bidualit\'e pour tout $X\in \Sch(S)$ et on d\'emontre formellement (?) les identit\'es suivantes pour un $S$-morphisme $f:X\to Y$:
\begin{gather}
D_X(f^*y)\simeq f^! D_Y(y), \quad D_X(f^!y)\simeq f^* D_Y(y)\label{eq5.9}\\
D_Y(f_*x) \simeq f_!D_X(x),\quad D_Y(f_!x)\simeq f_* D_X(x).\label{eq5.10}
\end{gather}

En pratique, on demande que les $T(X)$ soient des cat\'egories triangul\'ees et les foncteurs $f^*$, etc. des foncteurs triangul\'es.

\begin{rque}\phantomsection J'ai essay\'e de donner une liste de propri\'et\'es qui sont vraies dans tous les cas que je connais, mais je ne tente pas ici de donner un syst\`eme minimal d'axiomes, ni une collection maximale de propri\'et\'es, ni (sauf cas faciles) une id\'ee des d\'emonstrations possibles des implications. Non seulement cela d\'epasserait le cadre de ce cours, mais cette t\^ache semble impossible \`a l'heure actuelle. Il existe \`a ma connaissance quatre formalismes des six op\'erations en g\'eom\'etrie alg\'ebrique:
\begin{enumerate}
\item La dualit\'e coh\'erente de Grothendieck, expliqu\'ee par Robin Hartshorne dans \cite{hartshorne-res} et reprise par Brian Conrad \cite{conrad}, puis Amnon Neeman \cite{neeman-dual}. (Mais $j_!$ n'existe pas pour une immersion ouverte $j$\dots)
\item La dualit\'e en cohomologie \'etale ou $l$-adique\index{Cohomologie!$l$-adique}, qui est l'objet de SGA 4 et SGA 5 en passant par SGA 4 1/2 (voir ci-dessous).
\item La dualit\'e dans le monde motivique \`a la Voevodsky, construite par Joseph Ayoub dans \cite{ayoub-ast}.
\item La dualité pour les $D$-modules holonomes et pour les modules de 
Hodge mixtes (par exemple \cite{schnell}).
\end{enumerate}
Dans les quatre cas, les principes des axiomes de d\'epart et des d\'emonstrations se ressemblent sans tout \`a fait se recouvrir. On part de donn\'ees ``faciles'', \`a savoir $f^*$ et $f_*$ dans les cas 1 et 2, et on construit les autres en se fondant sur des th\'eorèmes non triviaux de changement de base et de finitude. Mais les strat\'egies sont diff\'erentes. Dans le cas \'etale on a changement de base pour tout morphisme propre et par tout morphgisme lisse; dans le cas coh\'erent, on a un th\'eor\`eme de changement de base par tout morphisme plat, alors que pour les morphismes propres on n'a qu'un th\'eor\`eme de semi-continuit\'e. Le cas 3 est encore diff\'erent: on utilise en plus la connaissance (facile) d'un adjoint \`a gauche de $f^*$ quand $f$ est lisse, plus des propri\'et\'es de $\A^1$-invariance par homotopie qui sont forc\'ees axiomatiquement (alors qu'elles sont d\'eduites des th\'eor\`emes dans le cas 2 et fausses dans le cas 1). Contrairement au cas \'etale, le changement de base par un morphisme lisse est facile à démontrer et le changement de base pour un morphisme propre s'en d\'eduit en passant par l'axiome de recollement faible (d\'ef. \ref{d5.1} \ref{essai2}), qui est la propri\'et\'e la moins formelle.  Je ne dirai rien du cas 4, faute de compétence.

Il semble (\`a l'auteur) qu'un travail de fondements clarifiant les liens entre ces constructions et d\'egageant un squelette commun \'eventuel serait utile \`a entreprendre.
\end{rque}

\subsubsection{Les six op\'erations en cohomologie \'etale} Pour tout sch\'ema $X$, tout nombre premier $l$ inversible dans $X$ et tout entier $n\ge 1$, notons
\[D^b_c(X,\Z/l^n)=D^b_c(X)\]
la sous-cat\'egorie pleine de $D(X_\et,\Z/l^n)$ (cat\'egorie d\'eriv\'ee des faisceaux \'etales de $\Z/l^n$-modules) form\'ee des complexes born\'es $C$ tels que $H^i(C)$ soit constructible pour tout $i\in \Z$. C'est une sous-cat\'egorie \'epaisse de $D(X_\et,\Z/l^n)$.

Soit $f:X\to Y$ un morphisme de sch\'emas. Il est \'evident que 
\[f^*D^b_c(Y)\subset D^b_c(X);\] 
le th\'eor\`eme de finitude pour la cohomologie \`a supports propres assure l'existence d'un foncteur triangul\'e
\[Rf_!:D^b_c(X)\to D^b_c(Y)\]
lorsque $f$ est compactifiable\footnote{Ou, plus g\'en\'eralement, s\'epar\'e de type fini \`a condition que $X$ et $Y$ soient quasi-compacts et quasi-s\'epar\'es: \cite[exp. XVII, \S 7]{SGA4}.}. Qu'en est-il pour les autres op\'erations $f_*,f^!,\uHom,\otimes$?

\begin{thm}[Deligne, \protect{\cite[Th. finitude, th. 1.1]{SGA412}}]\phantomsection\label{tfinitude} Soit $S$ un sch\'ema r\'egulier de dimension $\le 1$, et soit $\Sch(S)$ la cat\'egorie des $S$-sch\'emas de type fini. Donnons-nous un nombre premier $l$ inversible sur $S$. Soit $f:X\to Y$ un morphisme de $\Sch(S)$, et soit $F$ un faisceau \'etale constructible sur $X$, de $l^n$-torsion. Alors $R^qf_*F$ est constructible pour tout $q\ge 0$ et nul pour $q\gg0$. Par cons\'equent:
\[Rf_* D^b_c(X)\subset D^b_c(Y).\]
\end{thm}

\begin{cor}[ibid., cor. 1.5, cor. 1.6]\phantomsection Sous les m\^emes hypoth\`eses,
\[Rf^! D^b_c(Y)\subset D^b_c(X);\quad \RHom(D^b_c(X),D^b_c(X))\subset D^b_c(X).\]
\end{cor}

Pour la derni\`ere op\'eration $\otimes$, on a besoin d'une condition de Tor-dimension finie:

\begin{defn}[\protect{\cite[exp. XVII, 4.1.9]{SGA4}}]\phantomsection Un objet $K\in\break D^b(X_\et,\Z/l^n)$ est \emph{de \text{\rm Tor}-dimension $\le d$} s'il v\'erifie les conditions \'equivalentes suivantes:
\begin{thlist}
\item Pour tout $L\in D^b(X_\et,\Z/l^n)$ tel que $H^i(L)=0$ pour $i<l$, $H^i(K\oo\limits^L L)=0$ pour $i<l-d$.
\item Il existe un quasi-isomorphisme $K'\to K$, avec $K'$ \`a composantes plates et nulles pour $i<-d$.
\end{thlist}
On dit que $K$ est de Tor-dimension finie s'il est de Tor-dimension $\le d$ pour un entier $d$ convenable; on note 
\[D^b_{tf}(X)=\{K\in D^b(X_\et,\Z/l^n)\mid K \text{ est de Tor-dimension finie}\}\]
et 
\[D^b_{ctf}(X)=D^b_c(X)\cap D^b_{tf}(X).\]
\end{defn}

On voit facilement que cette condition est stable par $\oo\limits^L$.

\begin{cor}[\protect{\cite[Th. finitude, rem 1.7 et th. 4.3]{SGA412}}]\phantomsection\label{c5.2} $X\mapsto D^b_{ctf}(X)$ est stable par les six op\'erations $f^*,Rf_*,Rf_!,Rf^!, \oo\limits^L$, $\RHom$. De plus, le complexe
\[K_X=Rf_X^!\Z/l^n,\]
o\`u $f_X:X\to S$ est le morphisme structural, est dualisant.
\end{cor}

\subsubsection{Passage \`a la cohomologie  $l$-adique}\index{Cohomologie!$l$-adique}\label{5.3.3} Ce passage est non trivial: un point d\'elicat est de d\'efinir correctement des cat\'egories $D^b_c(X,\Z_l)$ et $D^b_c(X,\Q_l)$, voire $D^b_c(X,\bar \Q_l)$. 

Expliquons le probl\`eme de la premi\`ere. On aimerait d\'efinir 
\begin{equation}\label{eq5.11}
D^b_c(X,\Z/l^n)=2-\plim_n D^b_{ctf}(X,\Z/l^n)
\end{equation}
o\`u, \'etant donn\'e un syst\`eme projectif $(\sC_n,F_n:\sC_n\to \sC_{n-1})$ de cat\'egories, $2-\plim \sC_n$ est la cat\'egorie dont les objets sont les syst\`emes $(X_n, u_n:F_n(X_n)\iso X_{n-1})$ avec $X_n\in \sC_n$, les morphismes \'etant donn\'es ``composante par composante''. Malheureusement, si les $\sC_n$ sont des cat\'egories triangul\'ees et les $F_n$ des foncteurs triangul\'es, la structure triangul\'ee est en g\'en\'eral perdue apr\`es passage \`a $2-\plim$, \`a cause de la non-exactitude du foncteur limite projective (sur les groupes ab\'eliens). 

Heureusement, dans le cas o\`u $X$ est de type fini sur $\Spec \Z[1/l]$, le th\'eor\`eme \ref{tfinitude}, joint au fait que  $H^i_\et(\Spec \Z[1/l],F)]$ est fini pour tout $i\in \Z$ et tout faisceau constructible de $\Z/l^n$-modules $F$\footnote{Ce qui se r\'eduit essentiellement \`a la finitude du groupe des classes de $\Z[1/l,\mu_l]$ et au fait que son groupe d'unit\'es est un $\Z$-module de type fini.}, implique que $\Hom(K,L)$ est un groupe fini pour tous $K,L\in D^b_c(X,\Z/l^n)$. La condition de Mittag-Leffler implique alors que \eqref{eq5.11} d\'efinit bien une cat\'egorie triangul\'ee. C'est ce que fait Deligne dans \cite[1.1.2]{weilII}.

Le probl\`eme de la bonne d\'efinition de $D(X_\et,\Z_l)$ pour $X$ quelconque n'a \'et\'e r\'esolu que plus tard par Ofer Gabber (non publi\'e) et ind\'ependamment par Torsten Ekedahl \cite{ekedahl}.

Pour d\'efinir $D^b_c(X,\Q_l)$ ou $D^b_c(X,\bar\Q_l)$, on part de $D^b_c(X,\Z_l)$ et on tensorise les groupes d'homomorphismes par $\Q_l$ ou $\bar \Q_l$, en prenant ensuite l'enveloppe karoubienne si besoin est. On a alors:

\begin{thm}\phantomsection Pour $X$ variant dans $\Sch(\Z[1/l])$, $X\mapsto D^b(X,\Q_l)$ admet un formalisme des six op\'erations et de dualit\'e parall\`ele \`a celle du corollaire \ref{c5.2}.
\end{thm}

\subsubsection{L'\'equation fonctionnelle de Grothendieck}\label{s5.3.4} Nous disposons maintenant du langage qui permet de formuler l'\'equation fonctionnelle \'etudi\'ee par Grothendieck dans \cite[lettre du 30-9-64]{corresp}.\index{Equation fonctionnelle@\'Equation fonctionnelle}

Soit $X\in \Sch(\F_q)$. Donnons-nous un objet $E\in D^b_c(X,\Q_l)$. On aimerait relier la fonction $L(X,E)$ \`a la fonction $L(X,D_X(E))$, o\`u $D_X$ est la dualit\'e donn\'ee ci-dessus.

Pour expliquer le point d\'elicat, calculons cette derni\`ere fonction en appliquant la formule \eqref{eq5.12} avec $Y=\Spec \F_q$:
\[L(X,D_X(E))=L(\F_q,Rf_!D_X(E))=L(\F_q,D_{\F_q}(Rf_*E))=L(\F_q,(Rf_*E)^*)\]
o\`u $^*$ est la dualit\'e simple des $\Q_l$-espaces vectoriels gradu\'es (pour la r\`egle de Koszul). Ainsi, on trouve une expression non pas en termes de la cohomologie \`a supports propres de $E$, mais de sa cohomologie ordinaire. 

Cela n'arr\^ete pas Grothendieck, qui pose
\[R_\infty f(E) = [Rf_*E]-[Rf_!E]\]
dans le groupe de Grothendieck de $D^b_c(\F_q,\Q_l)$ (voir \S \ref{K0}). Il en tire l'\'equation fonctionnelle
\begin{equation}\label{eq5.13}
L(X,D_X(E),t)=(-t)^{-\chi(E)}\delta(E)L(X,E,t^{-1})A_\infty(t)
\end{equation}
o\`u 
\[\chi(E)= \sum_i (-1)^i \dim_{\Q_l} H^i_c(\bar X, E), \quad \delta(E) = \prod_i \det(\pi_{\F_q}\mid H^i_c(\bar X,E))^{(-1)^i}\]
et $A_\infty(t)=L(\F_q,R_\infty f(E),t)^{-1}$ est consid\'er\'e comme un ``terme correctif \`a l'infini''. (La formule \eqref{eq5.13} n'est pas difficile \`a d\'emontrer, en suivant la m\'ethode de d\'emonstration de \eqref{eq3.7}: voir le calcul de \eqref{eq3.6}.)

Grothendieck justifie cette expression en donnant une formule explicite pour $R_\infty f(E)$ en termes d'une compactification $\bar X$ de $X$:
\[\xymatrix{
X\ar[r]^j\ar[rd]_f& \bar X\ar[d]^{\bar f}&Y\ar[l]_i\ar[ld]^g\\
&\Spec \F_q
}\]
o\`u $Y$ est le ferm\'e compl\'ementaire:
\[R_\infty f(E) = [Rg_*(i^*Rj_*(E))].\]

Prouvons cette formule: pour tout $G\in D^b_c(\bar X,\Q_l)$, on a un triangle exact canonique:
\[j_! j^* G \to G\to i_*i^*G\by{+1}.\]

Prenant $G=Rj_* E$ et utilisant que $j^*Rj_*=Id_{X}$, on obtient
\[j_! E \to j_* E\to i_*i^*j_*E\by{+1}.\]

En appliquant maintenant $R\bar f_!$, on obtient
\[R\bar f_!j_! E \to R\bar f_!j_* E\to R\bar f_!i_*i^*Rj_*E\by{+1}.\]

Mais $\bar f$ est propre, donc $R\bar f_!=R\bar f_*$ et le triangle ci-dessus se r\'ecrit
\[R(\bar fj)_! E \to R(\bar fj)_* E\to R(\bar fi)_*i^*Rj_*E\by{+1}\]
soit
\[Rf_! E \to Rf_* E\to Rg_*i^*Rj_*E\by{+1}.\]

Lorsque $f$ est propre, la formule \eqref{eq5.13} devient plus propre (sic) puisqu'alors $A_\infty(t)=1$. 

Supposons maintenant $f$ lisse, purement de dimension $n$; alors
\[K_X = Rf^!\Q_l = f^*\Q_l(n)[2n]=\Q_l(n)[2n]\]
d'o\`u
\[D_X(E)= E^*(n)[2n]\]
o\`u $E^*=\RHom(E,\Q_l)$ est le dual ``na\"\i f'' de $E$. Dans ce cas, \eqref{eq5.13} se r\'ecrit
\begin{equation}\label{eq5.13lisse}
L(X,E^*,q^{-n}t)=(-t)^{-\chi(E)}\delta(E)L(X,E,t^{-1})A_\infty(t).
\end{equation}

Un cas particulier important, qui sera justifi\'e par un exemple ci-dessous, est celui o\`u $E$ est ``faiblement polarisable de poids $\rho$'' au sens que
\begin{equation}\label{eqpol}
E^*\simeq E(\rho)
\end{equation}
pour un certain entier $\rho\in\Z$, o\`u $E(\rho):=E\otimes \Q_l(\rho)$. Dans ce cas, \eqref{eq5.13lisse} prend la forme (en rempla\c cant $t$ par $1/t$)
\begin{equation}\label{eq5.13lissepol}
L(X,E,1/q^{n+\rho}t)=(-t)^{\chi(E)}\delta(E)L(X,E,t)A_\infty(t^{-1}).
\end{equation}

On a alors la g\'en\'eralisation suivante de \eqref{eq3.7} ``\`a coefficients'':

\begin{thm}[Grothendieck, ibid.]\phantomsection\label{tgef} En plus des hypoth\`eses ci-dessus, supposons $f$ propre. Donc:
\begin{itemize}
\item $f$ est propre et lisse;
\item $E$ est faiblement polarisable de poids $\rho$ au sens de \eqref{eqpol}.
\end{itemize}
Alors $L(X,E,t)$ v\'erifie l'\'equation fonctionnelle
\[L(X,E,1/q^{n+\rho}t)= (-t)^{\chi(E)}\delta(E)L(X,E,t)\]
avec
\[\chi(E)= \sum_i (-1)^i \dim_{\Q_l} H^i(\bar X, E), \quad \delta(E) = \prod_i \det(\pi_{\F_q}\mid H^i(\bar X,E))^{(-1)^i}.\]
De plus, on a $\delta(E)^2=q^{(\rho+n)\chi(E)}$, donc 
\[\delta(E)=\pm q^{\frac{(\rho+n)\chi(E)}{2}}.\]
\end{thm}

\begin{proof} La seule chose restant \`a d\'emontrer est la relation entre $\delta(E)$ et $\chi(E)$. Mais par dualit\'e de Poincar\'e, on peut aussi \'ecrire 
\[\delta(E) = \prod_i \det(\pi_{\F_q}\mid H^i(\bar X,D_X(E)))^{(-1)^{i+1}}.
\]

En rempla\c cant $D_X(E)$ par sa valeur $E(\rho)$ et en comparant \`a l'expression pr\'ec\'edente, on obtient ais\'ement la relation annonc\'ee.
\end{proof}

\subsubsection{Un cas particulier} Soit $K=\F_q(C)$ avec $C$ projective lisse, et soit $j:U\inj C$ un ouvert non vide. Donnons-nous un $\Q_l$-faisceau lisse $M$ sur $U$, et prenons $E=j_*M$. On a alors
\begin{align*}
E^* &= j_*(M^*)\\
D_C(E) &= E^*(1)[2].
\intertext{Si $M$ est faiblement polarisable de poids $\rho$, cette derni\`ere formule se r\'ecrit}
D_C(E) &= E(\rho+1)[2]
\end{align*}
et on trouve:

\begin{thm}[Grothendieck, ibid.]\phantomsection\label{tgef2} Avec les notations et hypoth\`eses ci-dessus, on a la m\^eme \'equation fonctionnelle que dans le th\'eor\`eme \ref{tgef}, avec $n=1$. De plus, on a une factorisation
\[L(X,E,t) = \frac{P_1(t)}{P_0(t)P_2(t)}\]
avec $P_i(t)= \det(1-t\pi_{\F_q}\mid H^i(\bar C, j_*M))$.
\end{thm}

\subsection{La th\'eorie des poids} R\'ef\'erence: Deligne \cite{weilII}.

Citons Grothendieck \cite[lettre du 31.10.61]{corresp},  r\'eponse \`a celle de Serre du 26 oct. 61 cit\'ee au \S\ref{5.2.1}:

\begin{quote} \it
Je n'ai pas bien compris ta question sur la cohomologie\index{Cohomologie!de Weil} (de Weil, je pr\'esume) des vari\'et\'es affines, et d'ailleurs je ne sais rien d'une \'eventuelle filtration naturelle, m\^eme pour une vari\'et\'e compl\`ete. On pourra en recauser quand tu seras \`a Harvard.
\end{quote}

Et encore (R\'ecoltes et Semailles, vol. 5, p. 792, note 4):

\begin{quote} {\it Je croyais, \`a tort, me rappeler que j'avais introduit la filtration par 
les poids d'un motif, se refl\'etant (pour tout $\ell$) en la filtration correspondante sur la r\'ealisation $\ell$-adique de ce motif (filtration d\'efinie en termes de 
valeurs absolues de valeurs propres de Frobenius). En fait, Deligne m'a rappel\'e 
que je n'avais travaill\'e qu'avec les notions de poids virtuelles (ce qui 
revenait à travailler avec des motifs virtuels, \'el\'ements d'un groupe de 
Grothendieck convenable\dots). C'est Deligne qui a d\'ecouvert ce fait important 
que la notion virtuelle avec laquelle je travaillais devrait correspondre \`a une 
filtration canonique, par poids croissants. Cette d\'ecouverte 
(toute aussi conjecturale que la th\'eorie conjecturale des motifs) a fourni 
aussit\^ot la clef d'une d\'efinition en forme des structures de 
Hodge-Deligne (dites aussi structures de Hodge mixtes) sur le 
corps des complexes, comme transcription \`a la Hodge des structures d\'ej\`a 
connues sur le motif et sur sa r\'ealisation de Hodge.} \end{quote}

Je vais d\'ecrire ci-dessous la th\'eorie des poids dans le cas $l$-adique, qui est le sujet de \cite{weilII}.

\subsubsection{Poids et faisceaux mixtes}

\begin{defn}  Soit $q$ une puissance d'un nombre premier, et soit $i\in\Z$. Un nombre $\alpha\in\bar \Q$ est  un \emph{$q$-nombre de Weil de poids $i$} si $|\sigma\alpha|=q^{i/2}$ pour tout $\sigma\in \Gal(\bar\Q/\Q)$.\index{Nombres!de Weil}
\end{defn}

\begin{rque}\phantomsection\label{r5.2} Soit $\alpha$ un $q$-nombre de Weil de poids $i$. Alors $\sigma \alpha$ est un $q$-nombre de Weil de poids $i$ pour tout $\sigma\in Gal(\bar \Q/\Q)$. Supposons $\sigma$ induit par la conjugaison complexe. Alors $\alpha\cdot \sigma \alpha$ est un nombre réel positif, de valeur absolue $q$, donc égal à $q$. Ainsi, $q/\alpha$ est \emph{conjugué} de $\alpha$.
\end{rque}

\begin{defn} 
Soit $X\in \Sch(\Z[1/l])$ et soit $F$ un faisceau $\Q_l$-adique sur $X$. 
\begin{thlist}
\item $F$ est (ponctuellement) \emph{pur de poids $i$} si, pour tout $x\in X_{(0)}$, les valeurs propres de $\pi_x\mid F_x$ sont des $N(x)$-nombres de Weil de poids $i$. ($\pi_x$ d\'esigne comme d'habitude le Frobenius g\'eom\'etrique de $x$.)
\item $F$ est \emph{mixte de poids $\le i$} s'il existe une filtration finie de $F$, de quotients successifs purs de poids $\le i$.
\end{thlist} 
\end{defn}

\begin{exs} $\Q_l(1)$ est pur de poids $-2$. Si $\sA$ est un sch\'ema ab\'elien sur $X$, $V_l(\sA)=T_l(\sA)\otimes \Q_l$ est pur de poids $-1$ gr\^ace au th\'eor\`eme \ref{t3.4} a). (Ce dernier r\'esultat n'est pas utilis\'e dans \cite{weilII}.)
\end{exs}

En fait, Deligne travaille avec une notion plus fine de $\iota$-poids et de faisceaux $\iota$-mixtes, o\`u $\iota$ est un isomorphisme fix\'e de $\bar \Q_l$ sur $\C$, de sorte que mixte $\iff$ $\iota$-mixte pour tout $\iota$: en effet, les valeurs propres de $\pi_x\mid F_x$ ci-dessus sont \emph{a priori} des \'el\'ements de $\Q_l$. Implicitement, ceci utilise l'axiome du choix! Citons-le, \cite[rem. 1.2.11]{weilII}:

\begin{quote} \it Je ne pr\'etends pas croire \`a l'existence d'isomorphismes entre $\bar \Q_l$ et $\C$, et ceux-ci ne sont qu'une commodit\'e d'exposition. Chaque fois que nous prouverons qu'un nombre est pur, un fragment facile de la d\'emonstration suffirait \`a \'etablir qu'il est alg\'ebrique. Pour le reste des arguments, il suffirait alors de consid\'erer les plongements complexes du sous-corps de $\bar \Q_l$ for\-m\'e des nombres alg\'ebriques, et ceci ne requiert pas l'axiome du choix.
\end{quote}

Pour la d\'efinition correcte de $\iota$-poids (\`a valeurs dans $\R$) et des faisceaux $\iota$-purs et $\iota$-mixtes, voir \cite[1.2.6]{weilII}. 

Il y a d'abord la conjecture frappante:

\begin{conj} Si $X\in \Sch(\F_p)$, tout $\Q_l$-faisceau constructible sur $X$ est mixte.
\end{conj}

Le th\'eor\`eme principal est:

\begin{thm}[\protect{\cite[th. 3.3.1]{weilII}}] Soit $f:X\to Y$ un morphisme de $\Sch(\Z[1/l])$, et soit $F$ un $\Q_l$-faisceau sur $X$, mixte de poids $\le \rho$. Alors $R^qf_! F$ est mixte de poids $\le \rho+q$ pour tout $q\ge 0$.
\end{thm}

Deligne en d\'eduit:

\begin{thm}[\protect{\cite[th. 3.4.1]{weilII}}] \phantomsection\label{twII} Soit $X\in \Sch(\F_q)$, et soit $F$ un $\Q_l$-faisceau lisse mixte sur $X$. Alors\\
a) $F$ admet une filtration croissante finie $W_\bullet F$, la \emph{filtration par le poids}, telle que pour tout $i\in \Z$, $\Gr_i F:=W_i F/W_{i-1} F$ soit lisse et pur de poids $i$. Cette filtration est unique et strictement fonctorielle en $F$.\\
b) Si $F$ est pur et $X$ est normal, $F_{|\bar X}$ est semi-simple.
\end{thm}

\begin{cor}[\protect{\cite[th. 6.1.11]{weilII}}] Pour $X\in \Sch(\F_q)$, notons \break $D^b_m(X,\Q_l)$ la sous-cat\'egorie \'epaisse de $D^b_c(X,\Q_l)$ form\'ee des objets $C$ tels que $H^i(C)$ soit mixte pour tout $i\in \Z$. Alors $D^b_m$ est stable par les six op\'erations.
\end{cor}

Le m\^eme corollaire vaut pour $X\in \Sch(\Q)$ \`a condition de d\'efinir un faisceau mixte sur $X$ comme un $\Q_l$-faisceau provenant d'un faiceau mixte sur $\sX_n$ pour $n$ multiplicativement assez grand, o\`u $\sX_n$ est un mod\`ele de type fini de $X$ sur $\Z[1/n]$ (\emph{ibid}).

Un autre r\'esultat important est:

\begin{thm}[\protect{\cite[th. 3.2.3]{weilII}}]\phantomsection\label{t5.4} Soient $C$ une courbe projective lisse sur $\F_q$, $j:U\to C$ un ouvert dense et $F$ un $\Q_l$-faisceau lisse sur $U$, pur de poids $\rho$. Alors les valeurs propres de Frobenius agissant sur $H^i(\bar C,j_*F)$ sont des $q$-nombres de Weil de poids $\rho+i$.
\end{thm}

On aura enfin besoin de l'\'enonc\'e suivant, qui n'est pas explicite dans \cite{weilII} (mais voir \cite[5.1.14 et l. 5 au-dessus de 5.1.9]{bbd}):

\begin{prop}\phantomsection\label{p5.1} Soit $p:\sX\to U$ un morphisme propre et lisse, o\`u $U$ est une vari\'et\'e lisse de dimension $d$ sur $\F_q$. Soit $F$ un $\Q_l$-faisceau lisse sur $\sX$, pur de poids $\rho$. Alors, pour tout $i\ge 0$, $R^ip_* F$ est lisse, et pur de poids $\rho+i$.
\end{prop}

Dans le cas $d=0$, $F=\Q_l$, on retrouve la derni\`ere conjecture de Weil (hypoth\`ese de Riemann)\index{Hypothèse de Riemann!pour les variétés projectives lisses} g\'en\'eralis\'ee du cas projectif lisse au cas propre et lisse, \cite[cor. 3.3.9]{weilII}.

\begin{proof} Que $R^ip_*F$ soit lisse r\'esulte du fait que $p$ est propre et lisse. Le th\'eor\`eme \ref{twII} implique qu'il est mixte de poids $\le \rho+i$; pour conclure, il suffit de voir que son dual est mixte de poids $\le -\rho-i$. Pour cela, on peut supposer $\sX$ connexe, donc \'equidimensionnel. On applique le formalisme \eqref{eq5.10}:
\[D_U(Rp_*F)\simeq Rp_*D_\sX(F)\]
qui donne
\[
\RHom(Rp_*F,\Q_l(d)[2d])\simeq Rp_*\RHom(F,\Q_l(n+d)[2n+2d])
\]
o\`u $n$ est la dimension relative de $p$, d'o\`u 
\[
\RHom(Rp_*F,\Q_l)\simeq Rp_*\RHom(F,\Q_l(n)[2n])
\]
et en prenant $H^{-i}$:
\[\uHom(R^ip_*F,\Q_l)\simeq R^{2n-i}p_*(\uHom(F,\Q_l))(n)\]
(noter qu'on a $\RHom(G,\Q_l)=\uHom(G,\Q_l)[0]$ pour tout faisceau lisse $G$). On conclut en r\'eappliquant le th\'eor\`eme \ref{twII}.
\end{proof}

\subsubsection{Le th\'eor\`eme de Lefschetz difficile}\label{s5.4.2}

Soit $X$ une vari\'et\'e projective lisse de dimension $n$ sur un corps $k$. Choisissons une section hyperplane lisse  $Y\subset X$ pour un plongement projectif convenable $i:X\inj \P^N$. Par le th\'eor\`eme de Bertini \cite{hartshorne}, l'ensemble de ces sections d\'efinit un ouvert non vide $U$ d'un espace affine: $Y$ existe donc toujours si $k$ est infini.

Si $k$ est fini, il se peut que $U(k)=\emptyset$: dans ce cas, on compose $i$ avec un \emph{plongement de Veronese}
\[\P^N = \P(V)\inj \P(S^r(V))\]
induit par $x\mapsto x^r$. Pour $r$ assez grand, l'ouvert $U_r$ correspondant a un point rationnel sur $k$ (\cf \cite[exp. XVII]{SGA7}). Ceci revient \`a remplacer les sections hyperplanes dans $\P^N$ par des sections par des hypersurfaces de degr\'e $r$.

Notons $L=\cl^1(Y)\in H^2(X)(1)$, o\`u $H$ est une cohomologie de Weil. Pour $i\le n$, on peut consid\'erer l'homomorphisme
\begin{equation}\label{eq5.14}
\begin{CD}
H^i(X)@>\cdot [L]^{n-i}>> H^{2n-i}(X)(n-i). 
\end{CD}
\end{equation}

Si $k=\C$ et $H=H_B$, \eqref{eq5.14} est un isomorphisme: cela r\'esulte de la th\'eorie de Hodge. En utilisant les th\'eor\`emes de comparaison du \S \ref{3.6.1}, on en d\'eduit que \eqref{eq5.14} est un isomorphisme si $H=H_l$ et $\car k=0$.

Le cas d'un corps fini est vrai mais beaucoup plus difficile:

\begin{thm}[\protect{\cite[th. 4.1.1]{weilII}}] \eqref{eq5.14} est un isomorphisme si $k\allowbreak=\F_q$ et $H=H_l$, avec $l\nmid q$.
\end{thm}

D'autre part, la dualit\'e de Poincar\'e fournit un accouplement parfait \eqref{eq:dP}. En combinant les deux, on obtient:

\begin{cor}\phantomsection\label{c5.5} Pour tout $i\le n$, le choix de $L$ fournit un accouplement parfait, Galois-\'equivariant et $(-1)^i$-sym\'etrique:
\[H^i_l(X)\times H^i_l(X)\to \Q_l(-i).\]
\end{cor}

\subsection{La fonction $L$ compl\'et\'ee d'une vari\'et\'e projective lisse sur un corps global}\label{s5.5} R\'ef\'erence: Serre \cite{serre-dpp}.\index{Fonction zêta!de Hasse-Weil}

\subsubsection{Le probl\`eme} Soit $K$ un corps global, et soit $X\in \V(K)$. Si $v\in \Sigma_K^f$ est une place de bonne r\'eduction pour $X$, on peut d\'efinir un \emph{facteur local en $v$} de la fonction $L$ (ou z\^eta) de $X$:
\[L_v(X,s) = \zeta(X(v),s)\]
o\`u $X(v)$ est la fibre sp\'eciale en $v$ d'un mod\`ele projectif lisse de $X$ sur $O_v$, puis une fonction $L$ ``de Hasse-Weil'' partielle
\[L_0(X,s) = \prod_{\begin{smallmatrix}v\in \Sigma_K^f\\v \text{ bonne}\end{smallmatrix}} L_v(X,s).\]

De mani\`ere essentiellement \'equivalente, le morphisme $X\to \Spec K$ s'\'etend en un morphisme projectif lisse $\sX\to U$, o\`u $U=\Spec O_S$ pour un anneau de $S$-entiers de $K$ convenable si $K$ est un corps de nombres, et $U$ est un ouvert du mod\`ele projectif lisse de $K$ si $K$ est un corps de fonctions; on a alors
\[L_0(X,s)=\zeta(\sX,s)\]
\`a un nombre fini de facteurs eul\'eriens pr\`es (d\'ependant du choix de $U$). 

Dans le cas de certaines courbes sur un corps de nombres, on peut d\'emontrer (\`a l'aide des caract\`eres de Hecke)\index{Caractère!de Hecke} que $L_0(X,s)$ admet un prolongement analytique \`a $\C$ et une \'equation fonctionnelle: \cf th\'eor\`eme \ref{twh} et \S \ref{5.1.4}. \index{Equation fonctionnelle@\'Equation fonctionnelle}

Dans \cite{serre-dpp}, Serre \'etudie la question suivante: peut-on associer \`a $X$ un ``facteur local'' en chaque place $v\in \Sigma_K$, et un ``terme exponentiel'', qui permettent de conjecturer une ``jolie'' \'equation fonctionnelle?

Il r\'epond \`a cette question de mani\`ere plus pr\'ecise, en attachant \`a chaque ``groupe de cohomologie de $X$'' des facteurs locaux et un termes exponentiel, qui lui permettent de formuler une conjecture comme ci-dessus. De plus, cette conjecture est vraie en caract\'eristique $p$ essentiellement par r\'eduction au th\'eor\`eme \ref{tgef}.

On fixe donc dans la suite une $K$-vari\'et\'e projective lisse $X$ de dimension $n$, et un entier $i\in [0,2n]$.

\subsubsection{Les facteurs locaux non archim\'ediens}

\begin{defn}\phantomsection\label{d5.2} Soit $v\in \Sigma_K^f$. On pose
\[L_v(K,H^i(X),s)= \det(1-\pi_vN(v)^{-s}\mid H^i_l(X)^{I_v})^{-1}\]
o\`u
\begin{itemize}
\item $l$ est un nombre premier ne divisant pas $N(v)$.
\item $H^i_l(X)=H^i_\et(\bar X,\Q_l)$ est la cohomologie\index{Cohomologie!de Weil} de Weil attach\'ee \`a $l$ (cohomologie $l$-adique\index{Cohomologie!$l$-adique} g\'eom\'etrique).
\item $I_v$ est le groupe d'inertie en $v$.
\item $\pi_v$ est (une classe de conjugaison de) l'\'el\'ement de Frobenius g\'eom\'etrique en $v$.
\end{itemize}
\end{defn}

Pr\'ecisons cette d\'efinition. Le groupe de Galois absolu $G_v$ de $K_v$, compl\'et\'e de $K$ en $v$, s'ins\`ere dans une suite exacte
\[1\to I_v\to G_v\to G(v)\to 1\]
o\`u $G(v)$ est le groupe de Galois absolu du corps r\'esiduel $\kappa(v)$. L'action de $G_v$ sur $H^i_l(X)^{I_v}$ se factorise par $G(v)$, et l'\'el\'ement de $G(v)$ intervenant dans la d\'efinition \ref{d5.2} est, comme toujours, l'inverse du Frobenius ``arithm\'etique''.

\begin{lemme} Supposons que $X$ ait bonne r\'eduction en $v$, et soit $X(v)$ la fibre sp\'eciale d'un mod\`ele projectif lisse $\sX_v$ de $X$ sur $O_v$. Alors
\[L_v(K,H^i(X),s)=P_i(X(v),N(v)^{-s})^{-1}\]
o\`u $P_i\in \Z[t]$ est le polyn\^ome intervenant dans la d\'ecomposition de la fonction $Z(X(v),t)$, \cf \eqref{eq3.4}. En particulier:
\begin{enumerate}
\item $L_v(K,H^i(X),s)$ ne d\'epend pas du choix de $l$;
\item $\prod_{i=0}^{2n} L_v(K,H^i(X),s)^{(-1)^i}=\zeta(X(v),s)$.
\end{enumerate}
\end{lemme}

\begin{proof} La premi\`ere affirmation r\'esulte du th\'eorème de changement de base propre et lisse, \cf \S \ref{3.6.1}: celui-ci contient le fait que (sous l'hypoth\`ese de bonne r\'eduction) $I_v$ op\`ere trivialement sur $H^i_l(X)$. L'int\'egralit\'e des $P_i$ et leur ind\'ependance de $l$ r\'esulte alors de l'hypoth\`ese de Riemann\index{Hypothèse de Riemann!pour les variétés projectives lisses} \cite{weilI}, \cf lemme \ref{l3.4}. La derni\`ere identit\'e est claire.
\end{proof}

\subsubsection{Le cas de mauvaise r\'eduction} Que se passe-t-il quand $X$ n'a pas bonne r\'eduction en $v$?  L'ind\'ependance de $l$ est alors une conjecture:

\begin{conj}[Serre \protect{\cite[conj. $C_5$ et $C_6$]{serre-dpp}}]\phantomsection\label{c5.4} Le polyn\^ome
\[\det(1-\pi_vt\mid H^i_l(X)^{I_v})\] 
est ind\'ependant du choix de $l$ et \`a coefficients entiers; ses racines inverses sont des nombres de Weil\index{Nombres!de Weil} de poids compris entre $0$ et $i$.
\end{conj}

\begin{thm}\phantomsection\label{t5.3} La conjecture \ref{c5.4} est vraie dans les cas suivants:
\begin{thlist}
\item Si $i\in \{0,1,2n-1,2n\}$.
\item Si $\car K>0$.
\end{thlist}
\end{thm}

\begin{proof} (i) est  trivial pour $i=0$; c'est un th\'eor\`eme de Grothendieck pour $i=1$: \cite[exp. IX, th. 4.3 et cor. 4.4]{SGA7}.  Les autres cas s'en d\'eduisent par dualit\'e de Poincar\'e.

(Pour appliquer les r\'esultats de Grothendieck, on observe que $H^1_l(X) = T_l(\Pic^0(X))$ o\`u $\Pic^0(X)$ est la vari\'et\'e de Picard de $X$, \cf th\'eor\`eme \ref{t3.5}.)

Le cas (ii) est d\^u \`a  Tomohide Terasoma \cite{terasoma}. 
\end{proof}

\subsubsection{La conjecture de monodromie-poids}\index{Conjecture!de monodromie-poids}\label{s5.5.4}
La conjecture \ref{c5.4} est li\'ee \`a cette conjecture, qui peut se résumer par le slogan: 

\begin{quote}
\it Sur $H^i_l(X)$, la filtration par le poids co\"\i ncide avec la filtration par la monodromie.
\end{quote}

Plus précisément, on d\'efinit sur $H^i_l(X)$ une certaine filtration $M_*H^i_l(X)$ $\Gal(\bar K/K)$-invariante: la \emph{filtration par la monodromie} \cite[1.7]{weilII}. Elle est sujette au  \emph{th\'eor\`eme de monodromie $l$-adique de Grothendieck} (théorème \ref{t5.7} ci-dessous):\index{Théorème!de monodromie $l$-adique}

\begin{defn} Soient $\Gamma$ un groupe profini, $l$ un nombre premier, et $\rho:\Gamma\to GL_n(\Q_l)$ une repr\'esentation continue. On dit que $\rho$ est \emph{unipotente} si, pour tout $g\in \Gamma$, l'endomorphisme $\rho(g)-1$ est nilpotent. On dit que $\rho$ est \emph{quasi-unipotente} s'il existe un sous-groupe ouvert $U\subset \Gamma$ tel que  la restriction de $\rho$ \`a $U$ soit unipotente.
\end{defn}

\begin{thm}[\protect{\cite[exp. I]{SGA7}}]\phantomsection\label{t5.7} Soit $K$ un corps complet pour une valuation discr\`ete, \`a corps r\'esiduel fini $k$ de caract\'eristique $p$, et soit $\rho:G\to GL_n(\Q_l)$ une repr\'esentation continue, o\`u $G=\Gal(K_s/K)$ et soit $l\ne p$. Alors la restriction de $\rho$ au groupe d'inertie $I$ est quasi-unipotente.
\end{thm}

Voir l'exercice \ref{exo5.3} pour une démonstration, et \cite[\emph{loc. cit.}]{SGA7} pour d'autres (avec des hypothèses plus faibles sur $k$).

Ce théorème implique que l'action de $I_v$ sur $\Gr_j^M H^i_l(X)$ se factorise par un quotient fini; si $F$ est un relev\'e du Frobenius g\'eom\'etrique dans $\Gal(\bar K_v/K_v)$, son action sur $\Gr_j^M H^i_l(X)$ est donc bien d\'efinie \`a une racine de l'unit\'e pr\`es. Ceci permet de formuler:

\begin{conj}\phantomsection\label{c5.6} Pour tout $j\in \Z$, les valeurs propres de Frobenius op\'erant sur $\Gr_j^M H^i_l(X)$ sont des nombres de Weil de poids $i+j$.
\end{conj}

Si $K=\F_q(C)$ pour une courbe $C$, cette conjecture est d\'emontr\'ee par Deligne dans  \cite[1.8.4]{weilII}; la preuve de Terasoma du théorème \ref{t5.3} (ii) repose sur ce r\'esultat. 

D'apr\`es Takeshi Saito \cite[cor. 0.6]{tsaito}, la conjecture \ref{c5.4} r\'esulte de la conjecture \ref{c5.6} et de l'alg\'ebricit\'e du $i$-\`eme projecteur de K\"unneth de $X$. Cela lui permet de d\'emontrer la conjecture \ref{c5.4} pour les \emph{surfaces sur un corps de nombres}: dans ce cas, on sait en effet que tous les projecteurs de K\"unneth sont alg\'ebriques, et la démonstration de la conjecture de monodromie-poids  est due \`a Rapoport-Zink \cite[Satz 2.13]{rz}.

De plus, Morihiko Saito a montr\'e que (sur un corps de nombres) la conjecture \ref{c5.6} r\'esulte des conjectures standard de Grothendieck \cite{msaito}. Enfin, Peter Scholze l'a d\'emontrée dans un certain nombre de cas nouveaux \cite{scholze}.

\begin{exo}\phantomsection\label{exo5.3}  Pour faire cet exercice, il faut conna\^\i tre un peu de th\'eorie de la ramification des corps locaux, qu'on peut trouver dans \cite[ch. IV, \S\S 1 et 2]{serre-cl}. %Je r\'esume ici ce qui sera n\'ecessaire:
L'exercice suit la démonstration de  \cite[prop. dans l'appendice]{serre-tate}. On pourra utiliser les exercices \ref{exo5.1} et \ref{exo5.2}.

\begin{enumerate}[label=(\alph*)]
\item Soit $\rho$ comme dans le théorème \ref{t5.7}, et soit  $V$ son espace de repr\'esentation. Montrer que $G$ laisse invariant un r\'eseau de $V$.  (Soit $L_0$ un r\'eseau quelconque de $V$: montrer que $\rho(G)\cdot L_0$ est compact, donc est un r\'eseau.)
\item En d\'eduire qu'on peut supposer que $\rho$ est \`a valeurs dans $GL_n(\Z_l)$.
\item Montrer qu'il suffit de d\'emontrer le th\'eor\`eme en rempla\c cant $G$ par un sous-groupe ouvert, c'est-\`a-dire $K$ par une extension finie.
\item En d\'eduire qu'on peut supposer que, pour tout $\sigma\in G$, on a $\rho(\sigma)\equiv 1_n\pmod{l^2M_n(\Z_l)}$ o\`u $1_n$ est la matrice unit\'e (consid\'erer la composition $G\by{\rho}GL_n(\Z_l)\to GL_n(\Z/l^2)$). On va alors montrer que $\rho_{|I}$ est unipotente.
\item Montrer que $\rho(P)=1$, où $P$ est le groupe d'inertie sauvage de $K$. (Utiliser l'exercice \ref{exo5.2} (f).)
\item Montrer que $\rho$ se factorise m\^eme \`a travers le quotient $I_m^l=\Z_l(1)$ de $I_m$.
\item Soit $s\in I_m^l$, et soit $\lambda\in L$ une valeur propre de $\rho(s)$ dans une extension finie $L$ de $\Q_l$. Montrer que $\lambda\in 1+l^2O_L$.
\item Soit $q$ le cardinal du corps r\'esiduel de $K$. Montrer que $\lambda^q$ est aussi une valeur propre de $\rho(s)$. (Utiliser l'exercice \ref{exo5.1} (a) et le fait que l'homomorphisme de \cite[ch. IV, \S 2, cor. 1]{serre-cl} est Galois-équivariant.)
\item En d\'eduire que $\lambda^{q^i}$ est valeur propre de $\rho(s)$ pour tout $i\ge 0$, puis que $\lambda$ est une racine de l'unit\'e.
\item En d\'eduire que $\lambda=1$. (Utiliser (g) et l'exercice \ref{exo5.2} (d).)
\item Conclure que $\rho(s)$ est bien unipotente.
\end{enumerate}
\end{exo}

\subsubsection{Les facteurs locaux archim\'ediens} Soit maintenant $v\in \Sigma_K^\infty$. Pour d\'efinir un facteur local $\Gamma_v(K,H^i(X),v)$, on utilise la \emph{d\'ecomposition de Hodge}
\[H^i(X_v(\C),\C)=\bigoplus_{p+q=i} H^{p,q}\]
o\`u $X_v=X\otimes_K K_v$.

Commen\c cons par le cas complexe:

\begin{defn} Supposons $K_v=\C$. Alors on pose
\[\Gamma_v(K,H^i(X),s)= \prod_{p,q} \Gamma_\C(s-\inf(p,q))^{h(p,q)}, \quad h(p,q)=\dim H^{p,q}.\]
\end{defn}

Le cas r\'eel est plus int\'eressant: la conjugaison complexe $\pi_v\in \Gal(\bar K_v/K_v)\allowbreak\simeq \Z/2$ op\`ere alors sur $X_v(\C)$ (on identifie $\bar K_v$ \`a $\C$), donc sur $H^i(X_v(\C),\C)$ en transformant $H^{p,q}$ en $H^{q,p}$.

\begin{defn} Supposons $K_v=\R$.\\
a) Si $i$ est impair, on pose
\[\Gamma_v(K,H^i(X),s)= \prod_{p<q} \Gamma_\C(s-p))^{h(p,q)}.\]
b) Si $i=2j$ est pair, on pose
\[\Gamma_v(K,H^i(X),s)= \prod_{p<q} \Gamma_\C(s-p))^{h(p,q)}\Gamma_\R(s-j)^{h(j,+)}\Gamma_\R(s-j+1)^{h(j,-)}\]
o\`u $h(j,\epsilon)$ est la dimension de l'espace propre de valeur propre $\epsilon$ pour l'action de $\pi_v$ sur $H^{j,j}$.
\end{defn}

\subsubsection{Le facteur exponentiel}\label{5.5.5} Sa d\'efinition fait intervenir le \emph{conducteur de $H^i(X)$}: c'est un diviseur effectif
\[\ff = \sum_{v\in \Sigma_K^f} f(v)v\]
o\`u $f(v)=\epsilon_v+ \delta_v$, ces entiers \'etant d\'efinis comme suit:
\[\epsilon_v = \dim H^i_l(X)- \dim H^i_l(X)^{I_v}.\]

Pour d\'efinir $\delta_v$, on remplace la repr\'esentation $H^i_l(X)$ de $G_v$ par sa \emph{semi-simplifi\'ee} $H^i_l(X)^{ss}$: le th\'eor\`eme de monodromie $l$-adique\index{Théorème!de monodromie $l$-adique} implique que l'action de $I_v$ sur $H^i_l(X)^{ss}$ se factorise par un quotient fini. On peut alors donner un sens \`a son \emph{conducteur de Swan} (une variante du conducteur d'Artin, \cite[p. 130]{raynaud}) ou \cite[\S 2.1]{serre-dpp}): $\delta_v$ est l'exposant de celui-ci.

On d\'efinit aussi le discriminant de $K$ par
\[D=\begin{cases}
|d_{K/\Q}| &\text{si $\car K=0$}\\
q^{2g-2} &\text{si $K=\F_q(C)$}
\end{cases}
\]
o\`u, dans le deuxi\`eme cas, $C$ est le mod\`ele projectif lisse de $K$, $g$ est son genre et $\F_q$ est son corps des constantes.

\begin{defn} $A=N(\ff)D^{B_i}$ o\`u $B_i=\dim_{\Q_l} H^i_l(X)$.
\end{defn}

\subsubsection{La fonction $L$ compl\'et\'ee}

\begin{defn}[Serre \protect{\cite{serre-dpp}}, \`a la notation pr\`es] On pose
\[L(K,H^i(X),s) = \prod_{v\in \Sigma_K^f} L_v(K,H^i(X),s)\]
\[\Lambda_\infty(K,H^i(X),s)=A^{s/2} \prod_{v\in \Sigma_K^\infty} \Gamma_v(K,H^i(X),s)\]
\[\Lambda(K,H^i(X),s)=\Lambda_\infty(K,H^i(X),s)L(K,H^i(X),s).\]
\end{defn}

\begin{rque} Le terme exponentiel $A^{s/2}$ se d\'ecompose en produit de facteurs locaux: T. Saito m'a sugg\'er\'e qu'il pourrait \^etre int\'eressant de les incorporer aux facteurs locaux $L_v(K,H^i(X),s)$ correspondants.
\end{rque}

\subsubsection{L'\'equation fonctionnelle}\index{Equation fonctionnelle@\'Equation fonctionnelle}

\begin{conj}[Serre, ibid.]\phantomsection\label{c5.3} La fonction $\Lambda(K,H^i(X),s)$ admet un prolongement m\'eromorphe \`a tout le plan complexe, et une \'equation fonctionnelle
\[\Lambda(K,H^i(X),s) = w\Lambda(K,H^i(X),i+1-s),\quad w=\pm1.\]
\end{conj}

Pour $X$ une courbe elliptique, cette conjecture (sous une forme non cohomologique) remonte \`a Weil \cite{weilma}.

\begin{thm}[Grothendieck, Deligne]\phantomsection\label{t5.5} La conjecture \ref{c5.3} est vraie si $\car K>0$. De plus, si $K=\F_q(C)$ comme ci-dessus, on a
\[\Lambda(K,H^i(X),s)= A^{s/2}\frac{P_1(q^{-s})}{P_0(q^{-s})P_2(q^{-s})} \]
avec $P_0,P_1,P_2\in \Z[t]$, o\`u les racines inverses de $P_j$ sont de valeur absolue complexe $q^{\frac{i+j}{2}}$.
\end{thm}

\begin{proof}  Le morphisme projectif lisse $X\to \Spec K$ s'\'etend en un morphisme projectif lisse $p:\sX\to U$ o\`u $U$ est un ouvert convenable de $C$; si $j_U:U\inj C$ est l'immersion correspondante, on a
\[j_*H^i_l(X)=(j_U)_*R^ip_*\Q_l\]
o\`u $j:\Spec K\to C$ est l'inclusion du point g\'en\'erique.

Ainsi, $j_*H^i_l(X)$ est un $\Q_l$-faisceau mixte sur $C$. Je dis qu'au terme exponentiel pr\`es, on a:
\[\Lambda(K,H^i(X),s) = L(C, j_*H^i_l(X),q^{-s}).\]

En effet, consid\'erons le facteur local du membre de droite en $x\in C_{(0)}$: 
\[L(\kappa(x),i_x^*j_*H^i_l(X),q^{-s})= L(\kappa(x),i_x^*(j_x)_*H^i_l(X),q^{-s})\]
o\`u $j_x$ est l'inclusion locale $\Spec K\to \Spec \sO_{C,x}$. Il r\'esulte de la description galoisienne de l'image directe par une telle immersion ouverte  (par exemple \cite[ch. II, ex. 3.15]{milne}) que 
\[(j_x)_*H^i_l(X)=H^i_l(X)^{I_x}\]
comme module sur le groupe de Galois absolu de $K$. Donc
\[L(\kappa(x),i_x^*(j_x)_*H^i_l(X),q^{-s})=L_x(K,H^i(X),s)\]
au sens de la d\'efinition \ref{d5.2}. 

D'apr\`es la proposition \ref{p5.1}, $R^ip_*\Q_l$ est lisse et pur de poids $i$. Le corollaire \ref{c5.5} implique que le faisceau $H^i_l(X)$ est ``faiblement polarisable de poids $i$''; quitte \`a restreindre $U$\footnote{Pour disposer d'une bonne polarisation relative \`a $p$.}, cet \'enonc\'e s'\'etend au faisceau $R^ip_*\Q_l$. 
En appliquant le th\'eor\`eme \ref{tgef2}, on trouve donc une \'equation fonctionnelle de la forme 
\begin{equation}\label{eq5.15}
L(C, j_*H^i_l(X),q^{i+1-s})= \pm(-q^{-s})^{\chi} q^{\frac{(i+1)\chi}{2}}L(C, j_*H^i_l(X),q^{s})  
\end{equation}
o\`u $\chi=\chi(C,j_*H^i_l(X))$, et la factorisation annonc\'ee (\emph{a priori}, \`a coefficients dans $\Q_l$). 

De plus, le th\'eor\`eme \ref{t5.4} implique que les racines inverses de $P_j$ sont des nombres de Weil de poids $i+j$. D'autre part, le th\'eor\`eme \ref{t5.3} implique que $L(C,j_*H^i_l(X),t)\in \Z[[t]]$: c'est donc une fraction rationnelle \`a coefficients rationnels  (\cf preuve du th\'eor\`eme \ref{t3.6}).   Le m\^eme raisonnement que dans la preuve du lemme \ref{l3.4} montre alors que les $P_j$ sont premiers entre eux deux \`a deux, \`a coefficients entiers et ind\'ependants du choix de $l$.

Il reste \`a d\'eduire de \eqref{eq5.15} l'\'equation fonctionnelle donn\'ee dans l'\'enonc\'e de la conjecture \ref{c5.3}. Cela r\'esulte de la \emph{formule de Grothendieck-Ogg-\v Safarevi\v c} \cite[th. 1]{raynaud}:
\[\chi(C,j_*H^i_l(X))=(2-2g)B_i - \deg \ff\]
(\cf \S\ref{5.5.5} pour $\ff$). Un anc\^etre de cette formule est d\^u (comme toujours\dots) \`a Weil: \cite[ch. VI, \S 4]{serre-cl}.

(Cet argument n'est pas tout \`a fait complet: la formule de Gro\-then\-dieck-Ogg-\v Savarevi\v c est  valable pour les $\Z/l^n$-faisceaux constructibles; il faut donc l'appliquer \`a $j_*H^i_\et(\bar X,\Z/l^n)$  pour tout $n\ge 1$, puis passer \`a la limite.) 
\end{proof}

Pour la constante de l'équation fonctionnelle, je renvoie à l'exposé de Deligne \cite{deligne-constante} qui repose essentiellement sur le théorème \ref{tld} et la conjecture de monodromie-poids du \S \ref{s5.5.4}.

\section{Motifs}

R\'ef\'erences: Kleiman \cite{kleiman-motives}, Scholl \cite{scholl}, Andr\'e \cite{andre}.

On se donne un corps de base $k$.

\subsection{La probl\'ematique} Citons Serre \cite{serre-motifs}:

\begin{quote} \it
La situation d\'ecrite ci-dessus n'est pas tout \`a fait satisfaisante. On dispose de trop de groupes de cohomologie qui ne sont pas suffisamment li\'es entre eux -- malgr\'e les isomorphismes de compatibilit\'e. Par exemple, si $X$ et $Y$ sont deux vari\'et\'es (projectives, lisses), et $f: H^i_\et(X,\Q_\ell)\to H^i_\et(Y,\Q_\ell)$ une application $\Q_\ell$-lin\'eaire, où $\ell$ est un nombre premier fix\'e, il n'est pas possible en g\'en\'eral de d\'eduire de $f$ une application analogue pour la cohomologie $\ell'$-adique, o\`u $\ell'$ est un autre nombre premier. Pourtant, on a le sentiment que c'est possible pour certains $f$, ceux qui sont ``motiv\'es'' (par exemple ceux qui proviennent d'un morphisme de $Y$ dans $X$, ou plus g\'en\'eralement d'une correspondance alg\'ebrique entre $X$ et $Y$). Encore faut-il savoir ce que ``motiv\'e'' veut dire!
\end{quote}

Les isomorphismes de compatibilit\'e dont parle Serre sont ceux du \S \ref{3.6.1} ainsi qu'un isomorphisme analogue entre cohomologie de Betti\index{Cohomologie!de Betti} et cohomologie de de Rham\index{Cohomologie!de de Rham}: ils n'existent qu'en caract\'eristique z\'ero. En caract\'eristique $p$ la situation est plus myst\'erieuse: on n'a pas d'isomorphismes de comparaison liant les diff\'erentes cohomologies $l$-adiques. Pourtant, le polyn\^ome caract\'eristique de l'action de Frobenius sur la cohomologie $l$-adique\index{Cohomologie!$l$-adique} est \`a coefficients entiers et ind\'ependant de $l$\dots

On pourrait imaginer l'existence d'une cohomologie de Weil initiale, \`a coefficients dans $\Q$ et s'envoyant vers toutes les cohomologies $l$-adiques: c'est impossible par l'observation de Serre d\'ecrite au \S \ref{3.5.9}. Alors?

L'id\'ee de Grothendieck est de trouver une cohomologie de Weil initiale, non pas \`a valeurs dans une cat\'egorie d'espaces vectoriels sur un corps puisque c'est impossible, mais dans une cat\'egorie ab\'elienne convenable. Cette cat\'egorie serait la \emph{cat\'egorie des motifs} et la cohomologie de Weil universelle la \emph{cohomologie motivique}.\footnote{Cette terminologie a \'et\'e adopt\'ee par Lichtenbaum, puis Friedlander, Suslin et Voevodsky, dans un sens tr\`es diff\'erent (groupes Hom dans la cat\'egorie des motifs triangul\'es): on prendra garde \`a ne pas confondre les deux notions.}

L'id\'ee de Grothendieck pour donner un sens \`a ``motiv\'e'' est, comme le sugg\`ere Serre, de partir des correspondances alg\'ebriques, point commun entre les diverses cohomologies de Weil par l'interm\'ediaire des applications classes de cycles. Ceci d\'efinit d\'ej\`a une cat\'egorie additive dont les objets sont les vari\'et\'es projectives lisses; en raffinant cette derni\`ere par des proc\'ed\'es purement cat\'egoriques, on arrive \`a une cat\'egorie mono\"\i dale sym\'etrique \emph{rigide} (th\'eor\`eme \ref{t6.5}): la cat\'egorie des motifs purs $\sM_\sim(k,F)$.

Implicitement, on a fait le choix d'un corps $F$ de coefficients et d'une relation d'\'equivalence ad\'equate $\sim$ (\S \ref{s6.2}). Si on choisit pour $\sim$ l'\'equivalence num\'erique, la cat\'egorie $\sM_\sim(k,F)$ est \emph{ab\'elienne semi-simple} (th\'eor\`eme \ref{tjannsen}). Si $k$ est fini, il r\'esulte de la derni\`ere conjecture de Weil\index{Conjectures!de Weil} (Deligne) que $\sM_\sim(k,F)$ admet une \emph{graduation par le poids} (th\'eor\`eme \ref{t6.3} et proposition \ref{p6.4}) et une partie du programme de Grothendieck est r\'ealis\'ee.  En particulier, tout motif simple a un poids bien d\'etermin\'e $i$, et sa fonction z\^eta\index{Fonction zêta!d'un motif (pur)} est un polyn\^ome ou l'inverse d'un polyn\^ome suivant la parit\'e de $i$, ses racines inverses \'etant des nombres de Weil de poids $i$  (th\'eor\`eme \ref{t6.4}). Comme cas particulier, on retrouve le th\'eor\`eme de Weil sur la conjecture d'Artin\index{Conjecture!d'Artin} en caract\'eristique $p$ (\S \ref{cor-loco}).

Malheureusement, des problèmes ouverts limitent la portée de la définition de Grothendieck: les \emph{conjectures standard}, déjà mentionnées au \S\ref{3.6.1}. Par exemple, on ne sait pas si les projecteurs qui définissent la graduation d'une cohomologie\index{Cohomologie!de Weil} de Weil (``composantes de K\"unneth de la diagonale'') sont toujours représentés par des cycles algébriques. Pour approcher les fonctions $L$ sur les corps de nombres, Deligne a alors introduit une variante de la cat\'egorie des motifs de Grothendieck: les Hom de cette cat\'egorie ne sont pas form\'es de cycles alg\'ebriques, mais de ``cycles de Hodge absolus'' \cite{deligneL,deligne900}, et les conjectures standard y sont vraies grâce au théorème de Lefschetz difficile et à la polarisabilité des structures de Hodge portées par la cohomologie de Betti\index{Cohomologie!de Betti} des vatiétés projectives lisses complexes. Ceci a été poussé plus loin par Yves André \cite{andreihes}, chez qui les cycles de Hodge absolus sont remplacés par des ``cycles motivés''. Je n'aborderai pas ces th\'eories, et encore moins les aspects les plus excitants de la th\'eorie des motifs qui d\'epassent l'enjeu des conjectures de Weil: th\'eorie tannakienne et groupe de Galois motivique \cite{saa,deligne900,andreihes}.

\subsection{Relations d'\'equivalence ad\'equates} \label{s6.2}

\begin{defn} Soit $\V(k)$ la cat\'egorie des $k$-vari\'et\'es projectives lisses, et soit $F$ un anneau commutatif. Une relation d'\'equivalence $\sim$ sur les cycles alg\'ebriques \`a coefficients dans $F$ est \emph{ad\'equate} si elle est $F$-lin\'eaire et si:
\begin{thlist}
\item {\bf Lemme de d\'eplacement.} Soient $\alpha,\beta\in Z^*(X,F)=Z^*(X)\otimes_\Z F$ pour $X\in \V(k)$. Alors il existe $\beta'\sim \beta$ tel que $\alpha$ et $\beta'$ se coupent proprement (c'est-\`a-dire avec les bonnes codimensions).
\item Soient $X,Y\in \V(k)$, et soient $\alpha\in Z^*(X)$, $\beta\in Z^*(X\times Y)$ tels que $\alpha\times Y$ et $\beta$ s'intersectent proprement. Alors $a\sim 0$ implique $(p_Y)_*(\beta\cdot (\alpha\times Y))\sim 0$.
\end{thlist}
On dira que $(F,\sim)$ est un \emph{couple adéquat}.
\end{defn}

\'Etant donn\'e un couple ad\'equat $(F,\sim)$ et $X\in \V(k)$, on note $A^*_\sim(X,F)=Z^*(X,F)/\sim$ et  $A_*^\sim(X,F)=Z_*(X,F)/\sim$: on peut faire de la th\'eorie de l'intersection sur $A^*_\sim(X,F)$.

\begin{exs} Voici les principales relations d'\'equivalences ad\'equates:
\[\sim_\rat\ge \sim_\alg\ge \sim_\tnil\ge \sim_H\ge \sim_\num\]
respectivement, l'\'equivalence rationnelle, l'\'equivalence alg\'ebrique, l'\'equivalence de smash-nilpotence de Voevodsky, l'\'equivalence homologique relative \`a une cohomologie de Weil $H$ (lorsque $F=\Q$), et l'\'equivalence num\'erique. Le signe $\ge$ indique la relation d'implication entre ces \'equivalences. L'\'equivalence rationnelle est la plus fine des \'equivalences ad\'equates, et l'\'equivalence num\'erique est la moins fine si $F$ est un corps. Pour plus de d\'etails, on pourra consulter Fulton \cite[ch. 19]{fulton} et Andr\'e \cite[ch. 3]{andre}. Extrayons seulement, pour plus tard, de la derni\`ere r\'ef\'erence:
\end{exs}

\begin{lemme}\phantomsection\label{l6.1} $\sim_\rat$ est la plus fine des \'equivalences ad\'equates.
\end{lemme}

\begin{proof}[Démonstration (d'apr\`es \protect{\cite[3.2.2.1]{andre}})]. Soit $\sim$ une \'equivalence ad\'equate pour les cycles alg\'ebriques \`a coefficients
dans $F$. La condition (ii) nous ram\`ene \`a prouver que $[0] \sim [\infty]$ sur $\P^1$. Par la condition (i), il existe un cycle $\sum n_i[x_i] \sim [1]$, $n_i \in F$, tel que $\sum n_i[x_i]\cdot [1]$ soit bien d\'efini, \ie $x_i \ne 1$. Appliquons (ii) en prenant pour $\beta$ le graphe de la fraction rationnelle $1-\prod \left(\frac{x-x_i}{1-x_i}\right)^{m_i}$ ($m_i >0$) et $\alpha=\sum n_i[x_i]-[1]$: on obtient que $mn[1]\sim m[0]$, o\`u $m=\sum m_i$, $n=
\sum n_i$; comme les $m_i$ sont arbitraires, on conclut que $n[1] \sim [0]$. En appliquant
l'automorphisme $x\mapsto 1/x$ (et de nouveau la condition (ii)), on obtient $n[1]\sim[\infty]$, d'o\`u $[0] \sim [\infty]$.
\end{proof}

On a aussi:

\begin{thm}\phantomsection\label{t6.9} Les groupes de cycles (\`a coefficients entiers) modulo l'\'equivalence num\'erique sont des $\Z$-modules libres de type fini.
\end{thm}

\begin{proof}[Démonstration (d'après \protect{\cite[3.4.6]{andre}})] Soit $H$ une cohomologie de Weil, de corps de coefficients $K$. Pour $X$ projective lisse de dimension $d$ et pour $r\ge 0$, plongeons $A^{d-r}_H(X)$ dans le $K$-espace vectoriel de dimension finie $H^{2(d-r)}(X)(d-r)$ via l'application classe de cycle $\cl$. Soit $(\beta_1,\dots,\beta_n)$ un système maximal d'éléments de $A^{d-r}_H(X)$ tel que $\cl(\beta_1),\dots,\cl(\beta_n)$ soient $K$-linéairement indépendants. Alors l'homomorphisme
\[\alpha\mapsto (\deg(\alpha\beta_1),\dots,\deg(\alpha\beta_n))\]
de $A^r_\num(X)$ dans $\Z^n$ est \emph{injectif}: en effet, si $\beta\in A^{d-r}_H(X)$, $\cl(\beta)$ est combinaison linéaire des $\cl(\beta_i)$, donc $\deg(\alpha\beta)=0$ par la compatibilité entre le produit d'intersection et l'accouplement de Poincaré (définition \ref{d3.4} (viii)).  Ceci conclut la démonstration puisque $\Z$ est noethérien.
\end{proof}

\subsection{Cat\'egorie des correspondances}

\begin{defn} Soit $\sim$ une relation d'\'equivalence ad\'equate. On d\'efinit la cat\'egorie $\Corr_\sim(k,F)$:
\begin{description}
\item[Objets] $[X]$, $X\in \V(k)$.
\item[Morphismes] $\Corr_\sim(k,F)([X],[Y])=\Corr_\sim^0(X,Y)$, o\`u $\Corr_\sim^0$ est le groupe des correspondances modulo $\sim$, \cf d\'efinition \ref{d3.3} dans le cas de $\sim_\rat$.
\end{description}
La composition des correspondances se d\'efinit comme dans \loccit
\end{defn}

\begin{defn} On a un foncteur contravariant de $\V(k)$ vers $\Corr_\sim(k,F)$:
\begin{align*}
\V(k)^\op&\to \Corr_\sim(k,F)\\
X&\mapsto [X]\\
f&\mapsto {}^t [\Gamma_f].
\end{align*}
\end{defn}

La cat\'egorie $\V(k)$ est mono\"\i dale sym\'etrique unitaire (d\'efinition \ref{dA.5}) pour le produit des vari\'et\'es (unit\'e: $\Spec k$), et

\begin{lemme} Le foncteur $[\ ]$ commute aux coproduits finis: la cat\'egorie $\Corr_\sim(k,F)$ est $F$-lin\'eaire. De plus, il existe une structure mono\"\i dale sym\'etrique unitaire $F$-lin\'eaire canonique sur $\Corr_\sim(k,F)$ telle que le foncteur ci-dessus soit mono\"\i dal sym\'etrique.
\end{lemme}

\begin{proof}  Par construction, $\Corr_\sim(k,F)$ est une $F$-cat\'egorie (ses $\Hom$ sont des $F$-modules et la composition est $F$-bilin\'eaire); si $X,Y\in \V(k)$, on v\'erifie tout de suite que $[X]\oplus [Y]$ est repr\'esent\'e par $[X\coprod Y]$ et que $[\emptyset]$ repr\'esente l'objet nul. 
Pour la seconde affirmation, on d\'efinit $[X]\otimes [Y]=[X\times Y]$, et le produit tensoriel des correspondances est donn\'e par le produit des cycles:
\[\Corr_\sim(X,Y)\otimes \Corr_\sim(X',Y')\to \Corr_\sim(X\times X',Y\times Y').\]

La v\'erification des axiomes se fait sans difficult\'e.
\end{proof}

\subsection{Motifs purs effectifs}

\begin{defn} On garde les notations ci-dessus. La cat\'egorie $\sM_\sim^\eff(k,F)$ des \emph{motifs effectifs} (\`a coefficients dans $F$, relativement \`a $\sim$) est l'enveloppe karoubienne de $\Corr_\sim(k,F)$ (d\'efinition \ref{dA.4}). On abr\`ege cette notation en $\sM_\sim^\eff(k)$ ou $\sM_\sim^\eff$ quand il n'y a pas de risque d'ambigu\"\i t\'e.
\end{defn}

D'apr\`es le lemme \ref{lA.2} et la proposition \ref{pA.2}, la structure mono\"\i dale sym\'etrique $F$-lin\'eaire de $\Corr_\sim(k,F)$ s'\'etend canoniquement \`a $\sM^\eff_\sim(k,F)$. On note
\[[X]\mapsto h(X)\]
le foncteur canonique $\Corr_\sim(k,F)\to \sM_\sim^\eff(k,F)$. S'il est besoin de pr\'eciser, on note $h_\sim$ au lieu de $h$.

D'apr\`es la d\'emonstration du lemme \ref{l6.1}, $[0]=[\infty]$ dans $A_0^\sim(\P^1,F)$; en appliquant une homographie convenable $z\mapsto \frac{az+b}{cz+d}$, on en d\'eduit que $[x]=[y]$ dans $A_0^\sim(\P^1,F)$ pour tous $x,y\in \P^1(k)$. Ainsi, la classe du projecteur
\[p:\P^1\to \Spec k\by{i_x}\P^1\]
dans $\Corr^0_\sim(\P^1,F)$ ne d\'epend pas du choix de $x\in \P^1(k)$.

\begin{defn} Le \emph{motif de Lefschetz} $\L$ est l'image du projecteur $1-p$. On a donc
\[h(\P^1)=\un\oplus \L.\]
Pour tout motif effectif $M\in \sM_\sim^\eff(k,F)$ et tout entier $n\ge 0$, on note (sic)
\[M(-n)=M\otimes \L^{\otimes n}.\] 
\end{defn}

\begin{prop}\phantomsection\label{p6.1} $\L$ est quasi-inversible (d\'efinition \ref{dA.7}).
\end{prop}

\begin{proof} Pour tout $X\in \V(k)$, les projecteurs $p$ et $1-p$ induisent la \emph{formule de la droite projective}
\begin{equation}\label{edp}
A^n_\sim(X\times\P^1,F)\simeq A^n_\sim(X,F)\oplus A^{n-1}_\sim(X,F)
\end{equation}
(se r\'eduire au cas de l'\'equivalence rationnelle, \cf lemme \ref{l6.1}, puis appliquer \cite[th. 3.3]{fulton}). En it\'erant une fois et en suivant le sens des fl\`eches, on voit facilement que ceci fournit un isomorphisme
\[\sM_\sim^\eff(h(X),h(Y))\iso \sM_\sim^\eff(h(X)(-1),h(Y)(-1))\]
induit par $-\otimes \L$.
\end{proof}

On a aussi:

\begin{prop}\phantomsection\label{p6.2} Pour tout $X\in \V(k)$ et tout $n\ge 0$, on a des isomorphismes naturels
\[A^n_\sim(X,F)\simeq \sM_\sim^\eff(\L^n,h(X)),\quad A_n^\sim(X,F)\simeq \sM_\sim^\eff(h(X),\L^n).\]
De plus, la composition
\[ \sM_\sim^\eff(\L^n,h(X))\times \sM_\sim^\eff(h(X),\L^n)\to \sM_\sim^\eff(\L^n,\L^n)\osi F\]
s'identifie au produit d'intersection
\[A^n_\sim(X,F)\times A_n^\sim(X,F)\to F.\]
\end{prop}

\begin{proof} Cela r\'esulte encore de la formule \eqref{edp}, par it\'eration.
\end{proof}

\subsection{Motifs purs}

\begin{defn} On note 
\[\sM_\sim(k,F)=\sM_\sim^\eff(k,F)[\L^{-1}]\]
\cf d\'efinition \ref{dA.8}.
\end{defn}

Compte tenu de la proposition \ref{p6.1}, le th\'eor\`eme \ref{tA.3} donne:

\begin{thm} Le foncteur canonique $\sM_\sim^\eff(k,F)\to \sM_\sim(k,F)$ est pleinement fid\`ele; la structure mono\"\i dale sym\'etrique de $\sM_\sim^\eff(k,F)$ s'\'etend \`a $\sM_\sim(k,F)$.\qed
\end{thm}

\begin{defn} On note $\T=\L^{-1}$ un quasi-inverse de $\L$: c'est le \emph{motif de Tate}. Pour $M\in \sM_\sim(k,F)$ et $n\in \Z$, on note
\[M(n)=
\begin{cases}
M\otimes \L^{\otimes -n}&\text{si $n\le 0$}\\
M\otimes \T^{\otimes n}&\text{si $n\ge 0$.}
\end{cases}\]
\end{defn}

Pour r\'esumer, on a construit une cha\^\i ne de cat\'egories et foncteurs mono\"\i daux sym\'etriques
\[\begin{CD}
\V(k)^\op@>>> \Corr_\sim(k,F)@>\natural>> \sM_\sim^\eff(k,F)@>\L^{-1}>> \sM_\sim(k,F)\\
X&\mapsto& [X]&\mapsto& h(X)&\mapsto & h(X)\\
f&\mapsto &[{}^t\Gamma_f]
\end{CD}\]
les deux derniers \'etant pleinement fid\`eles.

\begin{rque} Ce qui pr\'ed\`ede est la construction classique des motifs purs, comme la pr\'esentait Grothendieck. Dans \cite{jannsen}, Uwe Jannsen a introduit une construction directe de $\sM_\sim(k,F)$ comme la cat\'egorie des triplets $(X,p,n)$ o\`u $X\in \V(k)$, $p$ est un projecteur sur $X$ et $n\in \Z$. Cette description est reprise par exemple dans \cite{scholl} et \cite{andre}. Je pr\'ef\`ere suivre la construction initiale de Grothendieck, donn\'ee dans \cite{kleiman-motives}, pour deux raisons:
\begin{itemize}
\item elle met mieux en relief les enjeux cat\'egoriques;
\item elle est le mod\`ele de la construction ult\'erieure des motifs triangul\'es par Voevodsky \cite{voetri};
\item l'autre construction est id\'eale pour la structure tensorielle, mais moins pour la structure additive:  il est difficile de d\'ecrire $(X,p,m)\oplus (Y,q,n)$ si $m\ne n$. Le recours \`a la formule de la droite projective \eqref{edp} est n\'ecessaire et plus artificiel, \cf \cite[1.14]{scholl}.
\end{itemize}
\end{rque}

\subsection{Rigidit\'e}

\begin{thm}\phantomsection\label{t6.5} La cat\'egorie $\sM_\sim(k,F)$ est rigide (d\'efinition \break \ref{dDP}).
\end{thm}

\begin{proof} Soit $\sM'$ la sous-cat\'egorie pleine de $\sM_\sim(k,F)$ dont les objets sont les $h(X)(n)$ pour $X\in \V(k)$ \'equidimensionnel et $n\in \Z$. Son enveloppe karoubienne est $\sM_\sim(k,F)$. D'apr\`es les propositions \ref{pA.4} c) et \ref{pA.3}, il suffit de monter que $h(X)$ est fortement dualisable pour $X$ \'equidimensionnel.

Soit $n=\dim X$. Je dis que
\[h(X)^*=h(X)(n)\]
d\'efinit un dual fort de $X$. Pour cela, il faut d\'ecrire les morphismes $\eta$ et $\epsilon$ de la d\'efinition \ref{dDP}. \`A un twist \`a la Tate pr\`es, il s'agit de morphismes
\[\L^n\by{\eta}h(X)\otimes h(X), \quad h(X)\otimes h(X)\by{\epsilon} \L^n.\] 

En utilisant l'identit\'e $h(X)\otimes h(X)=h(X\times X)$ et la proposition \ref{p6.2}, ces morphismes correspondent \`a des classes de
\[A^n_\sim(X\times X,F)=A_n^\sim(X\times X,F).\]

Dans les deux cas, on prend la classe de la diagonale $[\Delta_X]$. La v\'erification des identit\'es de la d\'efinition \ref{dDP} sont laiss\'ees en exercice.
\end{proof}

\subsection{Le th\'eor\`eme de Jannsen} 

\begin{thm}[Jannsen, 1991]\phantomsection\label{tjannsen} Si $F$ est un corps de caract\'eristique z\'ero, la cat\'egorie $\sM_\num(k,F)$ est ab\'elienne semi-simple.
\end{thm}

Avant la perc\'ee de Jannsen, on savait que cet \'enonc\'e r\'esultait des conjectures standard de Grothendieck.\index{Conjectures!standard} Le fait qu'on puisse le d\'emontrer inconditionnellement a \'et\'e une surprise g\'en\'erale. Je renvoie \`a l'article de Jannsen \cite{jannsen} pour sa d\'emonstration originelle, et \`a \cite{ak2} pour une version abstraite de cette d\'emonstration.

Celle-ci repose sur l'existence d'une cohomologie de Weil; son principe est tr\`es proche de celui de la d\'emonstration du théorème \ref{t6.9}.

\subsection{Spécialisation} Soit $K$ un corps muni d'une valuation discrète $v$ de rang $1$, d'anneau de valuation $\sO$ et de corps résiduel $k$. On aimerait définir un foncteur de spécialisation
\[\sM_\sim(K,F)\to \sM_\sim(k,F)\]
au moins pour certains couples adéquats $(F,\sim)$. J'ignore si c'est possible dans cette généralité, mais c'est au moins le cas si on se restreint à une sous-catégorie pleine de $\sM_\sim(K,F)$:

\begin{defn} On note $\sM_\sim(K,v,F)$ la sous-catégorie épaisse de $\sM_\sim(K,F)$ engendrée par les motifs de la forme $h(X)$, où $X$ est une $K$-variété projective lisse qui a bonne réduction relativement à $v$ (cf. définition \ref{d5.3}). Les objets de $\sM_\sim(K,v,F)$ sont les \emph{motifs à bonne réduction} (relativement à $v$).
\end{defn}

\begin{rques}\phantomsection\label{r6.2} a) $\P^1$ a bonne réduction; si $X,X'$ ont bonne réduction, $X\times X'$ a bonne réduction. Cela  implique que la sous-catégorie $\sM_\sim(K,v,F)\subset \sM_\sim(K,F)$ est \emph{rigide}.\\
b) Supposons $\car k\ne 2$ et soit $C$ la conique projective anisotrope d'équation $aX_0^2+\pi X_1^2 -X_2^2=0$, où $a\in K^*- K^{*2}$ et $\pi$ est une uniformisante. Alors $C$ n'a pas bonne réduction, mais $h(C)\simeq h(\P^1)$ dans $\sM_\sim(K,F)$ dès que $2$ est inversible dans $F$. (Cet exemple marche aussi en caractéristique résiduelle $2$, avec une modification convenable.) Ainsi, le motif d'une $K$-variété projective lisse $X$ peut être à bonne réduction même si $X$ n'est pas à bonne réduction.
\end{rques}

\begin{thm}\phantomsection\label{t6.8} Il existe un unique $\otimes$-foncteur 
\[sp:\sM_\sim(K,v,F)\to \sM_\sim(k,F)\] 
envoyant le motif d'une $K$-variété projective lisse $X$ à bonne réduction sur celui de sa fibre spéciale (relative à un modèle lisse quelconque), dans les cas suivants:
\begin{enumerate}
\item $\sim$ est l'équivalence rationnelle, $F$ est quelconque.
\item $\sim$ est l'équivalence algébrique.
\item $\sim$ est l'équivalence de smash-nilpotence, $F$ est quelconque.
\item $\sim$ est l'équivalence homologique $l$-adique\index{Cohomologie!$l$-adique} où le nombre premier $l$ est inversible dans $k$, $F\supset \Q$.
\item $\car k=0$, $\sim$ est une cohomologie\index{Cohomologie!de Weil!classique} de Weil classique, $F\supset \Q$.
\end{enumerate}
\end{thm}

\begin{proof} Si le foncteur existe, son unicité est évidente ainsi que sa tensorialité.

Le cas 4 résulte du théorème de changement de base propre et lisse (cf. \S \ref{3.6.1}), et le cas 5 s'en déduit par les théorèmes de comparaison. Le cas 3 se déduit immédiatement du cas 1. Il reste à traiter les cas 1 et 2.

J'ai appris l'argument qui suit de Yves André. On utilise les homomorphismes de spécialisation \cite[ex. 20.3.1]{fulton}
\[\sigma:A_\sim^i(X,F)\to A_\sim^i(Y,F)\quad (\sim=\sim_\rat, \sim_\alg) \]
où $X/K$ a bonne réduction et $Y$ est la fibre spéciale d'un modèle lisse $\sX$ de $X$. 

Précisons: dans \emph{loc. cit.}, Fulton ne traite que le cas de l'équivalence rationnelle. Mais sa construction repose sur les suites exactes \cite[prop. 1.8]{fulton}
\begin{equation}\label{eq6.3}
A_\sim^{n-1}(Y,F)\by{i_*} A_\sim^n(\sX,F)\by{j^*} A_\sim^n(X,F)\to 0 
\end{equation}
(où $j:X\to \sX$, $i:Y\to \sX$ sont les immersions ouverte et fermée correspondant à la situation), et l'identité 
\begin{equation}\label{eq6.4}
i^*i_*=0
\end{equation}
déduite de \cite[cor. 6.3]{fulton}.  Or la suite exacte \eqref{eq6.1} vaut également pour l'équivalence algébrique \cite[ex. 10.3.4]{fulton}, ainsi évidemment que l'identité \eqref{eq6.4}.

Si $X'$ est une autre variété à bonne réduction, munie d'un modèle lisse $\sX'$ de fibre spéciale $Y'$, $\sX\otimes_\sO\sX'$ est un modèle lisse de $X\times X'$, de fibre spéciale $Y\times Y'$, et  $\sigma$ induit un homomorphisme
\[sp:\Corr_\sim(K,F)(X,X')\to \Corr_\sim(k,F)(Y,Y').\]

Si $X''$ est une troisième variété à bonne réduction, munie d'un modèle lisse $\sX''$ de fibre spéciale $Y''$, et si \[\alpha\in \Corr_\sim(K,F)(X,X'),\quad \beta\in \Corr_\sim(K,F)(X',X'')\] 
sont deux correspondances, les compatibilités démontrées dans \cite[20.3]{fulton} impliquent que
\[sp(\beta\circ \alpha) = sp(\beta)\circ sp(\alpha).\]

En particulier, prenons $X=X'=X''$, $\alpha=\beta=\Delta_X$, $\sX''=\sX$: on obtient que $sp(\alpha)$ et $sp(\beta)$ sont des isomorphismes inverses  de $h(Y)$ sur $h(Y')$. En faisant varier les modèles lisses de $X$, on obtient un \emph{système transitif d'isomorphismes} entre les motifs $h(Y)$, où $Y$ décrit les fibres spéciales correspondantes. Ceci montre que le motif $h(Y)$ est bien déterminé à isomorphisme unique près, et que $sp$ définit bien un foncteur; ce foncteur est clairement monoîdal, en réutilisant la compatibilité de $\sigma$ au produit \cite[20.3]{fulton}.

La construction de $sp$ s'étend maintenant automatiquement aux motifs effectifs, puis à tous les motifs. 
\end{proof}

\begin{rque} Le cas de l'équivalence numérique n'est pas clair; il résulterait des conjectures standard.\index{Conjectures!standard} On pourra consulter \cite[3]{aknote} pour une manière de contourner ce problème.
\end{rque}

\subsection{Th\'eorie motivique des poids (cas pur)}

On suppose que $F$ est un corps de caract\'eristique z\'ero, et on se donne une cohomologie\index{Cohomologie!de Weil} de Weil $H^*$ \`a coefficients dans $K\supset F$.

\begin{defn} Soit $X\in \V(k)$, et soit $i\in\Z$. Le \emph{$i$-\`eme projecteur de K\"unneth de $X$} (relativement \`a $H$) est la projection $p^i$ de $H^*(X)=\bigoplus_j H^j(X)$ sur $H^i(X)$. C'est un endomorphisme idempotent de $H^*(X)$.
\end{defn}

Remarquons que cette d\'efinition s'\'etend \`a tout motif $M\in \sM_H(k,F)$: on peut parler des \emph{projecteurs de K\"unneth de $M$}.

\begin{defn}\phantomsection\label{d6.2} Soit $M\in \sM_H(k,F)$, et soit $i\in \Z$.\\
a)  On dit que $M$ est \emph{purement de poids $i$} si $H^j(M)=0$ pour $j\ne i$. On abrège cette définition par la notation $w(M)=i$. \\
b) On dit que \emph{le $i$-\`eme projecteur de K\"unneth de $M$ est alg\'ebrique} si on a 
\[p^i=H(p^i_{alg})\quad \text{pour } p^i_{alg}\in \End(M).\]
On abr\'egera cette phrase en: ``$p^i_M$ est alg\'ebrique''. 
 Si $M=h(X)$ pour $X\in\V(k)$, on \'ecrit $p^i_X$ pour $p^i_M$.\\
\end{defn}

Si cette condition est v\'erifi\'ee, le repr\'esentant $p^i_{alg}$ est unique dans $\End_{\sM_H}(M)$; on peut donc le noter $p^i_M$ sans ambigu\"\i t\'e. Notons aussi $M^{(i)}=\IM p^i_M$: c'est la \emph{partie de poids $i$} de $M$. Si $M=h(X)$, on note $h^i(X)$ pour $M^{(i)}$.

Les conjectures standard de Grothendieck impliquent que, pour tout $M$ et pour tout $i$, le $i$-\`eme projecteur de K\"unneth de $M$ est alg\'ebrique. Indiquons maintenant ce qui est connu inconditionnellement, \ie sans supposer ces conjectures vraies:\index{Conjectures!standard}

\begin{lemme}\phantomsection\label{l6.2} Soient $M,N\in \sM_H(k,F)$ et soient $i,j\in \Z$.\\
a) Si $p^i_M$ et $p^i_N$ sont alg\'ebriques,  $p^i_{M\oplus N}$ est alg\'ebrique.\\
b) Si $p^i_M$ est alg\'ebrique, $p^i_{M'}$ est alg\'ebrique pour tout facteur direct $M'$ de $M$.\\
c) Si les $p^*_M$ et les $p^*_N$ sont alg\'ebriques, les $p^*_{M\otimes N}$ sont alg\'ebriques.\\
d) Si $p^i_M$ est alg\'ebrique, $p^{-i}_{M^*}$ est alg\'ebrique.
\end{lemme}

\begin{proof} a), c) et d) r\'esultent des identit\'es (dans $\Vec_K^*$)
\[p^i_{M\oplus N} = p^i_M\oplus p^i_N,\quad p^k_{M\otimes N}=\sum_{i+j=k}p^i_M\otimes p^j_N, \quad p^{-i}_{M^*}={}^tp^i_M.\]

Pour b), soit $p\in \End(M)$ un projecteur d'image $M'$. On remarque que $p^i_M$ est \emph{central} dans $\End(H^*(M))$, donc \emph{a fortiori} dans son sous-anneau $\End(M)$. Donc $p^i_M$ commute \`a $p$ et $pp^i_M$ est encore un projecteur: on voit tout de suite que c'est $p^i_{M'}$.
\end{proof}

\begin{thm}\phantomsection\label{t6.2} Soit $H$ une cohomologie\index{Cohomologie!de Weil!classique} de Weil vérifiant l'axiome de la remarque \ref{r3.4} (par exemple une cohomologie de Weil classique, \S \ref{3.5.9}).\\
a) Pour tout $X\in \V(k)$, purement de dimension $n$, $p^i_X$ est alg\'ebrique pour $i=0,1,2n-1,2n$. On a $h^{2n}(X)\simeq \L^n$ si $X$ est géométriquement connexe.\\
b) Si $n=2$, tous les $p^i_X$ sont alg\'ebriques.\\
c) Si $A$ est une vari\'et\'e ab\'elienne, tous les $p^i_A$ sont alg\'ebriques.
\end{thm}

\begin{proof} a)   Notons $k_X$ le corps des constantes de $X$, c'est-\`a-dire la fermeture alg\'ebrique de $k$ dans $k(X)$. Comme $X$ est lisse, $k_X/k$ est s\'eparable, donc $h(\Spec k_X)$ est d\'efini. Je dis que 
\begin{align*}
h^0_X&=h(k_X)\\
 &= \un \quad \text{si } k_X=k.
 \end{align*}

Notons d'abord que $H^0(k_X)\iso H^0(X)$ (axiome (v) d'une cohomologie de Weil). Consid\'erons maintenant un point ferm\'e $i_x:x\inj X$, de corps r\'esiduel $E$. Soit $f:\Spec E\to \Spec k_X$ le morphisme canonique. Alors la composition
\[h(k_X)\by{{p'}^*}h(X)\by{i_x^*} h(x)\by{f_*}h(k_X)\]
est \'egale \`a $\deg f\cdot 1_{h(k_X)}$ gr\^ace \`a la formule du corollaire \ref{c:proj} (qui est valable en toute g\'en\'eralit\'e). Donc  $\frac{1}{\deg f}{p'}^*f_*i_x^*$ est le projecteur $p^0_X$ cherch\'e. Par dualit\'e, on obtient l'alg\'ebricit\'e de $p^{2n}_X$ et la valeur de $h^{2n}(X)$.

Le cas de $p^1_X$ et $p^{2n-1}_X$ est beaucoup plus d\'elicat: il est d\^u \`a Grothendieck et g\'en\'eralis\'e par Jakob Murre \`a l'\'equivalence rationnelle (les deux preuves reposant sur un m\^eme th\'eor\`eme de Weil). Je renvoie \`a \cite[\S 4]{scholl} pour les d\'etails.

b) r\'esulte de a) puisque 
\[p^2_X=1_{h(X)}-p^0_X-p^1_X-p^3_X-p^4_X.\]

c) Ce th\'eor\`eme est d\^u \`a Lieberman-Kleiman, et g\'en\'eralis\'e par Shermenev, Deninger-Murre et K\"unnemann \`a l'\'equivalence rationnelle: je renvoie \`a  \cite{dm} et \cite{kunnemann} pour les d\'etails.
\end{proof}

\begin{cor} Pour tout $M\in \sM_H^\eff$, $p^i_M$ est alg\'ebrique pour $i=0,1$.
\end{cor}

\begin{proof} Par le lemme \ref{l6.2} b), on se ram\`ene \`a $M=h(X)$ pour $X\in \V(k)$ et on applique le th\'eor\`eme \ref{t6.2} a).
\end{proof}

Enfin, le th\'eor\`eme suivant est beaucoup plus profond: il repose sur la d\'emonstration par Deligne de l'hypoth\`ese de Riemann\index{Hypothèse de Riemann!pour les variétés projectives lisses} \cite{weilII}.

\begin{thm}[Katz-Messing, \protect{\cite{KM}}]\phantomsection\label{t6.3} Si $k$ est fini, les projecteurs de K\"unneth sont alg\'ebriques pour tout $X\in\V(k)$ et $H=H_l$ pour tout $l\ne \car k$, ainsi que pour la cohomologie cristalline\index{Cohomologie!cristalline}. De plus, $p^i_X$ est donn\'e par un cycle alg\'ebrique ind\'ependant de la cohomologie de Weil choisie, et combinaison lin\'eaire de puissances de Frobenius.
\end{thm}

\subsection{Exemple: motifs d'Artin}

\begin{defn} Un \emph{motif d'Artin} est un facteur direct de $h(X)$, o\`u $X$ est de dimension $0$. \end{defn}

Un $k$-sch\'ema lisse de dimension $0$ est de la forme $\Spec E$, o\`u $E$ est une $k$-alg\`ebre \'etale, c'est-\`a-dire produit $\prod_{i=1}^r E_i$ d'extensions finies s\'eparables de $k$. 

\begin{prop}\phantomsection\label{p6.5} Soit $\sM_\sim^0(k,F)\subset\sM_\sim(k,F)$ la sous-cat\'egorie pleine form\'ee des motifs d'Artin. Alors:
\begin{thlist}
\item $\sM^0_\sim(k,F)\subset \sM_\sim^\eff(k,F)$.
\item $\sM^0_\sim(k,F)$ est une sous-cat\'egorie mono\"\i dale sym\'etrique rigide de $\sM_\sim(k,F)$.
\item $\sM^0_\sim(k,F)$ ne d\'epend pas du choix de l'\'equivalence ad\'equate $\sim$, au sens suivant: si $\sim\ge \sim'$, le foncteur $\sM^0_\sim(k,F)\to \sM^0_{\sim'}(k,F)$ est une \'equivalence de cat\'egories.
\item Pour $\sim=\sim_H$ où $H$ est une cohomologie de Weil vérifiant l'axiome de la remarque \ref{r3.4}, on a
\[\sM_H^0(k,F)=\{M\in \sM_H^\eff\mid w(M)=0\}\]
(cf. définition \ref{d6.2} a)).
\end{thlist}
\end{prop}

\begin{proof} (i) et (ii) sont \'evidents; pour (iii) on se ram\`ene \`a montrer que 
\[\sM_\sim(h(X),h(Y))\iso \sM_{\sim'}(h(X),h(Y))\]
si $\dim X=\dim Y=0$, ce qui est \'evident puisqu'il s'agit de cycles de codimension $0$. Dans (iv), l'inclusion $\subset$ est évidente; inversement, si $M$ est effectif et $w(M)=0$, écrivons-le comme facteur direct de $h(X)$ pour $X$ projective lisse. Alors $M$ est facteur direct de $h^0(X)$ (cf. théorème \ref{t6.2} a)).
\end{proof}

\begin{thm}\phantomsection\label{t6.6} Notons $G=\Gal(k_s/k)$.  La cat\'egorie $\sM^0_\sim(k,F)$ est \'equivalente \`a la cat\'egorie (mono\"\i dale sym\'etrique) $\Rep_F(G)$ des repr\'esentations lin\'eaires continues de $G$ \`a coefficients dans $F$.
\end{thm}

\begin{proof} Si $\dim X=0$, la th\'eorie de Galois revue par Grothendieck associe \`a $X$ l'ensemble fini $X(k_s)$ muni de l'action (continue) de $G$, d'o\`u la repr\'esentation ``de permutation''
\[\rho_X=F[X(k_s)]\in \Rep_F(G).\]

Comme $\Rep_F(G)$ est pseudo-ab\'elienne, ceci s'\'etend en un foncteur
\[\rho:\sM_\sim^0(k,F)\to \Rep_F(G)\]
qui est visiblement mono\"\i dal sym\'etrique. Il reste \`a voir que $\rho$ est une \'equivalence de cat\'egories.

Pour la pleine fid\'elit\'e, on se ram\`ene par rigidit\'e (proposition \ref{p6.5} (ii)) \`a voir que
\begin{equation}\label{eq6.1}
\sM_\sim(\un,h(X))\iso \Rep_F(G)(F,\rho_X) 
\end{equation}
si $X=\Spec E$ o\`u $E$ est une extension finie s\'eparable de $k$. Le groupe de gauche est $A^0_\sim(X,F)\simeq F$ et celui de droite est $\Hom_G(F,F[X(k_s)])\simeq F$ puisque l'action de $G$ sur $X(k_s)=\Hom_k(E,k_s)$ est transitive. On voit tout de suite que \eqref{eq6.1} envoie le g\'en\'erateur $[X]$ du membre de gauche sur l'\'el\'ement $\sum\limits_{x\in X(k_s)} x$ du membre de droite.

Il reste \`a d\'emontrer l'essentielle surjectivit\'e. Soit $\rho\in \Rep_F(G)$. Par continuit\'e, $\rho$ se factorise \`a travers $\bar G=\Gal(E/k)$, o\`u $E$ est une extension galoisienne finie de $k$. Alors $\rho$ est facteur direct de la repr\'esentation r\'eguli\`ere $r_{\bar G}=\rho_{\Spec E}$.
\end{proof}

\subsection{Exemple: $h^1$ de vari\'et\'es ab\'eliennes}

Soit $A$ une vari\'et\'e ab\'elienne sur $k$. Dapr\`es le th\'eor\`eme \ref{t6.2} c), les projecteurs de K\"unneth de $h(A)$ sont alg\'ebriques; d'apr\`es Deninger-Murre \cite{dm}, on dispose m\^eme de projecteurs de Chow-K\"unneth $p^i_A$ modulo l'\'equivalence rationnelle, naturels pour les homomorphismes de vari\'et\'es ab\'eliennes. Ceci d\'efinit un foncteur (contravariant)
\[h^1:\Ab(k)^\op\to \sM_\rat(k,\Q)\]
o\`u $\Ab(k)$ est la cat\'egorie des $k$-vari\'et\'es ab\'eliennes (objets: vari\'et\'es ab\'eliennes; morphismes: homomorphismes de vari\'et\'es ab\'eliennes). 

\begin{thm}\phantomsection\label{t6.7} Le foncteur $h^1$ est additif. Il induit un foncteur pleinement fid\`ele
\[h^1:\Ab^0(k)^\op\to \sM_\rat(k,\Q)\]
o\`u $\Ab^0(k)$ a les m\^emes objets que $\Ab(k)$, les morphismes \'etant tensoris\'es par $\Q$. Sa compos\'ee avec la projection $\sM_\rat(k,\Q)\to \sM_\sim(k,\Q)$ est encore pleinement fid\`ele pour toute relation d'\'equivalence ad\'equate $\sim$. Si $H$ est une cohomologie de Weil vérifiant l'axiome de la remarque \ref{r3.4}, l'image essentielle de $h^1$ dans $\sM_H^\eff$ est formée des motifs purement de poids $1$.
\end{thm}

\begin{proof} Par construction des projecteurs $p^i_A$, on a
\[p^i_A\circ n_A = n^ip^i_A\]
pour tout $i\ge 0$ et $n\in\Z$, o\`u $n_A$ est l'endomorphisme de multiplication par $n$ sur $A$ \cite[th. 3.1]{dm}. On en d\'eduit facilement l'additivit\'e de $h^1$, qui s'\'etend donc \`a $\Ab^0(k)$. Comme $\sM_\rat(k,\Q)\to \sM_\sim(k,\Q)$ est plein pour tout $\sim$ et que $\sim_\num$ est l'\'equivalence ad\'equate la moins fine, il suffit de montrer que $h^1$ est plein pour $\sim=\sim_\rat$ et fid\`ele pour $\sim=\sim_\num$. Pour cela, on peut se ramener au cas de deux jacobiennes de courbes $J$ et $J'$, et les deux \'enonc\'es r\'esultent de l'isomorphisme de Weil (th\'eor\`eme \ref{t3.7}). Enfin, la dernière assertion se démontre comme pour la proposition \ref{p6.5} (iv). 
\end{proof}

\subsection{La fonction z\^eta d'un endomorphisme} \index{Fonction zêta!d'un endomorphisme}

\subsubsection{Cas d'une relation d'\'equivalence suffisamment fine}

En appliquant le th\'eor\`eme \ref{tA.4}, on obtient:

\begin{thm}\phantomsection\label{t6.1} Supposons que $F$ soit un corps de caract\'eristique z\'ero. Pour tout $M\in \sM_\rat(k,F)$ et tout endomorphisme $f\in \End(M)$, la fonction $Z(f,t)$ de la d\'efinition \ref{dA.9} est une fraction rationnelle \`a coefficients dans $F$. Si $f$ est un isomorphisme, on a une \'equation fonctionnelle\index{Equation fonctionnelle@\'Equation fonctionnelle} de la forme
\[Z({}^tf^{-1},t^{-1}) = \det(f)(-t)^{\chi(M)}Z(f,t)\]
o\`u ${}^tf\in \End(M^*)$ est le transpos\'e de $f$ et $\det(f)$ est un scalaire donné par la formule du théorème \ref{tA.4}. Dans cet \'enonc\'e, on peut remplacer $\sim_\rat$ par $\sim_\alg$ ou $\sim_\tnil$.
\end{thm}

\begin{proof} Il suffit de conna\^\i tre l'existence d'un foncteur mono\"\i dal sym\'etrique $H^*=\sM_\rat(k,F)\to \Vec_K^*$ pour un corps $K$ contenant $F$: un tel foncteur est donn\'e par une cohomologie\index{Cohomologie!de Weil} de Weil, par exemple la cohomologie $l$-adique\index{Cohomologie!$l$-adique} pour un $l\ne \car k$.
\end{proof}

\subsubsection{Cas de l'\'equivalence num\'erique} La conjecture standard\index{Conjectures!standard} principale de Grothendieck \cite{grothendieck-standard} pr\'edit que l'\'equivalence homologique co\"\i ncide avec l'\'equivalence num\'erique pour toute cohomologie de Weil. Faute de conna\^\i tre cette conjecture, l'extension du th\'eor\`eme \ref{t6.1} aux motifs modulo l'\'equivalence num\'erique est une question d\'elicate. Les r\'esultats qui suivent sont extraits de \cite{aknote}.

Soit $A$ un anneau; rappelons la d\'efinition de son \emph{radical de Jacobson}:
\[R=\{a\in A\mid \forall b\in A, 1-ab\text{ est inversible}\}.\]

\begin{prop}[\protect{\cite[cor. 3 et prop. 5]{aknote}}]\phantomsection\label{p6.3} Consid\'erons les motifs \`a coefficients dans un corps $F$ de caract\'eristique $0$. Soit $H$ une cohomologie\index{Cohomologie!de Weil!classique} de Weil classique (\S \ref{3.5.9}), et soit $M\in \sM_H(k,F)=\sM_H$. Alors\\
a) Le radical de Jacobson $R$ de $A_H=\End(M)$ est nilpotent, et $A_H/R$ est une $F$-alg\`ebre semi-simple.\\
b) \cite[cor. 1]{jannsen} Si la somme des projecteurs de K\"unneth de degr\'es pairs de $M$ est alg\'ebrique, $A_H/R\iso A_\num$ o\`u $A_\num=\End_{\sM_\num}(M_\num)$, $M_\num$ \'etant la projection de $M$ dans $\sM_\num$.\\
c)  Soit $M\in \sM_H$ tel que la somme des projecteurs de K\"unneth de degr\'es pairs de $M$ soit alg\'ebrique. Si $M_\num=0$, alors $M=0$.\\
d) (\cf \cite[rk. 4]{jannsen}) Soit $M\in \sM_\num$. Alors il existe $\tilde M\in \sM_H$ tel que $\tilde M_\num\simeq M$.
\end{prop}

On trouvera une version abstraite de cette proposition dans \cite[th. 1 b)]{ak2}.

\begin{prop}[\protect{\cite[prop. 6]{aknote}}]\phantomsection\label{p6.4} Notons $\sM_H^*$ la sous-cat\'egorie plei\-ne de $\sM_H$ form\'ee des $M$ dont tous les projecteurs de K\"unneth sont alg\'ebriques, et soit $\sM_\num^*$ son image essentielle dans $\sM_\num$. Alors\\
a) $\sM_\num^*$ ne d\'epend pas du choix de $H$ (classique).\\
b) Pour $M\in\sM_\num^*$, provenant de $\tilde M_H\in\sM_H^*$ (\cf proposition \ref{p6.3} d)), l'image $p^i_M$ de $p^i_{\tilde M_H}$ dans $\End(M)$ ne d\'epend pas du choix de $H$ ni de celui de $\tilde M_H$. \\
c) Si $k$ est fini, $\sM_\num^*=\sM_\num$.
\end{prop}

Cet \'enonc\'e est d\'emontr\'e incompl\`etement dans \cite{aknote}; la d\'emonstration (suivant O. Gabber) est compl\'et\'ee dans \cite[app. B]{ak}. Notons que c) ne fait que répéter le théorème \ref{t6.3}.

\begin{cor}\phantomsection\label{c6.1} Pour tout motif simple $S\in \sM_\num^*$, il existe un unique $i\in \Z$ tel que $p^i_S\ne 0$: c'est le \emph{poids} de $S$.\qed
\end{cor}

\begin{defn}\phantomsection\label{d6.1} Un motif $M\in \sM_\num^*$ est \emph{pur de poids $i$} s'il est somme directe de motifs simples de poids $i$.
\end{defn}

\begin{cor}\phantomsection\label{c6.2} Soit $M\in \sM_\num^*$, pur de poids $i$. Alors pour tout $f\in \End(M)$, $Z(f,t)=P(t)^{(-1)^{i+1}}$ o\`u $P\in F[t]$. On a l'équation fonctionnelle du théorème \ref{t6.1}.\index{Equation fonctionnelle@\'Equation fonctionnelle}
\end{cor}

\begin{proof} Cela r\'esulte imm\'ediatement de la formule des traces (lemme \ref{lA.6}).
\end{proof}

\subsection{Cas d'un corps de base fini}

Supposons maintenant $k=\F_q$. Pour tout $X\in \V(k)$, notons  $\pi_X$ l'endomorphisme de Frobenius vu comme correspondance alg\'ebrique (transpos\'e du graphe du Frobenius absolu de $X$). Alors $\pi_X$ est \emph{central} dans $\Corr_\rat^0(X,X)$; plus g\'en\'eralement, pour toute correspondance $\gamma\in\Corr_\rat^0(X,Y)$, le diagramme
\[\begin{CD}
h(X)@>\gamma>> h(Y)\\
@V{\pi_X}VV @V{\pi_Y}VV\\
h(X)@>\gamma>> h(Y) 
\end{CD}\]
est commutatif. De plus, $\pi_{X\times Y}=\pi_X\otimes \pi_Y$. On en d\'eduit que $X\mapsto \pi_X$ s'\'etend en un \emph{$\otimes$-automorphisme du foncteur identique}
\[\sM_\rat(k,\Q)\ni M\mapsto \pi_M\in \Aut(M)\]
vérifiant les identités
\begin{equation}\label{eq6.2}
\pi_{M\oplus N} = \pi_M\oplus \pi_N, \quad \pi_{M\otimes N} = \pi_M\otimes \pi_N, \quad \pi_{M^*} = {}^t \pi_M^{-1}.
\end{equation}
(la dernière identité étant une propriété générale d'un $\otimes$-automorphisme du foncteur identique, cf. \cite[I, (3.2.3.6)]{saa}).

Notons de plus que  $\sM_\num=\sM_\num^*$ par \cite{KM}.

\begin{prop}\phantomsection\label{l6.3} a) L'endomorphisme $\pi_\L$ du motif de Lefschetz $\L$ est la multiplication par $q$. \\
b) Si $X$ est géométriquement connexe de dimension $n$, l'endomorphisme $\pi_{h^{2n}(X)}$ est la multiplication par $q^n$.\\
c) Si $M\in \sM_\num^\eff(k,\Q)$ est pur de poids $i$ (définition \ref{d6.1}), on a
\[\det(\pi_M) = \pm q^{i\chi(M)/2}\]
où $\det(\pi_M)$ est le nombre rationnel apparaissant dans l'équation fonctionnelle du théorème \ref{t6.1}.\index{Equation fonctionnelle@\'Equation fonctionnelle} \end{prop}

\begin{proof} a) Soit $p=p^2_{\P^1}$ le second projecteur de Künneth de la droite projective: il découpe $\L$ sur le motif de $\P^1$. Il faut donc voir que
\[\pi_{\P^1}p = qp.\]
Mais $p$ est la classe de $0\times \P^1$ dans $CH^1(\P^1\times \P^1)$, transposé du graphe de la composition $\P^1\to \{0\}\inj \P^1$. Si $(t_1,t_2)$ sont les coordonnées de $\P^1\times \P^1$, l'équation de ce diviseur est $t_1=0$. Quand on lui applique $\pi_{\P^1}$, on obtient l'équation $t_1^q=0$, d'où la conclusion. b) résulte maintenant de a), du théorème \ref{t6.2} et de \eqref{eq6.2}. Enfin, pour c) on utilise la formule du théorème \ref{tA.4}, l'hypothèse de Riemann\index{Hypothèse de Riemann!pour les variétés projectives lisses} pour $H^i(X)$ et le fait que les racines inverses du polynôme caractéristique de $\pi_M$ forment un sous-ensemble (Galois-invariant!) de celui des racines inverses du polynôme caractéristique de $\pi_{h^i(X)}$.
\end{proof}

\begin{rques}\phantomsection\label{r6.3} a) En particulier, $\chi(M)$ est pair dans c) si $i$ est impair, au moins si $q$ n'est pas un carré. (C'est vrai conjecturalement même si $q$ est un carré: voir exercice \ref{exo6.1} c).)
\\
b) Si $M$ est de la forme $h^i(X)$ avec $i$ impair, le signe est $+1$ dans l'expression de $\det(\pi_M)$. En effet, la dualité de Poincaré et le théorème de Lefschetz difficile (\S \ref{s5.4.2}) munissent $H^i_l(X)$, pour $l$ premier ne divisant pas $q$, d'un accouplement parfait, Galois-équivariant
\[H^i_l(X)\times H^i_l(X)\to \Q_l(-i)\]
qui est \emph{alterné} puisque $i$ est impair; en particulier, $d=\dim H^i_l(X)$ est pair. Le déterminant d'une matrice symplectique étant égal à $1$, celui de $\pi_M$ est égal à $(\sqrt{q^i})^d$. %La parité de $d$ et la formule du théorème \ref{t6.1} impliquent alors que le signe de l'équation fonctionnelle de $\zeta(h^i(X),s)$ est $+1$. 
Voir remarque \ref{r3.6} pour le cas $i$ pair.\\
c) On pourra consulter \cite[III]{kahn-zeta} pour une approche différente de la proposition \ref{l6.3} c) (et une généralisation aux ``fonctions zêta motiviques''), utilisant la notion de déterminant d'un motif. %(On prendra garde au fait que la proposition 13.3 (3) de \loccit est fausse, comme le montre l'exercice \ref{exo6.1} c) (ii).\index{Fonction zêta!motivique}
\end{rques}

\begin{defn} Pour $M\in \sM_\num(k,\Q)$, on note
\[Z(M,t)=Z(\pi_M,t).\]
On pose aussi $\zeta(M,s)=Z(M,q^{-s})$: c'est la \emph{fonction z\^eta} de $M$.\index{Fonction zêta!d'un motif (pur)}
\end{defn}

\begin{lemme}\phantomsection\label{l6.4}
a) Si $M=h(X)$ pour $X\in \V(k)$, on a $Z(M,t)=Z(X,t)$.\\
b) Pour tout $M\in \sM_\num(k,\Q)$, on a
\[Z(M(1),t)=Z(M,q^{-1}t).\]
\end{lemme}

\begin{proof} a) est clair, \cf preuve du th\'eor\`eme \ref{t3.6}, et b) est un simple calcul.
\end{proof}

\begin{thm}\phantomsection\label{t6.4}\index{Equation fonctionnelle@\'Equation fonctionnelle} Supposons $M\in \sM_\num^\eff(k,\Q)$ pur de poids $i$. Alors $Z(M,t)=P(t)^{(-1)^{i+1}}$ o\`u $P(0)=1$ et  $P\in \Z[t]$. On a une \'equation fonctionnelle de la forme
\[Z(M^*,t^{-1}) = \det(\pi_M)(-t)^{\chi(M)}Z(M,t),\quad \det(\pi_M) = \pm q^{i\chi(M)/2}\]
où $\chi(M)$ est pair et le signe de $\det(\pi_M)$ est $+$ si $i$ est impair et $M$ de la forme $h^i(X)$, cf. remarques \ref{r6.3} a) et b). De plus, les racines inverses de $P$ sont des nombres de Weil de poids $i$. 
\end{thm}

\begin{proof} Le corollaire \ref{c6.2} donne la premi\`ere et la seconde partie de l'\'e\-non\-c\'e, en utilisant \eqref{eq6.2} pour l'équation fonctionnelle.
 
 Pour la dernière assertion (y compris $P\in \Z[t]$), on \'ecrit $M$ comme facteur direct de $h(X)$ pour $X\in\V(k)$ (non nécessairement connexe); alors $M$ est m\^eme facteur direct de $h^i(X)$, et $P$ divise le polyn\^ome $P_i$ attach\'e \`a $X$. La conclusion d\'ecoule alors de l'hypoth\`ese de Riemann\index{Hypothèse de Riemann!pour les courbes}.
\end{proof}

En particulier, pour $X$ projective lisse,
\begin{equation}\label{eq6.5}
Z(h^i(X),t) = P_i(t)^{(-1)^{i+1}}
\end{equation}
où $P_i$ est le $i$-ème polyn\^ome intervenant dans la factorisation de $Z(X,t)$ (\S \ref{s3.3}).

\begin{rques}\phantomsection\label{r6.1} a) Si on enl\`eve la condition que $S$ soit effectif, l'\'enonc\'e reste valable mais $P$ est alors \`a coefficients dans $\Z[1/q]$: par exemple, $Z(\T,t)=1-q^{-1}t$.\\
b) Le th\'eor\`eme \ref{t6.4} s'\'etend au cas o\`u les coefficients $\Q$ sont remplac\'es par un corps de nombres $F$; le polyn\^ome $P$ est alors \`a coefficients dans $O_F$ si $M$ est effectif, dans $O_F[1/q]$ en g\'en\'eral.\\
c) Pour $i=0,1$, le corollaire  \ref{c6.2} et le théorème \ref{t6.4} ne dépendent pas des résultats de Deligne et Katz-Messing, mais ``seulement'' de ceux de Weil (formule pour les morphismes entre jacobiennes de courbes et hypothèse de Riemann pour les courbes): cf. théorèmes \ref{t6.6} et \ref{t6.7}. 
\end{rques}

\subsection{La conjecture de Tate}\index{Conjecture!de Tate} On continue à supposer que $k=\F_q$.

\begin{conj}\label{co6.1} Pour toute $k$-variété projective lisse $X$ et tout entier $i\ge 0$, on a
\[\ord_{s=i} \zeta(X,s) = -\rg A^i_\num(X).\]
\end{conj}

(Rappelons que le groupe $A^i_\num(X)$ est libre de type fini d'après le théorème \ref{t6.9}.)

Cette conjecture est vraie pour $i=0$ d'après la partie ``hypothèse de Riemann'' des conjectures de Weil. Pour plus de détails sur le cas $i>0$, je revoie à l'exposé de Tate \cite{tate-seattle}. En particulier, pour un nombre premier $l\nmid q$, on a les énoncés suivants, où $G=\Gal(\bar k/k)$:
\begin{description}
\item[$T^i(X,l)$] L'homomorphisme déduit de la classe de cycle
\[CH^i(X)\otimes \Q_l\to H^{2i}_l(X)(i)^G,\]
est surjectif.
\item[$S^i(X,l)$] L'application composée
\[H_l^{2i}(X)(i)^G\inj H_l^{2i}(X)(i)\surj H_l^{2i}(X)(i)_G\]
est bijective.
\item[$SS^i(X,l)$] L'action de $G$ sur $H^i_l(X)$ est semi-simple.
\end{description}

On a évidemment $SS^{2i}(X,l)\Rightarrow S^i(X,l)$, et:

\begin{thm}[\protect{\cite[th. 2.9]{tate-seattle}}] Pour tout $l$, la conjecture \ref{co6.1} pour $(X,i)$ est équivalente à $T^i(X,l)+T^{d-i}(X,l)+S^i(X,l)$; elle implique la conjecture standard ``équivalence homologique = équivalence numérique'' en codimensions $i$ et $d-i$. De plus, $S^i(X,l) \iff S^{d-i}(X,l)$.
\end{thm}

\begin{thm}[\protect{\cite[th. 6]{kahn-ss}}] On a l'équivalence
\[S^d(X\times X,l)\iff SS^i(X,l) \text{ pour tout } i. \]
\end{thm}

Les énoncés $T^i(X,l)$ et $S^1(X,l)$ sont connus quand $i=1$ et $X$ est une variété abélienne \cite{tate-fini}: c'est une variante du théorème \ref{tt}. Par conséquent, la conjecture \ref{co6.1} est connue pour $i=1$ quand $X$ est une variété abélienne.

\begin{exo}\phantomsection\label{exo6.2} Dans cet exercice, on considère  le foncteur de réalisation enrichi
\[R_l^*:\sM_l(k,\Q_l)\to \Rep^*(G,\Q_l)\]
induit par la cohomologie de Weil $H_l$, à valeurs dans la catégorie abélienne des représentations $\Q_l$-adiques $\Z$-graduées de $G$.\\
a) (cf. \cite[prop. 7.3.1.3 1)]{andre}) Montrer que  la conjecture $T^i(X,l)$ pour tout $(X,i)$ équivaut à l'énoncé suivant, pour tout $(M,N)\in \sM_l(k,\Q_l)\times \sM_l(k,\Q_l)$:
\begin{description}
\item[T(M,N,l)] L'homomorphisme
\[\sM_l(k,\Q_l)(M,N)\to \Rep^*(G,\Q_l)(R_l^*(M),R_l^*(N))\]
est surjectif.
\end{description}
Autrement dit, $R_l^*$ est \emph{plein}. (Montrer que $T(M,N,l)\iff T(\un,M^*\otimes N,l)$, puis se ramener à $T^*(X,l)$ pour un $X$ convenable.)\\
b) Montrer de même que la conjecture $SS^i(X,l)$ pour tout $(X,i)$ implique pour tout $M\in \sM_l(k,\Q_l)$:
\begin{description}
\item[SS(M,l)] $R_l^*(M)$ est semi-simple.
\end{description}
c) Soient $M,N\in \sM_l(k,\Q_l)$. Supposons que $\dim R_l^i(M)=\dim R_l^i(N)$ pour tout $i$ et que les valeurs propres de l'action de Frobenius sur ces espaces vectoriels co\"\i ncident. Sous les conjectures $T(M,N,l)$, $SS(M,l)$ et $SS(N,l)$, en déduire que $M\simeq N$.\\
d) Montrer que la conjecture \ref{co6.1} implique l'énoncé suivant: pour tout $M\in \sM_\num(k,\Q)$ et tout $i\in \Z$, on a
\[\ord_{s=i} \zeta(M,s) = -\dim_\Q \sM_\num(M,\L^i).\]
(Reprendre la preuve de \cite[th. 2.9]{tate-seattle}.)
\end{exo}

\begin{exo}\phantomsection\label{exo6.1} Dans cet exercice, on examine ce que devient le signe de l'équation fonctionnelle dans le théorème \ref{t6.4} quand on remplace $h^i(X)$ par un motif quelconque de poids $i$ impair. \\
a) Soit $M\in \sM_\num^\eff(k,\Q)$, purement de poids $i$. Montrer qu'on peut écrire $M\simeq M'\oplus M''$, où les valeurs propres de $\pi_{M'}$ (resp. de $\pi_{M''}$) sont de la forme $\pm q^{i/2}$ (resp. ne sont pas de cette forme). (Construire l'idempotent correspondant comme polyn\^ome en $\pi_M$.)\\
b) Si $M'=0$, montrer que $\chi(M)$ est pair, que $\det(\pi_M)>0$, et que le signe de l'équation fonctionnelle de $\zeta(M,s)$ est $+1$. (Utiliser les formules du théorème \ref{t6.1} et du théorème \ref{tA.4}, ainsi que la remarque \ref{r5.2}.)\\
c) On suppose $M''=0$ et $i$ impair. Soit $\epsilon\in \{+,- \}$. D'après Deuring \cite{deuring2} (voir aussi Honda \cite{honda}), il existe une courbe elliptique supersingulière $E_\epsilon$, définie sur $k$ et telle que $\pi_{E_\epsilon}$ ait pour valeur propre $\epsilon q^{i/2}$. En utilisant l'exercice \ref{exo6.2} c), montrer que sous la conjecture \ref{co6.1}, $M$ est de la forme $\left(h^1(E_\epsilon)(\frac{1-i}{2})\right)^r$ pour un $r>0$. En déduire que les conclusions de b) restent vraies.\footnote{En examinant la démonstration, on voit que la conjecture \ref{co6.1} est en fait utilisée en codimension $\frac{i+1}{2}$.}\\
d) Que se passe-t-il quand $i$ est pair?
\end{exo}

\subsection{Coronidis loco}\label{cor-loco}\index{Conjecture!d'Artin}

D\'eduisons de ce qui pr\'ec\`ede une d\'emonstration motivique du th\'eor\`eme de Weil \ref{t4.3} (conjecture d'Artin sur les corps de fonctions): d'après la remarque \ref{r6.1} c), elle n'utilise que les résultats de Weil.

Soit $K$ un corps de fonctions d'une variable sur $\F_q$:  quitte \`a augmenter $\F_q$, on peut le supposer alg\'ebriquement ferm\'e dans $K$. Soit $L$ une extension finie régulière de $K$, galoisienne de groupe $G$. Soit $\rho:G\to GL_n(\C)$ une repr\'esentation lin\'eaire, irr\'eductible et non triviale. Alors $\rho$ est d\'efinie sur un corps de nombres $F$. En choisissant une place finie $\lambda$ de $F$ au-dessus d'un nombre premier $l$ qui ne divise pas $q$, on peut voir $\rho$ comme repr\'esentation \`a valeurs dans le compl\'et\'e $F_\lambda$, extension finie de $\Q_l$.

Soit $p:\tilde C\to C$ le rev\^etement (ramifi\'e), galoisien de groupe $G$, des mod\`eles projectifs lisses de $K$ et $L$. Le groupe $G$ op\`ere sur le motif $h(\tilde C)\in \sM_\num(\F_q,F_\lambda)$, et $\rho$ y d\'ecoupe un motif $M_\rho$; la formule des traces (\'etendue aux coefficients $F_\lambda$) montre que
\[L(\rho^*,t)=Z(M_\rho,t).\footnote{Prendre garde que les fonctions $L$ d'Artin sont d\'efinies en termes des Frobenius arithm\'etiques, contrairement aux fonctions $L$ de faisceaux $l$-adiques\index{Faisceau $l$-adique} et de motifs: pour passer d'une convention à l'autre, il faut  remplacer $\rho$ par sa duale.}
\]

Par hypoth\`ese sur $\rho$, $M_\rho$ est facteur direct de l'image du projecteur $\pi=1-\frac{1}{|G|}p^*p_*$. \'Ecrivons
\begin{align*}
h(C)&=\un\oplus h^1(C)\oplus \L\\
h(\tilde C)&=\un\oplus h^1(\tilde C)\oplus \L
\end{align*}
(on a utilis\'e l'hypoth\`ese que l'extension $L/K$ est r\'eguli\`ere). Comme $\IM(1-\pi)=h(C)$, on voit que $M_\rho$ est \emph{facteur direct de $h^1(\tilde C)$}; donc $Z(M_\rho,t)$ est un polyn\^ome divisant $P_1(\tilde C)$. Ainsi $\zeta(\rho^*,s)=Z(M_\rho,q^{-s})$ est une fonction enti\`ere, dont les z\'eros sont situ\'es sur la droite $\Re(s)=1/2$.

\appendix

\section{Cat\'egories karoubiennes et cat\'egories mono\"\i dales}

\subsection{Cat\'egories karoubiennes}

\subsubsection{Endomorphismes idempotents (ou projecteurs)}

\begin{defn} Soit $\sC$ une cat\'egorie, et soit $C$ un objet de $\sC$. Un endomorphisme $p$ de $C$ est \emph{idempotent} si $p^2=p$. On dit aussi que $p$ est un \emph{projecteur}.
\end{defn}

\begin{defn} Soient $\sC,C,p$ comme ci-dessus. L'\emph{image} de $p$ est, s'il existe, le noyau $\IM p$ du couple de fl\`eches parall\`eles $(1_X,p)$.
\end{defn}

(Cette d\'efinition est justifi\'ee par le calcul ensembliste suivant, lui-m\^eme justifi\'e par le lemme de Yoneda: si $x\in C$ est de la forme $p(y)$, alors $p(x)=p^2(y)=p(y)=x$.)

\begin{lemme}\phantomsection\label{l:retr} Supposons que $p$ ait une image. Alors $\IM p$ est canoniquement r\'etracte de $C$.
\end{lemme}

\begin{proof} Par d\'efinition, on a un morphisme canonique $\IM p\to C$, \'egalisant les deux fl\`eches $1_X,p$. Inversement, l'endomorphisme $p:C\to C$ \'egalise ces deux fl\`eches, donc se factorise canoniquement par $\IM p$. On voit imm\'ediatement que la composition
\[\IM p\to C\to \IM p\]
est \'egale \`a l'identit\'e.
\end{proof}

\begin{defn} La cat\'egorie $\sC$ est \emph{karoubienne} si tout projecteur a une image.
\end{defn}

\subsubsection{Enveloppes karoubiennes}

\begin{defn}\phantomsection\label{dA.4} Soit $\sC$ une cat\'egorie. Une \emph{enveloppe karoubienne} de $\sC$ (en anglais: \emph{idempotent completion}) est un foncteur $i:\sC\to \sC^\natural$, o\`u $\sC^\natural$ est pseudo-ab\'elienne, v\'erifiant la propri\'et\'e $2$-universelle suivante: pour toute cat\'egorie karoubienne $\sD$, le foncteur de restriction
\[i^*:\Fonct(\sC^\natural,\sD)\to \Fonct(\sC,\sD)\]
est une \'equivalence de cat\'egories.
\end{defn}

\begin{thm}\phantomsection\label{tA.2} Une enveloppe karoubienne existe toujours; le foncteur $i$ est pleinement fid\`ele et le foncteur $\sC^\natural\to (\sC^\natural)^\natural$ (enveloppe karoubienne de $\sC^\natural)$ est une \'equivalence de cat\'egories.
\end{thm}

\begin{proof} On en trouve deux dans \cite{SGA4}: l'une dans l'exercice 7.5 de l'expos\'e IV et l'autre dans l'exercice 8.7.8 de l'expos\'e I. Donnons la premi\`ere, qui a l'avantage d'\^etre constructive (la seconde, qui utilise les ind-objets, est plus conceptuelle):

D\'efinissons une cat\'egorie $\sC^\natural$ dont les objets sont les couples $(C,p)$ o\`u $C\in \sC$ et $p$ est un projecteur de $C$. Pour un autre objet $(D,q)$, on d\'efinit:
\[\sC^\natural((C,p),(D,q))=\{f\in \sC(C,D)\mid f=qfp.\}\]

La composition des morphismes est induite par celle de $\sC$; l'identit\'e de $(C,p)$ et $p$. Le foncteur $i$ est d\'efini par
\[i(C)=(C,1_C)\]
et $i$ est l'identit\'e sur les morphismes. La propri\'et\'e universelle se v\'erifie facilement, et implique la derni\`ere affirmation.
\end{proof}

\subsubsection{Cas des cat\'egories additives}

\begin{prop} Si $\sC$ est additive, $\sC^\natural$ est additive. De plus, pour tout objet $C\in \sC$ et tout projecteur $p$ de $C$, le morphisme canonique de $\sC^\natural$
\[\IM(p)\oplus \IM(1-p)\to C\]
est un isomorphisme.
\end{prop}

(Noter que, si $p$ est un projecteur, alors $1-p$ est un projecteur; de plus, $\IM(1-p)=\Ker(0,p)=:\Ker(p)$.)

\begin{proof} Par d\'efinition \cite[ch. VIII, \S 2]{mcl}, $\sC$ est additive si ses Hom sont des groupes ab\'eliens, la composition des morphismes \'etant bilin\'eaire, et si les produits finis y sont repr\'esentables. Ces propri\'et\'es sont clairement pr\'eserv\'ees par la construction donn\'ee dans la preuve du th\'eor\`eme \ref{tA.2}.

Pour la derni\`ere affirmation, on v\'erfie que les morphismes donn\'es par le lemme \ref{l:retr} (pour $p$ et $1-p$) font de $C$ un biproduit de $\IM(p)$ et $\IM(1-p)$ au sens de \cite[ch. VIII, d\'ef.]{mcl}.
\end{proof}

\subsubsection{Structures $F$-lin\'eaires}

\begin{defn}\phantomsection\label{dA.6} Soit $F$ un anneau unitaire. Une \emph{structure $F$-lin\'eaire} sur une cat\'egorie additive $\sA$  est la donn\'ee d'un homomorphisme d'anneaux $F\to\End(Id_{\sA})$.
\end{defn}

Explicitons: les endomorphismes du foncteur identique de $\sA$ ont une structure naturelle d'anneau. Un homomorphisme comme ci-dessus est la donn\'ee, pour tout $\lambda\in F$ et tout $A\in \sA$, d'un endomorphisme $\lambda_A:A\to A$, naturel en $A$, additif et multiplicatif en $\lambda$. Il revient au m\^eme de se donner des structures de $F$-modules bilat\`eres sur les $\sA(A,B)$, compatibles \`a la composition.

\begin{lemme}\phantomsection\label{lA.2} Soit $\sA$ une cat\'egorie additive: toute structure $F$-lin\'eaire sur $\sA$ s'\'etend canoniquement \`a son enveloppe karoubienne.
\end{lemme}

\begin{proof} Exercice.
\end{proof}

\subsection{Cat\'egories mono\"\i dales}\label{sA.2}

\subsubsection{D\'efinitions}

\begin{defn} Une \emph{cat\'egorie mono\"\i dale} est un triplet $(\sA,\otimes,a)$ o\`u 
\begin{itemize}
\item $\sA$ est une cat\'egorie;
\item $\otimes:\sA\times \sA\to \sA$ est un foncteur;
\item la \emph{contrainte d'associativit\'e} $a$ est un isomorphisme naturel
\[a_{A,B,C}: A\otimes(B\otimes C)\iso (A\otimes B)\otimes C\]
soumis \`a la condition de coh\'erence suivante: pour tous $A,B,C,D\in \sA$, le diagramme pentagonal
\[\xymatrix{
A\otimes(B\otimes(C\otimes D)) \ar[r]^{a_{A,B,C\otimes D}}\ar[d]^{1_A\otimes a_{B,C,D}}& (A\otimes B)\otimes (C\otimes D)\ar[r]^{a_{A\otimes B,C,D}} & ((A\otimes B)\otimes C)\otimes D\\
A\otimes ((B\otimes C)\otimes D)\ar[rr]^{a_{A,B\otimes C,D}}&& (A\otimes(B\otimes C))\otimes D\ar[u]^{a_{A,B,C}\otimes 1_D}}\]
est commutatif.
\end{itemize}
\end{defn}

\begin{rque} Cette d\'efinition se trouve dans Mac Lane \cite[ch. VII, \S 1]{mcl}, mais aussi dans Saavedra \cite[ch. I, 1.1.1]{saa} (\`a la terminologie pr\`es). Ceci s'applique aussi aux d\'efinitions qui suivent. J'ai repris de Saavedra la terminologie de \emph{contrainte}.
\end{rque}

Le th\'eor\`eme de coh\'erence de Mac Lane \cite[ch. VII, \S 2, th. 1]{mcl} \'enonce que la commutativit\'e ci-dessus entra\^\i ne ``toutes'' les autres commutativit\'es (i.e. contraintes d'associativité supérieures) imaginables. 

\begin{defn} Une cat\'egorie $\sA$ munie d'un bifoncteur $\otimes$ comme ci-dessus est \emph{unitaire} si elle est munie d'un \emph{objet unit\'e} $\un$ et, pour tout $A\in \sA$, d'isomorphismes naturels (contraintes d'unit\'e)
\[\un\otimes A\iso A\osi A\otimes \un.\]
Si $\sA$ est mono\"\i dale, on demande des compatibilit\'es ``\'evidentes'' entre la contrainte d'associativit\'e et les contraintes d'unit\'e.
\end{defn}

\begin{lemme}\phantomsection\label{lA.4} Soit $\sA$ une cat\'egorie mono\"\i dale unitaire. Alors le mono\"\i de $\End_\sA(\un)$ est  commutatif.
\end{lemme}

\begin{proof} On dispose de deux lois de composition sur $\End_\sA(\un)$: la composition des morphismes $\circ$ et leur produit tensoriel $\otimes$. On v\'erifie facilement que chaque loi est distributive par rapport \`a l'autre; un lemme bien connu des topologues (servant \`a montrer que les groupes d'homotopie sup\'erieurs des sph\`eres sont commutatifs) implique alors que ces deux lois sont \'egales et commutatives.
\end{proof}

\begin{defn} Une \emph{cat\'egorie mono\"\i dale sym\'etrique} est une cat\'egorie mono\"\i dale munie de \emph{contraintes de commutativit\'e}
\[\sigma_{A,B}:A\otimes B\iso B\otimes A\]
telles que $\sigma^2=1$, et compatibles \`a la containte de commutativit\'e en un sens ``\'evident''. On peut ajouter une unit\'e; on a alors une cat\'egorie mono\"\i dale sym\'etrique unitaire ($\otimes$-cat\'egorie ACU dans la terminologie de Saavedra \cite{saa}).
\end{defn}

On trouvera dans \cite[ch. I]{saa} une infinit\'e de sorites sur les diverses compatibilit\'es entre contraintes. %\footnote{Comme Serre me l'a fait observer, ces sorites sont engendrés par un nombre fini d'entre eux, cf. \cite[ch. VII, \S 2]{mcl} pour la containte d'associativité.} 
 En particulier, si $\sA$ est mono\"\i dale, une unit\'e est unique \`a isomorphisme pr\`es, et les contraintes d'unit\'e impliquent les autres contraintes.

\begin{lemme}\phantomsection\label{lsym} Soit $A$ un objet d'une cat\'egorie mono\"\i dale sym\'etrique $\sA$. Pour tout $n\ge 0$, la contrainte de commutativit\'e induit un homomorphisme canonique
\[\fS_n\to \Aut_\sA(A^{\otimes n})\]
o\`u $\fS_n$ est le groupe sym\'etrique.
\end{lemme}

\begin{proof} On se ram\`ene au cas $n=2$ en utilisant la pr\'esentation ``de Coxeter'' de $\fS_n$ pour les g\'en\'erateurs $\sigma_i=(i,i+1)$ ($1\le i<n$). La v\'erification des relations
\[\sigma_i\sigma_{i+1}\sigma_i = \sigma_{i+1}\sigma_i\sigma_{i+1}\]
provient de l'interaction entre les contraintes d'associativit\'e et de commutativit\'e.
\end{proof}

\begin{rque} Une notion plus faible que la sym\'etrie est la \emph{condition de tressage}: on ne demande plus \`a la contrainte de commutativit\'e $\sigma$ d'\^etre involutive, mais seulement de v\'erifier les \emph{relations de Yang-Baxter} \cite[ch. XI]{mcl}. Cette notion est importante dans la th\'eorie des groupes quantiques, ainsi que pour l'\'etude du groupe de Grothendieck-Teichm\"uller introduit par Drinfeld. L'action du groupe sym\'etrique est alors remplac\'ee par celle du \emph{groupe des tresses}.
\end{rque}

\subsubsection{Foncteurs mono\"\i daux}

\begin{defn} Soient $\sA$, $\sB$ deux cat\'egories mono\"\i dales. Un \emph{foncteur mono\"\i dal} de $\sA$ vers $\sB$ est un couple $(F,\phi)$ o\`u $F:\sA\to \sB$ est un foncteur et $\phi$ est un isomorphisme naturel
\[\phi_{A,B}:F(A)\otimes F(B)\iso F(A\otimes B)\]
compatible aux contraintes d'associativit\'e en un sens ``\'evident''. Si $\sA,\sB$ sont unitaires, sym\'etriques, on a de m\^eme la notion de foncteur mono\"\i dal unitaire, sym\'etrique.
\end{defn}

\begin{rque} Il arrive souvent qu'on ait une transformation naturelle comme ci-dessus, mais qui n'est pas un isomorphisme. On parle alors de foncteur pseudo-mono\"\i dal. On dit aussi parfois \emph{faiblement mono\"\i dal}, ou en anglais \emph{lax-monoidal} (mono\"\i dal au sens l\^ache). Par ailleurs, Mac Lane dit ``mono\"\i dal'' l\`a o\`u j'utilise ``pseudo-mono\"\i dal'' et ``strong mono\"\i dal'' l\`a o\`u j'utilise ``monoidal''. Il faut v\'erifier la terminologie selon les auteurs!

On dit enfin qu'un foncteur est \emph{strictement mono\"\i dal} si $\phi$ est l'identit\'e: ce choix semble faire consensus.
\end{rque}

Une cat\'egorie mono\"\i dale est dite \emph{stricte} si sa contrainte d'associativit\'e est \'egale \`a l'identit\'e. Un \'enonc\'e \'equivalent au th\'eor\`eme de Mac Lane est le suivant: 

\begin{thm}[\protect{\cite[ch. XI, \S 2, th. 1]{mcl}}]
Toute cat\'egorie mono\"\i dale est (mono\"\i dalement) \'equivalente \`a une cat\'egorie mono\"\i dale stricte.
\end{thm}

Gr\^ace \`a ce th\'eor\`eme, on s'autorise \`a ne pas noter les parenth\`eses dans une cat\'egorie mono\"\i dale pour les produits de trois objets ou plus.

\subsubsection{Cat\'egories mono\"\i dales ferm\'ees; objets dualisables}

\begin{defn}\phantomsection\label{dA.10} Une cat\'egorie mono\"\i dale $\sA$ est \emph{ferm\'ee} si le foncteur $\otimes$ a un adjoint \`a droite, nomm\'e \emph{$\Hom$ interne}.
\end{defn}

Plus correctement, la d\'efinition demande l'existence d'un bifoncteur $\uHom:\sA^\op\times \sA\to \sA$ et d'isomorphismes 
\[\sA(A\otimes B,C)\simeq \sA(A,\uHom(B,C))\]
naturels en $A,B,C$.

Si $\sA$ est unitaire, appliquons l'identit\'e ci-dessus avec $A=\un$; on trouve:
\[\sA(B,C)\simeq \sA(\un,\uHom(B,C))\]
ce qui justifie la terminologie ``Hom interne''.

\begin{defn}[Dold-Puppe] \phantomsection\label{dDP} Soit $\sA$ une cat\'egorie mono\"\i dale sym\'etrique unitaire. Un objet $A\in \sA$ est \emph{fortement dualisable} s'il existe un triplet $(A^*,\eta,\epsilon)$, avec $A^*\in \sA$ (le \emph{dual} de $A$), et
\[\eta:\un\to A\otimes A^* \text{ (unit\'e)}, \quad \epsilon:A^*\otimes A\to \un  \text{ (co\"unit\'e)}\]
tels que les compositions
\[\begin{CD}
A@>\eta\otimes 1_A>> A\otimes A^*\otimes A @>1_A\otimes \epsilon>> A\\
A^*@> 1_{A^*}\otimes\eta>> A^*\otimes A\otimes A^* @> \epsilon\otimes1_{A^*}>> A^*
\end{CD} \]
soient \'egales \`a l'identit\'e.\\
On dit que $\sA$ est \emph{rigide} si tout objet de $\sA$ est fortement dualisable.
\end{defn}

\begin{prop}\phantomsection\label{pA.4} a) Si $A$ est fortement dualisable, le triplet $(A^*,\eta,\epsilon)$ est unique \`a isomorphisme unique pr\`es. De plus, $A^*$ est fortement dualisable, de dual $A$ et de constantes de structure $\eta^*=\sigma\circ \eta$, $\epsilon^*=\epsilon\circ \sigma$ o\`u $\sigma$ est la contrainte de commutativit\'e.\\
b) Soit $A$ fortement dualisable. Alors pour tous $B,C\in \sA$, on a des isomorphismes naturels
\[\sA(B\otimes A,C)\simeq \sA(B,C\otimes A^*),\quad \sA(B,C\otimes A)\simeq \sA(B\otimes A^*,C).\]
c) Si $A,B$ sont fortement dualisables de duaux $A^*,B^*$, alors $A\otimes B$ est fortement dualisable de dual $A^*\otimes B^*$.
\end{prop}

\begin{proof} a) L'unicit\'e est laiss\'ee au lecteur. La seconde affirmation est obtenue en faisant subir la permutation $(13)$ aux termes centraux des deux compositions.

b) Construisons par exemple le premier isomorphisme; dans un sens:
\[\sA(B\otimes A,C)\by{\otimes1_{A^*}}\sA(B\otimes A\otimes A^*,C\otimes A^*)\by{(1_B\otimes\eta)^*} \sA(B,C\otimes A^*).\]

Dans l'autre:
\[\sA(B,C\otimes A^*)\by{\otimes 1_A}\sA(B\otimes A,C\otimes A^*\otimes A)\by{(1_C\otimes \epsilon)_*}\sA(B\otimes A,C).\]

Les axiomes impliquent que la composition dans les deux sens est \'egale \`a l'identit\'e.

c) est laiss\'e au lecteur.
\end{proof}

\begin{ex}\phantomsection\label{e5.1} En prenant $B=\un$ et $C$ fortement dualisable, on obtient un isomorphisme
\[\sA(A,C)\simeq \sA(\un,C\otimes A^*)\simeq \sA(\un,A^*\otimes C)\simeq \sA(C^*,A^*)\]
appel\'e \emph{transposition}: on le note $f\mapsto {}^tf$. Cette op\'eration est involutive et v\'erifie
\[{}^t(g\circ f) = {}^t f\circ {}^t g,\quad {}^t(f\otimes g) = {}^t f\otimes {}^tg.\]
\end{ex}

\begin{lemme} Soient $\sA$ une cat\'egorie mono\"\i dale sy\-m\'e\-tri\-que unitaire et  $A\in \sA$ un  objet fortement dualisable. Alors pour tout foncteur mono\"\i dal sym\'etrique unitaire $F:\sA\to \sB$ vers une autre cat\'egorie mono\"\i dale sym\'etrique unitaire,  $F(A)$ est fortement dualisable.
\end{lemme}

\begin{proof} Trivial.
\end{proof}

\subsubsection{Cat\'egories avec suspension}

\begin{defn}\phantomsection\label{dsusp} a) Une \emph{cat\'egorie avec suspension} est un couple $(\sA,\Sigma_\sA)$ o\`u $\sA$ est une cat\'egorie et $\Sigma_\sA$ est un endofoncteur de $\sA$. Un \emph{morphisme} de cat\'egories avec suspension $(\sA,\Sigma_\sA)\to (\sB,\Sigma_\sB)$ est un couple $(F,\rho)$ o\`u $F:\sA\to \sB$ est un foncteur et $\eta:F\circ \Sigma_A \Rightarrow \Sigma_\sB\circ F$ est un isomorphisme naturel. Un \emph{$2$-morphisme} entre morphismes parall\`eles $(F,\phi)$, $(G,\psi)$ est une transformation naturelle $u:F\Rightarrow F$ telle que le diagramme
\[\begin{CD}
F \Sigma_\sA(A)@>\phi_A>> \Sigma_\sB F(A)\\
@Vu_{\Sigma_\sA(A)}VV @V\Sigma_\sB(u_A)VV\\
G \Sigma_\sA(A)@>\psi_A>> \Sigma_\sB G(A)
\end{CD}\]
soit commutatif pour tout $A\in \sA$. \'Etant donn\'e deux cat\'egories avec suspension $(\sA,\Sigma_\sA)$, $(\sB,\Sigma_\sB)$, on note
\[\uHom((\sA,\Sigma_{\sA}),(\sB,\Sigma_\sB))\]
la cat\'egorie dont les objets sont les morphismes de $(\sA,\Sigma_\sA)$ vers $(\sB,\Sigma_\sB)$ et les morphismes sont les $2$-morphismes.
\\
b) Si $(\sA,\Sigma_\sA)$ est une cat\'egorie avec suspension, on dit que $\Sigma_\sA$ est \emph{inversible} si c'est une auto\'equivalence.
\end{defn}

\begin{lemme}\phantomsection\label{lsusp} Soit $(\sA,\Sigma_\sA)$ une cat\'egorie avec suspension. Il exis\-te un morphisme $\rho:(\sA,\Sigma_\sA)\to (\sA[\Sigma_\sA^{-1}],\tilde \Sigma_\sA)$ o\`u $\tilde\Sigma_\sA$ est inversible, $2$-universel au sens suivant: pour toute cat\'egorie avec suspension $(\sB,\Sigma_\sB)$ o\`u $\Sigma_\sB$ est inversible, le foncteur
\[\rho^*:\uHom((\sA[\Sigma_\sA^{-1}],\tilde \Sigma_\sA),(\sB,\Sigma_\sB))\to \uHom((\sA,\Sigma_{\sA}),(\sB,\Sigma_\sB))\]
est une \'equivalence de cat\'egories.\\
On peut choisir pour $\tilde \Sigma_\sA$ un automorphisme de cat\'egories.
\end{lemme}

\begin{proof} Soit $\sA[\Sigma_\sA^{-1}]$ la cat\'egorie dont les objets sont les couples $(A,m)$ pour $A\in \sA$, $m\in \Z$, les morphismes \'etant donn\'es par la formule
\[\sA[\Sigma_\sA^{-1}]((A,m),(B,n)) = \colim_{k\gg 0} \sA(\Sigma_\sA^{m+k}(A),\Sigma_\sA^{n+k}(B))\]
o\`u les morphismes de transition sont induits par $\Sigma_\sA$. La composition des morphismes est d\'efinie de mani\`ere \'evidente. Soit $\tilde\Sigma_{A}$ l'endofoncteur
\[(A,m)\mapsto (A,m)[1]=(A,m+1), \quad f[1] = f\]
pour l'identification canonique
\[\sA[\Sigma_\sA^{-1}]((A,m),(B,n))=\sA[\Sigma_\sA^{-1}]((A,m+1),(B,n+1)).\]

C'est \'evidemment un automorphisme de $\sA[\Sigma_\sA^{-1}]($. On a un foncteur canonique 
\begin{align*}
\rho:\sA&\to \sA[\Sigma_\sA^{-1}]\\
A&\mapsto (A,0)\\
f&\mapsto f
\end{align*}
L'\'egalit\'e $\Sigma_\sA^{m+k}(\Sigma_\sA(A))=\Sigma^{m+k+1}(A)$ donne un isomorphisme naturel $(A,1)\iso (\Sigma_\sA A,0)$, soit
\[\eta:[1]\circ \rho\iso \rho\circ \Sigma_\sA.\]

\'Etant donn\'e deux morphismes $(F,\phi), (G,\psi):(\sA[\Sigma_\sA^{-1}],[1 ])\rightrightarrows (\sB,\Sigma_\sB)$,  un $2$-morphisme $u:F\circ \rho\Rightarrow G\circ \rho$ s'\'etend uniquement en un $2$-morphisme $u:F\Rightarrow G$ par la formule
\[u_{(A,m)} = u_A[m].\]

Cela montre que $\rho^*$ est pleinement fid\`ele. D'autre part, si $(F,\phi):(\sA,\Sigma_\sA)\to (\sB,\Sigma_\sB)$ est un morphisme avec $\Sigma_\sB$ inversible, il s'\'etend \`a un $2$-isomorphisme pr\`es en un morphisme $(\tilde F,\tilde \phi):(\sA[\Sigma_\sA^{-1}],[1])\to (\sB,\Sigma_\sB)$ par la formule
\[F(A,m) = 
\begin{cases}
\Sigma_\sB^{m} F(A)&\text{si $m\ge 0$}\\
{\Sigma'_\sB}^{-m} F(A)&\text{si $m< 0$}
\end{cases}\]
o\`u $\Sigma'_\sB$ est un quasi-inverse de $\Sigma_\sB$. Ce qui montre que $\rho^*$ est essentiellement surjectif.
\end{proof}

\subsubsection{Objets inversibles; localisations multiplicatives} Soit $\sA$ une cat\'egorie mono\"\i dale sym\'etrique unitaire. 

\begin{lemme}\phantomsection\label{lA.3} Pour un objet $L\in \sA$, les conditions suivantes sont \'equivalentes:
\begin{thlist}
\item L'endofoncteur $t_L:A\mapsto A\otimes L$ est une \'equivalence de cat\'egories.
\item Il existe $L'\in \sA$ et un isomorphisme $\un\iso L'\otimes L$.
\end{thlist}
\end{lemme}

\begin{proof} (ii) est le cas particulier de (i): ``$\un$ est dans l'image essentielle de $t_L$''. Inversement, si (i) est vrai, un quasi-inverse de $t_L$ est donn\'e par $t_{L'}$.
\end{proof}

\begin{defn}\phantomsection\label{dA.7} Un objet $L$ de $\sA$ est \emph{inversible} s'il v\'erifie les conditions \'equivalentes du lemme \ref{lA.3}, \emph{quasi-inversible} si le foncteur $t_L$ de ce lemme est pleinement fid\`ele.
\end{defn}

\begin{lemme}\phantomsection\label{lA.5} a) Pour tout objet $L$ de $\sA$  on a un homomorphisme canonique
\[\End_\sA(\un)\to \End_\sA(L).\]
b) Si $L$ est quasi-inversible, c'est un isomorphisme; en particulier, le mono\"\i de $\End_\sA(L)$ est commutatif.\\
c) Si $L$ est quasi-inversible,, alors pour tout $n\ge 2$ l'action de $\fS_n$ sur $L^{\otimes n}$ (lemme \ref{lsym}) se factorise \`a travers la signature.
\end{lemme}

\begin{proof} a) L'homomorphisme est induit par $t_L$, modulo la contrainte d'unit\'e.

b) C'est \'evident.

c) r\'esulte de b) appliqu\'e \`a $L^{\otimes n}$.
\end{proof}

\begin{defn}\phantomsection\label{dA.8} Soit $L\in \sA$. On note $\sA[L^{-1}]$ la cat\'egorie $\sA[t_L^{-1}]$, o\`u $(\sA,t_L)$ est consid\'er\'ee comme une cat\'egorie avec suspension (\cf d\'efinition \ref{dsusp} et lemme \ref{lsusp}).
\end{defn}

Par construction, on a un foncteur canonique $\sA\to \sA[L^{-1}]$. Il n'est pas vrai en g\'en\'eral que la structure mono\"\i dale sym\'etrique de $\sA$ s'\'etende \`a $\sA[L^{-1}]$: par le lemme \ref{lA.5} c), une condition n\'ecessaire pour cela est que la permutation $(123)$ op\`ere trivialement sur l'image de $L^{\otimes 3}$ dans $\sA[L^{-1}]$. En fait:

\begin{prop}[Voevodsky \protect{\cite{voe-icm}}] La condition ci-dessus est suffisante.
\end{prop}

Pour une d\'emonstration \'el\'egante, je renvoie \`a Riou \cite[rem. 5.4 et th. 5.5]{rioudea}.\footnote{Cette probl\'ematique n'appara\^\i t pas explicitement dans Saavedra, mais semble implicite dans les calculs de \cite[ch. I, 2.6]{saa}.}

En particulier:

\begin{thm}\phantomsection\label{tA.3} Si $L$ est quasi-inversible, le foncteur $\sA\to \sA[L^{-1}]$ est pleinement fid\`ele et la structure mono\"\i dale sym\'etrique de $\sA$ s'\'etend \`a  $\sA[L^{-1}]$ le long de ce foncteur.
\end{thm}

\begin{proof} La premi\`ere affirmation est \'evidente, et la seconde r\'esulte de la proposition pr\'ec\'edente.
\end{proof}

\subsubsection{Cat\'egories mono\"\i dales et enveloppes karoubiennes}

\begin{prop}\phantomsection\label{pA.2} Soit $\sA$ une cat\'egorie. Toute structure mono\"\i dale sur $\sA$ s'\'etend canoniquement \`a son enveloppe karoubienne. De m\^eme pour une structure mono\"\i dale unitaire, mono\"\i dale sym\'etrique unitaire. Si $\sA$ est ferm\'ee, $\sA^\natural$ est ferm\'ee.
\end{prop}

\begin{proof} Par la propri\'et\'e universelle de l'enveloppe karoubienne, le bifoncteur
\[\sA\times \sA\by{\otimes} \sA\by{i}\sA^\natural\]
s'\'etend canoniquement en un foncteur de $(\sA\times \sA)^\natural\simeq \sA^\natural\times \sA^\natural$. De m\^eme les contraintes d'associativit\'e, etc., s'\'etendent.
\end{proof}

\begin{prop}\phantomsection\label{pA.3} Supposons $\sA$ mono\"\i dale, sym\'etrique, rigide et karoubienne. Soit $A\in \sA$, et soit $p=p^2$ un projecteur de $A$. Supposons $A$ fortement dualisable; alors $\IM p$ est fortement dualisable, de dual $\IM {}^tp$ (voir exemple \ref{e5.1}).
\end{prop}

\begin{proof} Exercice.
\end{proof}

\subsubsection{La trace d'un endomorphisme}

\begin{defn}\phantomsection\label{dA.3} Soient $\sA$ une cat\'egorie mono\"\i dale sym\'etrique unitaire, $A\in \sA$ un objet fortement dualisable et $f$ un endomorphisme de $A$. On lui associe un \'el\'ement
\[\Tr(f)\in \End_\sA(\un),\]
d\'efini comme la composition
\[\un\by{\eta} A\otimes A^*\by{f\otimes1} A\otimes A^*\by{\sigma} A^*\otimes A\by{\epsilon}\un.\]
C'est la \emph{trace} de $f$. Si on veut pr\'eciser, on note $\Tr_\sA(f)$.\\
On note $\chi(A)=\Tr(1_A)$: c'est la \emph{caract\'eristique d'Euler-Poincar\'e} de $A$.
\end{defn}

\begin{prop} Soient $A,B$ deux objets fortement dualisables de $\sA$. La trace a les propri\'et\'es suivantes:
\begin{gather*}
\Tr(fg)=\Tr(gf)\quad \text{si } f\in \sA(A,B), g\in \sB(B,A);\\
\Tr(f\otimes g) = \Tr(f)\Tr(g) \quad \text{si } f\in \End_\sA(A), g\in \End_\sA(B).
\end{gather*}
De plus, si $f\in \in \sA(A,B), g\in \sB(B,A)$, on a la formule suivante; notant $\iota_{AB}$ l'isomorphisme canonique $ \sA(\un,A^*\otimes B)\iso\sA(A,B)$:
\[\Tr(g\circ f) ={}^t (\iota_{AB}^{-1}(f))\circ \iota_{BA}^{-1}(g).\]
\end{prop}

\begin{proof} Laiss\'ee au lecteur en exercice.
\end{proof}

\begin{lemme}[Formule des traces]\label{lA.6} Soient $(\sA,A,f)$ comme \allowbreak dans la d\'efinition \ref{dA.3}, et soit $T:\sA\to \sB$ un foncteur mo\-no\-\"\i \-dal sym\'etrique unitaire. Alors
\[T(\Tr_\sA(f))=\Tr_\sB(T(f)).\]
\end{lemme}

\begin{proof} Trivial.
\end{proof}

\subsubsection{Cas additif: la fonction z\^eta d'un endomorphisme}\index{Fonction zêta!d'un endomorphisme}

\begin{defn}\phantomsection\label{dA.5} Soient $\sA$ une cat\'egorie et $F$ un anneau commutatif unitaire. Une structure mono\"\i dale et une structure $F$-lin\'eaire sur $\sA$ sont \emph{compatibles} si, pour deux morphismes $f,g$ et $\lambda\in F$, on a
\[\lambda(f\otimes g) = (\lambda f)\otimes g =f\lambda\otimes g=f\otimes \lambda g= f\otimes (g\lambda)=(f\otimes g)\lambda.\]
\end{defn}

Dans ce cas, $\End_\sA(\un)$ est une $F$-alg\`ebre et, si $A\in \sA$ est fortement dualisable, la trace $\Tr_\sA:\End_\sA(A)\to \End_\sA(\un)$ est $F$-lin\'eaire.

\begin{exs} 1) $F$ est un corps commutatif, $\sA=\Vec_F$ (espaces vectoriels non gradués) muni du produit tensoriel des espaces vectoriels. Alors $V\in \sA$ est fortement dualisable si et seulement s'il est de dimension finie; si $f\in \End(V)$, $\Tr_\sA(f)$ est la trace habituelle. Pour le voir ais\'ement, on peut se ramener (par fonctorialit\'e!) \`a $F$ alg\'ebriquement clos, mettre $f$ sous forme triangulaire, et enfin se ramener \`a $\dim V=1$ auquel cas l'\'enonc\'e est trivial. On peut aussi faire une d\'emonstration directe.\\
2) M\^eme exemple, mais $\sA=\Vec_F^*$, la contrainte de commutativit\'e \'etant donn\'ee par la r\`egle de Koszul: si $V,W\in \sA$ sont homog\`enes de degr\'es $i,j$, alors pour $(v,w)\in V\times W$
\[\sigma(v\otimes w) = (-1)^{ij} w\otimes v.\]
Si $V\in \Vec_F^*$ est de dimension finie et  $f\in \End(V)$, on trouve alors
\[\Tr_\sA(f)=\sum_{i\in\Z} (-1)^i \Tr(f\mid V^i).\]
\end{exs}

\begin{defn}\phantomsection\label{dA.9} Supposons que $F=\End_\sA(\un)$ soit une $\Q$-al\-g\`e\-bre. Soit $f$ un endomorphisme d'un objet fortement dualisable $A$. On d\'efinit sa \emph{fonction z\^eta}:
\[Z(f,t) = \exp\left(\sum_{n=1}^\infty \Tr_\sA(f^n) \frac{t^n}{n} \right)\in F[[t]].\]
\end{defn}

La th\'eor\`eme \ref{t3.6} et sa d\'emonstration prennent alors la forme abstraite suivante:

\begin{thm}\phantomsection\label{tA.4} Supposons que $F$ soit un corps et qu'il existe un foncteur mono\"\i dal sym\'etrique
\[H^*:\sA\to \Vec_K^*\]
o\`u $K$ est une extension de $F$. Alors, pour tout objet fortement dualisable $A\in \sA$ et tout endomorphisme $f$ de $A$, on a
\[Z(f,t)\in F(t).\]
Si $f$ est un isomorphisme, on a une \'equation fonctionnelle\index{Equation fonctionnelle@\'Equation fonctionnelle} de la forme
\[Z({}^tf^{-1},t^{-1}) = \det(f)(-t)^{\chi(A)}Z(f,t)\]
o\`u ${}^tf\in \End(A^*)$ est le transpos\'e de $f$ et $\det(f)$ est un scalaire; si $Z(f,t)=\displaystyle\frac{\prod_{i=1}^m (1-\alpha_i t)}{\prod_{j=1}^n (1-\beta_j t)}$ avec $\alpha_i,\beta_j\in \bar F$, on a
\[\det(f) = \frac{\prod_{i=1}^m \alpha_i}{\prod_{j=1}^n \beta_j}.\]
\end{thm}

\begin{proof} On peut supposer $\sA=\Vec_K^*$ et $H^*=Id$. Alors $f$ est homogène de degré $0$, ce qui permet de se ramener au cas où $A$ est homogène, disons de degré $i$. Sans perte de généralité, on peut de plus supposer $K$ algébriquement clos; mettant $f$ sous forme triangulaire supérieure, on se ramène alors au cas où $\dim A=1$ et où $f$ est la multiplication par un scalaire $\alpha$. On trouve alors
\[Z(f,t) =
\begin{cases}
\displaystyle \frac{1}{1-\alpha t} &\text{si $i$ est pair}\\
1-\alpha t &\text{si $i$ est impair.}
\end{cases}
\]

L'équation fonctionnelle, avec la valeur de $\det(f)$, s'en déduit en observant que $\chi(A)=(-1)^i$.
\end{proof}

On trouvera dans \cite[th. 3.2 et rem. 3.3]{kahn-zeta} des conditions plus faibles entra\^\i nant les conclusions de ce th\'eor\`eme.

\section{Cat\'egories triangul\'ees, cat\'egories d\'eriv\'ees et complexes parfaits}

\subsection{Localisation}  R\'ef\'erence:  Gabriel-Zisman \cite{gz}.

\subsubsection{Cat\'egories de fractions} Si $A$ est un anneau commutatif et $S$ est une partie de $A$, on sait d\'efinir le localis\'e de $A$, $S^{-1}A$, par rapport \`a $S$ \cite{bbki-ac}: il est muni d'un homomorphisme $A\to S^{-1} A$ qui est universel pour les homomorphismes rendant les \'el\'ements de $S$ inversibles. 

Si $A$ est int\`egre, la construction de $S^{-1} A$ est facile: c'est le sous-anneau du corps de fractions de $A$ engendr\'e par $A$ et $S^{-1}$. En g\'en\'eral, il faut \^etre plus soigneux avec sa construction \cite[\S 2, no 1, d\'ef. 2]{bbki-ac} .

Quand $A$ n'est pas commutatif (mais tout de m\^eme unitaire), le probl\`eme universel indiqu\'e ci-dessus a toujours une solution, m\^eme s'il n'est pas facile de trouver une r\'ef\'erence dans la litt\'erature \`a part le cas o\`u $S$ v\'erifie la ``condition de Ore''. 

Modulo des probl\`emes ensemblistes, cette situation se g\'en\'eralise\footnote{Voir exemple \ref{exA.1}.} \`a une cat\'egorie quelconque comme suit:

\begin{defn}\phantomsection Soit $\sC$ une cat\'egorie, et soit $S$ un ensemble de morphismes de $\sC$. Un foncteur $F:\sC\to \sD$ \emph{inverse $S$} si, pour tout $s\in S$, $F(s)$ est un isomorphisme de $\sD$.\\ 
Une \emph{localisation} de $\sC$ par rapport \`a $S$ est un foncteur $F_u:\sC\to \sC'$ v\'erifiant la propri\'et\'e $2$-universelle suivante: pour toute cat\'egorie $\sD$, le foncteur
\[\Fonct(\sC',\sD)\by{F_u^*} \Fonct(\sC,\sD)\]
est pleinement fid\`ele, d'image essentielle les foncteurs inversant $S$.
\end{defn}

\begin{thm}[\protect{\cite[ch. I, 1.1 et 1.2]{gz}}]\phantomsection Si $\sC$ est petite, elle admet une localisation. Plus pr\'ecis\'ement, il existe une cat\'egorie $S^{-1}\sC$ ayant les m\^emes objets que $\sC$ et un foncteur $P_S:\sC\to S^{-1}\sC$ qui est l'identit\'e sur les objets, v\'erifiant la propri\'et\'e $2$-universelle ci-dessus. \end{thm}

(Une cat\'egorie $\sC$ est \emph{petite} si la classe $Ob(\sC)$ de ses objets est un ensemble et, pour tout $(c,d)\in Ob(\sC)^2$, $\sC(c,d)$ est un ensemble. Dans le langage des univers de Grothendieck, on se donne un univers $\sU$ et on dit qu'un ensemble est $\sU$-petit s'il appartient \`a $\sU$.)

\begin{ex}\phantomsection\label{exA.1} Soit $M$ un mono\"\i de. On lui associe une cat\'egorie \`a un seul objet, not\'ee $BM$. Si $S$ est un sous-ensemble de $M$, la cat\'egorie $S^{-1}BM$ est de la forme $B(S^{-1}M)$ o\`u $S^{-1} M$ est le \emph{mono\"\i de des fractions de $M$ par rapport \`a $S$}. Si $M$ est un anneau, on retrouve l'anneau des fractions de $M$.
\end{ex}

\begin{thm}[\protect{\cite[ch. I, 1.3]{gz}}] \phantomsection Soit $G:\sC\to \sD$ un foncteur, d'adjoint \`a droite $D:\sD\to \sC$. Notons $\epsilon: GD\to Id_{\sD}$ la co\"unit\'e de cette adjonction. Les conditions suivantes sont \'equivalentes:
\begin{thlist}
\item $\epsilon$ est un isomorphisme de foncteurs.
\item $D$ est pleinement fid\`ele.
\item $G$ est une localisation relative \`a $S=\{s\in Fl(\sC)\mid G \text{ inverse } s\}$.
\item Pour toute cat\'egorie $\sE$, le foncteur
\[\Fonct(\sD,\sE)\by{G^*} \Fonct(\sC,\sE)\]
est pleinement fid\`ele.
\end{thlist}
\end{thm}

\subsubsection{Calcul des fractions} Lorsque $M$ est un mono\"\i de non commutatif, le mono\"\i de $S^{-1}M$ est en g\'en\'eral tr\`es difficile \`a d\'ecrire quand $S\subset M$ est quelconque: un \'el\'ement g\'en\'eral de $S^{-1}M$ est de la forme $s_1^{-1}m_1s_2^{-1}m_2\dots s_r^{-1} m_r$, o\`u $s_i\in S$, $m_i\in M$, l'entier $r$ \'etant \emph{a priori} non born\'e et les relations entre ces g\'en\'erateurs n'\'etant pas explicites. La situation est meilleure quand on a un \emph{calcul des fractions}: c'est la condition de Ore lorsque $M$ est ``sans torsion''. Cette situation se g\'en\'eralise au cas d'une cat\'egorie quelconque et donne lieu \`a:

\begin{defn}[\protect{\cite[ch. I, 2.2]{gz}}]\phantomsection Soit $\sC$ une cat\'egorie et soit $S\subset Fl(\sC)$ une classe de morphismes. On dit que $S$ \emph{admet un calcul des fractions \`a gauche} si les conditions suivantes sont v\'erifi\'ees:
\begin{description}
\item[Multiplicativit\'e]  $S$ contient les identit\'es des objets de $\sC$ et est stable par composition.
\item[Condition de Ore] Pour tout diagramme $X'\yb{s} X\by{u} Y$ o\`u $s\in S$, il existe un carr\'e commutatif
\[\begin{CD}
X@>u>> Y\\
@V{s}VV @V{s'}VV\\
X'@>u'>> Y'
\end{CD}\]
o\`u $s'\in S$.
\item[Simplification] Si $X\begin{smallmatrix}f\\ \rightrightarrows\\g \end{smallmatrix} Y$ sont deux morphismes parall\`eles et $s\in S$ est tel que $fs=gs$, il existe $t\in S$ tel que $tf=tg$.
\end{description}
D\'efinition duale pour le calcul des fractions \`a droite.
\end{defn}

Si $S$ admet un calcul des fractions \`a gauche, la condition de Ore implique que toute fl\`eche de $S^{-1} \sC$ peut s'\'ecrire $s^{-1}f$, o\`u $f\in Fl(\sC)$ et $s\in S$. Elle fournit aussi une r\`egle de composition de deux telles fl\`eches; enfin, la condition de simplification dit quand deux couples $(f,s)$ et $(f',s')$ d\'efinissent la m\^eme fl\`eche $s^{-1}f$: voir \cite[ch. I, 2.3]{gz} pour plus de pr\'ecisions.

\begin{ex}\phantomsection Soit $\sA$ une cat\'egorie ab\'elienne. Une sous-cat\'egorie $\sB\subset \sA$ est dite \emph{de Serre} si elle est pleine et si, \'etant donn\'e une suite exacte courte $0\to A'\to A\to A''\to 0$, $A\in \sB$ $\iff$ $A',A''\in \sB$. On associe \`a $\sB$ la classe de morphismes
\[S_\sB=\{s\mid \Ker s, \Coker s\in \sB\}.\]
Alors $S_\sB$ admet un calcul de fractions \`a gauche et \`a droite, et la cat\'egorie $\sA/\sB:=S_\sB^{-1}\sA$ est ab\'elienne: c'est le \emph{quotient de Serre} de $\sA$ par $\sB$. Ainsi, on peut dire que le premier exemple de localisation de cat\'egories apparaissant dans la litt\'erature est celui de Serre dans \cite{serre-classes} (o\`u $\sA=$ $\{$groupes ab\'eliens$\}$ et $\sB=$ par exemple $\{$groupes ab\'eliens finis de torsion $p$-primaire$\}$). Voir aussi \cite[nos 56--60]{FAC}.\\
Un autre exemple, ci-dessous, sera celui d'une sous-cat\'egorie \'epaisse d'une cat\'egorie triangul\'ee.
\end{ex}

\begin{thm}[\protect{\cite[ch. I, 4.1, prop.]{gz}}]\phantomsection\label{tA.1} Soit $\sC$ une petite cat\'egorie, et soit $S\subset Fl(\sC)$ un ensemble de fl\`eches admettant un calcul des fractions \`a gauche. Alors le foncteur de localisation
\[P_S:\sC\to S^{-1}\sC\]
admet un adjoint \`a droite si et seulement si la condition suivante est v\'erifi\'ee: pour tout objet $c\in \sC$, il existe un morphisme $s:c\to d$ tel que
\begin{thlist}
\item $P_S(s)$ soit inversible;
\item pour tout $t:x\to y$ de $S$, l'application $t^*:\sC(y,d)\to \sC(x,d)$ soit bijective.
\end{thlist}
Si cette condition est v\'erifi\'ee, $c\mapsto d$ d\'efinit l'adjoint \`a droite cherch\'e et la fl\`eche $s$ correspond \`a l'unit\'e de l'adjonction.
\end{thm}

\subsection{Cat\'egories triangul\'ees et d\'eriv\'ees}
R\'ef\'erences: la th\`ese de Verdier \cite{verdier-th}, son \'etat $0$ \cite[C.D.]{SGA412}, Residues and Duality \cite[ch. 1]{hartshorne-res}\dots

\subsubsection{Le c\^one d'un morphisme de complexes} Soit $\sA$ une cat\'egorie additive: notons $C(\sA)$ la cat\'egorie des complexes de cocha\^\i nes
\[\dots\to C^{n-1}\by{d^{n-1}} C^n\by{d^n}C^{n+1}\to \dots\]
 de $\sA$.

\begin{defn}\phantomsection\label{dA.1} Soit $f:C\to D$ un morphisme de $C(\sA)$. Le \emph{c\^one} de $f$ est le complexe $E$ o\`u
\[E^n=D^n\oplus C^{n+1}\]
et la diff\'erentielle $d_E^n: E^n\to E^{n+1}$ est donn\'ee par la matrice
\[d_E^n=\begin{pmatrix}d^n_D& f^{n+1}\\ 0 & -d_C^{n+1} \end{pmatrix}.\]
On a des morphismes de complexes
\[D\by{g} E\by{h} C[1]\]
o\`u $(C[1])^n=C^{n+1}$, $d_{C[1]}^n=-d_C^{n+1}$: 
\[g(d)=(d,0);\quad  h(d,c) = c.\]
\end{defn}

On trouvera dans \cite[ch. I, prop. 3.1.3]{verdier-th} une description de cette construction comme partie d'une adjonction.

\subsubsection{Homotopies}

\begin{defn}\phantomsection Soient $C\begin{smallmatrix}f\\ \rightrightarrows\\g \end{smallmatrix}D$ deux morphismes parall\`eles de $C(\sA)$. Une \emph{homotopie de $f$ \`a $g$} est un homomorphisme gradu\'e de degr\'e $-1$:
\[H^n:C^n\to D^{n-1}\]
tel que $g-f=d_DH+Hd_C$.\\
S'il existe une homotopie entre $f$ et $g$, on dit qu'ils sont \emph{homotopes}.
\end{defn}

\begin{lemme}\phantomsection\label{lA.1} La relation d'homotopie est une relation d'\'equivalence compatible \`a l'addition et \`a la composition des morphismes.\qed 
\end{lemme}

\begin{lemme}\phantomsection Soit $f:C\to D$ un morphisme de $C(\sA)$, de c\^one $E$, et soient $g,h$ comme dans la d\'efinition \ref{dA.1}. Alors $h\circ g=0$ et $g\circ f$, $f[1]\circ h$ sont homotopes \`a $0$.
\end{lemme}

\begin{proof} Une homotopie $H_1$ de $0$ \`a $g\circ f$ est donn\'ee par
\[H_1(c)= (0,c)\]
et une homotopie $H_2$ de $0$ \`a $f[1]\circ h$ est donn\'ee par
\[H_2(d,c)=d.\]
\end{proof}

\begin{defn}\phantomsection\label{dA.2} Soit $K(\sA)$ la cat\'egorie dont les objets sont ceux de $C(\sA)$, les morphismes \'etant donn\'es par
\[K(\sA)(C,D) = C(\sA)(C,D)/\sim_h\]
o\`u $\sim_h$ est la relation d'homotopie (\cf lemme \ref{lA.1}). C'est la \emph{cat\'egorie homotopique} de $\sA$. Si $f:C\to D$ est un morphisme de $C(\sA)$, on lui associe le triangle
\[t(f): C\by{f} D\by{g} E\by{h}C[1]\]
de la d\'efinition \ref{dA.1}, pouss\'e dans $K(\sA)$.
\end{defn}

\subsubsection{Cat\'egories triangul\'ees}

\begin{defn}\phantomsection a) Une \emph{$\Z$-cat\'egorie} est une cat\'egorie $\sT$ munie d'une auto\'equivalence de cat\'egories $X\mapsto X[1]$; si c'est un automorphisme de cat\'egories, on dit \emph{$\Z$-cat\'egorie stricte}.  On appelle cette auto\'equivalence \emph{d\'ecalage}, \emph{suspension} ou \emph{translation}.\\
b) Un \emph{triangle} de $\sT$ est un diagramme
\[X\to Y\to Z\to X[1]\]
souvent abr\'eg\'e en
\[X\to Y\to Z\by{+1}.\]
Un \emph{morphisme de triangles} est la notion \'evidente.
\end{defn}

Une autre notation courante dans la litt\'erature, qui explique la terminologie, est
\[\xymatrix{
&X\ar[dr]\\
Z\ar[ur]^{+1} && Y.\ar[ll]
}\]

\begin{defn}[Verdier \protect{\cite[ch. II, d\'ef. 1.1.1]{verdier-th}}]\phantomsection Une \emph{cat\'egorie triangul\'ee} est la donn\'ee d'une $\Z$-cat\'egorie additive stricte $\sT$ et d'une classe de triangles, appel\'es \emph{triangles exacts}, poss\'edant les propri\'et\'es suivantes:
\begin{description}
\item[TR1] Tout triangle isomorphe \`a un triangle exact est exact. Pour tout objet $X$ de $\sT$, le triangle $X\by{1_X} X\to 0\to X[1]$ est exact. Tout morphisme $u:X\to Y$ de $\sT$ est contenu dans un triangle exact $X\by{u} Y\by{v} Z\by{w} X[1]$.
\item[TR2] Un triangle  $X\by{u} Y\by{v} Z\by{w} X[1]$ est exact si et seulement si le triangle $Y\by{v} Z\by{w} X[1]\by{-u[1]} Y[1]$ est exact.
\item[TR3] Tout diagramme commutatif
\[\begin{CD}
X@>u>> Y@>v>> Z @>w>> X[1]\\
@V{f}VV @V{g}VV\\
X'@>u'>> Y'@>v'>> Z' @>w'>> X'[1]
\end{CD}\]
o\`u les lignes sont exactes, peut se compl\'eter en un diagramme commutatif
\[\begin{CD}
X@>u>> Y@>v>> Z @>w>> X[1]\\
@V{f}VV @V{g}VV@V{h}VV@V{f[1]}VV\\
X'@>u'>> Y'@>v'>> Z' @>w'>> X'[1].
\end{CD}\]
\item[TR4] (``axiome de l'octa\`edre'') Pour tout triangle commutatif
\[\xymatrix{
&X_2\ar[dr]^{u_1}\\
X_1\ar[ur]^{u_3}\ar[rr]^{u_2} && X_3
}\]
et tout triplet de triangles exacts
\[\begin{CD}
X_1@>u_2>> X_2@>v_3>> Z_3@>w_3>> X_1[1]\\
X_2@>u_1>> X_3@>v_1>> Z_1@>w_1>> X_2[1]\\
X_1@>u_3>> X_3@>v_2>> Z_2@>w_2>> X_1[1]
\end{CD}\]
il existe deux morphismes $m_1:Z_3\to Z_2$, $m_3:Z_2\to Z_1$ tels que $(1_{X_1},u_1,m_1)$ et $(u_3,1_{X_3}m_3)$ soient des morphismes de triangles et que le triangle
\[Z_3\by{m_1}Z_2\by{m_1}Z_1\by{v_3[1]w_1} Z_3[1]\]
soit exact.
\end{description}
\end{defn}

Voir \cite[1.1.7--1.1.14]{bbd} pour des commentaires sur l'axiome de l'octa\`edre.

\begin{thm}[\protect{\cite[ch. II, prop. 1.3.2]{verdier-th}}]\phantomsection Soit $\sA$ une cat\'egorie additive. La cat\'egorie $K(\sA)$ de la d\'efinition \ref{dA.2}, munie des triangles $t(f)$, est triangul\'ee.
\end{thm}

\begin{defn}\phantomsection a) Soient $\sT,\sT'$ deux cat\'egories triangul\'ees. Un \emph{foncteur triangul\'e} de $\sT$ vers $\sT'$ est un foncteur additif $F:\sT\to \sT'$, muni d'isomorphismes naturels de commutation aux foncteurs de d\'ecalage, et transformant triangles exacts en triangles exacts.\\
b) Soient $\sT$ une cat\'egorie triangul\'ee et $\sA$ une cat\'egorie ab\'elienne. Un \emph{foncteur homologique de $\sT$ vers $\sA$} est un foncteur $H:\sT\to \sA$ tel que, pour tout triangle exact $X\to Y\to Z\by{+1}$ de $\sT$, la suite de $\sA$
\[HX\to HY\to HZ\to H(X[1])\]
soit exacte. (Il est automatiquement additif: \cite[ch. II, rem. 1.2.7]{verdier-th}).\\
Un \emph{foncteur cohomologique de $\sT$ vers $\sA$} est un foncteur homologique de $\sT$ vers $\sA^\op$. 
\end{defn}

\begin{defn}\phantomsection Soient $\sT,\sT'$ deux cat\'egories triangul\'ees, et \allowbreak $F_1,F_2:\sT\rightrightarrows\sT'$ deux foncteurs triangul\'es parall\`eles. Une transformation naturelle $u:F_1\Rightarrow F_2$ est \emph{triangul\'ee} si elle commute aux d\'ecalages au sens suivant: pour tout $X\in \sT$, le diagramme
\[\begin{CD}
F_1(X[1])@>\sim>> F_1(X)[1]\\
@V{u_{X[1]}}VV @V{u_X[1]}VV\\
F_2(X[1])@>\sim>> F_2(X)[1]
\end{CD}\]
commute, o\`u les isomorphismes horizontaux sont les isomorphismes de commutation aux d\'ecalages inclus dans la structure de $F_1$ et $F_2$.
\end{defn}

\begin{thm}[\protect{\cite[ch. II, prop. 1.2.1]{verdier-th}}]\phantomsection Soit $\sT$ une cat\'egorie triangul\'ee, et soit $\Ab$ la cat\'egorie des groupes ab\'eliens. Pour tout objet $X\in \sT$, le foncteur $\sT(X,-):\sT\to \Ab$ est homologique et le foncteur $\sT(-,X):\sT\to \Ab^\op$ est cohomologique.
\end{thm}

\begin{rque}\phantomsection a) Dans \cite{verdier-th}, Verdier utilise la terminologie ``triangle distingu\'e'', ``foncteur exact'' et ``foncteur cohomologique'' respectivement pour (ici) ``triangle exact'', ``foncteur triangul\'e'' et ``foncteur homologique''. La terminologie choisie ici s'est progressivement impos\'ee: ``triangle exact'' refl\`ete plus fid\`element le fait qu'un tel triangle g\'en\`ere des suites exactes longues, et ``foncteur exact'' est particuli\`erement malheureux puisque les foncteurs triangul\'es usuels sont des d\'eriv\'es totaux de foncteurs non exacts entre cat\'egories ab\'eliennes! De m\^eme, il para\^\i t pr\'ef\'erable de r\'eserver ``homologique'' aux foncteurs covariants et ``cohomologique'' aux foncteurs contravariants.\\
b) L'hypoth\`ese que le foncteur de d\'ecalage soit un automorphisme de $\sT$ exclut des exemples de cat\'egories triangul\'ees la cat\'egorie homotopique stable des topologues: le foncteur suspension n'y est qu'une auto-\'equivalence de cat\'egories. C'est n\'eanmoins une diff\'erence inessentielle. J'ai pr\'ef\'er\'e garder l'hypothèse de Verdier pour \'eviter d'\'eventuels probl\`emes impr\'evus.
\end{rque}

\subsubsection{$K_0$ de cat\'egories ab\'eliennes et triangul\'ees}\label{K0} Si $\sA$ est une (petite) cat\'egorie ab\'elienne, son \emph{groupe de Grothendieck} $K_0(\sA)$ est le groupe ab\'elien d\'efini par g\'en\'erateurs et relations:
\begin{description}
\item[G\'en\'erateurs] $[A]$, o\`u $A$ d\'ecrit les classes d'isomorphismes d'objets de $\sA$.
\item[Relations] $[A]=[A']+[A'']$ pour toute suite exacte courte $0\to A'\to A\to A''\to 0$.
\end{description}

Si $\sT$ est une (petite) cat\'egorie triangul\'ee, on d\'efinit $K_0(\sT)$ de la m\^eme mani\`ere, les relations provenant des triangles exacts.

Si $D^b(\sA)$ est la cat\'egorie d\'eriv\'ee born\'ee d'une cat\'egorie ab\'elienne $\sA$, le plongement $A\mapsto A[0]$ de $\sA$ dans $D^b(\sA)$ induit un \emph{isomorphisme}
\[K_0(\sA)\iso K_0(D^b(\sA))\]
\cite[exp. VIII, \S 4]{SGA5}.

\subsubsection{Le quotient de Verdier} Dans ce num\'ero, je passe hypocritement sur les probl\`emes ensemblistes.

\begin{defn}\phantomsection a) Soit $\sT$ une cat\'egorie triangul\'ee. Une sous-ca\-t\'e\-go\-rie $\sS\subset \sT$ est 
\emph{triangul\'ee}  si elle est pleine, stable par d\'ecalage, et si pour tout morphisme $f:X\to Y$ de $\sS$, il existe un triangle exact $X\by{f}Y\by{g} Z\by{+1}$ de $\sT$ avec $Z\in \sS$. Le foncteur d'inclusion est alors triangul\'e.\\
b) $\sS$ est \emph{\'epaisse}\footnote{Verdier disait: \emph{satur\'ee}.} si, de plus, elle est stable par facteurs directs.
\end{defn}

Soit $\sS\subset \sT$ une sous-cat\'egorie triangul\'ee. On lui associe la classe de morphismes
\begin{multline*}S(\sS)= \{s:X\to Y\mid X,Y\in \sT; \text{ il existe un triangle exact }\\
 X\by{s} Y\to Z\by{+1} \text{ avec } Z\in \sS\}.
 \end{multline*}

\begin{thm}[\protect{\cite[ch. II, prop. 2.1.8 et th. 2.2.6]{verdier-th}}]\phantomsection La classe $S(\sS)$ admet un calcul des fractions \`a gauche et \`a droite. Notons $\sT/\sS$ la cat\'egorie localis\'ee: elle poss\`ede une unique structure de cat\'egorie triangul\'ee telle que le foncteur de localisation $P:\sT\to \sT/\sS$ soit triangul\'e; c'est celle dont les triangles exacts sont les images des triangles exacts de $\sT$. La cat\'egorie $\sT/\sS$ est le \emph{quotient de Verdier} de $\sS$ par $\sT$.
\end{thm}

\begin{rque}\phantomsection La cat\'egorie $\sS^\natural = \Ker P$ est la plus petite sous-cat\'egorie \'epaisse de $\sT$ qui contient $\sS$: ses objets sont les facteurs directs des objets de $\sS$. Pour cette raison, on se limite parfois \`a localiser par rapport \`a une sous-cat\'egorie \'epaisse (ce que ne fait pas Verdier).
\end{rque}

\begin{defn}\phantomsection Soit $\sC$ une classe d'objets de $\sT$. L'\emph{orthogonal \`a droite} (\resp \emph{\`a gauche}) de $\sC$ est
\[\sC^\perp = \{X\in \sT\mid \sT(\sC,X)=0\}\]
(\resp
\[{}^\perp\sC = \{X\in \sT\mid \sT(X,\sC)=0\}).\]
Si $\sC$ est stable par translation, ce sont des sous-cat\'egories \'epaisses de $\sT$.
\end{defn}

La th\'eor\`eme suivant retraduit le th\'eor\`eme \ref{tA.1} dans le contexte triangul\'e:

\begin{thm}[\protect{\cite[ch. II, prop. 2.3.3]{verdier-th}}]\phantomsection Soit $\sS$ une sous-ca\-t\'e\-go\-rie triangul\'ee de $\sT$. Le foncteur de localisation $P:\sT\to \sT/\sS$ admet un adjoint \`a droite (\resp \`a gauche) si et seulement si sa restriction \`a $\sS^\perp$ (\resp \`a ${}^\perp \sS$) est essentiellement surjective.
\end{thm}

\subsubsection{Cat\'egories d\'eriv\'ees}\label{catder}

\begin{defn}\phantomsection a) Soit $\sA$ une cat\'egorie ab\'elienne, et soient $C,D\in C(\sA)$. Un morphisme $f:C\to D$ est un \emph{quasi-isomorphisme} si $H^*(f):H^*(C)\to H^*(D)$ est un isomorphisme. Un objet $C$ est \emph{acyclique} si $H^*(C)=0$, \ie s'il est quasi-isomorphe \`a $0$.\\
b) On note $D(\sA)$ le quotient de Verdier $K(\sA)/\sS$, o\`u $\sS$ est la sous-cat\'egorie \'epaisse des complexes acycliques. C'est la \emph{cat\'egorie d\'eriv\'ee} de $\sA$.
\end{defn}

\begin{rque}\phantomsection On a des variantes ``born\'ees'' de $K(\sA)$ et $D(\sA)$:
\begin{align*}
K^+(\sA) &= \{C\in K(\sA)\mid C^n=0 \text{ pour } n\ll 0\}\\
K^-(\sA) &= \{C\in K(\sA)\mid C^n=0 \text{ pour } n\gg 0\}\\
K^b(\sA) &= K^+(\sA)\cap K^-(\sA)
\end{align*}
de m\^eme pour $D^\#(\sA)$ ($\#\in \{+,-,b\}$).
\end{rque}

\begin{prop}\phantomsection\label{pA.1} Soit $K_P(\sA)$ (\resp $K_I(\sA)$ la sous-cat\'egorie pleine de $K(\sA)$ form\'ee des complexes dont les termes sont projectifs (\resp injectifs): ce sont des sous-cat\'egories \'epaisses de $K(\sA)$.\\
a) Les compositions
\[K_P(\sA)\to K(\sA)\to D(\sA),\quad K_I(\sA)\to K(\sA)\to D(\sA)\]
sont pleinement fid\`eles.\\
b) Si tout objet de $\sA$ est un quotient d'un objet projectif (\resp se plonge dans un objet injectif), l'image essentielle de la premi\`ere (\resp seconde) composition contient $D^-(\sA)$ (\resp $D^+(\sA)$).
\end{prop}

\begin{rque}\phantomsection La proposition ci-dessus est la version d\'eriv\'ee des r\'esolutions injectives et projectives, qui permet de faire des calculs en alg\`ebre homologique. Dans les ann\'ees 60 et 70, on \'etait tr\`es attentif aux questions de ``bornitude'' et aux pathologies qui risquaient de se produire si on les supprimait. Cette situation a \'et\'e r\'evolutionn\'ee dans les ann\'ees 1990 par Spaltenstein \cite{spaltenstein}, puis B\"okstedt-Neeman \cite{bn} qui ont introduit des m\'ethodes permettant de travailler de la m\^eme mani\`ere avec des cat\'egories d\'eriv\'ees non born\'ees. Une \'etape fondamentale a \'et\'e l'article de Thomason-Trobaugh identifiant les complexes parfaits aux objets compacts de la cat\'egorie d\'eriv\'ee des faisceaux quasi-coh\'erents sur un sch\'ema raisonnable \cite{tt}. L'importance des cat\'egories d\'eriv\'ees non born\'ees a alors cr\^u avec les travaux ult\'erieurs de Neeman \cite{neeman}: par exemple, si $\sA$ est une cat\'egorie ab\'elienne admettant des sommes directes infinies exactes, alors $D(\sA)$ est stable par sommes directes infinies ce qui n'est le cas ni de $D^+(\sA)$, ni de $D^-(\sA)$, et encore moins de $D^b(\sA)$.
\end{rque}

\subsubsection{Foncteurs d\'eriv\'es}

Soit $F:\sA\to \sB$ un foncteur exact \`a gauche entre cat\'egories ab\'eliennes. Si $\sA$ a assez d'injectifs au sens de la proposition \ref{pA.1} b), on d\'efinit les \emph{foncteurs d\'eriv\'es \`a droite} de $F$ par la formule
\[R^qF(A) = H^q(I^*)\]
o\`u $I^*$ est une r\'esolution injective de $A$. L'id\'ee qui sous-tend le concept des cat\'egories d\'eriv\'ees est de conserver l'information contenue dans $I^*$ mais perdue en passant \`a sa cohomologie. Cela conduit \`a la notion de foncteur d\'eriv\'e total.

\begin{defn}\phantomsection $F:\sA\to \sB$ \'etant comme ci-dessus, on dit que $F$ est \emph{d\'erivable \`a droite} s'il existe un foncteur triangul\'e $RF:D(\sA)\to D(\sB)$ et une transformation naturelle
\[\phi_A:F(A)[0]\to RF(A)\]
(o\`u $B[0]\in D(\sB)$ d\'esigne l'objet $B$ de $\sB$, consid\'er\'e comme complexe concentr\'e en degr\'e $0$), ayant la propri\'et\'e universelle suivante: pour tout foncteur triangul\'e $G:D(\sA)\to D(\sB)$ muni d'une transformation naturelle $\psi:F[0]\Rightarrow G$, il existe une unique transformation naturelle $\theta:RF\Rightarrow G$ tel que $\psi = \theta \circ \phi$.\\
{\bf Variante}: on remplace $D$ par $D^+$.
\end{defn}

Cette notion est un cas particulier de celle d'extension de Kan \cite[X.3]{mcl}.

\begin{prop}\phantomsection Si $\sA$ a assez d'injectifs, $RF$ est d\'efini sur $D^+(\sA)$; il s'\'etend \`a $D(\sA)$ [au moins] si $R^qF = 0$ pour $q\gg 0$. (On dit alors que $F$ est \emph{de dimension cohomologique finie}.)
\end{prop}

\begin{thm}\phantomsection Soient $\sA\by{F}\sB\by{G}\sC$ des foncteurs exacts \`a gauche entre cat\'egories ab\'eliennes. Supposons que $\sA$ et $\sB$ aient assez d'injectifs et que $F$ transforme tout objet injectif de $\sA$ en un objet $G$-acyclique de $\sB$ (\ie un objet $B$ tel que $R^qG(B)=0$ pour tout $q>0$). Alors on a un isomorphisme canonique
\[R(G\circ F) \simeq RG\circ RF\]
sur $D^+(\sA)$ (et sur $D(\sA)$ sous des hypoth\`eses de dimension cohomologique finie).
\end{thm}

Ce th\'eor\`eme remplace et g\'en\'eralise la suite spectrale des foncteurs compos\'es:
\[R^pGR^qF(A)\Rightarrow R^{p+q}(G\circ F)(A).\footnote{Au moins pour des raisonnement formels! Il est difficile d'\'eviter les suites spectrales pour faire des calculs concrets\dots}\]

\subsubsection{D\'eriv\'es et adjoints} Ce qui est dit ci-dessus pour la d\'erivation \`a droite s'applique \`a la d\'erivation \`a gauche, \emph{mutatis mutandis}. Que se passe-t-il dans le cas de foncteurs adjoints?

Voici un exemple de r\'eponse:

\begin{prop}\phantomsection Soient $\sA,\sB$ deux cat\'egories ab\'eliennes, et $G:\sA\to \sB$, $D:\sB\to \sA$ un couple de foncteurs adjoints ($G$ est adjoint \`a gauche de $D$): donc $G$ est exact \`a droite et $D$ est exact \`a gauche. Supposons que $\sB$ ait assez d'injectifs et que $\sA$ ait assez de projectifs. Alors les foncteurs d\'eriv\'es totaux $LG:D^-(\sA)\to D^-(\sB)$, $RD:D^+(sB)\to D^+(\sA)$ existent et leurs restrictions \`a $D^b(\sA)$ et $D^b(\sB)$ sont adjointes l'une de l'autre.
\end{prop}

Pour une g\'en\'eralisation dans un cadre tr\`es abstrait, on peut consulter \cite{km}.

\subsubsection{Exemples: $\oo\limits^L$ et $\RHom$}

\begin{defn}\phantomsection Soit $\sA$ une cat\'egorie ab\'elienne munie d'une structure mono\"\i dale sym\'etrique $\otimes$, qu'on suppose \emph{exacte \`a droite}.\\
a) Un objet $A\in \sA$ est \emph{plat} si le foncteur $A\otimes-$ est exact.\\
b) Un \emph{Hom interne} est un adjoint \`a droite $\uHom$ de $\otimes$ (s'il existe). Un objet  $A\in \sA$ est \emph{coplat} si $\uHom(-,A)$ est exact.
\end{defn}

\begin{prop}\phantomsection Soit $(\sA,\otimes)$ comme ci-dessus, et supposons que tout objet de $\sA$ soit quotient d'un objet plat (on dit que $\sA$ a suffisamment d'objets plats). Alors le foncteur
\[\otimes:\sA\times \sA\to \sA\]
se d\'erive en
\[\oo^L:D^-(\sA)\times D^-(\sA)\to D^-(\sA).\]
Supposons que $\sA$ soit  ferm\'ee et admette suffisamment d'objets coplats. Alors $\uHom$ se d\'erive en
\[\RHom:D^-(\sA)\times D^+(\sA)\to D^+(\sA).\]
De plus, les restrictions de $\oo\limits^L$ et $\RHom$ \`a $D^b(\sA)\times D^b(\sA)$ sont adjointes l'une de l'autre.
\end{prop}

En pratique, on prend pour $\sA$ la cat\'egorie des faisceaux de $\Lambda$-Modules sur un site $X$, o\`u $\Lambda$ est un faisceau d'anneaux sur $X$: alors $\sA$ a suffisamment d'objets plats, et de plus tout objet injectif est coplat. La proposition ci-dessus s'applique donc.

\subsection{Complexes parfaits}

\begin{defn}\phantomsection Soit $\Lambda$ un anneau unitaire, et soit $\sA$ la cat\'egorie ab\'elienne des $\Lambda$-modules \`a gauche. Un complexe $C\in C(\sA)$ est \emph{parfait} s'il est isomorphe dans $D(\sA)$ \`a un complexe born\'e de $\Lambda$-modules projectifs de type fini.
\end{defn}

\begin{defn}\phantomsection\label{dparf} Soit $(S,\Lambda)$ un site localement annel\'e, et soit $\sA$ la cat\'egorie ab\'elienne des $\Lambda$-Modules \`a gauche.\\
a) Un complexe $C\in C(\sA)$ est \emph{strictement parfait} si c'est un complexe born\'e de $\Lambda$-Modules localement libres de type fini.\\
b) Un objet $C\in D(\sA)$ est \emph{parfait} s'il existe un recouvrement $(U_i)$ de $S$ tel que, pour tout $i$, $C_{|U_i}$ soit isomorphe \`a un complexe strictement parfait.
\end{defn}

Le théorème suivant est initialement dû à Thomason-Trobaugh dans un cas particulier \cite[th. 2.4.3]{tt}; l'énoncé général suivant a été dégagé par Amnon Neeman \cite{neeman-dual}.

\begin{thm}\phantomsection Supposons que $(S,\Lambda)=(X,\sO_X)$ pour un schéma $X$ quasi-compact et séparé. Alors un objet de $D(\sA)$ est un complexe parfait si et seulement s'il est compact.
\end{thm}

J'ignore si cet énoncé reste vrai dans le cas des complexes parfaits de \cite[Exp. XV]{SGA5} intervenant au \S \ref{s5.2.3}; il existe en tout cas des exemples de champs algébriques quasi-compacts et quasi-séparés sur lesquels certains complexes parfaits ne sont pas compacts, cf. \cite[\S 4]{hall-rydh} (je remercie Neeman pour cette référence). 

\newpage

\printindex

\end{document}